\documentclass[9pt]{amsart}
\usepackage{amsmath}
\usepackage{amsxtra}
\usepackage{amscd}
\usepackage{amsthm}
\usepackage{amsfonts}
\usepackage{amssymb}
\usepackage{eucal}
\usepackage[all]{xy}
\usepackage{graphicx}
\usepackage[usenames]{color}

\usepackage{hyperref}
\usepackage[normalem]{ulem}

\newtheorem{cor}[subsubsection]{Corollary}
\newtheorem{lem}[subsubsection]{Lemma}
\newtheorem{prop}[subsubsection]{Proposition}
\newtheorem{propconstr}[subsubsection]{Proposition-Construction}
\newtheorem{thmconstr}[subsubsection]{Theorem-Construction}

\newtheorem{conj}[subsubsection]{Conjecture}
\newtheorem{thm}[subsubsection]{Theorem}

\newtheorem{defn}[subsubsection]{Definition}

\theoremstyle{remark}
\newtheorem{rem}[subsubsection]{Remark}

\theoremstyle{definition}

\theoremstyle{remark}

\newcommand{\thmref}[1]{Theorem~\ref{#1}}

\newcommand{\secref}[1]{Sect.~\ref{#1}}
\newcommand{\lemref}[1]{Lemma~\ref{#1}}
\newcommand{\propref}[1]{Proposition~\ref{#1}}
\newcommand{\corref}[1]{Corollary~\ref{#1}}
\newcommand{\conjref}[1]{Conjecture~\ref{#1}}
\newcommand{\remref}[1]{Remark~\ref{#1}}

\numberwithin{equation}{section}

\newcommand{\nc}{\newcommand}
\nc{\renc}{\renewcommand}
\nc{\ssec}{\subsection}
\nc{\sssec}{\subsubsection}
\nc{\on}{\operatorname}

\nc{\ips}{{\iota_P^{(S)}}}
\nc{\ipms}{{\iota_{P^-}^{(S)}}}
\nc{\sfpps}{{\sfp_P^{(S)}}}
\nc{\sfppms}{{\sfp_{P^-}^{(S)}}}

\nc\ol{\overline}
\nc\ul{\underline}
\nc\wt{\widetilde}
\nc\tboxtimes{\wt{\boxtimes}}
\nc\tstar{\wt{\star}}
\nc{\alp}{\alpha}

\nc{\ZZ}{{\mathbb Z}}
\nc{\NN}{{\mathbb N}}
\nc{\OO}{{\mathbb O}}
\renc{\SS}{{\mathbb S}}
\nc{\DD}{{\mathbb D}}
\nc{\GG}{{\mathbb G}}

\nc{\Fq}{{\mathbb F}_q}
\nc{\Fqb}{\ol{\mathbb F}_q}
\nc{\Ql}{{\mathbb Q}_\ell}
\nc{\Qlb}{{\ol{\mathbb Q}_\ell}}
\nc{\id}{\text{id}}
\nc\X{\mathcal X}

\nc{\red}{\on{red}}
\nc{\Ho}{\on{Ho}}
\nc{\Hom}{\on{Hom}}
\nc{\coef}{\on{coef}}
\nc{\Lie}{\on{Lie}}
\nc{\Loc}{\on{Loc}}
\nc{\Pic}{\on{Pic}}
\nc{\Bun}{\on{Bun}}
\nc{\IC}{\on{IC}}
\nc{\Aut}{\on{Aut}}
\nc{\rk}{\on{rk}}
\nc{\Sh}{\on{Sh}}
\nc{\Perv}{\on{Perv}}
\nc{\pos}{{\on{pos}}}
\nc{\Conv}{\on{Conv}}
\nc{\Sph}{\on{Sph}}
\nc{\Sym}{\on{Sym}}
\nc{\BunBb}{\overline{\Bun}_B}
\nc{\BunNb}{\overline{\Bun}_N}
\nc{\BunTb}{\overline{\Bun}_T}
\nc{\BunBbm}{\overline{\Bun}_{B^-}}
\nc{\BunBbel}{\overline{\Bun}_{B,el}}
\nc{\BunBbmel}{\overline{\Bun}_{B^-,el}}
\nc{\Buno}{\overset{o}{\Bun}}
\nc{\BunPb}{{\overline{\Bun}_P}}
\nc{\BunBM}{\Bun_{B(M)}}
\nc{\BunBMb}{\overline{\Bun}_{B(M)}}
\nc{\BunPbw}{{\widetilde{\Bun}_P}}
\nc{\BunBP}{\widetilde{\Bun}_{B,P}}
\nc{\GUb}{\overline{G/U}}
\nc{\GUPb}{\overline{G/U(P)}}

\nc{\Hhom}{\underline{\on{Hom}}}
\nc\syminfty{\on{Sym}^{\infty}}
\nc\lal{\ol{\lambda}}
\nc\xl{\ol{x}}
\nc\thl{\ol{\theta}}
\nc\nul{\ol{\nu}}
\nc\mul{\ol{\mu}}
\nc\Sum\Sigma
\nc{\oX}{\overset{o}{X}{}}
\nc{\hl}{\overset{\leftarrow}h{}}
\nc{\hr}{\overset{\rightarrow}h{}}
\nc{\M}{{\mathcal M}}
\nc{\N}{{\mathcal N}}
\nc{\F}{{\mathcal F}}
\nc{\D}{{\mathcal D}}
\nc{\Q}{{\mathcal Q}}
\nc{\Y}{{\mathcal Y}}
\nc{\G}{{\mathcal G}}
\nc{\E}{{\mathcal E}}
\nc{\CalC}{{\mathcal C}}
\nc\Dh{\widehat{\D}}

\nc{\C}{{\mathcal C}}
\nc{\K}{{\mathcal K}}
\renewcommand{\H}{{\mathcal H}}

\nc{\T}{{\mathcal T}}
\nc{\V}{{\mathcal V}}
\renc{\P}{{\mathcal P}}
\nc{\A}{{\mathcal A}}
\nc{\B}{{\mathcal B}}
\nc{\U}{{\mathcal U}}

\nc{\Gr}{{\on{Gr}}}

\nc{\frn}{{\check{\mathfrak u}(P)}}

\nc{\fC}{\mathfrak C}
\nc{\fT}{\mathfrak T}
\nc{\p}{\mathfrak p}
\nc{\q}{\mathfrak q}
\nc\f{{\mathfrak f}}

\nc{\qo}{{\mathfrak q}}
\nc{\po}{{\mathfrak p}}
\nc{\s}{{\mathfrak s}}
\nc\w{\text{w}}

\renewcommand{\mod}{{\on{-mod}}}

\nc\mathi\iota
\nc\Spec{\on{Spec}}
\nc\Proj{\on{Proj}}
\nc\Mod{\on{Mod}}
\nc{\tw}{\widetilde{\mathfrak t}}
\nc{\pw}{\widetilde{\mathfrak p}}
\nc{\qw}{\widetilde{\mathfrak q}}
\nc{\jw}{\widetilde j}

\nc{\grb}{\overline{\Gr}}
\nc{\I}{\mathcal I}

\nc{\lambdach}{{\check\lambda}}
\nc{\Lambdach}{{\check\Lambda}{}}
\nc{\much}{{\check\mu}}
\nc{\omegach}{{\check\omega}}
\nc{\nuch}{{\check\nu}}
\nc{\etach}{{\check\eta}}
\nc{\alphach}{{\check\alpha}}
\nc{\oblvtach}{{\check\oblvta}}
\nc{\rhoch}{{\check\rho}}
\nc{\ch}{{\check h}}

\nc{\Hb}{\overline{\H}}

\nc{\BA}{{\mathbb{A}}}
\nc{\BC}{{\mathbb{C}}}
\nc{\BE}{{\mathbb{E}}}
\nc{\BF}{{\mathbb{F}}}
\nc{\BG}{{\mathbb{G}}}
\nc{\BM}{{\mathbb{M}}}
\nc{\BO}{{\mathbb{O}}}
\nc{\BD}{{\mathbb{D}}}
\nc{\BN}{{\mathbb{N}}}
\nc{\BP}{{\mathbb{P}}}
\nc{\BQ}{{\mathbb{Q}}}
\nc{\BR}{{\mathbb{R}}}
\nc{\BZ}{{\mathbb{Z}}}
\nc{\BS}{{\mathbb{S}}}
\nc{\Deep}{{\bf{deep}}}
\nc{\deep}{deep}

\nc{\CA}{{\mathcal{A}}}
\nc{\CB}{{\mathcal{B}}}

\nc{\CE}{{\mathcal{E}}}
\nc{\CF}{{\mathcal{F}}}
\nc{\CH}{{\mathcal{H}}}

\nc{\CL}{{\mathcal{L}}}
\nc{\CC}{{\mathcal{C}}}
\nc{\CG}{{\mathcal{G}}}
\nc{\CalD}{{\mathcal{D}}}
\nc{\CM}{{\mathcal{M}}}
\nc{\CN}{{\mathcal{N}}}
\nc{\CK}{{\mathcal{K}}}
\nc{\CO}{{\mathcal{O}}}
\nc{\CP}{{\mathcal{P}}}
\nc{\CQ}{{\mathcal{Q}}}
\nc{\CR}{{\mathcal{R}}}
\nc{\CS}{{\mathcal{S}}}
\nc{\CT}{{\mathcal{T}}}
\nc{\CU}{{\mathcal{U}}}
\nc{\CV}{{\mathcal{V}}}
\nc{\CW}{{\mathcal{W}}}
\nc{\CX}{{\mathcal{X}}}
\nc{\CY}{{\mathcal{Y}}}
\nc{\CZ}{{\mathcal{Z}}}
\nc{\CI}{{\mathcal{I}}}

\nc{\csM}{{\check{\mathcal A}}{}}
\nc{\oM}{{\overset{\circ}{\mathcal M}}{}}
\nc{\obM}{{\overset{\circ}{\mathbf M}}{}}
\nc{\oCA}{{\overset{\circ}{\mathcal A}}{}}
\nc{\obA}{{\overset{\circ}{\mathbf A}}{}}
\nc{\ooM}{{\overset{\circ}{M}}{}}
\nc{\osM}{{\overset{\circ}{\mathsf M}}{}}
\nc{\vM}{{\overset{\bullet}{\mathcal M}}{}}
\nc{\nM}{{\underset{\bullet}{\mathcal M}}{}}
\nc{\oD}{{\overset{\circ}{\mathcal D}}{}}
\nc{\obD}{{\overset{\circ}{\mathbf D}}{}}
\nc{\oA}{{\overset{\circ}{\mathbb A}}{}}
\nc{\op}{{\overset{\bullet}{\mathbf p}}{}}
\nc{\cp}{{\overset{\circ}{\mathbf p}}{}}
\nc{\oU}{{\overset{\bullet}{\mathcal U}}{}}
\nc{\oZ}{{\overset{\circ}{\mathcal Z}}{}}
\nc{\ofZ}{{\overset{\circ}{\mathfrak Z}}{}}
\nc{\oF}{{\overset{\circ}{\fF}}}

\nc{\fa}{{\mathfrak{a}}}
\nc{\fb}{{\mathfrak{b}}}
\nc{\fd}{{\mathfrak{d}}}
\nc{\ff}{{\mathfrak{f}}}
\nc{\fg}{{\mathfrak{g}}}
\nc{\fgl}{{\mathfrak{gl}}}
\nc{\fh}{{\mathfrak{h}}}
\nc{\fj}{{\mathfrak{j}}}
\nc{\fl}{{\mathfrak{l}}}
\nc{\fm}{{\mathfrak{m}}}
\nc{\fn}{{\mathfrak{n}}}
\nc{\fu}{{\mathfrak{u}}}
\nc{\fp}{{\mathfrak{p}}}
\nc{\fr}{{\mathfrak{r}}}
\nc{\fs}{{\mathfrak{s}}}
\nc{\ft}{{\mathfrak{t}}}
\nc{\fz}{{\mathfrak{z}}}
\nc{\fsl}{{\mathfrak{sl}}}
\nc{\hsl}{{\widehat{\mathfrak{sl}}}}
\nc{\hgl}{{\widehat{\mathfrak{gl}}}}
\nc{\hg}{{\widehat{\mathfrak{g}}}}
\nc{\chg}{{\widehat{\mathfrak{g}}}{}^\vee}
\nc{\hn}{{\widehat{\mathfrak{n}}}}
\nc{\chn}{{\widehat{\mathfrak{n}}}{}^\vee}

\nc{\fA}{{\mathfrak{A}}}
\nc{\fB}{{\mathfrak{B}}}
\nc{\fD}{{\mathfrak{D}}}
\nc{\fE}{{\mathfrak{E}}}
\nc{\fF}{{\mathfrak{F}}}
\nc{\fG}{{\mathfrak{G}}}
\nc{\fK}{{\mathfrak{K}}}
\nc{\fL}{{\mathfrak{L}}}
\nc{\fM}{{\mathfrak{M}}}
\nc{\fN}{{\mathfrak{N}}}
\nc{\fP}{{\mathfrak{P}}}
\nc{\fU}{{\mathfrak{U}}}
\nc{\fV}{{\mathfrak{V}}}
\nc{\fZ}{{\mathfrak{Z}}}

\nc{\ba}{{\mathbf{a}}}
\nc{\bb}{{\mathbf{b}}}
\nc{\bc}{{\mathbf{c}}}
\nc{\bd}{{\mathbf{d}}}
\nc{\bbf}{{\mathbf{f}}}
\nc{\be}{{\mathbf{e}}}
\nc{\bi}{{\mathbf{i}}}
\nc{\bj}{{\mathbf{j}}}
\nc{\bm}{{\mathbf{m}}}
\nc{\bn}{{\mathbf{n}}}
\nc{\bo}{{\mathbf{o}}}
\nc{\bp}{{\mathbf{p}}}
\nc{\bq}{{\mathbf{q}}}
\nc{\bu}{{\mathbf{u}}}
\nc{\bv}{{\mathbf{v}}}
\nc{\bx}{{\mathbf{x}}}
\nc{\bs}{{\mathbf{s}}}
\nc{\by}{{\mathbf{y}}}
\nc{\bw}{{\mathbf{w}}}
\nc{\bA}{{\mathbf{A}}}
\nc{\bK}{{\mathbf{K}}}
\nc{\bB}{{\mathbf{B}}}
\nc{\bC}{{\mathbf{C}}}
\nc{\bG}{{\mathbf{G}}}
\nc{\bD}{{\mathbf{D}}}
\nc{\bE}{{\mathbf{E}}}
\nc{\bH}{{\mathbf{H}}}
\nc{\bM}{{\mathbf{M}}}
\nc{\bN}{{\mathbf{N}}}
\nc{\bO}{{\mathbf{O}}}
\nc{\bQ}{{\mathbf{Q}}}
\nc{\bV}{{\mathbf{V}}}
\nc{\bW}{{\mathbf{W}}}
\nc{\bX}{{\mathbf{X}}}
\nc{\bZ}{{\mathbf{Z}}}
\nc{\bS}{{\mathbf{S}}}

\nc{\sA}{{\mathsf{A}}}
\nc{\sB}{{\mathsf{B}}}
\nc{\sC}{{\mathsf{C}}}
\nc{\sD}{{\mathsf{D}}}
\nc{\sF}{{\mathsf{F}}}
\nc{\sG}{{\mathsf{G}}}
\nc{\sH}{{\mathsf{H}}}
\nc{\sK}{{\mathsf{K}}}
\nc{\sM}{{\mathsf{M}}}
\nc{\sO}{{\mathsf{O}}}
\nc{\sW}{{\mathsf{W}}}
\nc{\sQ}{{\mathsf{Q}}}
\nc{\sP}{{\mathsf{P}}}
\nc{\sR}{{\mathsf{R}}}
\nc{\sZ}{{\mathsf{Z}}}
\nc{\sfp}{{\mathsf{p}}}
\nc{\sfq}{{\mathsf{q}}}
\nc{\sr}{{\mathsf{r}}}
\nc{\bk}{{\mathsf{k}}}
\nc{\sg}{{\mathsf{g}}}
\nc{\sff}{{\mathsf{f}}}
\nc{\sfb}{{\mathsf{b}}}
\nc{\sfc}{{\mathsf{c}}}
\nc{\sfe}{{\mathsf{e}}}
\nc{\sd}{{\mathsf{d}}}

\nc{\BK}{{\bar{K}}}

\nc{\tA}{{\widetilde{\mathbf{A}}}}
\nc{\tB}{{\widetilde{\mathcal{B}}}}
\nc{\tg}{{\widetilde{\mathfrak{g}}}}
\nc{\tG}{{\widetilde{G}}}
\nc{\TM}{{\widetilde{\mathbb{M}}}{}}
\nc{\tO}{{\widetilde{\mathsf{O}}}{}}
\nc{\tU}{{\widetilde{\mathfrak{U}}}{}}
\nc{\TZ}{{\tilde{Z}}}
\nc{\tx}{{\tilde{x}}}
\nc{\tbv}{{\tilde{\bv}}}
\nc{\tfP}{{\widetilde{\mathfrak{P}}}{}}
\nc{\tz}{{\tilde{\zeta}}}
\nc{\tmu}{{\tilde{\mu}}}

\nc{\urho}{\underline{\rho}}
\nc{\uB}{\underline{B}}
\nc{\uC}{{\underline{\mathbb{C}}}}
\nc{\ui}{\underline{i}}
\nc{\uj}{\underline{j}}
\nc{\ofP}{{\overline{\mathfrak{P}}}}
\nc{\oB}{{\overline{\mathcal{B}}}}
\nc{\og}{{\overline{\mathfrak{g}}}}
\nc{\oI}{{\overline{I}}}

\nc{\eps}{\varepsilon}
\nc{\hrho}{{\hat{\rho}}}

\nc{\one}{{\mathbf{1}}}
\nc{\two}{{\mathbf{t}}}

\nc{\Rep}{{\mathop{\operatorname{\rm Rep}}}}
\nc{\Tot}{{\mathop{\operatorname{\rm Tot}}}}
\nc{\Ker}{{\mathop{\operatorname{\rm Ker}}}}
\nc{\im}{{\mathop{\operatorname{\rm Im}}}}
\nc{\Hilb}{{\mathop{\operatorname{\rm Hilb}}}}
\nc{\End}{{\mathop{\operatorname{\rm End}}}}
\nc{\Ext}{{\mathop{\operatorname{\rm Ext}}}}
\nc{\CHom}{{\mathop{\operatorname{{\mathcal{H}}\it om}}}}
\nc{\CEnd}{{\mathop{\operatorname{{\mathcal{E}}\it nd}}}}
\nc{\GL}{{\mathop{\operatorname{\rm GL}}}}
\nc{\gr}{{\mathop{\operatorname{\rm gr}}}}
\nc{\HN}{{\mathop{\operatorname{\rm HN}}}}
\nc{\Id}{{\mathop{\operatorname{\rm Id}}}}
\nc{\de}{{\mathop{\operatorname{\rm def}}}}
\nc{\length}{{\mathop{\operatorname{\rm length}}}}
\nc{\supp}{{\mathop{\operatorname{\rm supp}}}}

\nc{\Cliff}{{\mathsf{Cliff}}}
\nc{\Fl}{\on{Fl}}
\nc{\Fib}{{\mathsf{Fib}}}
\nc{\Coh}{{\on{Coh}}}
\nc{\QCoh}{{\on{QCoh}}}
\nc{\IndCoh}{{\on{IndCoh}}}
\nc{\FCoh}{{\mathsf{FCoh}}}

\nc{\reg}{{\text{\rm reg}}}

\nc{\cplus}{{\mathbf{C}_+}}
\nc{\cminus}{{\mathbf{C}_-}}
\nc{\cthree}{{\mathbf{C}_\bullet}}
\nc{\Qbar}{{\bar{Q}}}
\nc\Eis{\on{Eis}}
\nc\Eisb{\ol\Eis{}}
\nc\Eisr{\on{Eis}^{rat}{}}
\nc\wh{\widehat}
\nc{\Def}{\on{Def_{\check{\fb}}(E)}}
\nc{\barZ}{\overline{Z}{}}
\nc{\barbarZ}{\overline{\barZ}{}}
\nc{\barpi}{\overline\pi}
\nc{\barbarpi}{\overline\barpi}
\nc{\barpip}{\overline\pi{}^+}
\nc{\barpim}{\overline\pi{}^-}

\nc{\fq}{\mathfrak q}

\nc{\fqb}{\ol{\sfq}{}}
\nc{\fpb}{\ol{\sfp}{}}
\nc{\fpr}{{\sfp^{rat}}{}}
\nc{\fqr}{{\sfq^{rat}}{}}

\nc{\hattimes}{\wh\otimes}

\nc{\bh}{{\bar{h}}}
\nc{\bOmega}{{\overline{\Omega(\check \fn)}}}

\nc{\seq}[1]{\stackrel{#1}{\sim}}

\nc{\cT}{{\check{T}}}
\nc{\cG}{{\check{G}}}
\nc{\cM}{{\check{M}}}
\nc{\cB}{{\check{B}}}

\nc{\ct}{{\check{\mathfrak t}}}
\nc{\cg}{{\check{\fg}}}
\nc{\cb}{{\check{\fb}}}
\nc{\cn}{{\check{\fn}}}

\nc{\cLambda}{{\check\Lambda}}

\nc{\cla}{{\check\lambda}}
\nc{\cmu}{{\check\mu}}
\nc{\cnu}{{\check\nu}}
\nc{\ceta}{{\check\eta}}

\nc{\DefbE}{{\on{Def}_{\cB}(E_\cT)}}

\nc{\imathb}{{\ol{\imath}}}
\nc{\rlr}{\overset{\longrightarrow}{\underset{\longrightarrow}\longleftarrow}}

\nc{\oBun}{\overset{\circ}\Bun}
\nc{\LocSys}{\on{LocSys}}
\nc{\BunBbb}{\ol{\ol{Bun}}_B}
\nc{\BunBr}{\Bun_B^{rat}}
\nc{\BunBrsg}{\Bun_B^{rat,\on{s.g.}}}
\nc{\BunBrp}{\Bun_B^{rat,polar}}
\nc{\BunBrpbg}{\Bun_B^{rat,polar,\on{b.g.}}}
\nc{\BunBrpsg}{\Bun_B^{rat,polar,\on{s.g.}}}
\nc{\BunTrp}{\Bun_T^{rat,polar}}
\nc{\BunTrpbg}{\Bun_T^{rat,polar,\on{b.g.}}}
\nc{\BunTrpsg}{\Bun_T^{rat,polar,\on{s.g.}}}
\nc{\BunNr}{\Bun_N^{rat}}
\nc{\BunNre}{\Bun_N^{enh,rat}}
\nc{\BunTr}{\Bun_T^{rat}}
\nc{\Vect}{\on{Vect}}
\nc{\Whit}{\on{Whit}}
\nc{\CTb}{\ol{\on{CT}}}
\nc{\Ran}{\on{Ran}}
\nc{\CTr}{\on{CT}^{rat}{}}
\nc\jmathr{\jmath^{rat}{}}
\nc{\ux}{\underline{x}}
\nc{\clambda}{{\check\lambda}}
\nc{\calpha}{{\check\alpha}}
\nc{\ind}{{\mathbf{ind}}}
\nc{\oblv}{{\mathbf{oblv}}}
\nc{\free}{{\mathbf{free}}}
\nc{\ox}{{\overline{x}}}
\nc{\cLa}{\check{\Lambda}}
\nc{\StinftyCat}{\on{DGCat}}
\nc{\inftyCat}{\infty\on{-Cat}}
\nc{\inftygroup}{\infty\on{-Grpd}}
\nc{\Dmod}{\on{D-mod}}
\nc{\CMaps}{{\mathcal Maps}}
\nc{\Maps}{\on{Maps}}
\nc{\affSch}{\on{Sch}^{\on{aff}}}
\nc{\dr}{{\on{dR}}}
\nc{\oCF}{\overset{\circ}\CF}
\nc{\oCY}{\overset{\circ}\CY}
\nc{\opi}{\overset{\circ}\pi}
\nc{\leqG}{\underset{G}\leq}
\nc{\leqM}{\underset{M}\leq}
\nc{\leqGad}{\underset{G_{ad}}\leq}
\nc{\leqMad}{\underset{M_{ad}}\leq}
\nc{\Tr}{\on{Tr}}
\nc{\Frob}{\on{Frob}}
\nc{\DGCat}{\on{DGCat}}
\nc{\tDGCat}{\on{Morita}(\DGCat)}
\nc{\ev}{\on{ev}}
\nc{\mmod}{\on{-}\mathbf{mod}}
\nc{\sotimes}{\overset{!}\otimes}
\nc{\Shv}{\on{Shv}}
\nc{\Spc}{\on{Spc}}
\nc{\LS}{\on{LS}}
\nc{\Res}{\on{Res}}
\nc{\bDelta}{{\mathbf{\Delta}}}
\nc{\bMaps}{{\mathbf{Maps}}}
\nc{\cD}{\mathcal D}
\nc{\ocD}{\overset{\circ}\cD}
\nc{\ppart}{(\!(u)\!)}
\nc{\qqart}{[\![u]\!]}
\nc{\oCU}{\overset{\circ}{\CU}}
\nc{\Ind}{\on{Ind}}
\nc{\coInd}{\on{coInd}}

\begin{document}


\title[A toy model for shtuka]{A toy model for the Drinfeld-Lafforgue shtuka construction}

\author{D.~Gaitsgory, D.~Kazhdan, N.~Rozenblyum and Y.~Varshavsky}

\address
{\newline
D.G.\footnote{Corresponding author. Full address: Department of Mathematics, Harvard University, 1 Oxford str, Cambridge MA 02138, USA}: 
Department of Mathematics, Harvard University; \newline
D.K.: Einstein School of Mathematics, Hebrew University of Jerusalem; \newline
N.R.: Department of Mathematics, The University of Chicago; \newline
Y.V.: Einstein School of Mathematics, Hebrew University of Jerusalem.} 

\date{\today}

\begin{abstract}
The goal of this paper is to provide a categorical framework that leads to the definition of shtukas \`a la Drinfeld and of
excursion operators \`a la V.~Lafforgue.
We take as the point of departure the Hecke action of $\Rep(\cG)$
on the category $\Shv(\Bun_G)$ of sheaves on $\Bun_G$, and also the endofunctor of the latter category, given by
the action of the geometric Frobenius.
The shtuka construction will be obtained by applying (various versions of)
categorical trace.
\end{abstract}

\maketitle

\tableofcontents

\section*{Introduction}

Our goal is to provide
a categorical framework that leads to the definition of shtukas \`a la Drinfeld and of excursion operators \`a la
V.~Lafforgue. We will capture the main ingredients of V.~Lafforgue's construction, which are:

\medskip

\noindent--The action of the algebra of functions on the stack of arithmetic local systems on the space of automorphic functions;

\medskip

\noindent--The ``S=T" identity.

\bigskip

However, all of this will be performed in a toy setting: the key technical(?) difficulty in V.~Lafforgue's work is that the sheaf-theoretic
context he needs is that of $\ell$-adic sheaves on schemes over $\BF_q$. By contrast, we will work in the topological context
in the spirit of \cite{BN1}.

\medskip

The $\ell$-adic context, which gives rise to actual shtukas, will be considered in subsequent work.

\ssec{Hecke action}

The point of view taken in this paper is that the geometric ingredient that gives rise to the Drinfeld-Lafforgue construction
is the \emph{categorical} Hecke action. In this subsection we will specify what we mean by this.

%
%

%

\sssec{}

We will consider the following three geometric/sheaf-theoretic contexts:

\medskip

\noindent--$\ell$-adic sheaves on schemes over any ground field $k$;

\medskip

\noindent--Sheaves (in the classical topology) with coefficients in a commutative ring $\sfe$
on schemes over $\BC$;

\medskip

\noindent--D-modules on schemes over a ground field $k$ of characteristic $0$.

\medskip

For a scheme/stack $\CY$, let $\Shv(\CY)$ denote the corresponding category of sheaves; this is a DG category
over our field of coefficients (i.e., over $\BQ_\ell$, $\sfe$ and $k$, respectively).

\sssec{}

Let $X$ be an algebraic curve and $G$ a reductive group (over our ground field).
Let $\Bun_G$ denote the stack of principal $G$-bundles
on $X$.

\medskip

Let $\cG$ be the Langlands dual group of $G$, which is a reductive group over our ring of coefficients.

\medskip

The point of departure is the Hecke action on $\Shv(\Bun_G)$ of the \emph{symmetric monoidal category} $\Rep(\cG)$
\emph{integrated over} $X$. This is not a completely straightforward notion, and we refer the reader to \secref{sss:action intro}
below or \secref{sss:monoidal actions} for a detailed discussion.

\medskip

Here is what this action gives us.

\sssec{}

In the context of D-modules, (the rather non-trivial) result of V.~Drinfeld and the first-named author (recorded in
\cite[Corollary 4.5.5]{Ga1}) says that this action gives rise to an action
of the category $\QCoh(\LocSys_\cG(X))$ on $\Shv(\Bun_G)$, where $\LocSys_\cG$ is the stack of \emph{de Rham} local systems
on $X$ with respect to $\cG$.

\sssec{}

In the context of sheaves on the classical topology, Theorem 6.3.5 of \cite{NY} says that this action gives rise to an action
of $\QCoh(\LocSys_\cG(X))$ on $\Shv_{\on{Nilp}}(\Bun_G)$, where:

\medskip

\noindent--$\LocSys_\cG(X)$ is the stack of \emph{Betti} local systems on $X$ with respect to $\cG$;

\medskip

\noindent--$\Shv_{\on{Nilp}}(\Bun_G)\subset \Shv(\Bun_G)$ is the full subcategory consisting of sheaves with
\emph{nilpotent singular support}.

\medskip

For the reader's convenience, we will review the construction of this action in \secref{s:NY}.

\medskip

We should remark that unlike the D-module context, the action of the Betti
$\QCoh(\LocSys_\cG(X))$ on $\Shv_{\on{Nilp}}(\Bun_G)$
is obtained relatively easily from what we state in the present paper as \thmref{t:integral to LocSys}
(which in itself is not a difficult assertion either), combined with the key observation of \cite[Theorem 6.1.1]{NY} about the behavior
of singular support.

\sssec{}

Our main interest, however, is when the sheaf-theoretic context is that of $\ell$-adic sheaves. The first conceptual
difficulty in this case is that we do not have a direct analog of $\LocSys_\cG(X)$ as an algebro-geometric object (over $\BQ_l$),
so we cannot talk about an action of $\QCoh(\LocSys_\cG(X))$ on $\Shv(\Bun_G)$.

\medskip

Nevertheless, the situation is not as hopeless as it might seem, and we will discuss it in detail in a subsequent
publication.

\ssec{What is done in this paper?}  \label{ss:outline}

We will now outline the actual mathematical contents of the present paper. Each of the steps we perform
is a toy (more precisely, Betti) analog of what one wishes to be able to do in the context of $\ell$-adic sheaves.

\sssec{}  \label{sss:action intro}

Let $\CA$ be a symmetric monoidal DG category, and let $Y$ be a \emph{space}, i.e., $Y$ is an object of the category
$\Spc$, see \secref{sss:higher categories} (we can also think of $Y$ as a homotopy type).
To this data one can associate a new symmetric monoidal category $\CA^{\otimes Y}$,
see \secref{ss:A power Y}. Sometimes one also uses the notation\footnote{Yet another name for $\CA^{\otimes Y}$ is
``chiral (or factorization) homology of $\CA$ along $Y$". }
$$\underset{Y}\int\, \CA:=\CA^{\otimes Y}.$$
For example, if $Y$ is a finite set $I$, we have
$$\CA^{\otimes Y}=\CA^{\otimes I},$$
i.e., the usual $I$-fold tensor product of copies of $\CA$. For a general $Y$, the construction can be described
by a colimit procedure off the finite set case.

\medskip

For example, we show (see \thmref{t:integral to LocSys}) that for $\CA=\Rep(\sG)$ (here $\sG$ is an algebraic group),
under some conditions\footnote{Specifically, if $Y$ has finitely many connected components and the ring of coefficients $\sfe$ contains $\BQ$.},
we have a canonical equivalence
\begin{equation} \label{e:LocSys intro}
\Rep(\sG)^{\otimes Y}\simeq \QCoh(\LocSys_\sG(Y)),
\end{equation}
where $\LocSys_\sG(Y)$ is the (derived) stack classifying $G$-local systems on $Y$, i.e., its value on a test affine derived scheme $S$
is the space of symmetric monoidal right t-exact functors
$\Rep(\sG)\to \QCoh(S)$, parameterized by $Y$.

\medskip

Our point of departure is a DG category $\CM$, equipped with an action of $\CA^{\otimes Y}$ as a monoidal category.
We emphasize that this is \emph{not} the same as a family of monoidal actions of $\CA$ on $\CM$, parameterized by $Y$,
see Remark \ref{r:mon vs sym}.

\medskip

We give an explicit description of what the datum of such an action amounts to (see \propref{p:describe functors}). Namely, it is equivalent
to a \emph{compatible} family of actions, one for each finite set $I$, of $\CA^{\otimes I}$ on $\CM$, parameterized by points of $Y^I$.
This description is useful, because the Hecke action of $\Rep(\cG)$ on $\Shv(\Bun_G)$ is given in exactly such form, see \secref{ss:Hecke}.

\medskip

Similarly, we show that the datum of a functor from $\CA^{\otimes Y}$ to some DG category $\CC$ is equivalent
to a \emph{compatible} family of functors, one for each finite set $I$,
$$\CS_I:Y^I\times \CA^{\otimes I}\to \CC.$$
This description is useful as it will explain the connection between the universal shtuka and $I$-legged shtukas, see \secref{ss:I-legged}.

\sssec{} \label{sss:excursions intro}

Next, we provide the general framework for excursion operators. Let $\CA$ and $Y$ be as before, and let
$$\CS_Y:\CA^{\otimes Y}\to \CC$$
be a functor of DG categories.

\medskip

Consider the object
$$\CS_Y(\one_{\CA^{\otimes Y}})\in \CC.$$

It carries an action of the (commutative) algebra $\CEnd_{\CA^{\otimes Y}}(\one_{\CA^{\otimes Y}})$.
We give an explicit description of this action, in the spirit of \cite[Sect. 9]{Laf}.

\medskip

First, we give an explicit description of the algebra
$\CEnd_{\CA^{\otimes Y}}(\one_{\CA^{\otimes Y}})$ as a colimit, in the case  when $Y$ is connected, and $\CA$
has \emph{an affine diagonal}, see \corref{c:endo}.
Choose a base point $y\in Y$. The index category in the colimit in question is that of pairs $(I,\gamma^I)$, where
$I$ is a finite set, and $\gamma^I$ is an $I$-tuple of loops in $Y$ based at $y$. (Note that this category is \emph{sifted},
so the colimit in the category of commutative algebras is the same as the colimit of underlying associative algebras
and is also the same as the colimit of the underlying vector spaces.)
The terms of the colimit are described as follows. The term corresponding to a finite set $I$
is given by
\begin{equation} \label{e:I term intro}
\CHom_{\CA}(\one_\CA,\on{mult}_{I_+}\circ \on{mult}_{I_+}^R(\one_\CA)),
\end{equation}
where:

\medskip

\noindent--$I_+=I\sqcup \{*\}$;

\medskip

\noindent--For a finite set $J$, we denote by $\on{mult}_J$ the tensor product map $\CA^{\otimes J}\to \CA$;

\medskip

\noindent--$\on{mult}^R_J$ denotes the right adjoint functor of $\on{mult}_J$.

\medskip

The algebra structure on \eqref{e:I term intro} comes from the right-lax symmetric monoidal structure on $\on{mult}^R_J$,
obtained by adjunction from the symmetric monoidal structure on $\on{mult}_J$ .

\sssec{}  \label{sss:excursions intro bis}

Next, given a functor $\CS_Y:\CA^{\otimes Y}\to \CC$, we show how each term \eqref{e:I term intro} acts on $\CS_Y(\one_{\CA^{\otimes Y}})$,
see \thmref{t:excurs}.

\medskip

Namely, given $\xi\in \CHom_{\CA}(\one_\CA,\on{mult}_{I_+}\circ \on{mult}_{I_+}^R(\one_\CA))$, the corresponding
endomorphism of $\CS_Y(\one_{\CA^{\otimes Y}})$ is the \emph{excursion operator}:
$$
\CD
\CS_Y(\one_{\CA^{\otimes Y}}) @>>{\sim}> \CS_{\{*\}}(y,\one_\CA) @>{\xi}>>
\CS_{\{*\}}(y,\on{mult}_{I_+}\circ (\on{mult}_{I_+})^R(\one_\CA))  \\
& & & & @VV{\sim}V   \\
& & & & \CS_{I_+}(y^{I_+},(\on{mult}_{I_+})^R(\one_\CA)) \\
& & & & @VV{\on{mon}_{\gamma^{I_+}}}V  \\
& & & & \CS_{I_+}(y^{I_+},(\on{mult}_{I_+})^R(\one_\CA))\\
& & & & @VV{\sim}V   \\
\CS_Y(\one_{\CA^{\otimes Y}})  @<{\sim}<<  \CS_{\{*\}}(y,\one_\CA) @<{\on{counit}}<<
\CS_{\{*\}}(y,\on{mult}_{I_+}\circ (\on{mult}_{I_+})^R(\one_\CA)),
\endCD
$$
where

\medskip

\noindent--$\gamma^{I_+}=(\gamma^I,\gamma_{\on{triv}})$;

\medskip

\noindent--For a finite set $J$, and a $J$-tuple $\gamma^J$ of loops in $Y$ based at $y$,
viewed as a loop into $Y^J$ based at $y^J$, we denote by $\on{mon}_{\gamma^J}$ the
corresponding automorphism of $\CS_J(y^J,-)$.

\medskip

In the particular case of $\CA=\Rep(\sG)$, from the colimit expression of \secref{sss:excursions intro}
and \eqref{e:LocSys intro}, we obtain an explicit description of the algebra
$$\Gamma(\LocSys_\sG(Y),\CO_{\LocSys_\sG(Y)})$$
in terms excursion operators. This recovers the analogs of the formulas from \cite[Sect. 10 and Proposition 11.7]{Laf}.

\sssec{}

We now come to the next main ingredient of this paper, namely the notion of \emph{categorical trace}.
First, we recall that given a symmetric monoidal category $\bO$ and a dualizable object $\bo\in \bO$
equipped with an endomorphism $F$, we can assign to this data a point
$$\Tr(F,\bo)\in \End _\bO(\one_\bO),$$
called the trace of $F$ on $\bo$, see \secref{sss:basic trace}.

\medskip

Suppose now that $\bO$ is actually a symmetric monoidal 2-category (i.e., we have not necessarily invertible 2-morphisms).
Let us be given a pair of dualizable objects $\bo_1,\bo_2\in \bO$, each equipped with an endomorphism $F_i$, $i=1,2$.
Assume in addition that we are given a 1-morphism $t:\bo_1\to \bo_2$ that admits a \emph{right adjoint}. Finally, suppose
that $t$ intertwines $F_1$ and $F_2$, \emph{up to a not necessarily invertible 2-morphism} $\alpha$, i.e.,
$$t\circ F_1 \overset{\alpha}\to F_2\circ t.$$

We show that to this data there corresponds a 2-morphism
$$\Tr(t,\alpha):\Tr(F_1,\bo_1)\to \Tr(F_2,\bo_2).$$

In fact, the 2-morphism $\Tr(t,\alpha)$ can be explicitly described as
\begin{equation} \label{e:map between Tr}
\Tr(F_1,\bo_1)\to \Tr(t^R\circ t\circ F_1,\bo_1)\simeq \Tr(t\circ F_1\circ t^R,\bo_2) \to
\Tr(F_2\circ t \circ t^R,\bo_2)\to  \Tr(F_2,\bo_2),
\end{equation}
where:

\medskip

\noindent--$t^R$ denotes the right adjoint of $t$;

\medskip

\noindent--The first arrow is induced by the unit of the adjunction $\on{id}_{\bo_1}\to t^R\circ t$;

\medskip

\noindent--The second arrow is given by the cyclicity property of trace;

\medskip

\noindent--The third arrow is induced by $\alpha$;

\medskip

\noindent--The fourth arrow is induced by the counit of the adjunction $t\circ t^R\to \on{id}_{\bo_2}$.

\sssec{}

We apply the above formalism in the following two main examples: $\bO=\DGCat$ and $\bO=\tDGCat$
(see right below for the definition of the latter).

\medskip

In the case of $\bO=\DGCat$, to a (dualizable) DG category $\CC$ equipped with an endofunctor $F$, we attach
$$\Tr(F,\CC)\in \Vect,$$
and to a pair of such, equipped with a functor $T:\CC_1\to \CC_2$ (that admits a continuous right adjoint)
and a natural transformation
$$\alpha:T\circ F_1\to F_2\circ T,$$
we attach a map in $\Vect$
\begin{equation} \label{e:categ trace intro}
\Tr(T,\alpha):\Tr(F_1,\CC_1)\to \Tr(F_2,\CC_2).
\end{equation}

\sssec{}

The case of $\bO=\tDGCat$ is obviously richer. By definition, the objects of $\bO=\tDGCat$ are 2-DG categories of the
form $\CR\mmod$, where $\CR$ is a monoidal DG category. Further, 1-morphisms in $\tDGCat$ are by definition
given by bi-module categories\footnote{We note that every object of $\tDGCat$ is dualizable: the dual of $\CR\mmod$
is $\CR^{\on{rev}}\mmod$, where $\CR^{\on{rev}}$ is obtained from $\CR$ by reversing the monoidal structure.}.

\medskip

To a 2-DG category $\fC$ equipped with an endofunctor
$\fF$ we now attach a \emph{DG category} $\Tr(\fF,\fC)$. To a pair of such, equipped with
a functor $\fT:\fC_1\to \fC_2$ (that admits a right adjoint in $\tDGCat$)
and a natural transformation
$$\alpha:\fT\circ \fF_1\to \fF_2\circ \fT,$$
we attach a functor between DG categories
\begin{equation} \label{e:2-categ trace intro}
\Tr(\fT,\alpha):\Tr(\fF_1,\fC_1)\to \Tr(\fF_2,\fC_2).
\end{equation}

\medskip

For $\fC=\CR\mmod$ and $\fF$ given by
$$\CQ\in (\CR\otimes \CR^{\on{rev}})\mmod,$$
the resulting DG category $\Tr(\fF,\fC)$ identifies with
$$\on{HH}_\bullet(\CR,\CQ):=\CR\underset{\CR\otimes \CR^{\on{rev}}}\otimes \CQ,$$
this is the \emph{category} of Hochschild chains on $\CR$ with coefficients in $\CQ$.

\sssec{}  \label{sss:two notions intro}

In \thmref{t:two notions of trace} we establish a basic compatibility between the categorical
and the 2-categorical trace constructions:

\medskip

Namely, let $\CR$ be a symmetric monoidal DG category; assume that $\CR$ is rigid.
Let $\CM$ be a module category over $\CR$. We can view $\CM$ as a 1-morphism
\begin{equation} \label{e:M as 1-mor}
\DGCat\to \CR\mmod.
\end{equation}
Assume that $\CM$ is dualizable as a DG category. This condition is equivalent requiring that
the above 1-morphism admit a right adjoint.

\medskip

Let $F_\CR$ be a symmetric monoidal endofunctor of $\CR$. Let $F_\CM$ be an endofunctor of $\CM$,
endowed with a \emph{datum of compatibility with $F_\CR$}; i.e., we have a datum of commutativity
of the diagram
$$
\CD
\CR \otimes \CM  @>>>  \CM  \\
@V{F_\CR\otimes F_\CM}VV  @VV{F_\CM}V   \\
\CR \otimes \CM  @>>>  \CM,
\endCD
$$
along with higher compatibilities. We can view this compatibility structure as a data of ``$\alpha$" for the
1-morphism \eqref{e:M as 1-mor}. Hence, by \eqref{e:2-categ trace intro}, to this data there corresponds
a map in $\DGCat$
$$\Vect\to \Tr(F_\CR,\CR\mmod).$$

In other words, we obtain an object
$$\on{cl}(\CM,F_\CM)\in \Tr(F_\CR,\CR\mmod)\simeq  \on{HH}_\bullet(\CR,F_\CR).$$

Note that since $\CR$ was \emph{symmetric} monoidal, and $F_\CR$ was also \emph{symmetric} monoidal,
the category $\on{HH}_\bullet(\CR,F_\CR)$ acquires a symmetric monoidal structure. Let
$\one_{\on{HH}_\bullet(\CR,F_\CR)}$ denote the unit object in $\on{HH}_\bullet(\CR,F_\CR)$.

\medskip

Our \thmref{t:two notions of trace} says that there exists a canonical isomorphism
\begin{equation} \label{e:two notions of trace trace intro}
\CHom_{\on{HH}_\bullet(\CR,F_\CR)}(\one_{\on{HH}_\bullet(\CR,F_\CR)},\on{cl}(\CM,F_\CM))\simeq
\Tr(F_\CM,\CM).
\end{equation}

Isomorphism \eqref{e:two notions of trace trace intro} shows that $\on{cl}(\CM,F_\CM)$ lifts
$\Tr(F_\CM,\CM)$ to an object of $\on{HH}_\bullet(\CR,F_\CR)$; this justifies the notation:
$$\Tr^{\on{enh}}_\CR(F_\CM,\CM):=\on{cl}(\CM,F_\CM).$$

So \thmref{t:two notions of trace} says that
$$\CHom_{\on{HH}_\bullet(\CR,F_\CR)}(\one_{\on{HH}_\bullet(\CR,F_\CR)},\Tr^{\on{enh}}_\CR(F_\CM,\CM))\simeq
\Tr(F_\CM,\CM).$$

\sssec{}

We now come to the central construction in this paper, which
is the prototype of the shtuka construction.

\medskip

Let us be given a symmetric monoidal category $\CA$ and a space $Y$. Let now $\phi$ be an endomorphism of $Y$,
which induces a symmetric monoidal endofunctor, denoted $\CA^{\otimes \phi}$, of $\CA^{\otimes Y}$. Let $Y/\phi$ be the
(homotopy) quotient of $Y$ by $\phi$.

\medskip

First, we note that there exists a canonical equivalence
\begin{equation} \label{e:Hochsch intro}
\on{HH}_\bullet(\CA^{\otimes Y},\CA^{\otimes \phi})\simeq \CA^{\otimes Y/\phi}.
\end{equation}

Next, let us be given an action of $\CA^{\otimes Y}$ on a (dualizable) DG category $\CM$.
Let $F_\CM$ be an endofunctor of $\CM$ compatible with $\CA^{\otimes \phi}$.

\medskip

Let
$$\on{Sht}_{\CM,F_\CM,\on{univ}}\in \CA^{\otimes Y/\phi}$$ be the object that corresponds under \eqref{e:Hochsch intro} to
$$\Tr^{\on{enh}}_{\CA^{\otimes Y}}(F_\CM,\CM)\in \on{HH}_\bullet(\CA^{\otimes Y},\CA^{\otimes \phi}).$$

\sssec{}  \label{sss:shtuka intro}

In \secref{s:sht}, we show how the object $\on{Sht}_{\CM,F_\CM,\on{univ}}$ encodes (i) the $I$-legged shtuka
construction\footnote{For us shtukas are algebraic objects. What we call ``shtukas" is more commonly called the ``cohomology of
sheaves arising by geometric Satake on (geometric) shtukas".  See \secref{sss:geometric shtukas} for an explanation of why this corresponds to cohomology of sheaves on the (geometric) moduli space of shtukas.}, (ii) partial Frobeniuses, (iii) excursion operators, and (iv) the ``S=T" relation.
In more detail:

\medskip

Using the rigidity of $\CA$ and hence of $\CA^{\otimes Y/\phi}$, we interpret the datum of $\on{Sht}_{\CM,F_\CM,\on{univ}}$
as a compatible family of functors
$$\on{Sht}_{\CM,F_\CM,Y/\phi,I}:(Y/\phi)^I\times \CA^{\otimes I}\to \Vect.$$

\medskip

Let
$$\on{Sht}_{\CM,F_\CM,Y,I}:Y^I\times \CA^{\otimes I}\to \Vect,$$
be the precomposition of $\on{Sht}_{\CM,F_\CM,Y/\phi,I}$ with the tautological projection $Y^I\to (Y/\phi)^I$.

\medskip

\noindent(i) In \propref{p:identify I-legged}, we show that the value of the functor
$\on{Sht}_{\CM,F_\CM,Y,I}$ on
$$\ul{y}\in Y^I,r\in \CA^{\otimes I},$$
identifies with
$$\Tr(H_{r_{\ul{y}}}\circ F_\CM,\CM),$$
where:

\medskip

\noindent--$r_{\ul{y}}$ denotes the object of $\CA^{\otimes Y}$ equal to the image of $r$ under $\CA^{\otimes I}\overset{\ul{y}}\longrightarrow \CA^{\otimes Y}$;

\medskip

\noindent--For $r'\in \CA^{\otimes Y}$, we let $H_{r'}$ denotes the endofunctor of $\CM$ given by the action of $r'$.

\medskip

This should be seen as a direct analog of the $I$-legged shtuka construction.

\medskip

\noindent(ii) The fact that $\on{Sht}_{\CM,F_\CM,Y,I}$ comes from $\on{Sht}_{\CM,F_\CM,Y/\phi,I}$ means that the former is equivariant under each of the endomorphisms
$\phi_i$ of $Y^I$, that acts as $\phi$ along the factor of $Y^I$ corresponding to $i\in I$ and as the identity along the other factors.
Following V.~Drinfeld, L.~Lafforgue and V.~Lafforgue, we call these endomorphisms ``partial Frobeniuses". We show (see \propref{p:partial Frob})
that, in terms of the isomorphism
$$\on{Sht}_{\CM,F_\CM,Y,I}(\ul{y}\in Y^I,r\in \CA^{\otimes I}) \simeq \Tr(H_{r_{\ul{y}}}\circ F_\CM,\CM),$$
the formula for the action of partial Frobeniuses translates
to the construction from \cite[Sect. 3]{Laf}. (We note that a salient feature of this formula is
cyclicity property of the trace construction.)

\medskip

\noindent(iii) By \thmref{t:two notions of trace} mentioned above, we have a canonical isomorphism
$$\Tr(F_\CM,\CM)\simeq  \CHom_{\CA^{\otimes Y/\phi}}(\one_{\CA^{\otimes Y/\phi}},\on{Sht}_{\CM,F_\CM,\on{univ}}) \simeq \on{Sht}_{\CM,F_\CM,Y/\phi,\emptyset}.$$

In particular, $\Tr(F_\CM,\CM)$ acquires an action of the algebra $\CEnd_{\CA^{\otimes Y/\phi}}(\one_{\CA^{\otimes Y/\phi}})$. In \propref{p:excursion shtukas}
we describe this action explicitly in terms of the \emph{excursion operators}, which are direct analogs of those in \cite[Sect. 9]{Laf}.

\medskip

\medskip

\noindent(iv) Arguably, the key conceptual and computational place in V.~Lafforgue's paper is the ``S=T" relation, which appears
as \cite[Proposition 6.2]{Laf}. In \thmref{t:S=T} we state and prove an analog of this result in our abstract context.  It states the equality
of two particular endomorphisms (one called $S$ and the other $T$) of $\Tr(F_\CM,\CM)$ corresponding to the data of $(y_0,a)$,
where $y_0$ is a $\phi$-fixed point on $Y$, and $a$ is a compact object in $\CA$.

\medskip

The $T$ operator is given by \eqref{e:categ trace intro}, for the endofunctor $H_{a_{y_0}}\circ F_\CM$ of $\CM$.

\medskip

The $S$ operator is an explicit excursion operator corresponding to $I=\{*\}$, the tautological loop based at the image $\bar{y}_0$
of $y_0$ under $Y\to Y/\phi$ (here we use the fact that $y_0$ is fixed by $\phi$), and a canonical map
$$\xi_a:\one_\CA \to \on{mult}\circ \on{mult}^R(\one_\CA)$$
attached to $a$, see \eqref{e:xi r}.

\ssec{Organization of the paper}

We will now briefly describe the structure of the paper.

\sssec{}

In \secref{s:int} we study the operation
$$Y,\CA\mapsto \CA^{\otimes Y},$$
where $Y$ is a space and $\CA$ is a symmetric monoidal category.

\medskip

The main results of this section are:

\medskip

\noindent--Presentation of $\CA^{\otimes Y}$ as a colimit with terms $\CA^{\otimes I}$ for finite sets $I$, as a symmetric
monoidal category/monoidal category/DG category (\thmref{t:integral});

\medskip

\noindent--Description of functors $\CA^{\otimes Y}\to \CC$ as compatible families of functors $\CA^{\otimes I}\to \CC\otimes \LS(Y^I)$,
where $\LS(-)$ denotes the category of (topological) local systems on a given space (\propref{p:describe functors});

\medskip

\noindent--For a group $\sG$ and a space $Y$ (with finitely many connected components), an equivalence
$$\Rep(\sG)^{\otimes Y}\to \QCoh(\LocSys_\sG(Y)),$$
where $\LocSys_\sG(Y)$ is the derived algebraic stack, classifying (topological) $\sG$-local systems on $Y$ (\thmref{t:integral to LocSys}),
and a generalization of this statement when $\Rep(\sG)=\QCoh(\on{pt}/\sG)$ is replaced by $\QCoh$ of a more general algebraic stack.

\sssec{}

The main theme of \secref{s:excurs} is excursion operators. The main results of this section are:

\medskip

\noindent--In the case when $Y$ is a connected pointed space, a presentation of $\CA^{\otimes Y}$ as a (sifted) colimit with terms
$$\CA^{\otimes \Sigma(I_+)},$$
where $I_+=I\sqcup *$ is a pointed finite set, and $\Sigma(I_+)\in \Spc$ denotes the suspension of $I_+$,
i.e., $\Sigma(I_+)\simeq \underset{I}\vee\, S^1$ (\corref{c:express connected});

\medskip

\noindent--Assuming that $\CA$ has an affine diagonal, a description of $\CEnd_{\CA^{\otimes Y}}(\one_{\CA^{\otimes Y}})$ as a colimit
with terms $\CEnd_{\CA^{\otimes \Sigma(I_+)}}(\one_{\CA^{\otimes \Sigma(I_+)}})$ (\corref{c:endo});

\medskip

\noindent--For $\CA$ rigid, an identification of
$\CEnd_{\CA^{\otimes \Sigma(J)}}(\one_{\CA^{\otimes \Sigma(J)}})$ with $\CHom_\CA(\one_\CA,\on{mult}_J\circ \on{mult}_J^R(\one_\CA))$
(\corref{c:endo term});

\medskip

\noindent--For a functor $\CS_Y:\CA^{\otimes Y}\to \CC$, an expression for the action of the terms
$$\CHom_\CA(\one_\CA,\on{mult}_J\circ \on{mult}_J^R(\one_\CA))\to \CEnd_{\CA^{\otimes \Sigma(J)}}(\one_{\CA^{\otimes \Sigma(J)}})\to
\CEnd_{\CA^{\otimes Y}}(\one_{\CA^{\otimes Y}})$$
on $\CS_Y(\one_{\CA^{\otimes Y}})$ in terms of the excursion operators (\thmref{t:excurs}).

\sssec{}

In \secref{s:Frob} we study the operation of categorical trace. The main results of this section are:

\medskip

\noindent--Construction of the categorical trace (\secref{ss:categ trace});

\medskip

\noindent--Construction of the 2-categorical trace (\secref{ss:2-categ trace});

\medskip

\noindent--The relationship between the two for rigid symmetric monoidal categories (\thmref{t:two notions of trace}).

\sssec{}

In \secref{s:wild} we study several generalizations and elaborations of \thmref{t:two notions of trace}. The main results are:

\medskip

\noindent--We connect the Chern character of a compact object of a rigid symmetric monoidal category with the excursion
operator (\corref{c:chern as excurs});

\medskip

\noindent--We formulate and prove a version of \thmref{t:two notions of trace} ``with observables" (\thmref{t:two notions of trace obs});

\medskip

\noindent--We introduce a mechanism that stands behind the Drinfeld-Lafforgue-Lafforgue definition of partial Frobeniuses
(\lemref{l:cyclicity and observables}).

\sssec{}

In \secref{s:sht} we combine the material of the previous sections to obtain our toy model 
for the shtuka construction. The material
in this section has been described already in \secref{sss:shtuka intro}.

\sssec{}

In the main body of this paper we study actions of categories of the form $\CA^{\otimes Y}$, where $\CA$ is a symmetric monoidal category,
and $Y$ is an object of $\Spc$. A key example is furnished by \cite{NY}: when working over the ground field $\BC$, we have an
action of $\Rep(\cG)^{\otimes X^{\on{top,sing}}}$ on $\Shv_{\on{Nilp}}(\Bun_G)$, where $X^{\on{top,alg}}$ is the object of $\Spc$ corresponding
to our curve $X$, and $\Shv(-)$ denotes the category of sheaves in the classical topology. (We recall that according to \thmref{t:integral to LocSys},
the category $\Rep(\cG)^{\otimes X^{\on{top,sing}}}$ identifies with $\QCoh(\LocSys_{\cG,\on{Betti}})$.)

\medskip

In the Appendix, Sects. \ref{s:sheaves}-\ref{s:ULA} we investigate the possibility of extending the construction of \cite{NY} to other
sheaf-theoretic contexts.

\sssec{}

In \secref{s:sheaves} we introduce a list of sheaf-theoretic contexts that we will consider. This list includes sheaves in the classical
topology (for schemes/stacks over $\BC$), D-modules (for schemes/stacks over a ground field of characteristic zero), as well as \'etale
sheaves (over any ground field).

\medskip

We recall the notion of singular support of an object $\CF\in \Shv(\CY)$, where $\CY$ is a scheme or algebraic stack.

\medskip

The main result of \secref{s:sheaves} is \thmref{t:vert sing supp}, which says that given an algebraic stack $\CY$
and a conical subset $\CN\subset T^*(\CY)$, for a scheme $X$, the functor of external tensor product defines an equivalence
$$\Shv_\CN(\CY)\otimes \Shv_{\on{lisse}}(X)\to \Shv_{\CN'}(\CY\times X), \quad \CN':=\CN \times \{\text{zero-section}\},$$
under the assumption that $X$ is proper\footnote{This assumption
is not necessary if $\Shv(-)$ is sheaves in the classical topology.}. In the above formula $\Shv_{\on{lisse}}(X)\subset \Shv(X)$
is the subcategory of lisse objects, see \secref{sss:defn lisse A}.

\sssec{}

In \secref{s:NY} we generalize the construction of \cite{NY} about the action of $\Rep(\cG)^{\otimes X^{\on{top,sing}}}$
on $\Shv_{\on{Nilp}}(\Bun_G)$.

\medskip

First, we provide details for the general pattern of Hecke action: we
show that in any sheaf-theoretic context, there is a compatible family of actions of $\Rep(\cG)^{\otimes I}$
on $\Shv(\Bun_G\times X^I)$ as $I$ ranges over the category of finite sets.

\medskip

Next, we show that the Hecke action gives rise to a compatible family of monoidal functors
\begin{equation} \label{e:Hecke nilp intro}
\Rep(\cG)^{\otimes I}\to \End(\Shv_{\on{Nilp}}(\Bun_G))\otimes \Shv_{\on{lisse}}(X), \quad I\in \on{fSet}.
\end{equation}

\medskip

Finally, if the sheaf-theoretic context is that of sheaves in the classical topology, we show that
the functors \eqref{e:Hecke nilp intro} assemble to an action of $\LocSys_\cG(X)$
on $\Shv_{\on{Nilp}}(\Bun_G)$.

\sssec{}

In \secref{s:Dmod} we specialize to the case when our sheaf-theoretic context is that of D-modules.
We formulate several conjectures pertaining to integrated actions in this context.

\medskip

First off, for a symmetric monoidal category $\CA$ and a scheme $X$, we \emph{define} an action of $\CA^{\otimes X}$
on a DG category $\CM$, to be the same as a compatible family of monoidal functors
$$\CA^{\otimes I}\to \End(\CM)\otimes \Dmod(X^I), \quad I\in \on{fSet}.$$

This is equivalent to having an action on $\CM$ of the symmetric monoidal category $\on{Fact}(\CA)_{\Ran(X)}$, defined as
in \cite[Sect. 2.5]{Ga3}.

\medskip

A key example of such an action is when $\CM=\Dmod(\Bun_G)$ and $\CA=\Rep(\cG)$.

\medskip

Let $\LocSys_\cG(X)$ be the stack of de Rham $\cG$-local systems on $X$, defined as in \cite[Sect. 10.1]{AG}. According to
\cite[Sect. 4.3]{Ga1}, we have a canonically defined symmetric monoidal functor
$$\Rep(\cG)^{\otimes X}\to \QCoh(\LocSys_\cG(X)),$$
which admits a fully faithful right adjoint. Hence, among categories equipped with an action of $\Rep(\cG)^{\otimes X}$ there
is a full subcategory formed by those categories on which this action comes from an action of $\QCoh(\LocSys_\cG(X))$.

\medskip

A result that we mentioned earlier on says that $\Dmod(\Bun_G)$ belongs to the above subcategory. However, the
proof of this theorem heavily relied on ``non-geometric" constructions, specifically on the interaction between $\Dmod(\Bun_G)$
and representation theory of the affine Kac-Moody Lie algebra corresponding to $\fg$ at the critical level. Given that, one would
have liked to have a more geometric proof of this result.

\medskip

Unfortunately, we \emph{do not} have a conjecture as to when an action of $\Rep(\cG)^{\otimes X}$ factors via
$\QCoh(\LocSys_\cG(X))$.

\medskip

Next, we introduce what it means for an action of $\CA^{\otimes X}$ on $\CM$ to be lisse. This means that for any $r\in \CA^{\otimes I}$
and $m\in \CM$, the resulting object of $\CM\otimes \Dmod(X^I)$ belongs to the full subcategory
$$\CM\otimes \Dmod_{\on{lisse}}(X^I)\subset \CM\otimes \Dmod(X^I),$$
where $\Dmod_{\on{lisse}}(-)\subset \Dmod(-)$ is the full subcategory of lisse D-modules (ind-completion of $\CO$-coherent D-modules).
For any $\CM$, one can consider its maximal full subcategory $\CM^{\on{lisse}}\subset \CM$, on which the action is lisse.

\medskip

Take $\CM=\Dmod(\Bun_G)$. \thmref{t:preserve nilp} (due to \cite{NY}) says that we have an inclusion
$$\Dmod_{\on{Nilp}}(\Bun_G)\subset (\Dmod(\Bun_G))^{\on{lisse}}.$$

In \conjref{c:lisse nilp} we propose that this inclusion is an equality.

\medskip

Finally, we propose \conjref{c:lisse as supp} that describes lisse actions in terms of set-theoretic support. Namely, let $\CM$
be a category equipped with an action of $\Rep(\cG)^{\otimes X}$ that factors via $\QCoh(\LocSys_\cG(X))$. Thus, given a compact
object of $\CM$, one can talk about its set-theoretic support, which is a Zariski-closed subset in $\LocSys_\cG(X)$.

\medskip

\conjref{c:lisse as supp} says that a compact object of $\CM$ belongs to $\CM^{\on{lisse}}$ if and only if its set-theoretic support
lies in the finite union of subsets \emph{induced} from irreducible local systems for Levi subgroups of $G$.

\sssec{}

In \secref{s:ULA} we discuss the abstract notion of ULA, which has been used in the definition of ULA actions in \secref{s:Dmod}.

\medskip

First, given a monoidal category $\CC$ and a module category $\CM$, we define what it means for an object $m\in \CM$
to be ULA with respect to $\CC$.

\medskip

The case of interest for us is $\CC=\Dmod(Y)$, with the monoidal structure given by the $\sotimes$ operation.

\medskip

We show that for $\CM:=\Dmod(Z)$ for a scheme $Z$ over $Y$ we recover the usual notion of what it means for
an object in $\Dmod(Z)$ to be ULA over $Y$.

\medskip

We also establish the following criterion for ULA-ness: we recall that the datum of an $\Dmod(Y)$-module
category is equivalent to that of \emph{sheaf of categories} over the de Rham prestack $Y_{\on{dR}}$ of $Y$. In particular, we can consider
the value of this sheaf of categories on $Y$ itself and obtain a $\QCoh(Y)$-module category $\CM_Y$, which can be recovered from $\CM$
as
$$\CM_Y\simeq \QCoh(Y)\underset{\Dmod(Y)}\otimes \CM.$$

In the case of $\CM:=\Dmod(Z)$, we have:
$$\CM_Y\simeq \QCoh(Y\underset{Y_{\on{dR}}}\times Z_{\on{dR}});$$
when $Z\to Y$ is smooth, this is the (derived) category of modules over the sheaf of \emph{vertical} differential operators.

\medskip

The categories $\CM$ and $\CM_Y$ are related by an adjoint pair of functors
$$\ind:\CM_Y\rightleftarrows \CM:\oblv.$$

Our \thmref{t:ULA Dmod} says that an object $m\in \CM$ is ULA with respect to $\Dmod(Y)$ if and only if
its image $\oblv(m)\in \CM_Y$ is compact.

\ssec{Notation and conventions}

\sssec{Higher categories}  \label{sss:higher categories}

This paper will substantially use the language of $\infty$-categories\footnote{We will often omit the adjective ``infinity" and refer to
$\infty$-categories simply as ``categories"}, as developed in \cite{Lu1}.

\medskip

We let $\Spc$ denote the $\infty$-category of spaces.

\medskip

Given an $\infty$-category $\bC$, and a pair of objects $\bc_1,\bc_2\in \bC$, we let $\Maps_\bC(\bc_1,\bc_2)\in \Spc$
be the mapping space between them. Given a space $Y$, by a $Y$-family of maps $\bc_1\to \bc_2$ we will mean a map
$Y\to \Maps_\bC(\bc_1,\bc_2)$ in $\Spc$.

\medskip

Recall that given an $\infty$-category $\bC$ that contains filtered colimits, an object $\bc\in \bC$ is said to be compact
if the Yoneda functor $\Maps_\bC(\bc,-):\bC\to \Spc$ preserves filtered colimits. We let $\bC^c\subset \bC$ denote
the full subcategory spanned by compact objects.

\medskip

Given a functor $F:\bC_1\to \bC_2$ between $\infty$-categories, we will denote by $F^R$ (resp., $F^L$) its right (resp., left) adjoint,
provided that it exists.

\sssec{Higher algebra}

Throughout this paper we will be concerned with \emph{higher algebra} over a commutative ring of coefficients, denoted $\sfe$.
Although $\sfe$ is not necessarily a field, we will denote by $\Vect$ the stable $\infty$-category of chain complexes of $\sfe$-modules,
see, e.g., \cite[Example 2.1.4.8]{GaLu}.

\medskip

We will regard $\Vect$ as equipped with a symmetric monoidal structure (in the sense on $\infty$-categories), see, e.g., \cite[Sect. 3.1.4]{GaLu}.
Thus, we can talk about commutative/associative algebra objects in $\Vect$, see, e.g., \cite[Sect. 3.1.3]{GaLu}.

\medskip

Whenever we talk about algebraic geometry over $\sfe$, we will mean \emph{derived algebraic geometry}, built off derived affine schemes,
the latter being by definition the category opposite to that of commutative algebras in $\Vect$, connective with respect to the natural
t-structure.

\medskip

We will denote by $\DGCat$ the $\infty$-category of (presentable) cocomplete stable $\infty$-categories, \emph{equipped with a module
structure over $\Vect$ with respect to the symmetric monoidal structure on the $\infty$-category of cocomplete stable $\infty$-categories
given by the Lurie tensor product}, see \cite[Sect. 4.8.1]{Lu2}. We will refer to objects of $\DGCat$ as ``DG categories".
We emphasize that as 1-morphisms in $\DGCat$, only colimit-preserving functors
are allowed.

\medskip

For a given DG category $\CC$, and a pair of objects $c_1,c_2\in \CC$, we have a well-defined ``inner Hom" object
$\CHom_\CC(c_1,c_2)\in \Vect$, characterized by the requirement that
$$\Maps_{\Vect}(V,\CHom_\CC(c_1,c_2))\simeq \Maps_\CC(V\otimes c_1,c_2), \quad V\in \Vect.$$

\medskip

The category $\DGCat$ itself carries a symmetric monoidal structure, given by Lurie tensor product over $\Vect$. In particular, we can talk
about the $\infty$-category of associative/commutative algebras in $\DGCat$, which we denote by $\DGCat^{\on{Mon}}$ (resp., $\DGCat^{\on{SymMon}}$),
and refer to as monoidal (resp., symmetric monoidal) DG categories.

\medskip

Unless specified otherwise, all monoidal/symmetric monoidal DG categories will be assumed unital. Given a monoidal/symmetric monoidal DG category
$\CA$, we will denote by $\one_\CA$ its unit object.

\sssec{Rigidity} \label{sss:rigid}

In multiple places in this paper we use the notion of \emph{rigidity} for a monoidal DG category $\CA$.
We refer the reader to \cite[Sect. 9.1.1]{GR1} for the general definition.

\medskip

That said, in most of the cases of interest, the DG category in question will be compactly generated. In this case,
according to \cite[Lemma 9.1.5]{GR1}, the condition of being rigid is equivalent to the fact that its classes of
compact objects and objects that are both left and right dualizable, coincide.

\sssec{Prestacks}

In Appendices \ref{s:sheaves}-\ref{s:NY} we will deal with classical (i.e., non-derived) algebro-geometric objects over a ground
field $k$ (which has nothing to do with our ring of coefficients $\sfe$).

\medskip

We will denote by $\affSch$ the category of affine schemes (of finite type (!)) over $k$.

\medskip

By a prestack (technically, prestack locally of finite type), we will mean an arbitrary functor
$$(\affSch)^{\on{op}}\to \Spc.$$

\ssec{Acknowledgements}

We are grateful to A.~Beilinson, V.~Drinfeld, V.~Lafforgue, J.~Lurie and P.~Scholze for valuable discussions,
which informed our thinking about the subject.

\medskip

The research of D.G. was supported by NSF grant DMS-1707662 and by Gelfand Chair at IHES.
The research of Y.V. was supported by the ISF grant 822/17.

\medskip

Part of the project was carried out while D.G, N.R and Y.V were at MSRI and were supported by NSF grant DMS-1440140.

\medskip

The project have received funding from ERC under grant agreement No 669655.

\section{Symmetric monoidal categories integrated over a space}  \label{s:int}

Let $\CA$ be a symmetric monoidal category, and let $Y$ be an object of $\Spc$.
The goal of this section is to give an explicit description of the category
$$\CA^{\otimes Y},$$
(sometimes also denoted $\underset{Y}\int\, \CA$), which can be thought of as
\emph{factorization homology} of $\CA$ along $Y$.

\medskip

We will describe $\CA^{\otimes Y}$
as a colimit as a (i) symmetric monoidal category, (ii) just monoidal category, (iii) plain
DG category (each time the colimit will be taken within the corresponding ambient category, i.e.,
inside the category of symmetric monoidal categories, monoidal categories or DG categories,
respectively).

\medskip

In particular, we will give an explicit description of what it takes for $\CA^{\otimes Y}$ to act (as a monoidal
category) on a DG category $\CM$, and what it takes to map out of $\CA^{\otimes Y}$ as a plain DG category.
Both descriptions will be formulated in terms of functors out of $\CA^{\otimes I}$ parameterized by points of $Y^I$,
for $I$ ranging over the category $\on{fSet}$ of finite sets.

\ssec{The integral}

\sssec{}

Let $\bC$ be a category with colimits. For an object $\bc\in \bC$ and $Y\in \Spc$ define the object
$$\underset{Y}\int\, \bc:= \underset{Y}{\on{colim}}\, \bc\in \bC.$$
%
%

I.e., we take the colimit along the index category $Y$ of the constant functor $Y\to \bC$ with value $\bc$.

\sssec{}

Tautologically, for $\bc,\bc'\in \bC$, we have
\begin{equation} \label{e:univ ppty int}
\Maps_\bC(\underset{Y}\int\, \bc,\bc')\simeq \Maps_{\Spc}(Y,\Maps_\bC(\bc,\bc')).
\end{equation}

\medskip

This shows that the assignment
\begin{equation} \label{e:integral functor two var}
\bc,Y\mapsto \underset{Y}\int\, \bc
\end{equation}
is a functor
$$\bC\times \Spc\to \bC.$$

\medskip

Moreover, the functor \eqref{e:integral functor two var} preserves colimits in each variable.

\sssec{}

For a fixed $\bc$, the functor
\begin{equation} \label{e:integral functor}
Y\mapsto \underset{Y}\int\, \bc, \quad \on{Spc}\to \bC
\end{equation}
can be characterized as the unique colimit-preserving functor whose
value on $\{*\}\in \Spc$ is $\bc$.

\sssec{}

For example for $Y$ a discrete set $I$, we have
\begin{equation} \label{e:int finite set}
\underset{I}\int\, \bc\simeq \underset{i\in I}\sqcup\, \bc.
\end{equation}

\sssec{}

Let us recall the notion of \emph{left Kan extension}. Let $i:\bD\to \bD'$ be a functor; let $\bC$ be a category with colimits,
and let $F:\bD\to \bC$ be a functor. The left Kan extension of $F$ along $i$ is a functor
$$\on{LKE}_i(F):\bD'\to \bC$$ with the
following universal property
$$\Maps_{\on{Funct}(\bD,\bC)}(F,G\circ i)\simeq \Maps_{\on{Funct}(\bD',\bC)}(\on{LKE}_i(F),G), \quad G\in \on{Funct}(\bD',\bC).$$

One can show that the value of $\on{LKE}_i(F)$ on an object $\bd'\in \bD'$ is given by
\begin{equation} \label{e:value LKE}
\underset{\bd\in \bD_{/\bd'}}{\on{colim}}\, F(\bd).
\end{equation}

Here and elsewhere, the notation $\bD_{/\bd'}$ is the slice category, i.e., the category of
\begin{equation} \label{e:slice}
\{\bd\in \bD,i(\bd)\to \bd'\}.
\end{equation}

\sssec{}

Tautologically, we can rewrite the functor \eqref{e:integral functor} as follows. Let $\{*\}$ be the point category, and let $*$ denote its single object.
By a slight abuse of notation we will also denote by $\{*\}$ the corresponding object of $\Spc$, i.e., the point space.

\medskip

Then \eqref{e:integral functor} is the left Kan extension along
\begin{equation} \label{e:incl point}
\{*\}\hookrightarrow \on{Spc},\quad *\mapsto \{*\}
\end{equation}
of the functor
$$\{*\}\to \bC, \quad *\mapsto \bc.$$

\sssec{}  \label{sss:LKE fin}

Let $\on{fSet}$ denote the category of finite sets, equipped with the embedding
\begin{equation} \label{e:finite sets to spaces}
\on{fSet}\hookrightarrow \on{Spc}.
\end{equation}

By transitivity of the procedure of left Kan extension with respect to the composition
$$\{*\}\hookrightarrow \on{fSet}\hookrightarrow \on{Spc},$$
we obtain that \eqref{e:integral functor} identifies with the left Kan extension along \eqref{e:finite sets to spaces}
of its restriction to $\on{fSet}$,
while the latter is given by \eqref{e:int finite set}.

\ssec{The tensor product $\CA^{\otimes Y}$ as a colimit}  \label{ss:A power Y}

\sssec{}

We take $\bC$ to be $\DGCat^{\on{SymMon}}$, so an object of $\bC$ is a symmetric monoidal DG category $\CA$.
We wish to give an explicit description of the resulting category
\begin{equation} \label{e:integral of category}
\CA^{\otimes Y}:=\underset{Y}\int\, \CA
\end{equation}
as a colimit, in three different contexts: (i) as a symmetric monoidal category, (ii) as a monoidal category,
(iii) as a plain DG category.

\sssec{}

For a category $\bD$, let $\on{TwArr}(\bD)$ be the corresponding twisted arrows category. I.e., its objects are 1-morphisms in $\bD$
$$\bd_s\to \bd_t,$$
and morphisms
$$(\bd^0_s\to \bd^0_t)\to (\bd^1_s\to \bd^1_t)$$
are commutative diagrams
$$
\CD
\bd^0_s  @>>>  \bd^0_t \\
@VVV   @AAA  \\
\bd^1_s  @>>>  \bd^1_t.
\endCD
$$

In practice we will take $\bD$ to be the category $\on{fSet}$ of finite sets.

\sssec{}  \label{sss:three cases}

Let $\CA$ be a symmetric monoidal category, and let $\bC$ be one of the categories: $\DGCat^{\on{SymMon}}$, $\DGCat^{\on{Mon}}$, $\DGCat$.
For a given $Y$, consider the functor
\begin{equation} \label{e:functor from TwArr}
\on{TwArr}(\on{fSet})\to \bC, \quad (I\to J)\mapsto \underset{\Maps(J,Y)}\int\, \CA^{\otimes I}.
\end{equation}

We emphasize that in the above formula $\CA^{\otimes I}$ is understood in the usual sense, i.e., the $I$-fold tensor product of $\CA$ with itself,
which is again a symmetric monoidal category, but viewed as an object in $\bC$ using the tautological forgetful functor
$$\DGCat^{\on{SymMon}}\to \bC,$$
while $\underset{\Maps(J,Y)}\int$ is taken \emph{within $\bC$}, i.e., the result of this operation looks different as a category,
depending on which $\bC$ we choose.

\medskip

We claim:

\begin{thm} \label{t:integral}
In each of the three cases in \secref{sss:three cases}, the image of $\CA^{\otimes Y}$ under the forgetful functor
$$\DGCat^{\on{SymMon}}\to \bC$$
identifies canonically with the colimit in $\bC$ of the functor \eqref{e:functor from TwArr} along $\on{TwArr}(\on{fSet})$.
\end{thm}

\begin{rem}
The assertion of \thmref{t:integral} with the same proof applies more generally, where instead of $\DGCat$ we consider a symmetric
monoidal category $\bO$ satisfying the conditions of Remark \ref{r:forget mon} below. So instead of $\DGCat^{\on{SymMon}}$
we will have $\on{ComAlg}(\bO)$ and $\bC$ can be any of $\on{ComAlg}(\bO)$, $\on{AssocAlg}(\bO)$ or $\bO$.
\end{rem}

\ssec{Proof of \thmref{t:integral}}

\sssec{Plan of the proof}

We will show that both sides, viewed as functors $\Spc\to \bC$, are left Kan extensions of their respective
restrictions along \eqref{e:finite sets to spaces}. Then we will show that there restrictions are canonically
isomorphic.

\sssec{}

The first two steps will use the following general assertion. Recall that a category $A$ is said to be \emph{sifted} if
the diagonal functor $A\to A\times A$ is cofinal.

\begin{lem} \label{l:coprod sifted}
If $A$ has coproducts, then it is sifted.
\end{lem}

\begin{proof}
We need to show that for any $a',a''\in A$, the category
$$A_{(a',a'')_{/}}:=(a\in A, a'\to a  \leftarrow a'')$$
is contractible\footnote{In our terminology ``contractible" is what in \cite{Lu1} is called ``weakly contractible".}.
Now, the fact that $A$ has coproducts means that $A_{(a',a'')_{/}}$ has an initial object.
\end{proof}

\begin{cor}  \label{c:over sifted}
Let $\bD$ have finite coproducts, and let $i:\bD\to \bD'$ be a functor
that preserves finite coproducts. Then for any $\bd'\in \bD$, the slice category
$\bD_{/\bd'}$ \emph{(}see Equation \eqref{e:slice}\emph{)} is sifted.
\end{cor}


Here is an application of \corref{c:over sifted} that we will use repeatedly:

\begin{prop} \label{p:LKE commutes}
Let $i:\bD\to \bD'$ be as in \corref{c:over sifted}. Let $\bC_1$ be a category that admits colimits, and let
$\Phi:\bC_1\to \bC_2$ be a functor that preserves sifted colimits. Then for $F:\bD\to \bC_1$, the natural
transformation
$$\on{LKE}_i(\Phi\circ F)\to \Phi\circ \on{LKE}_i(F),$$
obtained by the universal property of $\on{LKE}_i(-)$, is an isomorphism.
\end{prop}

\begin{proof}
Follows by \corref{c:over sifted} from the formula \eqref{e:value LKE} for the values of $\on{LKE}_i(-)$.
\end{proof}

We now proceed to the proof of \thmref{t:integral}.

\sssec{Step 1}  \label{sss:operad}

Let us show that the functor
$$\Spc \overset{Y\mapsto \CA^{\otimes Y}}\longrightarrow \DGCat^{\on{SymMon}}\to \bC$$
is the left Kan extension of its restriction along \eqref{e:finite sets to spaces}.

\medskip

We apply \propref{p:LKE commutes} to the functor $\on{fSet}\to \Spc$, which obviously satisfies the assumption
of \corref{c:over sifted}. 

\medskip

Since the functor $Y\mapsto \CA^{\otimes Y}$ is the left Kan extension of its restriction to $\on{fSet}$,
it remains to show that the forgetful functor
$\DGCat^{\on{SymMon}}\to \bC$ preserves sifted colimits. This follows from the fact that both forgetful functors
\begin{equation} \label{e:forget mon}
\DGCat^{\on{SymMon}}\to \DGCat \text{ and } \DGCat^{\on{Mon}}\to \DGCat
\end{equation}
preserve sifted colimits and are conservative.

\begin{rem} \label{r:forget mon}

The fact that the functors \eqref{e:forget mon} preserve sifted colimits and are conservative is a consequence
of the following:

\medskip

Let $\bO$ be a symmetric monoidal category which admits sifted colimits and such that the tensor product functor commutes with sifted colimits.
Then for an operad $\CP$, the forgetful functor
$$\oblv_{\CP}:\CP\on{-alg}(\bO)\to \bO$$
preserves sifted colimits and is conservative, see \cite[Prop. 3.2.3.1 and Lemma 3.3.2.6]{Lu2}.

\end{rem}

\sssec{Step 2}  \label{sss:f step 2}

Let us now show that the functor $\Spc\to \bC$ that sends $Y$ to
the colimit (in $\bC$) of the functor \eqref{e:functor from TwArr} along $\on{TwArr}(\on{fSet})$
is the left Kan extension of its restriction along \eqref{e:finite sets to spaces}.

\medskip

It suffices to show that for each $I$ and $J$, the functor that sends $Y$ to $\underset{\Maps(J,Y)}\int\, \CA^{\otimes I}$
is the left Kan extension of its restriction along \eqref{e:finite sets to spaces}. Since the functor \eqref{e:integral functor two var}
commutes with sifted colimits in the space variable, it suffices to show that the functor that sends $Y$ to $\Maps(J,Y)$
is the left Kan extension of its restriction along \eqref{e:finite sets to spaces}.

\medskip



Since the functor $\Maps(J,-)$ preserves sifted colimits, this follows again from
\propref{p:LKE commutes} applied to $\bC_1=\bC_2=\Spc$, using the fact that the identity
functor on $\Spc$ is the left Kan extension of its restriction along \eqref{e:finite sets to spaces}.

\sssec{Step 3}

It remains to show that the restrictions of both sides in \thmref{t:integral} to $\on{fSet}$ are
canonically isomorphic.

\medskip

However, this follows from the next version of the Yoneda lemma (which we prove below for the sake of completeness):

\begin{prop} \label{p:Yoneda}
Let $\bC$ be a category with colimits and let $\Phi:\bD\to \bC$ be a functor. Then for $\bd\in \bD$
there is a canonical isomorphism
$$\Phi(\bd) \simeq \underset{(\bd_s\to \bd_t)\in \on{TwArr}(\bD)}{\on{colim}}\, \underset{\Maps(\bd_t,\bd)}\int\, \Phi(\bd_s).$$
\end{prop}

\qed[\thmref{t:integral}]

\sssec{}

For the proof of \propref{p:Yoneda} we note the following general feature of the twisted arrows category:

\medskip

Let $F_1,F_2:\bD'\to \bE$ be a pair of functors. Consider the functor
$$\on{Tw}(F_1,F_2):\on{TwArr}(\bD')^{\on{op}}\to \Spc$$
that sends
$$(\bd'_s\to \bd'_t)\mapsto \Maps_\bE(F_1(\bd'_s),F_2(\bd'_t)).$$

\medskip

We have the following standard fact (see e.g. \cite[Prop. 5.1]{GHN}):

\begin{lem}  \label{l:twisted}
There exists a canonical isomorphism
$$\underset{\on{TwArr}(\bD')^{\on{op}}}{\on{lim}}\, \on{Tw}(F_1,F_2)\simeq \Maps_{\on{Funct}(\bD',\bE)}(F_1,F_2).$$
\end{lem}

\sssec{}

To prove \propref{p:Yoneda}, we apply \lemref{l:twisted} to $\bD':=\bD^{\on{op}}$, $\bE=\Spc$ and
$$F_1:=\Maps_{\bD}(-,\bd),\,\, F_2:=\Maps_\bC(\Phi(-),\bc), \quad \bc\in \bC.$$

Note that we have a tautological identification
$$\on{TwArr}(\bD)\simeq \on{TwArr}(\bD^{\on{op}}).$$

Using \lemref{l:twisted} and \eqref{e:univ ppty int}, we obtain that
\begin{multline*}
\Maps_\bC\left(\left(\underset{(\bd_s\to \bd_t)\in \on{TwArr}(\bD)}{\on{colim}}\, \underset{\Maps(\bd_t,\bd)}\int\, \Phi(\bd_s)\right),\bc\right)\simeq \\
\simeq \underset{(\bd_s\to \bd_t)\in (\on{TwArr}(\bD))^{\on{op}}}{\on{lim}}\,
\Maps_\bC(\underset{\Maps(\bd_t,\bd)}\int\, \Phi(\bd_s),\bc) \simeq \\
\simeq \underset{(\bd_s\to \bd_t)\in (\on{TwArr}(\bD))^{\on{op}}}{\on{lim}}\,
\Maps_{\Spc}(\Maps(\bd_t,\bd),\Maps(\Phi(\bd_s),\bc))\simeq  \\
\simeq\Maps_{\on{Funct}(\bD^{\on{op}},\Spc)}\left(\Maps_{\bD}(-,\bd),\Maps_{\bC}(\Phi(-),\bc)\right),
\end{multline*}
which by the usual Yoneda lemma identifies with
$$\Maps_\bC(\Phi(\bd),\bc),$$
as desired.

\qed

\ssec{The category of local systems}

In this subsection we will describe the right adjoint to the functor
$$\CC\mapsto \underset{Y}\int\, \CC, \quad \CC\in \bC$$
for $\bC$ being $\DGCat^{\on{SymMon}}$, $\DGCat^{\on{Mon}}$, $\DGCat$.

\sssec{}  \label{sss:int and power}

For a category $\bC$ with limits, an object $\bc\in \bC$ and $Y\in \Spc$, set
$$\bc^Y:=\underset{Y}{\on{lim}}\, \bc,$$
i.e., the limit of the functor $Y\to \bC$ with constant value $\bc$. For $Y$ being a discrete set $I$, we obtain
$$\bc^I\simeq \underset{i\in I}\Pi\, \bc.$$

\medskip

Note that for $\bc,\bc'\in \bC$, we have
$$\Maps_\bC(\bc',\bc^Y)\simeq \Maps_{\Spc}(Y,\Maps_\bC(\bc',\bc)).$$

\medskip

The latter expression shows that the assignment
$$\bc,Y\mapsto \bc^Y$$
is a functor
$$\bC\times \Spc^{\on{op}}\to \bC,$$
moreover, this functor preserves limits in each variable.

\medskip

Furthermore, by \eqref{e:univ ppty int}, for a fixed $Y$, the functor
$$\bc\mapsto \bc^Y,\quad \bC\to \bC$$
is the right adjoint of the functor
$$\bc'\mapsto \underset{Y}\int\, \bc'.$$

\sssec{}

Take $\bC=\DGCat^{\on{SymMon}}$ and set
$$\LS(Y):=\Vect^Y\in  \DGCat^{\on{SymMon}}.$$

\medskip

Since the forgetful functors
$$\DGCat^{\on{SymMon}}\to \DGCat^{\on{Mon}}\to \DGCat$$
commute with limits (this is valid in the general context of algebras over an operad in \secref{sss:operad}, see \cite[Prop. 3.2.2.1]{Lu2}),
the image of $\LS(Y)$ under these forgetful functors identifies with
$\Vect^Y$, thought of taking values in $\DGCat^{\on{Mon}}$ (resp., $\DGCat$).

\begin{rem}

Let $\CY$ be a topological space\footnote{We assume that $\CY$ is sufficiently nice; i.e. a paracompact topological space homotopy equivalent to a
CW complex and of finite homotopical dimension (see \cite[Sect. 7.1]{Lu1})} weakly equivalent to the geometric realization of $Y$.
One can show that $\LS(Y)$ is equivalent to the full subcategory $\Shv_{\on{lisse}}(\CY)$ of the (unbounded) derived category $\Shv(\CY)$
of sheaves of $\sfe$-vector spaces on $\CY$ consisting of objects with locally constant cohomology sheaves.

\end{rem}

\sssec{}

We will now describe the category $\LS(Y)$ as a plain DG category more explicitly:

\begin{prop} \label{p:LS}  \hfill

\smallskip

\noindent{\em(a)} For a map $f:Y_1\to Y_2$, the restriction map $f^\dagger:\LS(Y_2)\to \LS(Y_1)$
admits a left adjoint.

\smallskip

\noindent{\em(a')} The functor
\begin{equation} \label{e:LS left}
\Spc\to \DGCat, \quad Y\mapsto \LS(Y), \quad (Y_1\overset{f}\to Y_2)\rightsquigarrow
\left(\LS(Y_1)\overset{(f^\dagger)^L}\longrightarrow \LS(Y_2)\right)
\end{equation}
is canonically isomorphic to the functor
$$Y\mapsto \underset{Y}\int\, \Vect.$$

\smallskip

\noindent{\em(b)} The category $\LS(Y)$ is dualizable.

\smallskip

\noindent{\em(b')} The functor \eqref{e:LS left} is canonically isomorphic to the functor
\begin{equation} \label{e:LS dual}
\Spc\to \DGCat, \quad Y\mapsto \LS(Y)^\vee, \quad (Y_1\overset{f}\to Y_2)\rightsquigarrow
\left(\LS(Y_1)^\vee\overset{(f^\dagger)^\vee}\longrightarrow \LS(Y_2)^\vee\right)
\end{equation}

\end{prop}

The proof of this proposition is given below.

\begin{cor} \label{c:LS}
The category $\underset{Y}\int\, \Vect$ is dualizable. The functor
$$\Spc^{\on{op}}\to \DGCat, \quad Y\mapsto (\underset{Y}\int\, \Vect)^\vee, \quad
(Y_1\to Y_2)\rightsquigarrow \left((\underset{Y_2}\int\, \Vect)^\vee\to (\underset{Y_1}\int\, \Vect)^\vee\right)$$
is canonically isomorphic to the functor $Y\mapsto \LS(Y)$.
\end{cor}

\begin{proof}

Combine points (a') and (b') of \propref{p:LS}.

\end{proof}

Note that in particular, \propref{p:LS} and \corref{c:LS} imply that for an individual $Y$ we have the canonical equivalences
$$\LS(Y)\simeq \underset{Y}\int\, \Vect,\quad \LS(Y)^\vee\simeq \LS(Y),\quad (\underset{Y}\int\, \Vect)^\vee\simeq \underset{Y}\int\, \Vect.$$

\sssec{}

The proof of \propref{p:LS} is based on the next general lemma
(see \cite[Chapter 1, Propositions 2.5.7 and 6.3.4, and Lemma 2.6.4]{GR1}):

\begin{lem} \label{l:lim colim}
Let
\begin{equation} \label{e:index A}
a\mapsto \CC_a, \quad A\to \DGCat
\end{equation}
be a diagram of DG categories, such that for every arrow $a_1\to a_2$, the corresponding functor $$F_{a_1,a_2}:\CC_{a_1}\to \CC_{a_2}$$
admits a continuous right adjoint. Set
$$\CC:=\underset{a\in A}{\on{colim}}\, \CC_a.$$
Then:

\smallskip

\noindent{\em(a)} The tautological functors $\on{ins}_a:\CC_a\to \CC$ admit continuous right adjoints.

\smallskip

\noindent{\em(b)} Consider the functor
\begin{equation} \label{e:index A op}
A^{\on{op}}\to \DGCat, \quad a\mapsto \CC_a, \quad (a_1\to a_2)\rightsquigarrow (\CC_{a_2}\overset{F_{a_1,a_2}^R}\longrightarrow \CC_{a_1}),
\end{equation}
and set
$$\wt\CC:=\underset{a\in A^{\on{op}}}{\on{lim}}\, \CC_a.$$
Then the functor $\CC\to \wt\CC$, given by the system of functors
$$\on{ins}_a^R:\CC\to \CC_a$$
is an equivalence.

\smallskip

\noindent{\em(c)} Let $\phi:A'\to A$ be a functor of index categories. Composing with \eqref{e:index A} and \eqref{e:index A op}, respectively, we obtain the functors
$$A'\to \DGCat \text{ and } A'{}^{\on{op}}\to \DGCat.$$
Set
$$\CC':=\underset{a'\in A'}{\on{colim}}\, \CC_{\phi(a')} \text{ and } \wt\CC':=\underset{a'\in A'{}^{\on{op}}}{\on{lim}}\, \CC_{\phi(a')}.$$
Then under the identifications
$$\CC\simeq \wt\CC \text{ and } \CC'\simeq \wt\CC',$$
the tautological functor $\CC'\to \CC$ identifies with the left adjoint of the restriction functor $\wt\CC'\to \wt\CC$.

\smallskip

\noindent{\em(d)} Assume that all $\CC_a$ are dualizable. Consider the functor
\begin{equation} \label{e:index A dual}
A^{\on{op}}\to \DGCat, \quad a\mapsto \CC^\vee_a, \quad (a_1\to a_2)\rightsquigarrow (\CC^\vee_{a_2}\overset{F_{a_1,a_2}^\vee}\longrightarrow \CC^\vee_{a_1}).
\end{equation}
Then $\CC$ is dualizable, and the functor
$$\CC^\vee\to \underset{a\in A^{\on{op}}}{\on{lim}}\, \CC^\vee_a,$$
given by the system of functors $$\on{ins}_a^\vee:\CC^\vee\to \CC^\vee_a$$
is an equivalence.

\end{lem}

\begin{proof}[Proof of \propref{p:LS}]

Recall that
$$\underset{Y}\int\, \Vect\simeq \underset{Y}{\on{colim}}\, \Vect \text{ and } \LS(Y) \simeq \underset{Y}{\on{lim}}\, \Vect.$$

The identification
$$\underset{Y}\int\,\Vect \simeq \LS(Y)$$
follows from \lemref{l:lim colim}(b) applied to the index category $Y$ and the constant family with value $\Vect$.

\medskip

Point (a') of \propref{p:LS} follows from \lemref{l:lim colim}(c), implying also
point (a).

\medskip

Point (b') of \propref{p:LS} follows from \lemref{l:lim colim}(d), implying also
point (b).

\end{proof}

\sssec{}

Let $\bC$ be one of the three categories from \thmref{t:integral}, and let $\CC$ be an object of $\bC$. I.e.,
$\CC$ is either a symmetric monoidal category, or a monoidal category or a plain DG category.

\medskip

Note that for
a symmetric monoidal category $\CA$, the tensor product $\CC\otimes \CA$ is again naturally an object of $\bC$
(this is true for $\DGCat$ replaced by an arbitrary symmetric monoidal category $\bO$ and any operad $\CP$
as in \secref{sss:operad}; this follows, e.g., from \cite[Prop. 3.2.4.3]{Lu2}).

\medskip

We claim that we have a natural map
\begin{equation} \label{e:ten LocSys}
\CC\otimes \LS(Y)\to \CC^Y.
\end{equation}

Indeed, by the definition of $\CC^Y$, the datum of such a map is equivalent to that of a family of maps $\CC\otimes \LS(Y)\to \CC$,
parameterized by points of $Y$. For $y\in Y$ we take the corresponding map to be
$$\CC\otimes \LS(Y)\overset{\on{Id}_\CC\otimes \on{ev}_y}\longrightarrow \CC\otimes \Vect\simeq \CC,$$
where
$$\on{ev}_y:\LS(Y)\to \LS(\{*\})\simeq \Vect$$
is the functor of pullback corresponding to $\{*\}\overset{y}\to Y$.

\medskip

We claim:

\begin{prop}  \label{p:LocSys as right adj}
The map \eqref{e:ten LocSys} is an isomorphism.
\end{prop}

The proof of \propref{p:LocSys as right adj} is given below.

\medskip

Using \secref{sss:int and power}, from \propref{p:LocSys as right adj}, we obtain:

\begin{cor} \label{c:LocSys as right adj}
For a given $Y$, the right adjoint of the functor
$$\CC'\mapsto \underset{Y}\int\,\CC'$$
is given by
$$\CC\mapsto \CC\otimes \LS(Y).$$
\end{cor}

\begin{proof}[Proof of \propref{p:LocSys as right adj}]

Since the forgetful functor $\bC\to \DGCat$ commutes with the operation $$-\otimes \LS(Y),$$
it suffices to prove the assertion for $\bC=\DGCat$.

\medskip

The map in question is an isomorphism for $Y=\{*\}$. Since $Y\simeq \underset{Y}{\on{colim}}\, \{*\}$,
it suffices to show that (for a fixed $\CC$), the functor
$$\CC\mapsto \CC\otimes \LS(Y)$$
takes colimits in $Y$ to limits in $\DGCat$. However, this follows from \corref{c:LS}, as the functor in question can be rewritten
as
$$\on{Funct}(\underset{Y}\int\, \Vect,\CC).$$

\end{proof}

\ssec{Digression: the stack of $\sG$-local systems}  \label{ss:LocSys}

In this subsection we will assume that the ring $\sfe$ of coefficients is a field of characteristic $0$.

\sssec{}

Let $\sG$ be an affine group-scheme over $\sfe$. We let $\Rep(\sG)$ be the category of $\sG$-representations,
defined as
$$\Rep(\sG):=\QCoh(\on{pt}/\sG).$$

The assumption on $\sfe$ implies that the category $\Rep(\sG)$ is compactly generated; its compact objects are those
that are mapped to compact objects of $\Vect_\sfe$ under the forgetful functor
\begin{equation} \label{e:oblv G}
\oblv_\sG:\Rep(\sG)\to \Vect=\QCoh(\on{pt}).
\end{equation}

\sssec{}

For $Y\in \Spc$, let $\LocSys_\sG(Y)$ denote the (derived) Artin stack that classifies $\sG$-local systems on $Y$.
By definition, for a derived affine scheme $S$,
$$\Maps(S,\LocSys_\sG(Y))$$
is the space of right t-exact symmetric monoidal functors
$$\Rep(\sG)\to \QCoh(S)\otimes \on{LS}(Y),$$
(where the RHS is equipped with the tensor product t-structure, see \cite[Sect. C.4]{Lu3}).

\medskip

Consider the (symmetric monoidal) category $\QCoh(\LocSys_\sG(Y))$. We will now construct a (symmetric monoidal) functor
\begin{equation} \label{e:integral to LocSys}
\Rep(\sG)^{\otimes Y}\to \QCoh(\LocSys_\sG(Y)).
\end{equation}

\sssec{}  \label{sss:to LocSys}

First, we note that for $Y=\{*\}$, we have
$$\LocSys_\sG(\{*\})\simeq \on{pt}/\sG.$$

Indeed, $\Maps(S,\on{pt}/\sG)$ is the groupoid of $\sG$-torsors on $S$, and those can be described as
right t-exact symmetric monoidal functors
$$\Rep(\sG)\to \QCoh(S),$$
see \cite[Sect. 10.2]{AG}.

\medskip

Hence,
$$\QCoh(\LocSys_\sG(\{*\}))\simeq \QCoh(\on{pt}/\sG)=:\Rep(\sG).$$

\medskip

For any $Y$, we have a canonical map of spaces
\begin{multline*}
Y\simeq \Maps_{\Spc}(\{*\},Y) \to \Maps_{\on{PreStk}}(\LocSys_\sG(Y),\LocSys_\sG(\{*\}))\to  \\
\to \Maps_{\DGCat^{\on{SymMon}}}(\QCoh(\LocSys_\sG(\{*\})),\QCoh(\LocSys_\sG(Y)))\simeq \\
\simeq  \Maps_{\DGCat^{\on{SymMon}}}(\Rep(\sG),\QCoh(\LocSys_\sG(Y))).
\end{multline*}

Hence, by \eqref{e:univ ppty int}, we obtain the desired map
$$\Rep(\sG)^{\otimes Y}:=\underset{Y}\int\, \Rep(\sG)\to \QCoh(\LocSys_\sG(Y)).$$

\sssec{}

We claim:

\begin{thm}  \label{t:integral to LocSys}
Assume that $Y$ has finitely many connected components.
Then the functor \eqref{e:integral to LocSys} is an equivalence.
\end{thm}

\begin{rem}
The assertion of \thmref{t:integral to LocSys} would fail for $Y$ which is the infinite disjoint union of copies of $\{*\}$,
and $\sG=\BG_a$.
\end{rem}

\sssec{}

We will actually prove a generalization of \thmref{t:integral to LocSys}. Let $\CZ$ be a prestack satisfying the following conditions:

\begin{itemize}

\item The diagonal map $\CZ\to \CZ\times \CZ$ is affine;

\item The category $\QCoh(\CZ)$ is dualizable;


\end{itemize}

Let us call such prestacks \emph{quasi-passable}, cf. \cite[Chapter 3, Sect. 3.5.1]{GR1}.

\medskip

For $Y\in \Spc$
consider the prestack $\bMaps(Y,\CZ)$ defined as
$$\Maps_{\on{PreStk}}(S, \bMaps(Y,\CZ))=\Maps_{\Spc}(Y,\Maps_{\on{PreStk}}(S,\CZ)), \quad S\in \affSch.$$

\medskip

As in \secref{sss:to LocSys}, we have a canonically defined symmetric monoidal functor
\begin{equation} \label{e:integral to Maps}
\QCoh(\CZ)^{\otimes Y}\to \QCoh(\bMaps(Y,\CZ)).
\end{equation}

We will prove:

\begin{thm}  \label{t:integral to Maps}
Assume that $Y$ has finitely many connected components and that $\CZ$ is quasi-passable. Then the prestack $\bMaps(Y, \CZ)$ is quasi-passable and
the functor \eqref{e:integral to Maps} is an equivalence.
\end{thm}

\thmref{t:integral to LocSys} is a particular case of \thmref{t:integral to Maps} for $\CZ=\on{pt}/\sG$; note that in this case
$$\bMaps(Y,\on{pt}/\sG)\simeq \LocSys_\sG(Y).$$

\ssec{Proof of \thmref{t:integral to Maps}}

\sssec{}

We will first establish a few basic facts about quasi-passable prestacks.

\begin{lem}\label{l:quasi-passable affine}
Suppose that $f: \CZ_1 \to \CZ_2$ is an affine morphism, where $\CZ_2$ is quasi-passable.  Then $\CZ_1$
is quasi-passable.
\end{lem}
\begin{proof}
By the standard argument, $\CZ_1$ has affine diagonal; namely, by assumption, the maps $\CZ_1 \to \CZ_2 \to
\CZ_2 \times \CZ_2$ and $\CZ_1 \times \CZ_1 \to \CZ_2 \times \CZ_2$ are affine and by the
2-out-of-3 propery of affine morphisms, we obtain that $\CZ_1 \to \CZ_1 \times \CZ_1$ is affine, as
desired.

\medskip

It remains to show that $\QCoh(\CZ_1)$
is dualizable.  By \cite[Chapter 3, Proposition 3.3.3]{GR1}, we have
$$ \QCoh(\CZ_1) \simeq f_*(\mathcal{O}_{Z_1})\on{-mod}(\QCoh(\CZ_2)) .$$
The lemma follows from the following general fact (applied to the monad given by tensoring with
$f_*(\mathcal{O}_{Z_1})$).
\end{proof}

\begin{lem}
Let $\bC$ be a dualizable DG category and let $M: \bC \to \bC$ be a monad such that the underlying endo-functor of
$\bC$ preserves colimits.  Then the category $M\mod(\bC)$ is dualizable.
\end{lem}

\begin{proof}

It is easy to see that $M^\vee\mod(\bC^\vee)$ provides a dual for $M\mod(\bC)$.

\end{proof}

%
%
%
%
%

\sssec{}

As a consequence, we have following persistence property of quasi-passable prestacks:

\begin{cor}\label{c:quasi-passable pullback}
Suppose that we have a pullback diagram
$$
\xymatrix{\CZ \ar[r]\ar[d] & \CZ_2 \ar[d]^f \\
\CZ_1 \ar[r] & \CZ_0}
$$
such that $\CZ_0, \CZ_1$ and $\CZ_2$ are quasi-passable and the map $f:\CZ_2 \to \CZ_0$ is affine.  Then
$\CZ$ is quasi-passable and the functor
$$ \QCoh(\CZ_1) \underset{\QCoh(\CZ_0)}{\otimes} \QCoh(\CZ_2) \to \QCoh(\CZ) $$
is an equivalence.
\end{cor}
\begin{proof}
By base change, the map $\CZ \to \CZ_1$ is affine and therefore $\CZ$ is quasi-passable by \lemref{l:quasi-passable affine}.  Moreover, by base change and \cite[Chapter 3, Proposition 3.3.3]{GR1}, we have
$$ \QCoh(\CZ) \simeq f_*(\mathcal{O}_{\CZ_2})\on{-mod}(\QCoh(\CZ_1)) \simeq f_*(\mathcal{O}_{\CZ_2})\on{-mod}(\QCoh(\CZ_0)) \underset{\QCoh(\CZ_0)}{\otimes} \QCoh(\CZ_1)  \simeq$$
$$ \simeq \QCoh(\CZ_1) \underset{\QCoh(\CZ_0)}{\otimes} \QCoh(\CZ_2) ,$$
as desired.
\end{proof}

\sssec{}
We are now ready to prove the \thmref{t:integral to Maps}.  We first reduce to the case that $Y$ is connected.
Indeed, by \cite[Chapter 3, Theorem 3.1.7]{GR1} if $\CX_1$ and $\CX_2$ are quasi-passable prestacks, so is $\CX_1 \times \CX_2$ and the
functor
\begin{equation}\label{e:quasi-pass tensor prod}
\QCoh(\CX_1) \otimes \QCoh(\CX_2) \to \QCoh(\CX_1 \times \CX_2)
\end{equation}
is an equivalence.

\medskip

In particular, if $Y = Y_1 \sqcup Y_2$, we have
$$ \bMaps(Y_1 \sqcup Y_2, \CZ) \simeq \bMaps(Y_1, \CZ) \times \bMaps(Y_2, \CZ).$$
Therefore, if \thmref{t:integral to Maps} is true for $Y_1$ and $Y_2$, it is true for $Y$ (using the fact that the coproduct in $\DGCat^{\on{SymMon}}$ is given by tensor product).

\sssec{}
We can thus assume without loss of generality that $Y$ is connected. In this case, we will prove a slightly more precise version of the theorem (which implies \thmref{t:integral to Maps} by \lemref{l:quasi-passable affine}):

\begin{thm}\label{t:integral to Maps connected}
Let $(Y,y)$ be a pointed connected space and $\CZ$ a quasi-passable prestack.  Then
the evaluation map
$$ ev_y: \bMaps(Y, \CZ) \to \CZ $$
is affine and the functor \eqref{e:integral to Maps} is an equivalence.
\end{thm}

\sssec{}
By \propref{p:express connected} below, we can write $(Y,y)$ as a sifted colimit
$$(Y,y) = \underset{\alpha\in A}{\on{colim}} (Y_\alpha,y )$$
in $\Spc_{\on{Ptd}}$ (and therefore $\Spc$)
of objects of the form
$$(Y_\alpha, y) = \left(\{*\}\simeq \{*\}\underset{\{*\}}\sqcup\, \{*\}\to \{*\}\underset{I\sqcup \{*\}}\sqcup\, \{*\} \simeq
\underset{I}\vee\, S^1 \right), \quad I\in \on{fSet}.$$

This gives an isomorphism of prestacks
$$ \bMaps(Y, \CZ) \simeq \underset{\alpha\in A}{\on{lim}}\  \bMaps(Y_\alpha, \CZ).$$

\thmref{t:integral to Maps connected} now follows from the following two assertions:

\begin{lem} \label{l:quasi-affine 1}
\thmref{t:integral to Maps connected} is true for $Y=\underset{I}\vee\, S^1$, $I\in \on{fSet}$.
\end{lem}

\begin{lem} \label{l:quasi-affine 2}
Let $a\mapsto \CW_a$ be a sifted diagram of prestacks, affine over some base prestack $\CZ$; set
$$\CW:=\underset{a}{\on{lim}}\, \CW_a.$$
Then the map $\CW \to \CZ$ is affine and
the induced functor
$$\underset{a}{\on{colim}}\, \QCoh(\CW_a)\to \QCoh(\CW)$$
is an equivalence.
\end{lem}

\sssec{Proof of \lemref{l:quasi-affine 1}}
Since $\underset{I}\vee\,  S^1 \simeq \{*\}\underset{I\sqcup \{*\}}\sqcup\, \{*\}$, we have a pullback diagram of prestacks
$$
\xymatrix{
\bMaps(\underset{I}\vee\,  S^1, \CZ) \ar[r]\ar[d] & \CZ \ar[d] \\
\CZ \ar[r] & \CZ^{I\sqcup \{*\}}
}$$
Since $\CZ$ is quasi-passable, so is $\CZ^{I\sqcup \{*\}}$ and the diagonal map $\CZ \to \CZ^{I\sqcup \{*\}}$
is affine.  The result then follows by \corref{c:quasi-passable pullback} (using the fact the pushouts in $\DGCat^{\on{SymMon}}$ are given by tensor products).

\qed

\sssec{Proof of \lemref{l:quasi-affine 2}}
The assertion that $\CW \to \CZ$ is affine follows from the fact a limit of affine schemes is
an affine scheme and limits commute with pullbacks.

\medskip

Now, let $\CA_a$ (resp., $\CA$) be the object of $\on{ComAlg}(\QCoh(\CZ))$ equal to the image of the structure
sheaf on $\CW_\alpha$ (resp., $\CW$). Then (e.g., by \cite[Chapter 3, Prop. 3.3.3]{GR1}), direct image induces
an equivalence
$$\QCoh(\CW_a)\simeq \CA_a\mod(\QCoh(\CZ)) \text{ and } \QCoh(\CW)\simeq \CA\mod(\QCoh(\CZ)).$$

Now the assertion of the lemma follows from \lemref{l:alg colimit} below, combined with the fact that the forgetful functor
$$\on{ComAlg}\to \on{AssocAlg}$$
commutes with sifted colimits.

\qed

\ssec{Functors out of $\CA^{\otimes Y}$ as diagrams parameterized by finite sets}

In this subsection we will derive some corollaries of \thmref{t:integral}.

\sssec{}

Let $\bC$ be one of the three categories from \thmref{t:integral}, and let $\CC$ be an object of $\bC$. I.e.,
$\CC$ can be a symmetric monoidal DG category, or a monoidal DG category, or a plain DG category.

\medskip

Consider the following two functors
$$\on{fSet}\to \bC,$$
\begin{equation} \label{e:functor 1}
I\mapsto \CA^{\otimes I};
\end{equation}
and
\begin{equation} \label{e:functor 2}
I\mapsto \CC\otimes \on{LS}(Y^I).
\end{equation}

We claim:

\begin{prop} \label{p:describe functors}
The space of maps $\CA^{\otimes Y}\to \CC$ within $\bC$ (i.e., functors $\CA^{\otimes Y}\to \CC$ that preserve the corresponding structure)
is canonically isomorphic to the space of natural transformations from \eqref{e:functor 1} to \eqref{e:functor 2}.
\end{prop}

\begin{proof}

Recall (see \lemref{l:twisted}) that given a pair of functors
$$\Phi_1,\Phi_2:\bD\to \bC,$$
the space of natural transformations $\Phi_1\to \Phi_2$ identifies with the limit
$$\underset{(\bd_s\to \bd_t)\in \on{TwArr}(\bD)^{\on{op}}}{\on{lim}}\, \Maps_\bC(\Phi_1(\bd_s),\Phi_2(\bd_t)).$$

Hence, we obtain that the space of natural transformations from \eqref{e:functor 1} to \eqref{e:functor 2}
identifies canonically with
$$\underset{(I\to J)\in \on{TwArr}(\on{fSet})^{\on{op}}}{\on{lim}}\, \Maps_\bC(\CA^{\otimes I},\CC\otimes \on{LS}(Y^J)).$$

Applying \corref{c:LocSys as right adj}, we rewrite this as
$$\underset{(I\to J)\in \on{TwArr}(\on{fSet})^{\on{op}}}{\on{lim}}\,
\Maps_{\bC}(\underset{Y^J}\int\, \CA^{\otimes I},\CC).$$

Now the assertion follows from \thmref{t:integral}, which says that
$$\underset{(I\to J)\in \on{TwArr}(\on{fSet})}{\on{colim}}\,
\underset{Y^J}\int\, \CA^{\otimes I} \simeq \CA^{\otimes Y}.$$

\end{proof}

\sssec{}  \label{sss:monoidal actions}

As a first application of \propref{p:describe functors}, we will describe what it takes to have an action
of $\CA^{\otimes Y}$ on some DG category $\CM$.

\medskip

We take $\bC=\DGCat^{\on{Mon}}$. Set $\CC:=\End(\CM)$; this is a monoidal category. An action of
$\CA^{\otimes Y}$ on $\CM$ is the same as a map of \emph{monoidal categories}
$$\CA^{\otimes Y}\to \End(\CM).$$

According to \propref{p:describe functors}, the latter amounts to a system of monoidal functors
\begin{equation} \label{e:action I}
\CA^{\otimes I}\to \End(\CM)\otimes \LS(Y^I), \quad I\in \on{fSet},
\end{equation}
compatible in the sense that they make the following diagrams commute
\begin{equation} \label{e:compat I J}
\CD
\CA^{\otimes I}   @>>>   \End(\CM)\otimes \LS(Y^I)  \\
@VVV   @VVV   \\
\CA^{\otimes J}   @>>>   \End(\CM)\otimes \LS(Y^J)
\endCD
\end{equation}
for every $I\to J$, along with a system of higher compatibilities.

\begin{rem} \label{r:mon vs sym}
Let us note the difference between the notion of action of $\CA^{\otimes Y}$ on $\CM$, and that of action
of $\underset{Y}\int\, \CA$ on $\CM$, where $\underset{Y}\int\,-$ is understood \emph{within} $\DGCat^{\on{Mon}}$:

\medskip

The former contains the data of the homomorphisms \eqref{e:action I} for all $I$ with the compatibilities expressed
by diagrams \eqref{e:compat I J}. The latter contains just the data of homomorphism
$$\CA\to \End(\CM)\otimes \LS(Y),$$
i.e., the case $I=\{*\}$.

\end{rem}

\sssec{Example of $Y=S^1$}

Let us elaborate further on Remark \ref{r:mon vs sym} in the case $Y=S^1$. Note that by identifying
$$S^1=\{*\}\underset{\{*\}\sqcup \{*\}}\sqcup \{*\},$$
we obtain that $\LS(S^1)$ identifies with the category of vector spaces equipped with an automorphism.

\medskip

The monoidal category $\End(\CM)\otimes \LS(S^1)$ is that of endofunctors of $\CM$ equipped with an automorphism.
Hence, an action of $\underset{S^1}\int\, \CA$ (with $\int$ understood as taking place inside $\DGCat^{\on{Mon}}$)
on $\CM$ is an assignment
$$a\in \CA \mapsto (H_a,\on{mon}_a),$$
where $H_a$ is an endofunctor of $\CM$ equipped with an endomorphism $\on{mon}_a$. These data must be compatible in the sense that for
$m\in \CM$ the diagram

\smallskip

\begin{equation} \label{e:tensor compat}
\CD
H_{a_1}\circ H_{a_2}(m)  @>{H_{a_1}(\on{mon}_{a_2}|_{H_{a_2}(m)})}>> H_{a_1}\circ H_{a_2}(m)
@>{\on{mon}_{a_1}|_{H_{a_1}(H_{a_2}(m))}}>> H_{a_1}\circ H_{a_2}(m)  \\
@V{\sim}VV & & @V{\sim}VV  \\
H_{a_1\otimes a_2}(m)  & @>{\on{mon}_{a_1\otimes a_2}|_{H_{a_1\otimes a_2}(m)}}>>  & H_{a_1\otimes a_2}(m)
\endCD
\end{equation}
should commute, along with higher compatibilities. In the case when $\CM$ and $\CA$ are
derived categories of abelian categories and the action is t-exact, the commutativity of the diagrams
\eqref{e:tensor compat} for objects in the heart is sufficient to ensure the higher compatibilities.

\medskip

By contrast, an action of $\CA^{\otimes S^1}$ on $\CM$ requires additional compatibilities. The first of these is that under the identification
$$H_{a_1}\circ H_{a_2}(m) \simeq H_{a_1\otimes a_2}(m) \simeq H_{a_2\otimes a_1}(m) \simeq H_{a_2}\circ H_{a_1}(m),$$
the automorphism
$$\on{mon}_{a_1}|_{H_{a_1}(H_{a_2}(m))}$$
of the LHS should correspond to the automorphism
$$H_{a_2}(\on{mon}_{a_1}|_{H_{a_1}(m)})$$
of the RHS. Again, in the situation arising from abelian categories, this condition is sufficient for the higher compatibilities.

\ssec{Objects in and functors out of $\CA^{\otimes Y}$}

\sssec{}  \label{sss:functors out of integral}

Let us now take $\bC=\DGCat$. For a target DG category $\CC$, \propref{p:describe functors} describes what it takes
to construct a functor
$$\CS_Y:\CA^{\otimes Y}\to \CC.$$

Namely, such a datum is equivalent to a natural transformation between the functors
\begin{equation} \label{e:functor one functor}
I\mapsto \CA^{\otimes I}
\end{equation}
and
\begin{equation} \label{e:functor two functor}
I\mapsto \CC\otimes \LS(Y^I).
\end{equation}

\medskip

In other words, this is a data of functors
\begin{equation} \label{e:functor I}
\CS_I:\CA^{\otimes I}\to \CC \otimes \LS(Y^I), \quad I\in \on{fSet},
\end{equation}
that make the following diagrams commute
\begin{equation} \label{e:compat I J functor}
\CD
\CA^{\otimes I}   @>{\CS_I}>>   \CC \otimes \LS(Y^I)  \\
@VVV   @VVV   \\
\CA^{\otimes J}   @>{\CS_J}>>   \CC \otimes \LS(Y^J)
\endCD
\end{equation}
for every $I\to J$, along with a system of higher compatibilities.

\sssec{}

We now make the following observation (which is a special case of \propref{p:colimit of rigid cats} below
applied to the constant diagram $Y \to \DGCat^{\on{SymMon}}$).

\medskip

Recall the notion of a \emph{rigid} monoidal category, see \secref{sss:rigid}.

\begin{prop} \label{p:rigidity}
Let $\CA$ be a rigid monoidal DG category.  Then $\CA^{\otimes Y}$ is also rigid for any $Y \in \Spc$.
Moreover, if $\CA$ is compactly generated so is $\CA^{\otimes Y}$.
\end{prop}

%
%
%
%
%
%
%
%
%
%
%

Applying \cite[Chapter 1, Sect. 9.2.1]{GR1}, we obtain:

\begin{cor}  \label{c:rigidity}
Assume that $\CA$ is rigid. Then the functors
$$\CA^{\otimes Y}\otimes \CA^{\otimes Y}\overset{\on{mult}}\longrightarrow \CA^{\otimes Y} \overset{\CHom(\one_{\CA^{\otimes Y}},-)}\longrightarrow \Vect$$
and
$$\Vect \overset{\one_{\CA^{\otimes Y}}}\longrightarrow \CA^{\otimes Y}\overset{\on{mult}^R}\longrightarrow \CA^{\otimes Y}\otimes \CA^{\otimes Y},$$
where the first functor in the first line is given by the monoidal operation, and the second functor in the
second line is its right adjoint, define an identification
$$(\CA^{\otimes Y})^\vee\simeq \CA^{\otimes Y}.$$
\end{cor}

\sssec{}  \label{sss:objects in integral}

Taking $\CC=\Vect$, and applying \corref{c:rigidity}, we obtain that
the category
$$\on{Funct}(\CA^{\otimes Y},\Vect)\simeq (\CA^{\otimes Y})^\vee$$
is canonically equivalent to $\CA^{\otimes Y}$ itself via the map $r\mapsto\CHom_{\CA^{\otimes Y}}(\one_{\CA^{\otimes Y}},r\otimes -)$.

\medskip

Note that in the description of \secref{sss:functors out of integral}, for
$$\CS_{\on{univ}}\in \CA^{\otimes Y},$$
the corresponding functor
$$\CS_Y:\CA^{\otimes Y}\to \Vect$$
is given in terms of \eqref{e:functor I} by
the family of functors
\begin{equation} \label{e:SI}
\CS_I:\CA^{\otimes I}\to \on{LS}(Y^I)
\end{equation}
that are described as follows:

\medskip

The datum of \eqref{e:SI} is equivalent to the map from $\Maps(I,Y)\simeq Y^I$ to space of functors
\begin{equation} \label{e:SI adj}
\CA^{\otimes I}\to \Vect.
\end{equation}

For a point $\ul{y}\in Y^I\simeq \Maps(I,Y)$ and $r\in \CA^{\otimes I}$, denote by $r_{\ul{y}} $ the corresponding object of $\CA^{\otimes Y}$.
By unwinding the definitions, we obtain that \eqref{e:SI adj} sends
$$r\mapsto \CHom_{\CA^{\otimes Y}}(\one_{\CA^{\otimes Y}},r_{\ul{y}}\otimes \CS_{\on{univ}}).$$

\sssec{}  \label{sss:objects in integral vac}

In particular, under the correspondence of \secref{sss:objects in integral}, the object
$$\CS_\emptyset \in \on{Funct}(\CA^{\otimes \emptyset},\on{LS}(Y^{\emptyset}))\simeq \on{Funct}(\Vect,\Vect)\simeq \Vect$$
identifies with
$$\CHom(\one_{\CA^{\otimes Y}},\CS_{\on{univ}}).$$

\section{Excursions} \label{s:excurs}

The goal of this section is to give a (more explicit) description of the category $\CA^{\otimes Y}$ in the case
when $Y$ is a pointed connected space, and, more substantially, of the endomorphism algebra of its unit object.
The latter will be described in terms of what we will call, following V.~Lafforgue, ``excursion operators".

\ssec{Description of $\CA^{\otimes Y}$ via the fundamental group}

\sssec{}

Let $\BE_1(\Spc)$ denote the category of monoid objects in $\Spc$
(a.k.a. $\BE_1$-objects in $\Spc$ with respect to the Cartesian monoidal structure).
I.e., these are objects of $\Spc$ equipped with a binary operation and a unit that satisfy the axioms
of monoid up to coherent homotopy.

\medskip

We have an adjoint pair (see \cite[Example 3.1.3.6, Lemma 3.2.2.6 and Prop. 3.2.3.1]{Lu2})
\begin{equation} \label{e:free E1}
\free_{\BE_1}:\Spc\rightleftarrows \BE_1(\Spc):\oblv_{\BE_1},
\end{equation}
where $\oblv_{\BE_1}$ preserves sifted colimits and is conservative. 
Explicitly, $$\free_{\BE_1}(Y)=\underset{n}\sqcup\, Y^n.$$

\sssec{}

Recall also that the category $\BE_1(\Spc)$ is connected by a pair of mutually adjoint functors
to the category of pointed spaces
$$\on{Spc}_{\on{Ptd}}:=\Spc_{\{*\}/},$$
$$B:\BE_1(\Spc)\rightleftarrows \on{Spc}_{\on{Ptd}}:\Omega.$$

This adjunction factors as
$$\BE_1(\Spc)\rightleftarrows
\on{Spc}_{\on{cnctd,Ptd}} \rightleftarrows \on{Spc}_{\on{Ptd}},$$
where $\on{Spc}_{\on{cnctd,Ptd}}\subset \on{Spc}_{\on{Ptd}}$ is the full subcategory of connected pointed spaces.

\medskip

The functor $\Omega|_{\on{Spc}_{\on{cnctd,Ptd}}}$ is fully faithful, and its essential image consists of
group-like objects in $\BE_1(\Spc)$, i.e., those objects $Y$, for which $\pi_0(Y)$, equipped with
a monoid structure induced by that on $Y$, is actually a group (see \cite[Theorem 5.2.6.10]{Lu2}
for an $\infty$-categorical account of this basic fact, which is originally due to Stasheff \cite{St}).

\sssec{}

Let $\on{FFM}$ denote the \emph{full} subcategory of $\BE_1(\Spc)$ equal to the essential
image of the functor $\free_{\BE_1}$ applied to the full subcategory
$$\on{fSet}\subset \Spc.$$

Note that the category $\on{FFM}$ is discrete (a.k.a. ordinary). Its objects are the usual
free finitely generated monoids, and morphisms are morphisms between such.

\sssec{}

The following assertion is probably well-known, but we will supply a proof for completeness:

\begin{prop} \label{p:express connected}
For $Y\in \on{Spc}_{\on{cnctd,Ptd}}$, the slice category $\on{FFM}_{/\Omega(Y,y)}$ is sifted,
and the natural map
\begin{equation} \label{e:spaces as colim}
\underset{H\in \on{FFM}_{/\Omega(Y,y)}}{\on{colim}}\, B(H)\to Y
\end{equation}
is an isomorphism, where the colimit is taken in the category of pointed spaces.
\end{prop}

As an immediate corollary we obtain:
\begin{cor} \label{c:express connected}
For $Y\in \on{Spc}_{\on{cnctd,Ptd}}$, we have a canonical identification
$$\CA^{\otimes Y}\simeq \underset{H\in \on{FFM}_{/\Omega(Y,y)}}{\on{colim}}\, \CA^{\otimes B(H)}.$$
\end{cor}

\begin{proof}
Since the functor
$$Y\mapsto \CA^{\otimes Y}$$
preserves colimits, we only need to show that the isomorphism \eqref{e:spaces as colim} remains valid if the colimit is understood
as taking place in $\Spc$ rather than $\Spc_{\on{Ptd}}$. However, this is always the case any time the index category
is contractible, which in our case it is: indeed, it is sifted, and any sifted category is contractible (see \cite[Prop. 5.5.8.7]{Lu1}).
\end{proof}

In \secref{ss:endo} we will use \corref{c:express connected} to give an explicit description of
the algebra $\End_{\CA^{\otimes Y}}(\one_{\CA^{\otimes Y}})$.

\ssec{Proof of \propref{p:express connected}}

\sssec{}

We will prove the following:

\begin{prop}  \label{p:monoid as colim}
For every $H_0\in \BE_1(\Spc)$, the slice category
$\on{FFM}_{/H_0}$ is sifted, and the map
\begin{equation} \label{e:monoid as colim}
\underset{H\in \on{FFM}_{/H_0}}{\on{colim}}\, H\to H_0
\end{equation}
is an isomorphism.
\end{prop}

\propref{p:monoid as colim} implies \propref{p:express connected}: take
$H_0=\Omega(Y,y)$ for $Y\in \on{Spc}_{\on{cnctd,Ptd}}$, and apply the functor $B$ to \eqref{e:monoid as colim}.

\begin{rem}
Instead of monoids in \propref{p:monoid as colim}, we can just as well consider semi-groups; i.e. $\BE_1$ spaces
which are not required to have a unit element.  Indeed, by \cite[Corollary 5.4.3.6]{Lu2}, the forgetful functor
from grouplike $\BE_1$-algebras to $\BE_1$-semi-groups is fully faithful.  A slight adavantage of this is that
there are fewer maps between free semi-groups and therefore we obtain a slightly smaller indexing category.
\end{rem}

\sssec{}

The rest of this subsection is devoted to the proof of \propref{p:monoid as colim}. We claim:

\begin{lem} \label{l:Y1} \hfill

\smallskip

\noindent{\em(a)} The category $\on{FFM}$ has finite coproducts and the functor
$\on{FFM}\to \BE_1(\Spc)$ preserves finite coproducts.

\smallskip

\noindent{\em(b)} The Yoneda embedding $j:\BE_1(\Spc)\to \on{PShv}(\on{FFM})$
is conservative and preserves sifted colimits.
\end{lem}

In point (b) of the lemma, for a small category $\bC$ (in our case $\bC=\on{FFM}$),
we are using the notation $\on{PShv}(\bC)$ to denote all functors $\bC^{\on{op}}\to \on{Spc}$.

\sssec{}

Let us show how \lemref{l:Y1} implies \propref{p:monoid as colim}:

\medskip

By \lemref{l:Y1} and \corref{c:over sifted}, the category $\on{FFM}_{/H_0}$ is sifted.
Next, since $j$ is conservative it suffices to show that the map
$$j\left(\underset{H\in \on{FFM}_{/H_0}}{\on{colim}}\, H\right)\to j(H_0)$$
is an isomorphism.

\medskip

Since $j$ preserves sifted colimits, we have to show that the map
$$\underset{H\in \on{FFM}_{/H_0}}{\on{colim}}\, j(H)\to j(H_0)$$
is an isomorphism. But this follows from the Yoneda Lemma.

\qed[\propref{p:monoid as colim}]

\sssec{}

We now proceed to the proof of \lemref{l:Y1}. First, we state its analog for spaces:

\begin{lem} \label{l:Y2} The Yoneda embedding $j:\Spc\to \on{PShv}(\on{fSet})$
is conservative and preserves sifted colimits.
\end{lem}

\begin{proof}

To prove that $j$ preserves sifted colimits, its suffices to show that for every $I\in \on{fSet}$ the map
$\Maps_{\Spc}(I,-)$ preserves sifted colimits. But this is standard (and has been used already in \secref{sss:f step 2}).

\medskip

The functor $j$ is conservative, since it has a left inverse
$$\on{PShv}(\on{fSet})\to \on{PShv}(\{*\})=\Spc,$$ induced by the inclusion $\{*\}\to \on{fSet}$.

\end{proof}

Let us now prove \lemref{l:Y1}:

\begin{proof}

Point (a) follows from the fact that $\free_{\BE_1}$ is a left adjoint. We now prove point (b).

\medskip

Consider the functor $\free_{\BE_1}:\on{fSet}\to \on{FFM}$, and the corresponding pullback functor
$$\free_{\BE_1}^*:\on{PShv}(\on{FFM})\to \on{PShv}(\on{fSet}).$$
The functor $\free_{\BE_1}^*$ preserves all colimits (tautologically), and is conservative, since $\free_{\BE_1}$
is essentially surjective.

\medskip

Hence, it suffices to show that the composition:
$$\free_{\BE_1}^*\circ j:\BE_1(\Spc)\to \on{PShv}(\on{fSet})$$
is conservative and preserves sifted colimits. Since $\free_{\BE_1}$ is the left adjoint of $\oblv_{\BE_1}$,
the above composition factors as
$$j\circ \oblv_{\BE_1}:\BE_1(\Spc)\to \Spc\to \on{PShv}(\on{fSet}).$$

Since $\oblv_{\BE_1}$ is conservative and preserves sifted colimits (see \cite[Lemma 3.2.2.6 and Prop. 3.2.3.1]{Lu2}), the assertion follows from
\lemref{l:Y2}.

\end{proof}

\ssec{Digression: affine functors}

\sssec{}

Let $F:\CC_0\to \CC$ be a map monoidal categories, which admits a continuous right adjoint
as a map of plain DG categories. The functor $F^R$ has a natural right-lax monoidal structure;
hence the object
$$F^R(\one_\CC)\in \CC_0$$
has a natural structure of associative algebra in $\CC^0$ (see \cite[Cor. 7.3.2.7]{Lu2}).

\medskip

Moreover, the functor $F^R$ naturally upgrades to a functor
\begin{equation} \label{e:affine functor}
(F^R)^{\on{enh}}:\CC\to F^R(\one_\CC)\mod(\CC_0).
\end{equation}

\begin{defn}
We shall say that $F$ is \emph{affine} if \eqref{e:affine functor} is an equivalence.
\end{defn}

For example, if $f:\CZ_1\to \CZ_2$ is a quasi-affine map between prestacks, the functor
$$f^*:\QCoh(\CZ_2)\to \QCoh(\CZ_1)$$
is affine, see \cite[Chapter 3, Prop. 3.3.3]{GR1}.

\sssec{}

Here is a handy criterion of affineness (which follows immediately from the Barr-Beck-Lurie theorem):

\begin{lem} \label{l:crit aff}
A functor $F$ is affine if and only if its right adjoint $F^R$ has the following properties:

\smallskip

\noindent{\em(i)} It is continuous;

 \smallskip

\noindent{\em(ii)} It is conservative;

 \smallskip

\noindent{\em(iii)} The right-lax compatibility of $F^R$ with the action of $\CC_0$ is strict.

\end{lem}

\sssec{}

Let $\CA$ be a symmetric monoidal category.

\begin{defn}
We shall say that $\CA$ \emph{has an affine diagonal} if the tensor
product functor
$$\CA\otimes \CA\to \CA$$
is affine.
\end{defn}

\medskip

For example (assuming that $\sfe$ contains $\BQ$), for an algebraic group $\sG$, the category
$$\Rep(\sG)\simeq \QCoh(\on{pt}/\sG)$$
has affine diagonal, because the morphism
$$\on{pt}/\sG\to \on{pt}/\sG\times \on{pt}/\sG$$
is affine (here we are using the equivalence \eqref{e:quasi-pass tensor prod}).

\begin{rem} \label{r:rigid vs affine diagonal}
It is easy to see that a symmetric monoidal category $\CA$ is rigid (see \secref{sss:rigid} for what this means)
if and only if it has an affine diagonal and its unit object is compact. This follows immediately from \lemref{l:crit aff} above.
\end{rem}

\sssec{}

We can use the characterization of rigid symmetric monoidal categories in \remref{r:rigid vs affine diagonal} to prove that rigid symmetric monoidal
categories are stable under colimits:

\begin{prop}\label{p:colimit of rigid cats}
Let $I \to \DGCat^{\on{SymMon}}$ be a diagram of symmetric monoidal DG categories such that for each $i\in I$,
the corresponding category $\CC_i$ is rigid and for each morphism $i\to j \in I$, the corresponding functor
$\CC_i \to \CC_j$ admits a continuous right adjoint.  Then the colimit $\underset{i\in I}{\on{colim}}\  \CC_i$
is rigid.  Moreover, if each $\CC_i$ is compactly generated, so is $\underset{i\in I}{\on{colim}}\  \CC_i$.
\end{prop}

In the proof of \propref{p:colimit of rigid cats}, we will use the following basic fact about colimits:

\begin{prop}\label{p:colimit computation}
Let $\bC$ be a category that admits all colimits.  Suppose that we have a diagram $F: I \times [1] \to \bC$,
where $I$ is a contractible category:
$$
\xymatrix{
\ldots\ar[r] & X_i \ar[r]\ar[d] & X_j\ar[r]\ar[d] & \ldots \\
\ldots\ar[r] & Y_i \ar[r] & Y_j \ar[r] & \ldots
}
$$
Let $X = \on{colim} X_i$. Then the natural map
$$\underset{i\in I}{\on{colim}}\, Y_i \to \underset{i\in I}{\on{colim}}\, (X \underset{X_i}{\sqcup} Y_i)$$
is an isomorphism.
\end{prop}

\begin{proof}

Since push-outs commute with colimits, we have
$$\underset{i\in I}{\on{colim}}\, (X \underset{X_i}{\sqcup} Y_i) \simeq
(\underset{i\in I}{\on{colim}}\, X) \underset{\underset{i\in I}{\on{colim}}\, X_i}\sqcup\, (\underset{i\in I}{\on{colim}}\, Y_i ).$$

Now the assertion follows from the fact that
$$\underset{i\in I}{\on{colim}}\, X\simeq X,$$
since $I$ is contractible, so
$$(\underset{i\in I}{\on{colim}}\, X) \underset{\underset{i\in I}{\on{colim}}\, X_i}\sqcup\, (\underset{i\in I}{\on{colim}}\, Y_i )\simeq
X \underset{X}\sqcup\, (\underset{i\in I}{\on{colim}}\, Y_i )\simeq (\underset{i\in I}{\on{colim}}\, Y_i ).$$

\end{proof}

%
%
%

\begin{proof}[Proof of \propref{p:colimit of rigid cats}]
Since all colimits decompose as coproducts and sifted colimits, we can treat each case separately.
Suppose $\CC_0$ and $\CC_1$ are rigid symmetric monoidal DG categories.
Clearly, $\CC_0 \otimes \CC_1$ (which is the coproduct in $\DGCat^{\on{SymMon}}$) is also rigid.
Moreover, by \cite[Chapter 1, Proposition 7.4.2]{GR1} it is compactly generated if $\CC_0$ and $\CC_1$ are.

\medskip

Now, suppose that $I$ is sifted.  In this case, the underlying DG category of $\CC: = \underset{i\in I}{\on{colim}}\  \CC_i$
is the corresponding colimit in $\DGCat$.  In particular, for every $i\in I$, the functor
$$\on{ins}_i:\CC_i \to \CC$$ admits a continuous right adjoint and therefore preserves compact objects.  It follows,
that $\one_{\CC} \simeq \on{ins}_i(\one_{\CC_i})$ (for any $i\in I$) is compact.  Moreover,
if each $\CC_i$ is compactly generated, so is $\CC$.

\medskip

It remains to show that $\CC$ has affine diagonal.  Note that since $I$ is sifted, we have
$$ \underset{i\in I}{\on{colim}}\ (\CC_i \otimes \CC_i) \simeq \CC \otimes \CC.$$
Therefore, since pushouts in $\DGCat^{\on{SymMon}}$ are given by tensor products, we have by \propref{p:colimit computation},
$$ \CC \simeq \underset{i \in I}{\on{colim}}\ (\CC \otimes \CC) \underset{\CC_i \otimes \CC_i}{\otimes} \CC_i .$$

\medskip

For each $i\in I$, the tensor product functor $\CC_i \otimes \CC_i \to \CC_i$ is affine by assumption.  Therefore,
so is the functor
$$ \CC \otimes \CC \to (\CC \otimes \CC) \underset{\CC_i \otimes \CC_i}{\otimes} \CC_i .$$
The proposition now follows from the following general assertion.
\end{proof}

\begin{prop} \label{p:over C}
Let $\CC$ be a symmetric monoidal DG category.  The functor
$$ \on{ComAlg}(\CC) \to \DGCat^{\on{SymMon}}_{\CC/} \simeq \CC\on{-ComAlg}(\DGCat) $$
given by $A \mapsto A\on{-mod}(\CC)$ is fully faithful and commutes with colimits.
\end{prop}
\begin{proof}
The functor admits a right adjoint given by
$$(\CC\to \CalD)\mapsto \ul\Hom_{\CalD}(\one_\CalD,\one_\CalD),$$
where $\ul\Hom_{\CalD}(-,-)$ denotes internal Hom relative to the right action of $\CC$ on $\CalD$, i.e.,
$$\Hom_\CC(c,\ul\Hom_{\CalD}(d_1,d_2))=\Hom_{\CalD}(d_1\otimes c,d_2).$$
The fully faithfulness of the functor in question follows from the isomorphism
$$ \ul\Hom_{A\on{-mod}(\CC)}(\one_{A\on{-mod}(\CC)}, \one_{A\on{-mod}(\CC)}) \simeq A .$$
\end{proof}

\ssec{Endomorphisms of the unit object in $\CA^{\otimes Y}$ as a colimit} \label{ss:endo}

\sssec{}

Let us be in the situation of \lemref{l:lim colim}. Assume that $A$ has an initial object, denoted $a_0$.
Fix two objects $\bc'_{a_0},\bc''_{a_0}\in \CC_{a_0}$ with $\bc'_{a_0}$ compact. Set
$$\bc':=\on{ins}_{a_0}(\bc'_{a_0}),\,\,  \bc'':=\on{ins}_{a_0}(\bc''_{a_0}).$$

\medskip

We can form an $A$-family of objects of $\Vect$
$$a\mapsto \CHom_{\CC_a}(F_{a_0,a}(\bc'_{a_0}),F_{a_0,a}(\bc''_{a_0})),$$
and we have a natural map
\begin{equation} \label{e:colim Hom}
\underset{a}{\on{colim}}\, \CHom_{\CC_a}(F_{a_0,a}(\bc'_{a_0}),F_{a_0,a}(\bc''_{a_0}))\to \CHom_\CC(\bc',\bc'').
\end{equation}

\sssec{}

It is not difficult to show that when the index category $A$ is \emph{filtered}, then the map
\eqref{e:colim Hom} is an isomorphism (see \cite{Ro}). However, it is also easy to give an example when the map
\eqref{e:colim Hom} is a not an isomorphism.

\medskip

However, we claim:

\begin{thm} \label{t:endo}
Let $\CA$ have an affine diagonal. Then the map \eqref{e:colim Hom} is an isomorphism for the diagram of \corref{c:express connected}.
\end{thm}

A consequence of this theorem that will play a role for us in the sequel is the following:

\begin{cor} \label{c:endo}
Let $\CA$ have an affine diagonal, and assume that $\one_\CA\in \CA$ is compact.
Then for a connected $Y\in \Spc$ with a marked point $y\in Y$,
we have the following expression for $\CEnd_{\CA^{\otimes Y}}(\one_{\CA^{\otimes Y}})$:
$$\CEnd_{\CA^{\otimes Y}}(\one_{\CA^{\otimes Y}})\simeq
\underset{H\in \on{FFM}_{/\Omega(Y,y)}}{\on{colim}}\, \CEnd_{\CA^{\otimes B(H)}}(\one_{\CA^{\otimes B(H)}}).$$
\end{cor}

In \secref{ss:endo term} we will give an explicit description of the algebras $\CEnd_{\CA^{\otimes B(H)}}(\one_{\CA^{\otimes B(H)}})$
for $H\in \on{FFM}$, i.e., for $H$ of the form $\free_{\BE_1}(I)$ for $I\in\on{fSet}$.

\medskip

The next two subsections are devoted to the proof of \thmref{t:endo}.

\sssec{}

We will prove \thmref{t:endo} in the following general context:

\begin{prop} \label{p:endo affine}
Assume that in the situation of \eqref{e:colim Hom} the functor
$$a\mapsto \CC_a$$
comes from a functor $A\to \DGCat^{\on{Mon}}$. Assume that $A$ is sifted and all the functors $\CC_{a_0}\to \CC_a$
are affine. Then the map \eqref{e:colim Hom} is an isomorphism.
\end{prop}

\ssec{Proof of \propref{p:endo affine}}

\sssec{}

The assumption that the functors $\CC_{a_0}\to \CC_a$ are affine implies that the assignment
$$a\mapsto \CC_a,$$
viewed as a functor
$$A\to \on{DGCat}$$
canonically factors as
$$A\mapsto \on{AssocAlg}(\CC_{a_0}) \to  \on{DGCat},$$
where:

\medskip

\noindent--the first arrow is the functor $a\mapsto F_{a_0,a}^R(\one_{\CC_a})$;

\medskip

\noindent--the second arrow is the functor
$$R\mapsto R\mod(\CC_0), \quad (R_1\to R_2)\rightsquigarrow \left(M\mapsto R_2\underset{R_1}\otimes M\right).$$

\medskip

In particular, we obtain an $A$-diagram
$$a\mapsto R_a\in \on{AssocAlg}(\CC_{a_0}).$$

Set
$$R:=\underset{a}{\on{colim}}\, R_a,$$
where the colimit is taken in $\on{AssocAlg}(\CC_{a_0})$.

\medskip

We have a naturally defined functor
\begin{equation} \label{e:alg colimit}
\underset{a\in A}{\on{colim}}\, R_a\mod(\CC_{a_0})\to R\mod(\CC_{a_0}).
\end{equation}

\begin{lem} \label{l:alg colimit}
The functor \eqref{e:alg colimit} is an equivalence.
\end{lem}

The proof of the lemma is given below.

\sssec{}

Thus, returning to the proof of \propref{p:endo affine}, we need to show that for $M'_0,M''_0\in \CC_{a_0}$ with $M'_0$ compact, the map
$$\underset{a\in A}{\on{colim}}\, \Maps_{R_a\mod(\CC_{a_0})}(R_a\otimes M'_0,R_a\otimes M''_0)\to
\Maps_{R\mod(\CC_{a_0})}(R\otimes M'_0,R\otimes M''_0)$$
is an isomorphism. We rewrite the LHS as
$$\underset{a\in A}{\on{colim}}\, \Maps_{\CC_{a_0}}(M'_0,\oblv_{\on{Assoc}}(R_a)\otimes M''_0),$$
and the RHS as
$$\Maps_{\CC_{a_0}}(M'_0,\oblv_{\on{Assoc}}(R)\otimes M''_0).$$

Since, $M'_0$ was assumed compact, we rewrite the LHS further as
$$\Maps_{\CC_{a_0}}\left(M'_0,\underset{a\in A}{\on{colim}}\, \oblv_{\on{Assoc}}(R_a)\otimes M''_0 \right)\simeq
\Maps_{\CC_{a_0}}\left(M'_0,(\underset{a\in A}{\on{colim}}\, \oblv_{\on{Assoc}}(R_a)) \otimes M''_0 \right).$$

Now, the assumption that $A$ is sifted implies that
$$\underset{a\in A}{\on{colim}}\, \oblv_{\on{Assoc}}(R_a)\simeq
\oblv_{\on{Assoc}}(\underset{a\in A}{\on{colim}}\, R_a)=\oblv_{\on{Assoc}}(R),$$
and the assertion follows.

\qed[\propref{p:endo affine}]

\sssec{First proof of \lemref{l:alg colimit}}

Recall the equivalence of \lemref{l:lim colim}. The corresponding category $\wt\CC$ is the limit
$$\underset{a\in A}{\on{lim}}\, R_a\mod(\CC_{a_0}),$$
where for $a_1\to a_2$, the corresponding functor
$$R_{a_2}\mod(\CC_{a_0})\to R_{a_1}\mod(\CC_{a_0})$$
is given by restriction along $R_{a_1}\to R_{a_2}$.

\medskip

In terms of the equivalence
$$\CC\simeq \wt\CC,$$
the functor right adjoint to \eqref{e:alg colimit}, viewed as a functor
$$R\mod(\CC_{a_0})\to \underset{a\in A}{\on{lim}}\, R_a\mod(\CC_{a_0}),$$
is given by restrictions along $R_a\to R$.

\medskip

Hence, \lemref{l:alg colimit} is equivalent to the following statement:

\begin{lem} \label{l:alg limit}
Let $\CC$ be a monoidal category, in which the monoidal operation is compatible with colimits
in each variable. Then for a sifted colimit diagram of algebras in $\CC$,
$$\underset{a}{\on{colim}}\, R_a\simeq R,$$
the functor
\begin{equation} \label{e:alg limit}
R\mod(\CC)\to \underset{a}{\on{lim}}\, R_a\mod(\CC)
\end{equation}
is an equivalence.
\end{lem}

\begin{proof}[Proof of \lemref{l:alg limit}]

Note that functor \eqref{e:alg limit} is compatible with the natural forgetful functors
of both sides to $\CC$. We first show that \eqref{e:alg limit} induces an equivalence
between the fiber of both categories over a given object of $\CC$.

\medskip

For a pair of objects $M_1,M_2\in \CC$, let $\ul\Hom(M_1,M_2)\in \CC$ be their
internal $\Hom$, i.e.,
$$\Maps_{\CC}(\bc,\ul\Hom(M_1,M_2))\simeq \Maps(\bc\otimes M_1,M_2)$$
(this object exists by the Brown representability theorem, \cite[Prop. 5.5.2.2]{Lu1}).

\medskip

For $M_1=M_2=M$, the object $\ul\End(M):=\ul\Hom(M,M)$ has a structure of associative algebra.
For $R'\in  \on{AssocAlg}(\CC)$, the datum of action of $R'$ on $M$ is equivalent to a map
$$R'\to \ul\End(M)$$ in $\on{AssocAlg}(\CC)$ (see \cite[Cor. 4.7.1.40]{Lu2}).

\medskip

Hence, the datum of structure of $R$-module on $M$ is equivalent to that of a homomorphism $$R\to \ul\End(M),$$
and a compatible system of data of $R_a$-modules on $M$ is equivalent to that of a point in
$$\underset{\on{a}}{\on{lim}}\, \Maps_{\on{AssocAlg}(\CC)}(R_a,\ul\End(M)).$$

We rewrite the latter as
$$\Maps_{\on{AssocAlg}(\CC)}\left(\underset{\on{a}}{\on{colim}}\, R_a,\ul\End(M)\right),$$
which is by definition
$$\Maps_{\on{AssocAlg}(\CC)}(R,\ul\End(M)),$$
implying the assertion.

\medskip

The equivalence of fibers of the two sides of \eqref{e:alg limit} over a given $M$ implies that \eqref{e:alg limit}
induces an equivalence of the underlying groupoids. To show that \eqref{e:alg limit} is an equivalence of categories,
it remains to show that for $M_1,M_2\in R\mod$, the map
\begin{equation} \label{e:ff in limit}
\CHom_{R\mod(\CC)}(M_1,M_2)\to \underset{a\in A}{\on{lim}}\,\CHom_{R_a\mod(\CC)}(M_1,M_2)
\end{equation}
is an isomorphism.

\medskip

Note that for a given $R'\in  \on{AssocAlg}(\CC)$, and $M'_1,M'_2\in R'\mod(\CC)$, the object
$$\CHom_{R'\mod(\CC)}(M'_1,M'_2)$$ is calculated as the limit (a.k.a. totalization) of the co-Bar complex, whose terms
are
$$\CHom_\CC(\oblv_{\on{Assoc}}(R')^{\otimes n}\otimes M'_1,M'_2).$$

Note that, since $A$ is sifted, for each $n$, the map
$$\CHom_\CC(\oblv_{\on{Assoc}}(R)^{\otimes n}\otimes M'_1,M'_2) \to \underset{a\in A}{\on{lim}}\,\CHom_\CC(\oblv_{\on{Assoc}}(R_a)^{\otimes n}\otimes M_1,M_2)$$
is an isomorphism.

\medskip

Interchanging the order of $\underset{a\in A}{\on{lim}}$ and the totalization we obtain that \eqref{e:ff in limit}
is an isomorphism.

\end{proof}

\sssec{Second proof of \lemref{l:alg colimit}}

In fact, \lemref{l:alg colimit} is a corollary of the following more general assertion (see \corref{c:categ of modules} below):

\begin{lem}
Let $\CC$ be a monoidal DG category. Then the functor
$$\on{AssocAlg}(\CC)\to (\CC\mod^r(\DGCat))_{\CC/},\quad R\mapsto R\mod(\CC)$$
is fully faithful and commutes with colimits\footnote{In the above formula, the superscript ``$r$" stands for "right modules".}.
\end{lem}

\begin{proof}

Same as the proof of \propref{p:over C}.

\end{proof}

\begin{cor} \label{c:categ of modules}
Let $\CC$ be a monoidal DG category. Then the functor
$$\on{AssocAlg}(\CC)\to \DGCat,\quad R\mapsto R\mod(\CC)$$
commutes with colimits over contractible index categories.
\end{cor}

\begin{proof}

Follows from the fact that the forgetful functor
$$\CC\mod^r(\DGCat)\to \DGCat$$
commutes with colimits.

\end{proof}

\ssec{Applying the paradigm}

\sssec{}

In order to prove \thmref{t:endo}, it suffices to show that it fits into the paradigm of \propref{p:endo affine}.
Hence, it suffices to prove the following:

\medskip

\begin{prop} \label{p:bouquet}
Let $\CA$ have an affine diagonal. Then for a finite set $I$, the functor
$$\CA=\CA^{\otimes \{*\}}\simeq \CA^{\otimes B(\free_{\BE_1}(\emptyset))}\to
\CA^{\otimes B(\free_{\BE_1}(I))}$$
is affine.
\end{prop}

In the process of proving this proposition, we will describe what the category $\CA^{\otimes B(\free_{\BE_1}(I))}$
looks like.

\begin{rem}
Using the argument in the proof of \thmref{t:integral to Maps connected}, one can strengthen the assertion of \propref{p:bouquet} as follows:
for a map of spaces $Y_1\to Y_2$ which is surjective on $\pi_0$, the resulting map
$$\CA^{Y_1}\to \CA^{Y_2}$$
is affine.
\end{rem}

\sssec{}

Since both functors $B(-)$ and $\free_{\BE_1}$ are left adjoints, we have:
$$B(\free_{\BE_1}(I))\simeq \underset{I}{\vee}\, B(\free_{\BE_1}(\{*\})),$$
where $\vee$ means push-out with respect to the common base point. Since pushouts in $\on{DGCat}^{\on{SymMon}}$ are tensor products, we obtain:
$$\CA^{\otimes B(\free_{\BE_1}(I))}\simeq
\underset{I\text{-fold product}}{\underbrace{\CA^{\otimes B(\free_{\BE_1}(\{*\}))}\underset{\CA}\otimes...\underset{\CA}\otimes \CA^{\otimes B(\free_{\BE_1}(\{*\}))}}}.$$

\medskip

Hence, it suffices to show that $\CA^{\otimes B(\free_{\BE_1}(\{*\}))}$ is affine over $\CA$: indeed, for a pair of commutative algebras $A_1$ and $A_2$
in $\CA$, we have $$A_1\mod(\CA)\underset{\CA}\otimes A_2\mod(\CA)\simeq (A_1\otimes A_2)\mod(\CA).$$

\sssec{}

Note that $B(\free_{\BE_1}(\{*\}))\simeq S^1$. Indeed, by adjunction
\begin{multline} \label{e:B free}
\Maps_{\Spc_{\on{Ptd}}}(B(\free_{\BE_1}(\{*\})),Y)\simeq \Maps_{\BE_1(\Spc)}(\free_{\BE_1}(\{*\}),\Omega(Y,y))\simeq \\
\simeq \Maps_{\Spc}(\{*\},\Omega(Y,y))\simeq \Omega(Y,y)\simeq \Maps_{\Spc_{\on{Ptd}}}(S^1,Y).
\end{multline}

\sssec{}

Write $S^1$ as a push-out
$$\{*\}\underset{\{*\}\sqcup \{*\}}\sqcup\, \{*\}.$$

Hence,
$$\CA^{\otimes B(\free_{\BE_1}(\{*\}))}\simeq \CA^{\otimes S^1}\simeq \CA\underset{\CA\otimes \CA}\otimes \CA.$$

\sssec{}

Hence, it remains to show that if $\CA$ has an affine diagonal, then the functor
$$\CA\to \CA\underset{\CA\otimes \CA}\otimes \CA$$
is affine.

\medskip

Note that the above functor is obtained by base change $\CA\underset{\CA\otimes \CA}\otimes -$
from the functor $\CA\otimes \CA\to \CA$. Hence, it suffices to prove the following assertion:

\begin{lem}
Let $\CA_0\to \CA$ be an affine functor between symmetric monoidal categories.
Then for any diagram of symmetric monoidal categories
$$\CB' \leftarrow \CB\to \CA_0,$$
the resulting functor
$$\CB'\underset{\CB}\otimes \CA_0\to \CB'\underset{\CB}\otimes \CA$$
is also affine.
\end{lem}

\begin{proof}
Follows from the fact that for $R\in \on{AssocAlg}(\CA_0)$, we have
$$\CB'\underset{\CB}\otimes R\mod(\CA_0)\simeq R'\mod(\CB'\underset{\CB}\otimes \CA_0),$$
where $R'$ denotes the image of $R$ under
$$\CA_0\to \CB'\underset{\CB}\otimes \CA_0.$$

Indeed, the adjunction
$$(\on{id}\otimes \ind_R):\CB'\underset{\CB}\otimes \CA_0\rightleftarrows \CB'\underset{\CB}\otimes R\mod(\CA_0):
(\on{id}\otimes \oblv_R)$$
satisfies the assumption of the Barr-Beck-Lurie theorem (see \cite[Theorem 4.7.3.5]{Lu2}) and hence is monadic, with the monad given by $R'\otimes -$.
\end{proof}

\ssec{Endomorphisms of the unit, term-wise}  \label{ss:endo term}

\sssec{}  \label{sss:I+}

For a finite set $I$, let $I_+$ denote the pointed finite set
$$I_+=I\sqcup \{*\}.$$

Note that we have a canonical identification of pointed spaces
$$B(\free_{\BE_1}(I))\simeq \{*\}\underset{I_+}\sqcup \{*\}=:\Sigma(I_+).$$

\medskip

Indeed, by adjunction, for a pointed space $(Y,y)$, the following five pieces of data are equivalent:

\medskip

\noindent(i) A map of pointed spaces $\Sigma(I_+)\to (Y,y)$;

\medskip

\noindent(ii) A map of pointed spaces $(I_+,*)\to \Omega(Y,y)$;

\medskip

\noindent(iii) A map of spaces $I\to \Omega(Y,y)$;

\medskip

\noindent(iv) A map of $\BE_1$-objects $\free_{\BE_1}(I)\to \Omega(Y,y)$;

\medskip

\noindent(v) A map of pointed spaces $B(\free_{\BE_1}(I))\to (Y,y)$.

\sssec{}  \label{sss:map gamma I+}

In what follows, for $\gamma^I\in \Maps(I,\Omega(Y,y))$, we let $\gamma^{I_+}$
denote the corresponding pointed map $I_+\to \Omega(Y,y)$, i.e.,
$$\gamma^{I_+}=(\gamma^I,\gamma_{\on{triv}}).$$

\medskip

Thus, for each $\gamma^I$ we obtain a map $\Sigma(I_+)\to Y$
and the corresponding map of algebras, to be denoted
$$\gamma^{I_+}:\End_{\CA^{\otimes \Sigma(I_+)}}(\one_{\CA^{\otimes \Sigma(I_+)}})\to
\End_{\CA^{\otimes Y}}(\one_{\CA^{\otimes Y}}).$$

According to \corref{c:endo}, the resulting map from the colimit of the terms
\begin{equation} \label{e:terms iso}
\End_{\CA^{\otimes \Sigma(I_+)}}(\one_{\CA^{\otimes \Sigma(I_+)}})
\end{equation}
to $\End_{\CA^{\otimes Y}}(\one_{\CA^{\otimes Y}})$, is an isomorphism, provided that
$\CA$ has an affine diagonal.

\medskip

In this subsection we will describe the terms \eqref{e:terms iso} more explicitly, assuming that $\CA$ is \emph{rigid}
(note that according to Remark \ref{r:rigid vs affine diagonal}, a rigid symmetric monoidal
category automatically has an affine diagonal).

\sssec{}

For a (not necessarily pointed) finite set $J$, consider commutative diagram
\begin{equation} \label{e:bouquet diag}
\CD
\CA^{\otimes J}  @>{\on{mult}_{J}}>> \CA \\
@V{\on{mult}_{J}}VV   @VV{\iota_t}V  \\
\CA  @>{\iota_s}>>  \CA^{\otimes \Sigma(J)},
\endCD
\end{equation}
where the two arrows $\iota_s,\iota_t:\CA\to \CA^{\otimes \Sigma(J)}$ correspond to the two maps
$$\{*\}\rightrightarrows \Sigma(J).$$

Note that since the functor \eqref{e:integral functor} preserves colimits,
\eqref{e:bouquet diag} is in fact a push-out square in $\DGCat^{\on{SymMon}}$.  Base change defines a natural transformation
\begin{equation} \label{e:base change}
\on{BC}_{J}:\on{mult}_{J}\circ (\on{mult}_{J})^R\to \iota_s^R\circ \iota_t.
\end{equation}

\begin{lem} \label{l:base change}
Assume that $\CA$ is rigid. Then the natural transformation $\on{BC}_{J}$ is an isomorphism.
\end{lem}

\begin{proof}

We will prove the assertion more generally for a push-out diagram of rigid symmetric monoidal categories
$$
\CD
\CA  @>>> \CA_2 \\
@VVV  @VVV  \\
\CA_1  @>>> \CA_1\underset{\CA}\otimes \CA_2.
\endCD
$$

Indeed, we rewrite the latter diagram as
$$
\CD
(\CA\otimes \CA)\underset{\CA\otimes \CA}\otimes \CA  @>>> (\CA\otimes \CA_2)\underset{\CA\otimes \CA}\otimes \CA \\
@VVV  @VVV  \\
(\CA_1\otimes \CA)\underset{\CA\otimes \CA}\otimes \CA @>>> (\CA_1\otimes \CA_2)\underset{\CA\otimes \CA}\otimes \CA.
\endCD
$$

Now, the assertion follows using \cite[Chapter 1, Lemma 9.3.6]{GR1} from the base change property of the diagram
$$
\CD
\CA\otimes \CA @>>> \CA\otimes \CA_2 \\
@VVV  @VVV  \\
\CA_1\otimes \CA @>>> \CA_1\otimes \CA_2.
\endCD
$$
Namely, the assertion of {\it loc.cit.} says that the right adjoint of a functor between module categories over a rigid
monoidal category, which is a priori only right-lax compatible with the action, is actually strictly compatible. Hence,
the resulting (commutative) diagram
$$
\CD
\CA\otimes \CA @<<< \CA\otimes \CA_2 \\
@VVV  @VVV  \\
\CA_1\otimes \CA @<<< \CA_1\otimes \CA_2
\endCD
$$
takes place in $(\CA\otimes \CA)\mmod$, and hence stays commutative after applying $-\underset{\CA\otimes \CA}\otimes \CA$.

\end{proof}

\sssec{}

Note that
$$\iota_s(\one_\CA)\simeq \one_{\CA^{\otimes \Sigma(J)}}\simeq \iota_t(\one_\CA).$$
Hence,
$$\CHom_\CA(\one_\CA,\iota_s^R\circ \iota_t(\one_\CA))\simeq
\CEnd_{\CA^{\otimes \Sigma(J)}}(\one_{\CA^{\otimes \Sigma(J)}}).$$

Therefore, as a corollary of \lemref{l:base change}, we obtain:

\begin{cor}  \label{c:endo term}
For $\CA$ rigid, we have a canonical isomorphism (in $\Vect$)
$$\CHom_{\CA}(\one_\CA,\on{mult}_{J}\circ \on{mult}_{J}^R(\one_\CA))\to
\CEnd_{\CA^{\otimes \Sigma(J)}}(\one_{\CA^{\otimes \Sigma(J)}}).$$
\end{cor}

\sssec{}

Note that the functor $\on{mult}_{J}$ is symmetric monoidal. Hence, the functor $\on{mult}^R_{J}$
has a natural right-lax symmetric monoidal structure. In particular, the object
$$(\on{mult}_{J})^R(\one_\CA)$$
has a natural structure of commutative algebra in $\CA^{\otimes J}$, and hence
$$ \on{mult}_{J}\circ (\on{mult}_{J})^R(\one_\CA)$$
has a natural structure of commutative algebra in $\CA$.

\medskip

Hence,
$$\Maps_{\CA}(\one_\CA, \on{mult}_{J}\circ (\on{mult}_{J})^R(\one_\CA))$$
acquires a structure of commutative algebra (in $\Vect$).

\sssec{}  \label{sss:xi ast}

For a pair of elements $\xi_1,\xi_2\in \Maps_{\CA}(\one_\CA, \on{mult}_{J}\circ \on{mult}_J^R(\one_\CA))$,
let us denote by $\xi_1\ast \xi_2$ their product in $\Maps_{\CA}(\one_\CA, \on{mult}_{J}\circ \on{mult}_J^R(\one_\CA))$.

\medskip

Explicitly, $\xi_1\ast\xi_2$ is given by the composition
\begin{multline*}
\one_\CA\simeq \one_\CA\otimes \one_\CA \overset{\xi_1\otimes \xi_2}\longrightarrow
\on{mult}_{J}\circ \on{mult}_J^R(\one_\CA)\otimes \on{mult}_{J}\circ \on{mult}_J^R(\one_\CA)\simeq \\
\simeq \on{mult}_J\left(\on{mult}_J^R(\one_\CA)\otimes \on{mult}_J^R(\one_\CA)\right)
\to \on{mult}_J\circ \on{mult}_J^R(\one_\CA\otimes \one_\CA)\simeq  \on{mult}_J\circ \on{mult}_J^R(\one_\CA).
\end{multline*}

\sssec{}

We now claim:

\begin{lem} \label{l:endo term alg}
The isomorphism of \corref{c:endo term} respects the algebra structures on the two sides.
\end{lem}

\begin{proof}

Since \eqref{e:bouquet diag} is a commutative diagram of symmetric monoidal DG categories, the base change
morphism \eqref{e:base change} respects the right-lax symmetric monoidal structures on the two sides. Hence,
the induced morphism
$$\CHom_{\CA}(\one_\CA,\on{mult}_J\circ \on{mult}_J^R(\one_\CA))\to
\CHom_\CA(\one_\CA,\iota_s^R \circ \iota_r(\one_\CA))\simeq
\CEnd_{\CA^{\otimes \Sigma(J)}}(\one_{\CA^{\otimes \Sigma(J)}})$$
respects the commutative algebra structures.

\medskip

Finally, the resulting algebra structure on $\CEnd_{\CA^{\otimes \Sigma(J)}}(\one_{\CA^{\otimes \Sigma(J)}})$
equals one given by the composition, by the Eckmann-Hilton argument.

\end{proof}

\sssec{}

In what follows for
\begin{equation} \label{e:xi elt}
\xi\in \Maps_{\CA}(\one_\CA, \on{mult}_{J}\circ \on{mult}_J^R(\one_\CA)),
\end{equation}
we let
\begin{equation} \label{e:E elt}
E_\xi\in \CEnd_{\CA^{\otimes \Sigma(J)}}(\one_{\CA^{\otimes \Sigma(J)}})
\end{equation}
denote the corresponding element.

\medskip

Let us describe explicitly the product operation on the elements $E_\xi$ of \eqref{e:E elt}. By \lemref{l:endo term alg},
we have
$$E_{\xi_1}\cdot E_{\xi_2}=E_{\xi_1\ast\xi_2},$$
where $\xi_1\ast\xi_2$
is the product of $\xi_1$ and $\xi_2$ in $\Maps_{\CA}(\one_\CA, \on{mult}_{J}\circ \on{mult}_J^R(\one_\CA))$,
which is in turn described explicitly in \secref{sss:xi ast}.

\ssec{Action on a module via excursions}  \label{ss:exc}

\sssec{}

Let us be given a functor of DG categories
$$\CS_Y:\CA^{\otimes Y}\to \CC.$$

In particular, the algebra $\CEnd_{\CA^{\otimes Y}}(\one_{\CA^{\otimes Y}})$  acts on the object
$\CS_Y(\one_{\CA^{\otimes Y}})$.

\medskip

Recall (see \secref{sss:map gamma I+}) that given a map
$$\gamma^I:I\to \Omega(Y,y),$$
we have a map
$$\gamma^{I_+}:\End_{\CA^{\otimes \Sigma(I_+)}}(\one_{\CA^{\otimes \Sigma(I_+)}})\to
\End_{\CA^{\otimes Y}}(\one_{\CA^{\otimes Y}}).$$

Recall that to an element
$$\xi\in \Maps_{\CA}(\one_\CA, \on{mult}_{J_+}\circ (\on{mult}_{I_+})^R(\one_\CA))$$
we associated an element
$$E_\xi\in \CEnd_{\CA^{\otimes \Sigma(I_+)}}(\one_{\CA^{\otimes \Sigma(J_+)}}).$$

In this subsection we will give an explicit formula for the action of the element
$$\gamma^{I_+}(E_\xi)\in \CEnd_{\CA^{\otimes Y}}(\one_{\CA^{\otimes Y}})$$
on $\CS_Y(\one_{\CA^{\otimes Y}})$.

\medskip

We will do so in a slightly more general context: instead of the pointed finite set $I_+$ we will
consider a non-pointed one.

\sssec{}

For an object $Y'\in \Spc$ and a point $y'\in Y'$ let $\on{ev}_{y'}$ denote the restriction functor
$$\on{LS}(Y')\to \Vect$$
corresponding to
$$\{*\} \overset{y'}\hookrightarrow Y'.$$

\medskip

Given two points $y'_1$ and $y'_2$ and a path $\gamma'$ between them, we have a natural transformation
\begin{equation} \label{e:parallel trans}
\on{mon}_{\gamma'}:\on{ev}_{y'_1}\to \on{ev}_{y_2'},
\end{equation}
given by ``parallel transport'' along $\gamma$. Here is the formal construction:

\medskip

By definition, a path $\gamma'$ is a point
of $\{*\}\underset{Y'}\times \{*\}$, where the two maps
$$\{*\}\rightrightarrows Y'$$
are given by $y'_1$ and $y'_2$, respectively.

\medskip

Restriction on $\LS(-)$ along the maps in the commutative square
$$
\CD
\{*\}\underset{Y'}\times \{*\} @>>> \{*\} \\
@VVV  @VV{y'_2}V  \\
\{*\}  @>{y'_1}>> Y',
\endCD
$$
followed by restriction along
$$\{*\} \overset{\gamma'}\to \{*\}\underset{Y'}\times \{*\},$$
defines the desired isomorphism of functors \eqref{e:parallel trans}.

\sssec{}  \label{sss:excurs}

Ley $y_1$ and $y_2$ be two points of $Y$. Let $J$ be a finite set, and let us be given a $J$-tuple
$\gamma^J$ of paths from $y_1$ to $y_2$. We can consider the points $y_1^J,y_2^J\in Y^J$
and regard $\gamma^J$ as a path from $y_1^J$ to $y_2^J$.

\medskip


Fix a point
$$\xi\in \Maps_{\CA}(\one_\CA, \on{mult}_J\circ \on{mult}_J^R(\one_\CA)).$$

\medskip

Recall that according to \propref{p:describe functors}, the datum of a functor $\CS_Y$ is equivalent to the
datum of a collection of functors
$$\CS_I:\CA^{\otimes I}\to \CC\otimes \LS(Y^I),$$
which depends functorially on $I\in \on{fSet}$.

\medskip

Define the \emph{excursion operator} $\on{Exc}_{\CS_Y}(\gamma^J,\xi)$ to be the following endomorphism
of $\CS_Y(\one_{\CA^{\otimes Y}})$
$$
\CD
\CS_Y(\one_{\CA^{\otimes Y}}) @>>{\sim}> \on{ev}_{y_1}\left(\CS_{\{*\}}(\one_\CA)\right) @>{\xi}>>
\on{ev}_{y_1}\left(\CS_{\{*\}}(\on{mult}_J\circ \on{mult}_J^R(\one_\CA))\right)  \\
& & & & @VV{\sim}V   \\
& & & & \on{ev}_{y^J_1}\left(\CS_J(\on{mult}_J^R(\one_\CA))\right)  \\
& & & & @VV{\on{mon}_{\gamma^J}}V  \\
& & & & \on{ev}_{y^J_2}\left(\CS_J(\on{mult}_J^R(\one_\CA))\right) \\
& & & & @VV{\sim}V   \\
\CS_Y(\one_{\CA^{\otimes Y}})  @<{\sim}<< \on{ev}_{y_2}\left(\CS_{\{*\}}(\one_\CA)\right) @<{\on{counit}}<<
\on{ev}_{y_2}\left(\CS_{\{*\}}(\on{mult}_J\circ \on{mult}_J^R(\one_\CA))\right),
\endCD
$$
where:

\begin{itemize}

\item the first and last isomorphisms are obtained by identifying $\one_{\CA^{\otimes Y}}$ with the image
of $\one_\CA$ under
$$\CA\simeq \CA^{\{*\}}\to \CA^Y,$$
where $\{*\}\to Y$ is $y_i$, $i=1,2$;

\medskip

\item the third and the fifth isomorphisms are obtained by functoriality with respect to the map $J\to \{*\}$
in $\on{fSet}$;



\end{itemize}

\sssec{}  \label{sss:univ exc}

In the particular case when $\CC=\CA^{\otimes Y}$ and $\CS_Y$ is the identity functor, we will denote the corresponding
$\on{Exc}_{\CS_Y}(\gamma^J,\xi)$ by
$$\on{Exc}_{\on{univ}}(\gamma^J,\xi)\in \CEnd_{\CA^{\otimes Y}}(\one_{\CA^{\otimes Y}}).$$

Explicitly, $\on{Exc}_{\on{univ}}(\gamma^J,\xi)$ is given by the composition
\begin{multline} \label{e:exc univ}
\one_{\CA^{\otimes Y}} \simeq \iota_{y_1}(\one_\CA) \overset{\xi}\to \iota_{y_1}\circ \on{mult}_J\circ \on{mult}^R_J (\one_\CA)
\simeq \iota_{y^J_1}\circ \on{mult}^R_J (\one_\CA) \overset{\on{mon}_{\gamma^J}}\longrightarrow \\
\to \iota_{y^J_2}\circ \on{mult}^R_J (\one_\CA) \simeq \iota_{y_2}\circ \on{mult}_J\circ \on{mult}^R_J (\one_\CA)
\overset{\on{counit}}\longrightarrow \iota_{y_2}(\one_\CA) \simeq \one_{\CA^{\otimes Y}},
\end{multline}
where $y_i$ and $y_i^J$ denote the functors
$$\CA\to \CA^Y \text{ and } \CA^J\to \CA^Y,$$
corresponding to
$$\{*\}\overset{y_i}\to Y \text{ and } J\overset{y^J_i}\to Y,$$
respectively.

\medskip

By functoriality, for a general $\CC$, the map
$\on{Exc}_{\CS_Y}(\gamma^J,\xi)$, is the image of $\on{Exc}_{\on{univ}}(\gamma^J,\xi)$
under $\CS_Y$.

\sssec{Example}

Let $\sfe$ be a field of characteristic $0$.
Take $\CA=\Rep(\sG)$, and assume that $Y$ has finitely many connected components, so that we have
an equivalence
$$\Rep(\sG)^{\otimes Y}\simeq \QCoh(\LocSys_\sG(Y)),$$
see \thmref{t:integral to LocSys}.

\medskip

Then for $(\gamma^J,\xi)$
as above, we can think of $\on{Exc}_{\on{univ}}(\gamma^J,\xi)$ as an element of
$$\Gamma(\LocSys_\sG(Y),\CO_{\LocSys_\sG(Y)}).$$

Let us describe this element explicitly for a particular (in fact, a generating family) of choices of $\xi$. Let $V_J$ be a representation
of $G^J$. Fix an invariant vector and an invariant covector in $\Res^{G^J}_G(V_J)$, i.e.,
$$v:\sfe\to  \Res^{G^J}_G(V_J) \text{ and } v^*:\Res^{G^J}_G(V_J)\to \sfe.$$

The datum of $v^*$ defines by adjunction a map
$$V_J\to \on{coInd}^{G^J}_G(\sfe)=\on{mult}_J^R(\sfe).$$
Let $\xi_{v,v^*}$ denote the composite
$$\sfe \overset{v}\to  \Res^{G^J}_G(V_J)=\on{mult}_J(V_J)
\overset{v^*}\to \on{mult}_J\circ \on{mult}_J^R(\sfe).$$

\medskip

Let us describe explicitly the element
$$\on{Exc}_{\on{univ}}(\gamma^J,\xi_{v,v^*})\in \Gamma(\LocSys_\sG(Y),\CO_{\LocSys_\sG(Y)})$$
as a function on $\LocSys_\sG(Y)$.

\medskip

Namely, the value of this function at a point $\sigma$ of $\LocSys_\sG(Y)$ is the composite
$$
\CD
\sfe  @>{v}>> \on{ev}_{y_1}((\Res^{G^J}_G(V_J))_\sigma) @>{\sim}>>  \on{ev}_{y_1^J}((V_J)_\sigma)  \\
& & & & @VV{\on{mon}_{\gamma^J}}V  \\
\sfe  @<{v^*}<< \on{ev}_{y_2}((\Res^{G^J}_G(V_J))_\sigma) @<{\sim}<< \on{ev}_{y_2^J}((V_J)_\sigma),
\endCD
$$
where for $W_{J'}\in \Rep(G^{J'})$ we denote by $(W_{J'})_\sigma$ the corresponding object
of $\on{LS}(Y^{J'})$.

\sssec{}

We are now ready to state the main result of this subsection:

\begin{thm} \label{t:excurs}
For $(y_1,y_2,\gamma^J)$ as above, consider the corresponding map
$$\gamma^J:\Sigma(J)\to Y.$$
Then the excursion operator
$$\on{Exc}_{\CS_Y}(\gamma^J,\xi)\in \End(\CS_Y(\one_{\CA^{\otimes Y}}))$$
of \secref{sss:excurs}
equals the action of the element of $\End_{\CA^{\otimes Y}}(\one_{\CA^{\otimes Y}})$ obtained as the image of $E_\xi$
(see \eqref{e:E elt}) under the map
$$\End_{\CA^{\otimes \Sigma(J)}}(\one_{\CA^{\otimes \Sigma(J)}}) \overset{\gamma^J}\to
\End_{\CA^{\otimes Y}}(\one_{\CA^{\otimes Y}}).$$
\end{thm}

\begin{rem}
We emphasize that the assertion of \thmref{t:excurs} holds for any $\CA$: we do not need either $\CA$ to have an
affine diagonal. We only use the \emph{existence} of the base change map
$$\CHom_{\CA}(\one_\CA,\on{mult}_{J}\circ \on{mult}_{J}^R(\one_\CA))\to
\CEnd_{\CA^{\otimes \Sigma(J)}}(\one_{\CA^{\otimes \Sigma(J)}}),$$
but we do not need this map to be an isomorphism.
\end{rem}

\ssec{Proof of \thmref{t:excurs}}

\sssec{}

First off, since the assertion is functorial in $\CC$, it suffices to consider the universal case, namely,
$\CC=\CA^{\otimes Y}$ and $\CS_Y$ is the identity functor. Second, since the statement is functorial in $Y$,
we can assume that
$$Y=\Sigma(J),$$
and $\gamma^J$ is the tautological $J$-tuple $\gamma^J_{\on{taut}}$ of paths
$$*_s \to *_t,$$
where $*_s,*_t$ are the two points of $\Sigma(J)$.

\medskip

Thus, we need to show that
\begin{equation} \label{e:identify Exc}
\on{Exc}_{\on{univ}}(\gamma^J_{\on{taut}},\xi)=E_\xi
\end{equation}
as objects in $\End_{\CA^{\otimes \Sigma(J)}}(\one_{\CA^{\otimes \Sigma(J)}})$.

\sssec{}

Let $\sfq$ denote the map
$$\CA^{\otimes J}\to \CA^{\otimes \Sigma(J)},$$
corresponding to either circuit in \eqref{e:bouquet diag}.

\medskip

The path $\gamma^J_{\on{can} }$ defines a $J$-tuple of isomorphisms of functors
$$\on{mon}^j_{\on{can} }:\iota_s\to \iota_t, \quad j\in J$$
so that the composite
$$\sfq\simeq \iota_s\circ \on{mult}_J\overset{\on{mon}^j_{\on{can}}}\longrightarrow \iota_t\circ \on{mult}_J\simeq \sfq$$
is the identity map for all $j\in J$.

\sssec{}

Unwinding the definitions, we obtain that the LHS of \eqref{e:identify Exc} is the endomorphism of $\one_{\CA^{\otimes \Sigma(J)}}$ given by
$$
\CD
\one_{\CA^{\otimes \Sigma(J)}} @>{\sim}>>
\iota_s(\one_\CA) @>{\xi}>> \iota_s\circ \on{mult}_{J}\circ (\on{mult}_{J})^R(\one_\CA)  @>{\sim}>>
\on{mult}_J\circ \iota_s^{\otimes J}\circ (\on{mult}_{J})^R(\one_\CA) \\
& & & & & & @VV{\underset{j\in J}\bigotimes\, \on{mon}^j_{\on{can} }}V \\
\one_{\CA^{\otimes \Sigma(J)}} @<{\sim}<<
\iota_t(\one_\CA) @<{\on{counit}}<<  \iota_t\circ \on{mult}_{J}\circ (\on{mult}_{J})^R(\one_\CA)  @<{\sim}<<
\on{mult}_J\circ \iota_t^{\otimes J}\circ (\on{mult}_{J})^R(\one_\CA),
\endCD
$$

\medskip

and the RHS is
$$
\one_{\CA^{\otimes \Sigma(J)}}
\simeq \iota_s(\one_\CA) \overset{\xi}\to  \iota_s\circ \on{mult}_{J}\circ (\on{mult}_{J})^R(\one_\CA)  \overset{\on{BC}_{J}}\longrightarrow
\iota_s\circ \iota_s^R\circ \iota_t(\one_\CA) \overset{\on{counit}}\longrightarrow \iota_t(\one_\CA)\simeq \one_{\CA^{\otimes \Sigma(J)}},
$$
where $\on{BC}_J$ is an is \eqref{e:base change}.

\sssec{}

We obtain that it suffices to show that the following two natural transformations
$$\iota_s\circ \on{mult}_{J}\circ (\on{mult}_{J})^R\rightrightarrows \iota_t$$
coincide:

\medskip

One is the composite
\begin{multline} \label{e:comp1'}
\iota_s\circ \on{mult}_{J}\circ (\on{mult}_{J})^R  \simeq \on{mult}_{J}\circ \iota_s^{\otimes J} \circ (\on{mult}_{J})^R
\overset{\underset{j\in J}\bigotimes \, \on{mon}^j_{\on{can}}}\simeq \on{mult}_{J}\circ \iota^{\otimes J}_t\circ (\on{mult}_{J})^R \simeq \\
\simeq \iota_t\circ \on{mult}_{J} \circ (\on{mult}_{J})^R \to \iota_t.
\end{multline}
and the other is the composite
\begin{equation} \label{e:comp2'}
\iota_s\circ \on{mult}_{J}\circ (\on{mult}_{J})^R  \overset{\on{BC}_{J}}\longrightarrow
\iota_s\circ \iota_s^R\circ \iota_t \to \iota_t.
\end{equation}

\sssec{}

However, unwinding the definition of $\on{BC}_J$, we obtain that both  \eqref{e:comp1'} and \eqref{e:comp2'} identify with
$$\iota_s\circ \on{mult}_{J}\circ (\on{mult}_{J})^R \simeq \sfq\circ (\on{mult}_{J})^R \simeq
\iota_t\circ \on{mult}_{J}\circ (\on{mult}_{J})^R \to \iota_t.$$

\qed[\thmref{t:excurs}]

\section{Taking the trace} \label{s:Frob}

In this section we will approach the central theme of this paper: the operation of taking the trace. The usual trace construction
assigns to an endomorphism $F$ of a dualizable object $\bo$ in a symmetric monoidal category $\bO$ its trace $\Tr(F,\bo)$, which
is an endomorphism of the unit object $\one_\bO$, see \secref{sss:basic trace}.

\medskip

However, our primary interest will be the notion of
\emph{higher trace}, when $\bO$ is actually a symmetric monoidal 2-category. In this case, the trace construction has an
additional functoriality, see \secref{sss:map of traces}. We will apply this formalism in the following two contexts: $\bO=\DGCat$,
in which our traces will be vector spaces, and $\bO=\tDGCat$, in which our traces will be DG-categories.

\medskip

The main result of this section is \thmref{t:two notions of trace}, which describes the
interaction between the trace operations at different categorical levels.

\ssec{The usual trace}

\sssec{}  \label{sss:basic trace}

Let $\bO$ be a symmetric monoidal category. Given a dualizable object $\bo\in \bO$ and a point
$F\in \End_\bO(\bo)$, we define its trace
$$\Tr(F,\bo)\in \End_\bO(\one_\bO)$$ to be the composite
$$\one_\bO \overset{\on{unit}}\longrightarrow \bo\otimes \bo^\vee \overset{F\otimes \on{id}_{\bo^\vee}}\longrightarrow
\bo\otimes \bo^\vee  \overset{\on{counit}}\longrightarrow \one_\bO.$$

\sssec{}  \label{sss:tr sym mon}

The assignment
$$(\bo,F)\mapsto \Tr(F,\bo)$$ is symmetric monoidal, i.e., we have a canonical isomorphism
$$\Tr(F_1\otimes F_2,\bo_1\otimes \bo_2)\simeq \Tr(F_1,\bo_1)\cdot \Tr(F_2,\bo_2),$$
where $\cdot$ denotes the structure of commutative monoid on $\End(\one_\bO)$ induced by
the symmetric monoidal structure on $\bO$, see \cite[Sect. 2.5]{TV}, along with higher compatibilities.

\medskip

In particular
\begin{equation} \label{e:Tr unit}
\Tr(\on{id}_{\one_\bO},\one_\bO)=\on{id}_{\one_\bO}=\one_{\End_\bO(\bo)}.
\end{equation}

\sssec{}

For a morphism $F:\bo_1\to \bo_2$ between dualizable objects let $F^\vee$ denote the dual morphism $\bo_2^\vee\to \bo_1^\vee$.

\medskip

Let $q_F$ be the point in $\Maps(\one_\bO,\bo_2\otimes \bo_1^\vee)$ that represents $F$, i.e., the composite
$$\one_\bO \overset{\on{unit}}\longrightarrow \bo_1\otimes \bo_1^\vee \overset{F\otimes \on{id}_{\bo_1^\vee}}\longrightarrow
\bo_2\otimes \bo_1^\vee.$$

We have
$$q_F=q_{F^\vee}.$$

From here it follows that for every $\bo$ and $F$ as in \secref{sss:basic trace}, we have
$$\Tr(F,\bo)=\Tr(F^\vee,\bo^\vee).$$

\sssec{} \label{sss:cyclicity}

Similarly, the trace map has the following cyclicity property: we claim that for morphisms
$$F_{1,2}:\bo_1\to \bo_2 \text{ and } F_{2,1}:\bo_2\to \bo_1,$$
there is a canonical isomorphism
\begin{equation} \label{e:cyclicity}
\Tr(F_{1,2}\circ F_{2,1},\bo_2) \simeq \Tr(F_{2,1}\circ F_{1,2},\bo_1).
\end{equation}

\medskip

Indeed, let $q_{i,j}\in \Maps(\one_\bO,\bo_j\otimes \bo_i^\vee)$ be the point that represents $F_{i,j}$.
Then $F_{i,j}\circ F_{j,i}$ is represented by the map
$$\one_\bO\simeq \one_\bO\otimes \one_\bO \overset{q_{i,j}\otimes q_{j,i}}\longrightarrow
\bo_j\otimes \bo_i^\vee \otimes \bo_i\otimes \bo_j^\vee\overset{\on{id}_{\bo_j}\otimes \on{counit}\otimes \on{id}_{\bo_j^\vee}}\longrightarrow
\bo_j\otimes \one_\bO \otimes \bo^\vee_j\simeq \bo_j\otimes \bo^\vee_j.$$

Hence,
$$\Tr(F_{i,j}\circ F_{j,i},\bo_j)$$ is the composite
$$
\one_\bO\simeq \one_\bO\otimes \one_\bO \overset{q_{i,j}\otimes q_{j,i}}\longrightarrow
\bo_j\otimes \bo_i^\vee \otimes \bo_i\otimes \bo_j^\vee\overset{\on{id}_{\bo_j}\otimes \on{counit}\otimes \on{id}_{\bo_j^\vee}}\longrightarrow
\bo_j\otimes \one_\bO \otimes \bo^\vee_j\simeq \bo_j\otimes \bo^\vee_j \overset{\on{counit}}\longrightarrow \one_\bO,$$
and the latter expression is manifestly symmetric in $i$ and $j$.

\ssec{Trace in a 2-category}  \label{ss:categ trace}

\sssec{}  \label{sss:map of traces}

Let now $\bO$ be a symmetric monoidal 2-category (we will be assuming the formalism of
$(\infty,2)$-categories from \cite[Chapter 10]{GR1}).

\medskip

Let $\bo_1$ and $\bo_2$ be a pair of dualizable objects, each endowed
with an endomorphism $F_i$. Let $t:\bo_1\to \bo_2$ be a 1-morphism that \emph{admits a right adjoint}. This means that there exists
a 1-morphism $t^R:\bo_2\to \bo_1$ and 2-morphisms
$$\on{id}_{\bo_1}\to t^R\circ t \text{ and } t\circ t^R\to \on{id}_{\bo_2}$$
that satisfy the usual axioms, see \cite[Chapter 12, Sect. 1]{GR1}.

\medskip

Let $F_i$ be a 1-endomorphism of $\bo_i$. In addition, let us be given a 2-morphism
$$\alpha:t\circ F_1\to F_2\circ t$$
\begin{equation} \label{e:2-morph}
\xy
(0,0)*+{\bo_1}="A";
(20,0)*+{\bo_1}="B";
(0,-20)*+{\bo_2}="C";
(20,-20)*+{\bo_2}="D";
{\ar@{->}^{F_1} "A";"B"};
{\ar@{->}^{t} "A";"C"};
{\ar@{->}^{t} "B";"D"};
{\ar@{->}_{F_2} "C";"D"};
{\ar@{=>}^\alpha "B";"C"};
\endxy
\end{equation}

In this case, following \cite[Definition 2.24]{BN1} or \cite[Example 1.2.5]{KP1}, we define the 2-morphism
$$\Tr(t,\alpha):\Tr(F_1,\bo_1)\to \Tr(F_2,\bo_2)$$
to be the composite
$$
\xy
(0,0)*+{\one_\bO}="A";
(20,0)*+{\bo_1\otimes \bo_1^\vee}="B";
(40,0)*+{\bo_1\otimes \bo_1^\vee}="C";
(60,0)*+{\one_\bO}="D";
(0,-20)*+{\one_\bO}="A'";
(20,-20)*+{\bo_2\otimes \bo_2^\vee}="B'";
(40,-20)*+{\bo_2\otimes \bo_2^\vee}="C'";
(60,-20)*+{\one_\bO,}="D'";
{\ar@{->}^{\on{unit}} "A";"B"};
{\ar@{->}_{\on{unit}} "A'";"B'"};
{\ar@{->}^{\on{counit}} "C";"D"};
{\ar@{->}_{\on{counit}} "C'";"D'"};
{\ar@{->}^{F_1\otimes \on{id}} "B";"C"};
{\ar@{->}_{F_2\otimes \on{id}} "B'";"C'"};
{\ar@{->}^{\on{id}} "A";"A'"};
{\ar@{->}^{\on{id}} "D";"D'"};
{\ar@{->}^{t\otimes t^{\on{op}}} "B";"B'"};
{\ar@{->}^{t\otimes t^{\on{op}}} "C";"C'"};
{\ar@{=>} "B";"A'"};
{\ar@{=>} "C";"B'"};
{\ar@{=>} "D";"C'"};
\endxy
$$
where:

\medskip

\noindent--$t^{\on{op}}$ denotes the 1-morphism $\bo_1^\vee\to \bo_2^\vee$ equal to $(t^R)^\vee$;

\medskip

\noindent--the 2-morphism in the left square is given by the $(t,t^R)$-adjunction;

\medskip

\noindent--the 2-morphism in the middle square is given by $\alpha$;

\medskip

\noindent--the 2-morphism in the right square is given by the $(t,t^R)$-adjunction.

\begin{rem}
The above construction is equivalent to formula \eqref{e:map between Tr} given earlier.
\end{rem}

\sssec{} \label{sss:funct trace}

The construction of \secref{sss:map of traces} is compatible with compositions:

\medskip

For a composition
$$\bo_1\overset{t_{1,2}}\to \bo_2 \overset{t_{2,3}}\to \bo_3,$$
the 2-morphism
$$\Tr(F_1,\bo_1) \overset{\Tr(t_{1,2},\alpha_{1,2})}\longrightarrow \Tr(F_2,\bo_2)
\overset{\Tr(t_{2,3},\alpha_{2,3})}\longrightarrow \Tr(F_3,\bo_3)$$
identifies with $\Tr(t_{1,3},\alpha_{1,3})$, where
$$t_{1,3}=t_{2,3}\circ t_{1,2},$$
and $\alpha_{1,3}$ is obtained by composing $\alpha_{1,2}$ and $\alpha_{2,3}$.

\sssec{} \label{sss:cyclicity morphisms}

The construction of \secref{sss:map of traces} has the following cyclicity property:

\medskip

Given a pair of diagrams of 2-morphisms

\begin{equation} \label{e:two 2-morph}
\xy
(0,0)*+{\bo_1}="A";
(20,0)*+{\bo_2}="B";
(0,-20)*+{\bo'_1}="C";
(20,-20)*+{\bo'_2}="D";
{\ar@{->}^{F_{1,2}} "A";"B"};
{\ar@{->}^{t_1} "A";"C"};
{\ar@{->}^{t_2} "B";"D"};
{\ar@{->}_{F'_{1,2}} "C";"D"};
{\ar@{=>}^{\alpha} "B";"C"};
(40,0)*+{\bo_2}="A'";
(60,0)*+{\bo_1}="B'";
(40,-20)*+{\bo'_2}="C'";
(60,-20)*+{\bo'_1,}="D'";
{\ar@{->}^{F_{2,1}} "A'";"B'"};
{\ar@{->}^{t_2} "A'";"C'"};
{\ar@{->}^{t_1} "B'";"D'"};
{\ar@{->}_{F'_{2,1}} "C'";"D'"};
{\ar@{=>}^{\beta} "B'";"C'"};
\endxy
\end{equation}
we can compose them horizontally in two ways, thus getting diagrams

\begin{equation*}
\xy
(0,0)*+{\bo_1}="A";
(20,0)*+{\bo_1}="B";
(0,-20)*+{\bo'_1}="C";
(20,-20)*+{\bo'_1}="D";
{\ar@{->}^{F_{2,1}\circ F_{1,2}} "A";"B"};
{\ar@{->}^{t_1} "A";"C"};
{\ar@{->}^{t_1} "B";"D"};
{\ar@{->}_{F'_{2,1}\circ F'_{1,2}} "C";"D"};
{\ar@{=>}^{\alpha\circ\beta} "B";"C"};
(40,0)*+{\bo_2}="A'";
(60,0)*+{\bo_2}="B'";
(40,-20)*+{\bo'_2}="C'";
(60,-20)*+{\bo'_2.}="D'";
{\ar@{->}^{F_{1,2}\circ F_{2,1}} "A'";"B'"};
{\ar@{->}^{t_2} "A'";"C'"};
{\ar@{->}^{t_2} "B'";"D'"};
{\ar@{->}_{F'_{1,2}\circ F'_{2,1}} "C'";"D'"};
{\ar@{=>}^{\beta\circ \alpha} "B'";"C'"};
\endxy
\end{equation*}

Assume now that objects $\bo_1,\bo_2,\bo'_1$ and $\bo'_2$ are dualizable, while
1-morphisms $t_1$ and $t_2$ admit right adjoints. Then we claim that the trace maps
\[\Tr(t_1,\alpha\circ\beta):\Tr(F_{2,1}\circ F_{1,2},\bo_1)\to \Tr(F'_{2,1}\circ F'_{1,2},\bo'_1)\]
and
\[\Tr(t_2,\beta\circ\alpha):\Tr(F_{1,2}\circ F_{2,1},\bo_2)\to \Tr(F'_{1,2}\circ F'_{2,1},\bo'_2)\]
match up under identifications
\[
\Tr(F_{1,2}\circ F_{2,1},\bo_2) \simeq \Tr(F_{2,1}\circ F_{1,2},\bo_1) \text{ and } \Tr(F'_{1,2}\circ F'_{2,1},\bo'_2) \simeq \Tr(F'_{2,1}\circ F'_{1,2},\bo'_1)
\]
of \eqref{e:cyclicity}.
\medskip

More precisely, we claim that there is a canonical isomorphism
\begin{equation} \label{e:cycl isomorp}
\Tr(t_1,\alpha\circ\beta)\simeq \Tr(t_2,\beta\circ\alpha)
\end{equation}
in
\[
\on{Funct}(\Tr(F_{2,1}\circ F_{1,2},\bo_1),\Tr(F'_{2,1}\circ F'_{1,2},\bo'_1))\simeq \on{Funct}(\Tr(F_{1,2}\circ F_{2,1},\bo_2),\Tr(F'_{1,2}\circ F'_{2,1},\bo'_2)).
\]

Indeed, let $q_{i,j}\in \Maps(\one_\bO,\bo_j\otimes \bo_i^\vee)$ and $q'_{i,j}\in \Maps(\one_\bO,\bo'_j\otimes \bo'{}_i^\vee)$ be the points that  represent $F_{i,j}$ and $F'_{i,j}$, respectively. Then arguing as in \secref{sss:cyclicity} one sees that both maps
\[
\Tr(t_1,\alpha\circ\beta)\text{ and }\Tr(t_2,\beta\circ\alpha)
\]
are canonically isomorphic to the composite
$$
\xy
(0,0)*+{\one_\bO}="A";
(40,0)*+{\bo_1\otimes \bo_2^\vee\otimes\bo_2\otimes \bo_1^\vee}="B";
(80,0)*+{\one_\bO}="C";
(0,-20)*+{\one_\bO}="A'";
(40,-20)*+{\bo_1\otimes \bo_2^\vee\otimes\bo_2\otimes \bo_1^\vee}="B'";
(80,-20)*+{\one_\bO,}="C'";
{\ar@{->}^{q_{2,1}\otimes q_{1,2}} "A";"B"};
{\ar@{->}_{q_{2,1}\otimes q_{1,2}} "A'";"B'"};
{\ar@{->}^{\on{counit}} "B";"C"};
{\ar@{->}_{\on{counit}} "B'";"C'"};
{\ar@{->}^{\on{id}} "A";"A'"};
{\ar@{->}^{\on{id}} "C";"C'"};
{\ar@{->}^{t_1\otimes t_2^{\on{op}}\otimes t_2\otimes t_1^{\on{op}}} "B";"B'"};
{\ar@{=>} "B";"A'"};
{\ar@{=>} "C";"B'"};
\endxy
$$
where:

\medskip

\noindent--$t_i^{\on{op}}$ denotes the 1-morphism $\bo_i^\vee\to \bo'{}_i^\vee$ equal to $(t_i^R)^\vee$;

\medskip

\noindent--the 2-morphism in the left square is given by the $(t_i,t_j^R)$-adjunctions and 2-morphisms $\alpha$ and $\beta$;

\medskip

\noindent--the 2-morphism in the right square is given by the $(t_i,t_i^R)$-adjunction.

\sssec{} \label{sss:cyclicity composition}

By functoriality, isomorphisms \eqref{e:cycl isomorp} from  \secref{sss:cyclicity morphisms} are compatible with {\em vertical} compositions:

\medskip

Consider a pair of diagrams of 2-morphisms

\begin{equation} \label{e:two composable 2-morph}
\xy
(0,0)*+{\bo_1}="A";
(20,0)*+{\bo_2}="B";
(0,-20)*+{\bo'_1}="C";
(20,-20)*+{\bo'_2}="D";
(0,-40)*+{\bo''_1}="E";
(20,-40)*+{\bo''_2}="F";
{\ar@{->}^{F_{1,2}} "A";"B"};
{\ar@{->}^{t_1} "A";"C"};
{\ar@{->}^{t_2} "B";"D"};
{\ar@{->}^{F'_{1,2}} "C";"D"};
{\ar@{=>}^{\alpha} "B";"C"};
{\ar@{->}^{t'_1} "C";"E"};
{\ar@{->}^{t'_2} "D";"F"};
{\ar@{->}_{F''_{1,2}} "E";"F"};
{\ar@{=>}^{\alpha'} "D";"E"};
(40,0)*+{\bo_2}="A'";
(60,0)*+{\bo_1}="B'";
(40,-20)*+{\bo'_2}="C'";
(60,-20)*+{\bo'_1}="D'";
(40,-40)*+{\bo''_2}="E'";
(60,-40)*+{\bo''_1,}="F'";
{\ar@{->}^{F_{2,1}} "A'";"B'"};
{\ar@{->}^{t_2} "A'";"C'"};
{\ar@{->}^{t_1} "B'";"D'"};
{\ar@{->}^{F'_{2,1}} "C'";"D'"};
{\ar@{=>}^{\beta} "B'";"C'"};
{\ar@{->}^{t'_2} "C'";"E'"};
{\ar@{->}^{t'_1} "D'";"F'"};
{\ar@{->}_{F''_{2,1}} "E'";"F'"};
{\ar@{=>}^{\beta'} "D'";"E'"};
\endxy
\end{equation}
in which all objects are dualizable and all vertical 1-morphisms admit right adjoints.

\medskip

Then isomorphisms \eqref{e:cycl isomorp}, corresponding to the top part, the bottom part and the vertical composition of the diagrams
\eqref{e:two composable 2-morph}, respectively, are
\begin{equation} \label{e:cycl isomorp1}
\Tr(t_1,\alpha\circ\beta)\simeq \Tr(t_2,\beta\circ\alpha), \hskip 8truept \Tr(t'_1,\alpha'\circ\beta')\simeq \Tr(t'_2,\beta'\circ\alpha')
\end{equation}
and
\begin{equation} \label{e:cycl isomorp2}
\Tr(t'_1\circ t_1,(\alpha'\circ \alpha)\circ(\beta'\circ \beta))\simeq \Tr(t'_2\circ t_2,(\beta'\circ \beta)\circ(\alpha'\circ \alpha)).
\end{equation}

\medskip

Moreover, by the fuctoriality of \eqref{e:cycl isomorp} the following diagram is commutative:

\begin{equation} \label{e:cycl isomorp3}
\begin{CD}
\Tr(t'_1\circ t_1,(\alpha'\circ \alpha)\circ(\beta'\circ \beta)) @>{\text{\secref{sss:funct trace}}}>\sim>  \Tr(t'_1,\alpha'\circ\beta')\circ \Tr(t_1,\alpha\circ\beta)  \\
@V\eqref{e:cycl isomorp2}V{\sim}V     @V\eqref{e:cycl isomorp1}V{\sim}V \\
\Tr(t'_2\circ t_2,(\beta'\circ \beta)\circ(\alpha'\circ \alpha)) @>{\text{\secref{sss:funct trace}}}>\sim> \Tr(t'_2,\beta'\circ\alpha')\circ \Tr(t_2,\beta\circ\alpha).
\end{CD}
\end{equation}

\ssec{Properties of the 2-categorical trace}

In this subsection we will explore further functoriality properties of the construction
of \secref{sss:map of traces}. We recommend the reader to skip this subsection and return
to it when necessary.

\sssec{}  \label{sss:Lo}

The functoriality mentioned in \secref{sss:funct trace} can be promoted to a symmetric monoidal functor
between symmetric monoidal $(\infty,1)$-categories:

\medskip

Consider the category, to be denoted $L(\bO)$,
whose objects are pairs $(\bo,F)$, where $\bo\in \bO$ and $F\in \End_\bO(\bo)$, and whose morphisms are given by diagrams \eqref{e:2-morph},
see \cite[Sect. 1.2]{KP1}\footnote{To see the full construction of $L(\bO)$ as an $\infty$-category see
\cite[Chapter 10, Sect. 4.1]{GR1}.}.

\medskip

The symmetric monoidial structure on $\bO$ induces one on $L(\bO)$.
Let $L(\bO)_{\on{rgd}}\subset L(\bO)$ be the 1-full subcategory, where we allow as objects those $(\bo,F)$
for which $\bo$ is dualizable as an object of $\bO$, and where we restrict 1-morphisms to those pairs
$(t,\alpha)$, for which $t$ admits a right adjoint.

\medskip

Then the assignment
$$(\bo,F)\mapsto \Tr(F,\bo)$$
is a \emph{symmetric monoidal} functor
\begin{equation} \label{e:2-tr as functor}
\Tr:L(\bO)_{\on{rgd}}\to \End_\bO(\one_\bO).
\end{equation}

The construction of the functor \eqref{e:2-tr as functor} can be either performed directly using the definition
of $\infty$-categorical symmetric monoidal structures as in \cite[Chapter 1, Sect. 3.3]{GR1}, or using the device of \cite[Theorem 1.7]{HSS}.

\sssec{}  \label{sss:2-categ sym mon}

Let $\ba\in \bO$ be an associative/commutative algebra object in $\bO$, so that the multiplication map
$$\ba\otimes \ba\to \ba$$
admits a right adjoint. Assume also that $\ba$ is dualizable as an object of $\bO$.

\medskip

Let $F_\ba$ be a right-lax monoidal/symmetric monoidal endomorphism of $\ba$ (see, e.g., \secref{sss:trace on sym mon dg categ}).
Then $(\ba,F_\ba)$ acquires a structure of associative/commutative algebra object in $L(\bO)_{\on{rgd}}$.

\medskip

Hence, we
obtain that $\Tr(F_\ba,\ba)$ acquires a structure of associative/commutative algebra object in $\End_\bO(\one_\bO)$.

\sssec{}  \label{sss:2-categ sym mon mod}

Let $\ba$ be as above, and let $\bm\in \bO$ be an $\ba$-module object. Assume that the action map
$$\ba\otimes \bm\to \ba$$
admits a right adjoint. Assume also that $\bm$ is dualizable as an object of $\bO$.

\medskip

Let $F_\bm$ be an endomorphism of $\bm$ that is right-lax compatible with $F_\ba$ (see, e.g., \secref{sss:trace on sym mon dg categ}).
Then $(\bm,F_\bm)$ acquires a structure of module over $(\ba,F_\ba)$ in $L(\bO)_{\on{rgd}}$.

\medskip

Hence, we obtain that
$\Tr(F_\bm,\bm)\in \End_\bO(\one_\bO)$ acquires a structure of module over $\Tr(F_\ba,\ba)$.

\sssec{}  \label{sss:2-categ sym mon hom}

Let $\ba$ and $\ba'$ be a pair of associative/commutative algebra objects in $\bO$ as in \secref{sss:2-categ sym mon},
each endowed with a right-lax monoidal/symmetric monoidal endomorphism. Let $\varphi$ be a right-lax monoidal/symmetric monoidal map
$\ba\to \ba'$, equipped with a 2-morphism in the diagram
$$
\xy
(0,0)*+{\ba}="A";
(20,0)*+{\ba}="B";
(0,-20)*+{\ba'}="C";
(20,-20)*+{\ba',}="D";
{\ar@{->}^{F_{\ba}} "A";"B"};
{\ar@{->}^{\varphi} "A";"C"};
{\ar@{->}^{\varphi} "B";"D"};
{\ar@{->}_{F_{\ba'}} "C";"D"};
{\ar@{=>}^\alpha "B";"C"};
\endxy
$$
compatible with the right-lax monoidal/symmetric monoidal structures on the edges. Assume that $\varphi$ admits a right adjoint.

\medskip

We can view the data of $(\varphi,\alpha)$ as a morphism of associative/commutative algebra objects $(\ba,F_\ba)\to (\ba',F_{\ba'})$
in $L(\bO)_{\on{rgd}}$. Hence, the above data induces a map
$$\Tr(F_\ba,\ba)\to \Tr(F_{\ba'},\ba'),$$
which is a map of associative/commutative algebras, as follows from the functoriality of the trace construction,
see \secref{sss:funct trace}.

\sssec{}  \label{sss:2-categ sym mon unit}

Note that $(\one_\bO,\on{id})$ is the unit in $L(\bO)_{\on{rgd}}$ (see \secref{sss:Lo}). Since $\Tr$ is a symmetric monoidal
functor, we can identify
$$\Tr(\on{id},\one_\bO)=\one_{\End_\bO(\one_\bO)}$$
as objects of $\End_\bO(\one_\bO)$.

\medskip

Let now $\ba$ be as in \secref{sss:2-categ sym mon}. Consider the unit map
\begin{equation} \label{e:unit a}
\one_\bO\to \ba,
\end{equation}
equipped with the 2-morphism
\begin{equation} \label{e:2-categ sym mon hom}
\xy
(0,0)*+{\one_\bO}="A";
(20,0)*+{\one_\bO}="B";
(0,-20)*+{\ba}="C";
(20,-20)*+{\ba,}="D";
{\ar@{->}^{\on{id}} "A";"B"};
{\ar@{->} "A";"C"};
{\ar@{->} "B";"D"};
{\ar@{->}_{F_{\ba}} "C";"D"};
{\ar@{=>} "B";"C"};
\endxy
\end{equation}
provided by the right-lax monoidal structure on $F_\ba$. Assume that \eqref{e:unit a} admits a right adjoint.

\medskip

It follows from \secref{sss:2-categ sym mon hom} that the resulting map
$$\one_{\End_\bO(\one_\bO)}\simeq \Tr(\on{id},\one_\bO)\to \Tr(F_\ba,\ba)$$
is the unit of $\Tr(F_\ba,\ba)$ as an associative/commutative algebra object in
$\End_\bO(\one_\bO)$.

\ssec{Trace on DG categories}

\sssec{}

The example of primary interest for us is $\bO=\DGCat$, with its natural symmetric monoidal structure.

\medskip

Note that $\one_{\DGCat}=\Vect$, so
\begin{equation} \label{e:end Vect}
\End(\one_{\DGCat})\simeq \Vect,
\end{equation}
as a category, equipped with a symmetric monoidal structure.

\medskip

Hence, for a dualizable DG category $\CC$ equipped with an endofunctor $F$, we obtain an object
$$\Tr(F,\CC)\in \Vect.$$

Furthermore, if $T:\CC_1\to \CC_2$ is a morphism in $\DGCat$ (i.e., a colimit-preserving $\sfe$-linear exact functor) between dualizable DG categories that admits a \emph{continuous} right adjoint, and given a natural transformation
$$\alpha:T\circ F_1\to F_2\circ T,$$
we obtain a map in $\Vect$
$$\Tr(T,\alpha):\Tr(F_1,\CC_1)\to Tr(F_2,\CC_2).$$

\sssec{}

Take $\CC=\Vect$ and $F=\on{Id}$. Then by \eqref{e:Tr unit}, we have
$$\Tr(\on{Id},\Vect)\simeq \sfe,$$
as an object in \eqref{e:end Vect}.

\medskip

More generally, for $V\in \Vect$, the trace of the endofunctor of $\Vect$, given by
$-\otimes V$, identifies with $V$ as an object in \eqref{e:end Vect}.

\sssec{}  \label{sss:class}

Take $\CC_2=\CC$ and $\CC_1=\Vect$ with the functor $T$ corresponding to an object $c\in \CC$,
i.e., the (unique) colimit preserving functor satisfying
\begin{equation} \label{e:functor corr to object}
\sfe\mapsto c.
\end{equation}

Note that the condition that $T$ admit a continuous right adjoint is equivalent to the condition that
$c$ be compact.

\medskip

Let $F_2=F$ be some endofunctor of $\CC$ and take $F_1=\on{Id}$. Then the datum of $\alpha$ amounts to a morphism
$$\alpha:c\to F(c).$$

The resulting map
$$\sfe\simeq \on{Tr}(\on{Id},\Vect)\overset{\Tr(T,\alpha)}\longrightarrow \Tr(F,\CC)$$
corresponds to a point in $\Tr(F,\CC)$, which we will denote by $\on{cl}(c,\alpha)$.

\sssec{} \label{sss:trace on sym mon dg categ}

Let $\CR$ be a monoidal/symmetric monoidal DG category, and $\CM$ an $\CR$-module category. Assume that the functors
$$\CR\otimes \CR\to \CR \text{ and } \CR\otimes \CM\to \CM,$$
viewed as functors of plain DG categories admit right adjoints
(if $\CR$ and $\CM$ are compactly generated, the condition of admitting a right adjoint is equivalent to preserving compactness).

\medskip

Let $\CR$ be endowed with a right-lax monoidal/symmetric monoidal endofunctor $F_\CR$, i.e., we have a natural transformation
$$F_\CR(a_1)\otimes F_\CR(a_2)\to F_\CR(a_1\otimes a_2), \quad a_1,a_2\in \CR$$
equipped with higher compatibilities.

\medskip

Assume also that $\CM$ is dualizable and is endowed with an endofunctor $\phi_\CM$ that is right-lax compatible with the $\CR$-action.
I.e., we have a natural transformation
$$F_\CR(a)\otimes F_\CM(m)\to F_\CM(a\otimes m), \quad a\in \CR, m\in \CM,$$
equipped with higher compatibilities.

\medskip

In this case, applying the construction of Sects. \ref{sss:2-categ sym mon} and \ref{sss:2-categ sym mon mod},  we obtain that $\Tr(F_\CR,\CR)$
acquires a structure of associative/commutative algebra, and $\Tr(F_\CM,\CM)$ acquires a structure of $\Tr(F_\CR,\CR)$-module.

\sssec{} \label{sss:Hecke a}

%

Let $\CR$ and $\CM$ be as above. Let $r\in \CR$ be a compact object equipped with a map $\alpha:r\to F_\CR(r)$.
On the one hand, we can consider
$$\on{cl}(r,\alpha)\in \Tr(F_\CR,\CR).$$

On the other hand, let $H_r$ denote the endofunctor of $\CM$ given by the action of $r$ (here ``$H$" should be evocative
of ``Hecke"). The right-lax compatibility of $F_\CM$ with the action defines a natural transformation
\begin{equation} \label{e:Hecke a}
H_r\circ F_\CM\overset{\alpha}\to H_{F_\CR(r)}\circ F_\CM \to F_\CM\circ H_r
\end{equation}
which we denote by $\alpha_{r,\CM}$.

\medskip

By \secref{sss:map of traces}, to the pair $(H_r,\alpha_{r,\CM})$ we can assign the map
$$\Tr(H_r,\alpha_{r,\CM}):\Tr(F_\CM,\CM)\to \Tr(F_\CM,\CM).$$

\medskip

We claim:

\begin{prop}  \label{p:S=T, one}
The action of $\on{cl}(r,\alpha)$ on $\Tr(F_\CM,\CM)$ equals $\Tr(H_r,\alpha_{r,\CM})$.
\end{prop}

\begin{proof}

Follows from the functoriality of the trace construction (see \secref{sss:funct trace}) corresponding to the composition
$$\CM\simeq \Vect\otimes \CM \overset{r\otimes \on{Id}}\longrightarrow \CR\otimes \CM \overset{\on{act}}\longrightarrow \CM.$$

\end{proof}

\ssec{Examples}

The results from this subsection will not be used in the rest of the paper.  However, they are meant to provide an intuition
for the behavior of the categorical trace construction.

\sssec{}  \label{sss:HC}

Consider the example $\CC=R\mod$, where $R\in \on{AssocAlg}(\Vect)$. We have
$$R_1\mod\otimes R_2\mod\simeq (R_1\otimes R_2)\mod,$$
see \cite[Theorem 4.8.5.16]{Lu2}.

\medskip

In particular, $(R\mod)^\vee$ identifies with $R^{\on{rev}}\mod$, where $R^{\on{rev}}$ is obtained from $R$ by reversing
the multiplication. The unit and counit map are given by the functors
$$\Vect \to (R\otimes R^{\on{rev}})\mod, \quad \sfe\mapsto R,$$
and
$$(R\otimes R^{\on{rev}})\mod \to \Vect, \quad Q\mapsto R\underset{R\otimes R^{\on{rev}}}\otimes Q,$$
respectively.

\medskip

Identifying
$$\End(R\mod)\simeq R\mod\otimes (R\mod)^\vee \simeq (R\otimes R^{\on{rev}})\mod,$$
we obtain that every continuous endofunctor of $R\mod$ is of the form
$$M\mapsto F_Q(M):=Q\underset{R}\otimes M$$
for $Q\in (R\otimes R^{\on{rev}})\mod$.

\medskip

The trace of such an endofunctor is given by
$$\on{HH}_\bullet(R,Q):=R \underset{R\otimes R^{\on{rev}}}\otimes Q.$$

\medskip

The identity endofunctor of $R\mod$ corresponds to $Q=R$. In this case we use a simplified notation
$$\on{HH}_\bullet(R):=\on{HH}_\bullet(R,Q).$$

This is the vector space of Hochschild chains on $R$.

\sssec{}  \label{sss:prestack}

Let $\CY$ be a prestack (see \cite[Chapter 2, Sect. 1]{GR1}), such that:

\medskip

\noindent--the category $\QCoh(\CY)$ is dualizable;

\medskip

\noindent--the object $\CO_\CY\in \QCoh(\CY)$ is compact;

\medskip

\noindent--the diagonal morphism $\Delta:\CY\to \CY\times \CY$ is schematic and qsqc (quasi-separated and quasi-compact).

\medskip

For example, these conditions are satisfied for a quasi-compact algebraic stack with affine diagonal,
of finite type over a field of characteristic $0$ (see \cite[Theorem 1.4.2]{DrGa1}).

\medskip

The condition that $\CO_\CY\in \QCoh(\CY)$ is compact is equivalent to one saying that the functor of global sections
$$\Gamma(\CY,-):\QCoh(\CY)\to \Vect$$
is continuous. The condition that $\Delta$ is schematic and qsqc implies that the direct image functor
$$\Delta_*: \QCoh(\CY)\to \QCoh(\CY\times \CY)$$
is continuous and satisfies base change.

\medskip

In this case, the functors
$$\Vect \overset{\sfe\mapsto \CO_\CY}\to \QCoh(\CY) \overset{\Delta_*}\longrightarrow  \QCoh(\CY\times \CY)\simeq \QCoh(\CY)\otimes \QCoh(\CY)$$
and
$$\QCoh(\CY)\otimes \QCoh(\CY)\simeq \QCoh(\CY\times \CY)\overset{\Delta^*}\longrightarrow  \QCoh(\CY)\overset{\Gamma(\CY,-)}\to \Vect$$
define an identification
$$\QCoh(\CY)^\vee \simeq \QCoh(\CY)$$
(the proof is a diagram chase using base change).

\sssec{}  \label{sss:Tr QCoh}

Let $\phi$ be an endomorphism of $\CY$ and consider the endofunctor of $\QCoh(\CY)$ given by $\phi^*$.

\medskip

We claim that $\on{Tr}(\phi^*,\QCoh(\CY))$ identifies canonically with
$$\Gamma(\CY^\phi,\CO_{\CY^\phi}),$$
where
$$\CY^\phi:=\CY\underset{\on{Graph}_\phi,\CY\times \CY,\Delta}\times \CY$$
is the fixed point locus of $\phi$.

\medskip

Indeed, let $\iota$ denote the forgetful map $\CY^\phi\to \CY$, so that we have
$$\iota\circ \phi\simeq \iota.$$

\medskip

We calculate $\on{Tr}(\phi^*,\QCoh(\CY))$ as the composite
$$
\CD
& &  \QCoh(\CY) @>{(p_\CY)_*}>> \QCoh(\on{pt})  \\
& & @AA{\Delta^*}A \\
& &  \QCoh(\CY\times \CY)  \\
& & @AA{(\phi\times \on{id})^*}A  \\
\QCoh(\CY) @>{\Delta_*}>> \QCoh(\CY\times \CY)  \\
@A{p_\CY^*}AA  \\
\QCoh(\on{pt}),
\endCD
$$
where $p_\CY:\CY\to \on{pt}$ is the projection map.

\medskip

By base change, we rewrite this functor as
$$
\CD
\QCoh(\CY^\phi)  @>{\iota_*}>> \QCoh(\CY) @>{(p_\CY)_*}>> \QCoh(\on{pt}) \\
@A{\iota^*}AA \\
\QCoh(\CY)  \\
@A{p_\CY^*}AA  \\
\QCoh(\on{pt}),
\endCD
$$
which sends $\sfe\in \Vect=\QCoh(\on{pt})$ to $\Gamma(\CY^\phi,\CO_{\CY^\phi})$, as desired.

\sssec{}

Let now $\CF\in \QCoh(\CY)$ be a compact object, equipped with a map
$$\alpha:\CF\to \phi^*(\CF).$$

Consider the corresponding object
$$\on{cl}(\CF,\alpha)\in \on{Tr}(\phi^*,\QCoh(\CY)),$$
see \secref{sss:class}.

\medskip

We will now describe explicitly the image of $\on{cl}(\CF,\alpha)$ under the identification
$$\on{Tr}(\phi^*,\QCoh(\CY))\simeq \Gamma(\CY^\phi,\CO_{\CY^\phi}).$$

\medskip

First, we claim:

\begin{lem}
Every compact object $\CF\in \QCoh(\CY)$ is dualizable in the sense of the symmetric monoidal structure
on $\QCoh(\CY)$.
\end{lem}

\begin{proof}

The proof runs parallel to \cite[Proposition 3.6]{BFN}:

\medskip

Since the diagonal morphism of $\CY$ is schematic and qsqc, any map
$$f:S\to \CY,$$
where $S$ an affine scheme, is itself schematic and qsqc. In particular, the functor
$f_*$, right adjoint to $f^*$, is continuous. Hence, the functor $f^*$ preserves compactness.

\medskip

This implies that the pullback of every compact object $\CF\in \QCoh(\CY)$ to
every affine scheme $S$ is perfect, and hence dualizable as an object of the symmetric monoidal category
$\QCoh(S)$.

\medskip

Since
$$\QCoh(\CY)\simeq \underset{S\to \CY}{\on{lim}}\,\QCoh(S),$$
we obtain that $\CF$ is dualizable in $\QCoh(\CY)$.

\end{proof}

\sssec{}

%

Consider the pullback of $\alpha$ along the map
$$\iota:\CY^\phi\to \CY.$$

Using the fact that $\iota=\phi\circ \iota$, we obtain a map
$$\iota^*(\CF)\overset{\alpha}\to\iota^*\circ \phi^*(\CF)=
 (\phi\circ \iota)^*(\CF)\simeq \iota^*(\CF).$$

This is an endomorphism of $\iota^*(\CF)$, which we denote $\alpha^\phi$. We claim:

\begin{prop}  \label{p:class QCoh}
The element
$$\on{cl}(\CF,\alpha)\in \on{Tr}(\phi^*,\QCoh(\CY))\simeq \Gamma(\CY^\phi,\CO_{\CY^\phi})$$
identifies with
$$\on{Tr}(\alpha^\phi,\iota^*(\CF)),$$
where the latter trace is taken in the symmetric monoidal category $\QCoh(\CY^\phi)$, and where we
identify
$$\End_{\QCoh(\CY^\phi)}(\one_{\QCoh(\CY^\phi)})=\End_{\QCoh(\CY^\phi)}(\CO_{\CY^\phi})\simeq \Gamma(\CY^\phi,\CO_{\CY^\phi}).$$
\end{prop}

This is proved in \cite[Prop. 2.2.3]{KP1}; see \secref{ss:KPgen} below for a proof in a more general context.

\sssec{}

Let now $\CY$ be a quasi-compact algebraic stack with an affine diagonal over  a ground field of characteristic $0$.
Consider the category $\Dmod(\CY)$.

\medskip

In this case, the functors
$$\Vect \overset{\sfe\mapsto \omega_\CY}\longrightarrow \Dmod(\CY)\overset{\Delta_{\on{dR},*}}\longrightarrow
\Dmod(\CY\times \CY) \simeq \Dmod(\CY)\otimes \Dmod(\CY)$$
and
$$\Dmod(\CY)\otimes \Dmod(\CY)\simeq \Dmod(\CY\times \CY) \overset{\Delta^!}\longrightarrow
\Dmod(\CY) \overset{\Gamma_{\on{dR}}(\CY,-)}\longrightarrow \Vect$$
define an identification
$$\Dmod(\CY)^\vee\simeq \Dmod(\CY).$$

\medskip

Let $\phi$ be an endomorphism of $\CF$. Consider the endofunctor of $\Dmod(\CY)$, given by $\phi_{\on{dR},*}$. Then as in \secref{sss:Tr QCoh}
one shows that there exists a canonical identification
$$\on{Tr}(\phi_{\on{dR},*},\Dmod(\CY))\simeq \Gamma_{\on{dR}}(\CY^\phi,\omega_{\CY^\phi}),$$
where the latter object is usually called the Borel-Moore homology of $\CY^\phi$.

\sssec{}

Let $\CF$ be a compact object of $\Dmod(\CY)$ equipped with a map
$$\alpha:\CF\to \phi_{\on{dR},*}(\CF).$$

To this object there corresponds an element
$$\on{cl}(\CF,\alpha)\in \on{Tr}(\phi_{\on{dR},*},\Dmod(\CY))\simeq \Gamma_{\on{dR}}(\CY^\phi,\omega_{\CY^\phi}).$$

However, we do not at the moment know how to give an explicit formula for this element, which would
be reminiscent of that of \propref{p:class QCoh}.

\medskip

However, according to \cite{Va}, the following particular case is known:

\begin{thm}
Assume that $\CY$ is a scheme, and let $y\in \CY^\phi$ be an isolated fixed point.
Assume, moreover, that $d\phi|_{T_y(Y)}$ does not have eigenvalue $1$ (i.e., the derived fixed point locus
$\CY^\phi$ is smooth at $y$). Assume that $\CF$ is holonomic. Then the image of $\on{cl}(\CF,\alpha)$ under the projection
on the direct summand corresponding to $y$
$$\Gamma_{\on{dR}}(\CY^\phi,\omega_{\CY^\phi})\to \Vect$$
equals
$$\on{Tr}(\alpha^\phi,\iota_y^{\on{dR},*}(\CF)),$$
where $\iota_y$ denotes the embedding $\on{pt}\overset{y}\to \CY$, and $\alpha^\phi$ denotes the induced endomorphism
of $\iota^{\on{dR},*}_y(\CF)$ equal to
$$\iota_y^{\on{dR},*}(\CF)\simeq (\phi\circ \iota_y)^{\on{dR},*}(\CF)=\iota_y^{\on{dR},*}\circ \phi^{\on{dR},*}(\CF)
\overset{\iota_y^{\on{dR},*}(\alpha')}\longrightarrow \iota_y^{\on{dR},*}(\CF),$$
and where
$$\alpha':\phi^{\on{dR},*}(\CF)\to \CF$$ is obtained from $\alpha$ by the $(\phi^{\on{dR},*},\phi_{\on{dR},*})$-adjunction.
\end{thm}

\sssec{A prototype for the moduli of shtukas}\label{sss:geometric shtukas}
Here we generalize some aspects of the above example and relate it to the construction of the moduli space of shtukas
(in the setting of D-modules in characteristic 0).  Let $\CY$ be a quasi-compact algebraic stack with affine diagonal over a field of characteristic 0.
Suppose we have a ``Frobenius'' endomorphism
$$ \phi: \CY \to \CY .$$
Moreover, suppose that we have a correspondence
$$ \xymatrix{
\CZ \ar[r]^p\ar[d]_q & \CY \\
\CY
}$$
and a sheaf $K \in \Dmod(\CZ)$.  Let $H_K$ denote the functor
$$ H_K = q_*(p^!(-) \otimes K): \Dmod(\CY) \to \Dmod(\CY) .$$
By the same considerations as above, we obtain
$$ \on{Tr}(H_K \circ \phi_{*, \dr}, \Dmod(\CY)) \simeq \Gamma_{\dr}(\CZ^{\phi}, i^!(K)), $$
where $\CZ^{\phi}$ is the pullback
$$ \xymatrix{
\CZ^{\phi} \ar[r]^{i}\ar[d] & \CZ \ar[d]^{(p,q)}\\
\CY \ar[r]^-{(\phi, \on{id})} & \CY \times \CY
}
$$

\medskip

Analogously, consider $\CY = \on{Bun}_G(X)$, where $X$ is an algebraic curve over a finite field $\mathbb{F}_q$,
$\phi$ is the morphism induced by the Frobenius on $X$, $\CZ$ is the Hecke stack (of modifications at some legs $x^I \in X^I$)
and $K$ is the Satake image of a representation of $G^{\vee}$.  In this case, $\CZ^{\phi}$ is the moduli stack of shtukas and
$\Gamma(\CZ^{\phi}, i^!(K))$ are the corresponding cohomologies of interest on the moduli of shtukas.

\ssec{Trace on DG 2-categories} \label{ss:2-categ trace}

\sssec{}  \label{sss:2-DGCat}

We will now study a different instance of the abstract formalism of \secref{sss:map of traces}. Namely, we take
$$\bO:=\tDGCat.$$

\medskip

By definition, objects of this category are parameterized by $\CR\in \DGCat^{\on{Mon}}$. We will denote the corresponding object of
$\tDGCat$ by $\ul{\CR\mmod}$.

\medskip

For two objects $\ul{\CR_1\mmod},\ul{\CR_2\mmod}\in \tDGCat$, the $(\infty,1)$-category of 1-morphisms
\begin{equation} \label{e:maps of mods}
\bMaps_{\tDGCat}(\ul{\CR_1\mmod},\ul{\CR_2\mmod})
\end{equation}
is by definition
$$\CR_2\otimes \CR_1^{\on{rev}}\mmod.$$

For an object $\CQ\in \CR_2\otimes \CR_1^{\on{rev}}\mmod$, we will denote the corresponding object of
\eqref{e:maps of mods} by $\fT_\CQ$.
The identity 1-morphism is given by $\fT_\CR$, where we view $\CR$ as an $\CR\otimes \CR^{\on{rev}}$-module category.

\medskip

The composition of two 1-morphisms
$$\ul{\CR_1\mmod} \overset{\fT_{\CQ_{1,2}}}\longrightarrow \ul{\CR_2\mmod} \overset{\fT_{\CQ_{2,3}}}\longrightarrow \ul{\CR_3\mmod}$$
is set to be $\fT_{\CQ_{1,3}}$, where
$$\fT_{\CQ_{1,3}}:=\fT_{\CQ_{2,3}}\underset{\CR_2}\otimes \fT_{\CQ_{1,2}}.$$

\sssec{} \label{sss:evaluate}

Note that a 1-morphism
$$\fT:\ul{\CR_1\mmod}\to \ul{\CR_2\mmod}$$
defines a functor of 2-categories
\begin{equation} \label{e:realize morphism}
\CR_1\mmod\to \CR_2\mmod,
\end{equation}
which, by a slight abuse of notation, we will denote by the same character $\fT$.

\medskip

Namely, for $\fT=\fT_\CQ$, the functor \eqref{e:realize morphism} is given by
$$\CM\mapsto \CQ\underset{\CR_1}\otimes \CM.$$

Sometimes, we will refer to this construction as ``evaluation of $\fT$ on $\CM$".

\begin{rem}
We use $\tDGCat$ as defined above as a substitute for the $(\infty,2)$-category of
``2-DG categories", hence the title of this subsection.

\medskip

Whatever the latter is, it contains $\tDGCat$ as a full subcategory, which consists
of uni-generated 2-DG categories.

\medskip

By way of analogy, $\on{Morita}(\Vect)$ is a full subcategory of $\DGCat$ that consists
of DG categories that can be generated by a single compact object.

\end{rem}

%
%

\sssec{}

The symmetric monoidal structure on $\tDGCat$ is given by
$$\ul{\CR_1\mmod}\otimes \ul{\CR_2\mmod}:=\ul{(\CR_1\otimes \CR_2)\mmod}.$$

\medskip

Note that the unit object of $\tDGCat$ is
$$\ul{\Vect\mmod}=:\ul{\DGCat}.$$

\sssec{}  \label{sss:dual 2categ}

Every object of $\tDGCat$ is dualizable. We have
$$\ul{\CR\mmod}^\vee \simeq \ul{\CR^{\on{rev}}\mmod},$$
where the unit and counit map are both given by
$$\CR\in (\CR\otimes \CR^{\on{rev}})\mmod.$$

\medskip

Under this identification, for a 1-morphism $\ul{\CR_1\mmod}\to \ul{\CR_2\mmod}$ corresponding to
$$\CM\in (\CR_1^{\on{rev}}\otimes \CR_2)\mmod,$$
the dual 1-morphism $\ul{\CR^{\on{rev}}_2\mmod}\to \ul{\CR^{\on{rev}}_1\mmod}$ is given by the same $\CM$.

\sssec{}

We have
$$\End_{\tDGCat}(\one_{\tDGCat})\simeq \DGCat,$$
as a symmetric monoidal category.

\medskip

Further,
$$\bMaps_{\tDGCat}(\one_{\tDGCat},\ul{\CR\mmod})\simeq \CR\mmod.$$

\sssec{}  \label{sss:funct trace 2categ}

Thus, we obtain that for every object $\fC\in \tDGCat$ and its endomorphism $\fF$, we can attach
$$\Tr(\fF,\fC)\in \DGCat.$$

\medskip

Furthermore, if
$$\fT:\fC_1\to \fC_2$$
is a 1-morphism that admits a right adjoint, and given
$$\alpha:\fT\circ \fF_1\to \fF_2\circ \fT,$$
we obtain a functor
$$\Tr(\fT,\alpha):\Tr(\fF_1,\fC_1)\to \Tr(\fF_2,\fC_2).$$

\sssec{}  \label{sss:right-dual}

Note that for $\CQ\in \CR_2\otimes \CR_1^{\on{rev}}\mmod$ the condition that $\fF_\CQ$ admit a right adjoint means that
$\CQ$ is \emph{right-dualizable} as a bimodule category.

\medskip

By definition, this means that there exists an object
$$\CQ^R\in \CR_1\otimes \CR_2^{\on{rev}}\mmod$$
equipped the map
$$\CR_1\to \CQ^R\underset{\CR_2}\otimes \CQ$$
in $(\CR_1\otimes \CR_1^{\on{rev}})\mmod$ and a map
$$\CQ \underset{\CR_1}\otimes \CQ^R \to \CR_2$$
in $(\CR_2\otimes \CR_2^{\on{rev}})\mmod$ that satisfy the usual axioms.

%
%
%

\sssec{} \label{sss:functor from homom}

The case of particular interest for us is when the 1-morphism
$$\ul{\CR_2\mmod}\to \ul{\CR_1\mmod}$$
is given by a monoidal functor
$$\Psi:\CR_1\to \CR_2,$$
i.e., it is given by the $(\CR_1,\CR_2)$-bimodule $\CQ_\Psi$, which is isomorphic to $\CR_2$ as a DG category,
on which $\CR_2$ acts by right multiplication, and $\CR_1$ acts by left multiplication via $\Psi$.

\medskip

Denote this 1-morphism by $\Res_\Psi$. The corresponding functor
$$\Res_\Psi:\CR_2\mmod\to \CR_1\mmod$$
is given by restriction via $\Psi$.

\medskip

The 1-morphism $\Res_\Psi$ tautologically admits a left adjoint, denoted $\Ind_\Psi$. It is given by
the $(\CR_2,\CR_1)$-bimodule $\CR_2$, on which $\CR_2$ acts by left multiplication,
and $\CR_1$ acts by right multiplication via $\Psi$. The unit map for the adjunction is given by
$$\CR_1\underset{\Psi}\to \CR_2\simeq \CR_2\underset{\CR_2}\otimes \CR_2,$$
and the counit of the adjunction is given by the multiplication map
$$\CR_2 \underset{\CR_1}\otimes \CR_2\to \CR_2.$$

The corresponding functor
$$\Ind_\Psi:\CR_1\mmod\to \CR_2\mmod$$
is given by
$$\CM\mapsto \CR_2\underset{\CR_1}\otimes \CM.$$

\sssec{}  \label{sss:right adj rigid}

Assume now that $\CR_1$ and $\CR_2$ are rigid (see \cite[Chapter 1, Sect. 9.1]{GR1}).
In this case, the 1-morphism $\Res_\Psi$ admits also a \emph{right} adjoint,
denoted $\coInd_\Psi$.

\medskip

The corresponding $(\CR_2,\CR_1)$-bimodule is again $\CR_2$.
The unit map for the adjunction.
$$\CR_2 \to \CR_2\underset{\CR_1}\otimes \CR_2$$
is the \emph{right adjoint} to the multiplication functor
$$\CR_2\underset{\CR_1}\otimes \CR_2\to \CR_2$$
(the right adjoint is a functor of $\CR_2$-bimodule categories due to rigidity).
The counit of the adjunction is the functor
$$\CR_2\underset{\CR_2}\otimes \CR_2\simeq \CR_2\to \CR_1,$$
right adjoint to $\Psi$ (again, this right adjoint is a functor of $\CR_1$-bimodule categories due to rigidity).

\begin{rem} \label{r:Ind and coInd}
Note that we obtain that in the situation when $\CR_1$ and $\CR_2$ are rigid, the left and right adjoints of
$\Res_\Psi$ are canonically isomorphic.
\end{rem}

\ssec{The 2-categorical trace and (categorical) Hochschild chains}

\sssec{}  \label{sss:HC 2categ}

For $\CR\in \DGCat^{\on{Mon}}$ and $\CQ\in (\CR\otimes \CR^{\on{rev}})\mmod$, denote
$$\on{HH}_\bullet(\CR,\CQ):=\CR\underset{\CR\otimes \CR^{\on{rev}}}\otimes \CQ\in \DGCat;$$
$$\on{HH}_\bullet(\CR):=\CR\underset{\CR\otimes \CR^{\on{rev}}}\otimes \CR.$$

As in \secref{sss:HC} we obtain formally that
\begin{equation} \label{e:Tr as HH}
\Tr(\fT_\CQ,\ul{\CR\mod})\simeq \on{HH}_\bullet(\CR,\CQ).
\end{equation}

\sssec{}  \label{sss:mon endo}

Let $\CQ$ be given by a monoidal endofunctor $F_\CR$ of $\CR$, see \secref{sss:functor from homom}, i.e.,
$\CQ=\CQ_{F_\CR}$. In this case, by a slight abuse of notation, we will write
$$\on{HH}_\bullet(\CR,F_\CR)$$
instead of $\on{HH}_\bullet(\CR,\CQ_{F_\CR})$.

\sssec{Example}  \label{sss:2-categ of QCoh}

Let $\CR=\QCoh(\CY)$ for $\CY$ as in \secref{sss:prestack} and with affine diagonal\footnote{One can show that the extra
hypothesis of having an affine diagonal is not necessary here.}, and let $F_\CR$ be given by $\phi^*$ for an endomorphism $\phi$ of $\CY$.

\medskip

By \corref{c:quasi-passable pullback},
 we have:
$$\on{HH}_\bullet(\QCoh(\CY),\phi^*)\simeq \QCoh(\CY^\phi).$$

\sssec{}  \label{sss:Y/phi}

Let us consider another example:

\medskip

Let $\CA$ be a symmetric monoidal category, and for a space $Y$ consider $\CA^{\otimes Y}$. Let
$\phi$ be an endomorphism of $Y$. By functoriality, it induces a symmetric monoidal functor
$$\CA^{\otimes \phi}:\CA^{\otimes Y}\to \CA^{\otimes Y}.$$


\medskip

Note that since the monoidal structures involved are symmetric, the category
$$\on{HH}_\bullet(\CA^{\otimes Y},\CA^{\otimes \phi})$$
also acquires a symmetric monoidal structure.

\medskip

Let $Y/\phi$ denote the quotient of $Y$ by $\phi$, i.e.,
$$Y/\phi:=Y\underset{\on{id}\sqcup \on{id},Y\sqcup Y,\phi\sqcup \on{id}}\sqcup\, Y,$$
where the subscripts indicate the morphisms with respect to which we form the pushout.

\medskip

Note that when $\phi$ is an automorphism, $Y/\phi$ is isomorphic to the quotient $Y/\BZ$, i.e., to
the geometric realization of the bar simplicial space
\begin{equation} \label{e:bar phi}
...\BZ\times Y\rightrightarrows Y.
\end{equation}

\medskip

We claim:

\begin{prop} \label{p:HC twist}
There exists a canonical equivalence
$$\on{HH}_\bullet(\CA^{\otimes Y},\CA^{\otimes \phi}) \simeq \CA^{\otimes Y/\phi}.$$
\end{prop}

\begin{proof}
Follows from the commutation of the functor \eqref{e:integral functor} with colimits in $Y$.
\end{proof}

\sssec{}  \label{sss:Frob LocSys}

Let us specialize further the example considered in \secref{sss:Y/phi} above.

\medskip

Let $\sfe$ be a field of characteristic $0$, and take $\CA=\Rep(\sG)$. Assume that $Y$ has finitely many connected components;
then the same is true for $Y/\phi$.

\medskip

Recall the identifications
$$\Rep(\sG)^{\otimes Y}\simeq \QCoh(\LocSys_\sG(Y)) \text{ and } \Rep(\sG)^{\otimes Y/\phi}\simeq \QCoh(\LocSys_\sG(Y/\phi))$$
of \thmref{t:integral to LocSys}.

\medskip

With respect to the above identifications, the functor $\Rep(\sG)^{\otimes \phi}$ corresponds to the functor
$$\LocSys(\phi)^*:\QCoh(\LocSys_\sG(Y))\to \QCoh(\LocSys_\sG(Y)),$$
where
$$\LocSys(\phi):\LocSys_\sG(Y)\to \LocSys_\sG(Y)$$
is the map induced by $\phi$.

\medskip

By \secref{sss:2-categ of QCoh}, we have
$$\on{HH}_\bullet(\QCoh(\LocSys_\sG(Y)),\LocSys(\phi)^*)\simeq \QCoh(\LocSys_\sG(Y)^{\LocSys(\phi)}).$$

Now,
$$\LocSys_\sG(Y)^{\LocSys(\phi)}\simeq \LocSys_\sG(Y/\phi).$$

To summarize, we obtain
\begin{multline*}
\on{HH}_\bullet(\Rep(\sG)^{\otimes Y},\Rep(\sG)^{\otimes \phi}) \simeq
\on{HH}_\bullet(\QCoh(\LocSys_\sG(Y)),\LocSys(\phi)^*)\simeq \\
\simeq \QCoh(\LocSys_\sG(Y)^{\LocSys(\phi)})
\simeq \QCoh(\LocSys_\sG(Y/\phi))\simeq \Rep(\sG)^{\otimes Y/\phi},
\end{multline*}
which is what \propref{p:HC twist} says in this case.

\ssec{The 2-categorical class map}

\sssec{}  \label{sss:2-categ class}

Let $\CR$ be a monoidal category, and let $\CM$ be an $\CR$-module. Assume that $\CM$ is \emph{right-dualizable}
as an $\CR$-module (see \secref{sss:right-dual}).

\medskip

Then the corresponding functor
$$\fT_\CM:\ul{\DGCat}\to \ul{\CR\mmod}$$
admits a right adjoint.

\medskip

Let $\CQ$ be an object of $(\CR\otimes \CR^{\on{rev}})\mmod$, and let us be given a map
$$\alpha:\CM\to \CQ\underset{\CR}\otimes \CM$$
in $\CR\mmod$.

\medskip

Applying the functoriality of 2-categorical trace from \ref{sss:funct trace 2categ} and repeating the construction of \secref{sss:class},
to the above datum we can assign an object
$$\on{cl}(\CM,\alpha)\in \Tr(\fT_\CQ,\ul{\CR\mod})\simeq \on{HH}_\bullet(\CR,\CQ).$$

%
%
%
%
%

\sssec{}  \label{sss:enhanced trace}

Note that for $\CM\in \CR\mmod$ and $\CQ=\CQ_{F_\CR}$ from \secref{sss:mon endo}, the datum of $\alpha$
as in \secref{sss:2-categ class} amounts to an endofunctor $F_\CM$, which is compatible with the
action of $\CR$:
$$F_\CM(a\otimes m) \simeq F_\CR(a)\otimes F_\CM(m).$$
We denote this correspondence by
$$F_\CM \rightsquigarrow \alpha_{F_\CM}.$$

\medskip

We will denote the corresponding object $\on{cl}(\CM,\alpha_{F_\CM})\in \on{HH}_\bullet(\CR,F_\CR)$ also by
$$\Tr^{\on{enh}}_\CR(F_\CM,\CM).$$

The reason for the notation $\Tr_\CR^{\on{enh}}$ will be explained in Remark \ref{r:enh notation} below.

\medskip

Consider the particular case $\CM=\CR$ and $F_\CM=F_\CR$. We will denote
$$\one_{\on{HH}_\bullet(\CR,F_\CR)}:=\Tr^{\on{enh}}_\CR(F_\CR,\CR)\in \on{HH}_\bullet(\CR,F_\CR).$$

\begin{rem}  \label{r:unit in Hochschild}
The reason for the notation $\one_{\on{HH}_\bullet(\CR,F_\CR)}$ is the following:

\medskip

Assume for a moment that $\CR$ is \emph{symmetric monoidal}. The symmetric monoidal structure on $\CR$ endows
$\ul{\CR\mmod}$ with a structure of commutative algebra object of $\tDGCat$.

\medskip

Assume that
$F_\CR$ is a symmetric monoidal endofunctor of $\CR$. Consider the corresponding
endomorphism $F_\CR$ of $\ul{\CR\mmod}$ (see \secref{sss:mon endo}). This endomorphism
is right-lax symmetric monoidal for the above commutative algebra structure on $\ul{\CR\mmod}$.
Hence, by \secref{sss:2-categ sym mon}, the category $\on{HH}_\bullet(\CR,F_\CR)$ acquires a symmetric monoidal structure.

\medskip

Now, it follows from \eqref{e:2-categ sym mon hom} that the object that we have denoted
$$\one_{\on{HH}_\bullet(\CR,F_\CR)}\in \on{HH}_\bullet(\CR,F_\CR)$$
is the monoidal unit.

\end{rem}

\sssec{}

We claim:

\begin{thm} \label{t:two notions of trace}
Assume that $\CR$ is rigid. Then there is a canonical isomorphism of associative algebras
\begin{equation} \label{e:trace endo}
\Tr(F_\CR,\CR) \simeq \CEnd_{\on{HH}_\bullet(\CR,F_\CR)}(\one_{\on{HH}_\bullet(\CR,F_\CR)}),
\end{equation}
and of modules over these algebras
\begin{equation} \label{e:trace module}
\Tr(F_\CM,\CM) \simeq \CHom_{\on{HH}_\bullet(\CR,F_\CR)}(\one_{\on{HH}_\bullet(\CR,F_\CR)},\Tr^{\on{enh}}_\CR(F_\CM,\CM)).
\end{equation}
\end{thm}

\begin{rem}  \label{r:enh notation}
The reason for the notation $\Tr^{\on{enh}}_\CR(F_\CM,\CM)$ is explained by \thmref{t:two notions of trace}: this theorem says that
the object
$$\Tr(F_\CM,\CM)\in \Vect$$
upgrades to an object of $\on{HH}_\bullet(\CR,F_\CR)$ (namely, $\Tr^{\on{enh}}_\CR(F_\CM,\CM)$),
where ``upgrades" means that the former is the image of the latter under the  functor
$$\CHom_{\on{HH}_\bullet(\CR,F_\CR)}(\one_{\on{HH}_\bullet(\CR,F_\CR)},-):
\on{HH}_\bullet(\CR,F_\CR)\to \Vect,$$
which we can think of as a kind of forgetful functor.
\end{rem}

\begin{rem}
It will follow from the proof that when $\CR$ is \emph{symmetric} monoidal, the isomorphism \eqref{e:trace endo}
that we will construct respects the \emph{commutative} algebra structure on both sides.
\end{rem}

\sssec{Example}

Let $\CY$ be as in \secref{sss:prestack}, and take $\CR=\QCoh(\CY)$. Note the conditions on $\CY$ imply
that $\QCoh(\CY)$ is rigid, see \cite[Chapter 3, Proposition 3.5.3]{GR1}. We let $F_\CR$ be given by $\phi^*$.

\medskip

Let $\CM$ be an $\QCoh(\CY)$-module category, which is dualizable as a plain DG category. Let
$F_\CM$ be an endofunctor of $\CM$ that is compatible with $\phi^*$, i.e., it makes the diagram
$$
\CD
\QCoh(\CY) \otimes \CM  @>>>  \CM  \\
@V{\phi^*\otimes F_\CM}VV   @VV{F_\CM}V  \\
\QCoh(\CY) \otimes \CM  @>>>  \CM
\endCD
$$
commute.

\medskip

Then the construction of \secref{sss:2-categ class} produces an object
$$\Tr^{\on{enh}}_{\QCoh(\CY)}(F_\CM,\CM)\in \QCoh(\CY^\phi)$$
and \thmref{t:two notions of trace} says that
$$\Gamma(\CY^\phi,\Tr^{\on{enh}}_{\QCoh(\CY)}(F_\CM,\CM))\simeq \Tr(F_\CM,\CM).$$

\ssec{A framework for the proof of  \thmref{t:two notions of trace}}  \label{ss:L tDGCat}

\sssec{}

In order to prove \thmref{t:two notions of trace} we will use the formalism of \secref{sss:Lo}, applied to
$$\bO:=\tDGCat,\,\, \End_{\bO}(\one_\bO)=\DGCat.$$ However, we note that both sides in
\begin{equation} \label{e:L-Tr}
\Tr:L(\tDGCat)_{\on{rgd}} \to \DGCat
\end{equation}
are naturally $(\infty,2)$-categories, and functor \eqref{e:L-Tr} upgrades to a functor between $(\infty,2)$-categories.

\sssec{}

The structure of $(\infty,2)$-category on $L(\tDGCat)_{\on{rgd}}$ comes from the structure of
$(\infty,3)$-category on $\tDGCat$.

\medskip

Indeed, note that for a pair of objects $\ul{\CR_i\mmod}\in \tDGCat$,
the $(\infty,1)$-category
$$\bMaps_{\tDGCat}(\ul{\CR_1\mmod},\ul{\CR_2\mmod}):=(\CR_2\otimes \CR_1^{\on{op}})\mmod$$
naturally upgrades to an $(\infty,2)$-category. Namely, for
$$\CQ',\CQ''\in (\CR_2\otimes \CR_1^{\on{op}})\mmod,$$
we can consider the $(\infty,1)$-category
$$\bMaps_{(\CR_2\otimes \CR_1^{\on{op}})\mmod}(\CQ',\CQ'').$$

\sssec{} \label{sss:3-categ}

Let us specify what are the 2-morphisms
in $L(\tDGCat)_{\on{rgd}}$ .

\medskip

By definition, 1-morphisms in $L(\tDGCat)_{\on{rgd}}$ are diagrams
$$
\xy
(0,0)*+{\bo_1}="A";
(20,0)*+{\bo_1}="B";
(0,-20)*+{\bo_2}="C";
(20,-20)*+{\bo_2,}="D";
{\ar@{->}^{f_1} "A";"B"};
{\ar@{->}^{t} "A";"C"};
{\ar@{->}^{t} "B";"D"};
{\ar@{->}_{f_2} "C";"D"};
{\ar@{=>}^{\alpha} "B";"C"};
\endxy
$$
where $\bo_i\in \tDGCat$, and the edges in the above diagram are 1-morphisms in $\tDGCat$,
where we require the 1-morphism $t$ to admit a right adjoint.

\medskip

Given another 1-morphism represented by
\begin{equation} \label{e:2-morph'}
\xy
(0,0)*+{\bo_1}="A";
(20,0)*+{\bo_1}="B";
(0,-20)*+{\bo_2}="C";
(20,-20)*+{\bo_2,}="D";
{\ar@{->}^{f_1} "A";"B"};
{\ar@{->}^{t'} "A";"C"};
{\ar@{->}^{t'} "B";"D"};
{\ar@{->}_{f_2} "C";"D"};
{\ar@{=>}^{\alpha'} "B";"C"};
\endxy
\end{equation}
a 2-morphism between them is a diagram
$$
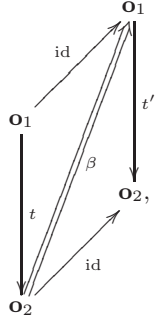
\xy
(0,0)*+{\bo_1}="A";
(15,15)*+{\bo_1}="B";
(0,-25)*+{\bo_2}="C";
(15,-10)*+{\bo_2,}="D";
{\ar@{->}^{\on{id}} "A";"B"};
{\ar@{->}^{t} "A";"C"};
{\ar@{->}^{t'} "B";"D"};
{\ar@{->}_{\on{id}} "C";"D"};
{\ar@{=>}_\beta "C";"B"};
\endxy
$$
where $\beta:t\to t'$ is a 2-morphism in $\tDGCat$ that \emph{admits a right adjoint}\footnote{This is a notion that exists in an
$(\infty,3)$-category.}, and which is equipped with a 3-morphism $\gamma$ for the cube:
$$
\xy
(0,0)*+{\bo_1}="A";
(15,15)*+{\bo_1}="B";
(0,-25)*+{\bo_2}="C";
(15,-10)*+{\bo_2}="D";
(55,0)*+{\bo_1}="A'";
(70,15)*+{\bo_1}="B'";
(55,-25)*+{\bo_2.}="C'";
(70,-10)*+{\bo_2}="D'";
{\ar@{->}^{\on{id}} "A";"B"};
{\ar@{->}^{t} "A";"C"};
{\ar@{->}^{t'} "B";"D"};
{\ar@{->}_{\on{id}} "C";"D"};
{\ar@{=>}^\beta "C";"B"};
{\ar@{->}^{\on{id}} "A'";"B'"};
{\ar@{->}^{t} "A'";"C'"};
{\ar@{->}^{t'} "B'";"D'"};
{\ar@{->}_{\on{id}} "C'";"D'"};
{\ar@{=>}^\beta "C'";"B'"};
{\ar@{->}_{f_1} "A";"A'"};
{\ar@{->}_{f_1} "B";"B'"};
{\ar@{->}_{f_2} "C";"C'"};
{\ar@{->}_{f_2} "D";"D'"};
{\ar@{=>}_{\alpha'} "A'";"C"};
{\ar@{=>}_{\alpha} "B'";"D"};
\endxy
$$

I.e., $\gamma$ is a 3-morphism
\begin{equation} \label{e:gamma morph}
\xy
(0,0)*+{t\circ f_1}="A";
(20,0)*+{f_2\circ t}="B";
(0,-20)*+{t'\circ f_1}="C";
(20,-20)*+{f_2\circ t'.}="D";
{\ar@{->}^{\alpha} "A";"B"};
{\ar@{->}^{\beta} "A";"C"};
{\ar@{->}^{\beta} "B";"D"};
{\ar@{->}_{\alpha'} "C";"D"};
{\ar@{=>}^{\gamma} "B";"C"};
\endxy
\end{equation}

\sssec{}

Let us show how the datum of $(\beta,\gamma)$ as above gives rise to a 2-morphism
$$\Tr(t,\alpha)\to \Tr(t',\alpha')$$
in
$$\End_{\tDGCat}(\ul{\DGCat})\simeq \DGCat,$$
which is a 3-morphism in $\ul{\DGCat}$.

\medskip

It is obtained from the diagram

$$
\xy
(0,0)*+{\Tr(f_1,\bo_1)}="A";
(25,0)*+{\Tr(t^R\circ t\circ f_1,\bo_1)}="B";
(60,0)*+{\Tr(t^R\circ f_2\circ t,\bo_1)}="C";
(100,0)*+{\Tr(f_2\circ t\circ t^R,\bo_2)}="D";
(125,0)*+{\Tr(f_2,\bo_2)}="E";
(0,-20)*+{\Tr(f_1,\bo_2)}="A'";
(25,-20)*+{\Tr(t'{}^R\circ t'\circ f_1,\bo_1)}="B'";
(60,-20)*+{\Tr(t'{}^R\circ f_2\circ t',\bo_1)}="C'";
(100,-20)*+{\Tr(f_2\circ t'\circ t'{}^R,\bo_2)}="D'";
(125,-20)*+{\Tr(f_2,\bo_1),}="E'";
{\ar@{->} "A";"B"};
{\ar@{->} "A'";"B'"};
{\ar@{->}^{\on{cycl}} "C";"D"};
{\ar@{->}_{\on{cycl}} "C'";"D'"};
{\ar@{->} "D";"E"};
{\ar@{->} "D'";"E'"};
{\ar@{->}^{\alpha} "B";"C"};
{\ar@{->}_{\alpha'} "B'";"C'"};
{\ar@{->}^{\on{Id}} "A";"A'"};
{\ar@{->} "D";"D'"};
{\ar@{->}^{\on{Id}} "E";"E'"};
{\ar@{->} "B";"B'"};
{\ar@{->} "C";"C'"};
{\ar@{=>} "B";"A'"};
{\ar@{=>} "E";"D'"};
{\ar@{=>}_{\gamma} "C";"B'"};
{\ar@{=>} "D";"C'"};
\endxy
$$
where the middle vertical arrows are induced by the 2-morphisms:

\begin{itemize}

\item $\beta:t\to t'$;

\item The 2-morphism $t^R\to t'{}^R$, obtained by passing to right adjoints at the level of 1-morphisms
from $\beta^R:t'\to t$, i.e., we consider the partially defined functor
``passing to the right adjoint"
$$\bMaps(\bo_2,\bo_1)\to \bMaps(\bo_1,\bo_2), \quad t\mapsto t^R$$
and apply it to
$$\beta^R\in \Maps_{\bMaps(\bo_2,\bo_1)}(t',t)$$ to obtain an object of
$$\Maps_{\bMaps(\bo_1,\bo_2)}(t^R,t'{}^R).$$

\end{itemize}

\ssec{Proof of \thmref{t:two notions of trace}: isomorphism of the underlying objects of $\Vect$}

\sssec{}

As a first step, in the setting of \thmref{t:two notions of trace}, we will show that we have a canonical isomorphism
\begin{equation} \label{e:trace module again}
\Tr(F_\CM,\CM) \simeq \CHom_{\on{HH}_\bullet(\CR,F_\CR)}(\one_{\on{HH}_\bullet(\CR,F_\CR)},\Tr^{\on{enh}}_\CR(F_\CM,\CM)).
\end{equation}
as objects of $\Vect$.

\medskip

We will do it in the following general framework.

\sssec{}

Let $\CR_1$ and $\CR_2$ be a pair of monoidal DG categories, each equipped with a monoidal endofunctor,
denoted $F_{\CR_1}$ and $F_{\CR_2}$, respectively. Let $\Psi:\CR_1\to \CR_2$ be a monoidal functor,
equipped with an isomorphism
\begin{equation} \label{e:psi intertw}
F_{\CR_2}\circ \Psi\simeq \Psi\circ F_{\CR_1}.
\end{equation}

Restriction and induction along $\Psi$ define an adjoint pair of 1-morphisms in $\tDGCat$
\begin{equation} \label{e:Ind Res adj}
\Ind_\Psi:\ul{\CR_1\mmod}\rightleftarrows \ul{\CR_2\mmod}:\Res_\Psi,
\end{equation}
see \secref{sss:functor from homom}.

\medskip

The isomorphism \eqref{e:psi intertw} give rise to a \emph{commutative square}
\begin{equation} \label{e:restr diag}
\CD
\ul{\CR_2\mmod}   @>{\Res_{F_{\CR_2}}}>>  \ul{\CR_2\mmod}    \\
@V{\Res_\Psi}VV   @VV{\Res_\Psi}V  \\
\ul{\CR_1\mmod}   @>{\Res_{F_{\CR_1}}}>>  \ul{\CR_1\mmod}.
\endCD
\end{equation}

\medskip

By passing to left adjoints along the vertical arrows in \eqref{e:restr diag}, we obtain a lax-commutative square:
(i.e., a square that commutes up to a 2-morphism)
\begin{equation} \label{e:ind diag}
\xy
(0,0)*+{\ul{\CR_1\mmod}}="A";
(30,0)*+{\ul{\CR_1\mmod}}="B";
(0,-20)*+{\ul{\CR_2\mmod}}="C";
(30,-20)*+{\ul{\CR_2\mmod}.}="D";
{\ar@{->}^{\Res_{F_{\CR_1}}} "A";"B"};
{\ar@{->}_{\Ind_\Psi} "A";"C"};
{\ar@{->}^{\Ind_\Psi} "B";"D"};
{\ar@{->}_{\Res_{F_{\CR_2}}} "C";"D"};
{\ar@{=>} "B";"C"};
\endxy
\end{equation}

\sssec{}

Assume now that $\CR_1$ and $\CR_2$ are rigid. In this case, by \secref{sss:right adj rigid}, the 1-morphism $\Res_\Psi$ admits
also a right adjoint.

\medskip

Furthermore, in this case the unit and the counit of the adjunction $(\Ind_\Psi,\Res_\Psi)$, which are 2-morphisms
$$\on{Id} \to \Res_\Psi\circ \Ind_\Psi \text{ and } \Ind_\Psi \circ \Res_\Psi\to \on{Id}$$
in $\tDGCat$, admit \emph{right} adjoints.

\medskip

Indeed, the above 2-morphisms are given by the functors
$$\CR_1\overset{\Psi}\to \CR_2 \text{ and } \CR_2\underset{\CR_1}\otimes \CR_2 \overset{\on{mult}}\to \CR_2$$
(as $(\CR_1,\CR_1)$- and $(\CR_2,\CR_2)$-bimodule categories-bimodule categories, respectively) and the rigidity
assumption implies that these functors admit right adjoints that respect the bimodule structures.

\medskip

From here it follows that the adjunction \eqref{e:Ind Res adj} and the diagrams \eqref{e:restr diag} and \eqref{e:ind diag}
give rise to an adjunction
\begin{equation} \label{e:Ind Res adj End}
\Ind_\Psi:(\ul{\CR_1\mmod},\Res_{F_{\CR_1}})\rightleftarrows (\ul{\CR_2\mmod},\Res_{F_{\CR_2}}):\Res_\Psi
\end{equation}
in the $(\infty,2)$-category $L(\tDGCat)_{\on{rgd}}$. In the unit and counit 2-morphisms
that encode the adjunction \eqref{e:Ind Res adj End}, the corresponding 3-morphisms
(denoted $\gamma$ in \secref{sss:3-categ}, see \eqref{e:gamma morph}) are actually isomorphisms.

\sssec{}

Hence, by \secref{ss:L tDGCat}, from the adjunction \eqref{e:Ind Res adj End}, by applying the functor $\Tr$ of
\eqref{e:L-Tr}, we obtain an adjunction
\begin{equation} \label{e:Ind Res adj L}
\Tr(\Ind_\Psi):\on{HH}_\bullet(\CR_1,F_{\CR_1}) \rightleftarrows \on{HH}_\bullet(\CR_2,F_{\CR_2}):\Tr(\Res_\Psi).
\end{equation}


\sssec{}  \label{sss:setting for gen two trace}

Let $\CM$ be an object of $\CR_2\mmod$; assume that it is dualizable as a DG category. Since $\CR_2$
was assumed rigid, we obtain that $\CM$ is dualizable as an object of $\CR_2\mmod$.

\medskip

Let $F_\CM$ be its endofunctor as in \secref{sss:enhanced trace}.
Then to it there corresponds an object
$$\Tr^{\on{enh}}_{\CR_2}(F_\CM,\CM)\in \on{HH}_\bullet(\CR_2,F_{\CR_2}).$$

\medskip

Consider now the object $\Res_\Psi(\CM)\in \CR_1\mmod$. Since $\CR_1$ is rigid and $\CM$ is dualizable as a DG category,
we obtain that $\Res_\Psi(\CM)$ is dualizable as an object of $\CR_1\mmod$.

\medskip

The datum of $F_\CM$ defines the corresponding datum for
$\Res_\Psi(\CM)\in \CR_1\mmod$.

\medskip

We will deduce \thmref{t:two notions of trace} (at the level of the underlying objects of $\Vect$) from the following more general relative statement:

\begin{thm}\label{t:rel two notions of trace}
In the situation above, there is a canonical isomorphism
$$ \Tr(\Ind_\Psi)^R(\Tr^{\on{enh}}_{\CR_2}(F_\CM,\CM))\simeq
\Tr^{\on{enh}}_{\CR_1}(F_\CM,\Res_\Psi(\CM)) .$$
\end{thm}

\sssec{Proof of \thmref{t:rel two notions of trace}}

We can view the pair $(\CM,F_\CM)$ as a 1-morphism
\begin{equation} \label{e:cl M pre}
(\ul{\DGCat},\on{Id}) \to (\ul{\CR_2\mmod},\Res_{F_{\CR_2}})
\end{equation}
in $L(\tDGCat)_{\on{rgd}}$.

\medskip

By construction,
$$\Tr^{\on{enh}}_{\CR_2}(F_\CM,\CM)\in \Tr(\Res_{F_{\CR_2}},\ul{\CR_2\mmod}):=\on{HH}_\bullet(\CR_2,F_{\CR_2})$$
is obtained by applying the functor $\Tr$ of \eqref{e:L-Tr} to \eqref{e:cl M pre}.

\medskip

The pair $(\Res_\Psi(\CM),F_\CM)$ is obtained from \eqref{e:cl M pre} as the composition
\begin{equation} \label{e:Res M}
(\ul{\DGCat},\on{Id}) \overset{(\CM,F_\CM)}\longrightarrow (\ul{\CR_2\mmod},\Res_{F_{\CR_2}})
\overset{\Res_\Psi}\longrightarrow  (\ul{\CR_1\mmod},\Res_{F_{\CR_1}}).
\end{equation}

\medskip

Hence,
$$\Tr^{\on{enh}}_{\CR_1}(F_\CM,\Res_\Psi(\CM))\in \Tr(\Res_{F_{\CR_1}},\ul{\CR_1\mmod}):=\on{HH}_\bullet(\CR_1,F_{\CR_1})$$
is obtained by applying the functor $\Tr$ of \eqref{e:L-Tr} to \eqref{e:Res M}.

\sssec{}

From here we obtain a canonical identification
\begin{equation} \label{e:trace module gen}
\Tr(\Res_\Psi)(\Tr^{\on{enh}}_{\CR_2}(F_\CM,\CM))\simeq
\Tr^{\on{enh}}_{\CR_1}(F_\CM,\Res_\Psi(\CM)).
\end{equation}

Note that by that by \eqref{e:Ind Res adj L}, we have:
$$\Tr(\Res_\Psi)\simeq \Tr(\Ind_\Psi)^R.$$

Hence, we can rewrite \eqref{e:trace module gen} as:
\begin{equation} \label{e:trace module gen R}
\Tr(\Ind_\Psi)^R(\Tr^{\on{enh}}_{\CR_2}(F_\CM,\CM))\simeq
\Tr^{\on{enh}}_{\CR_1}(F_\CM,\Res_\Psi(\CM)).
\end{equation}

\qed[\thmref{t:rel two notions of trace}]

\begin{rem}
Note that for the construction of the equivalence \eqref{e:trace module gen R} we use less than the full force of the assumption
that both $\CR_1$ and $\CR_2$ be rigid. What we actually use is that the symmetric monoidal functor $\Psi:\CR_1\to \CR_2$
is rigid (we leave it to the reader to work out what this means).

\medskip

For example, when $\CR_1$ is symmetric monoidal and $\Psi$
makes $\CR_2$ into a $\CR_1$-algebra object in $\DGCat$, the assumption we need is that $\CR_2$ be rigid \emph{over} $\CR_1$.
For $\CR_1=\Vect$ and $\CR_2=\CR$, this just means that $\CR$ is rigid.
\end{rem}

\sssec{} \label{sss:trace module again}

Let us now deduce the isomorphism \eqref{e:trace module again}. Take
$$(\CR_2,F_{\CR_2})=(\CR,F_{\CR}) \text{ and } (\CR_1,F_{\CR_1})=(\Vect,\on{Id})$$
with $\Psi$ being the unit functor $\Vect\to \CR$.

\medskip

Note that by construction, the functor
$$\Tr(\Ind_\Psi):\Vect\simeq \Tr(\on{Id},\ul{\DGCat})\to
\Tr(\Res_{F_\CR},\ul{\CR\mmod}):=\on{HH}_\bullet(\CR,F_\CR)$$
sends
$$\sfe \mapsto \one_{\on{HH}_\bullet(\CR,F_\CR)},$$
see Remark \ref{r:unit in Hochschild}.

\medskip

Hence, $\Tr(\Ind_\Psi)^R$ is given by
$$\CHom_{\on{HH}_\bullet(\CR,F_\CR)}(\one_{\on{HH}_\bullet(\CR,F_\CR)},-).$$

Finally, apply \eqref{e:trace module gen R}.


%
%
%
%

\ssec{Proof of \thmref{t:two notions of trace}: algebra and module structure}

\sssec{}

The functor $\Tr$ of \eqref{e:L-Tr} induces a functor
\begin{equation} \label{e:L-Tr unit}
\Tr:\End_{L(\tDGCat)_{\on{rgd}}}(\ul{\DGCat},\on{Id})\to \End_{\DGCat}(\Vect)\simeq \Vect.
\end{equation}

We note that we have a canonical identification of symmetric monoidal categories
\begin{equation} \label{e:L unit}
L(\DGCat)_{\on{rgd}}\simeq \End_{L(\tDGCat)_{\on{rgd}}}(\ul{\DGCat},\on{Id}),
\end{equation}
and the resulting functor
$$L(\DGCat)_{\on{rgd}}\to \Vect$$
is the functor \eqref{e:2-tr as functor} for $\bO=\DGCat$.

\sssec{}

Let $\Psi$ be the unit functor
$$\DGCat\to \CR.$$

\medskip

The adjunction
\begin{equation} \label{e:L unit Psi}
\Ind_\Psi:(\ul{\DGCat},\on{Id}) \leftrightarrows (\ul{\CR\mmod},\Res_{F_\CR}):\Res_\Psi
\end{equation}
defines a monad on $(\ul{\DGCat},\on{Id})\in L(\tDGCat)_{\on{rgd}}$, i.e., an associative algebra
in the (symmetric) monoidal category
$$\End_{L(\tDGCat)_{\on{rgd}}}(\ul{\DGCat},\on{Id}).$$

Under the identification \eqref{e:L unit} this algebra is given by
$$(\CR,F_\CR)\in L(\DGCat)_{\on{rgd}}.$$

Applying the functor $\Tr$ of \eqref{e:L-Tr unit} (which, by the above, is the same as the functor $\Tr$ of
\eqref{e:2-tr as functor} for $\bO=\DGCat$) we recover
$$\Tr(F_\CR,\CR)\in \Vect$$
with its associative algebra structure.

\sssec{}

Now, by the functoriality of \eqref{e:L-Tr}, this associative algebra, regarded as a monad on $\Vect$, identifies
with the monad corresponding to the adjunction
\begin{equation} \label{e:L unit Psi Tr}
\Tr(\Ind_\Psi):\Tr((\ul{\DGCat},\on{Id})) \leftrightarrows \Tr((\ul{\CR\mmod},\Res_{F_\CR})):\Tr(\Res_\Psi),
\end{equation}
i.e., the adjunction obtained from \eqref{e:L unit Psi} by applying the functor $\Tr$.

\medskip

We identify
$$\Tr((\ul{\DGCat},\on{Id})) \simeq \Vect \text{ and } \Tr((\ul{\CR\mmod},\Res_{F_\CR}))\simeq \on{HH}_\bullet(\CR,F_{\CR}),$$
where the 1-morphism $\Tr(\Ind_\Psi)$ identifies with $\sfe\mapsto \one_{\on{HH}_\bullet(\CR,F_\CR)}$.

\medskip

Hence, we obtain that the associative algebra $\Tr(F_\CR,\CR)\in \Vect$ corresponds to the adjunction
$$\Vect\leftrightarrows \on{HH}_\bullet(\CR,F_{\CR}), \quad \sfe\mapsto \one_{\on{HH}_\bullet(\CR,F_\CR)}.$$

This establishes the isomorphism \eqref{e:trace endo} as associative algebras.

\sssec{}

Similarly, a pair $(\CM,F_\CM)$ can be viewed as a 1-morphism
$$(\ul{\DGCat},\on{Id}) \to (\ul{\CR\mmod},\Res_{F_\CR}).$$

Composing with $\Res_\Psi$ we obtain an object in $\End_{L(\tDGCat)_{\on{rgd}}}(\ul{\DGCat},\on{Id})$,
which is a module over the monad $\Res_\Psi\circ \Ind_\Psi$, and
which under the identification \eqref{e:L unit} corresponds to
$$(\CM,F_\CM)\in (\CR,F_\CR)\mod(L(\DGCat)_{\on{rgd}}).$$

Applying the functor $\Tr$ of \eqref{e:2-tr as functor} for $\bO=\DGCat$, we recover
$\Tr(F_\CM,\CM)$ as a module over $\Tr(F_\CR,\CR)$.

\medskip

Now, by the functoriality of \eqref{e:L-Tr}, the pair
$$\Tr(F_\CR,\CR), \,\,\Tr(F_\CM,\CM)\in \Tr(F_\CR,\CR)\mod$$
is the same as one obtained from
$$(\CR,F_\CR), \,\, (\CM,F_\CM)\in (\CR,F_\CR)\mod(L(\DGCat)_{\on{rgd}})$$
by applying the functor $\Tr$ of \eqref{e:L-Tr}.

\medskip

This establishes the isomorphism \eqref{e:trace module} as modules over the two sides of
\eqref{e:trace endo}.

\qed[\thmref{t:two notions of trace}]

\ssec{A more elementary proof of \thmref{t:two notions of trace}} \label{ss:two notions of trace expl}

For the convenience of the reader and future reference, in this subsection we will
outline a more elementary proof (of a particular case) of \thmref{t:two notions of trace}, which does not use the machinery of
$(\infty,3)$-categories.

\medskip

We first establish the stated isomorphism for the underlying objects of $\Vect$.

\sssec{}

Consider the following composition of 1-morphisms in $\tDGCat$
$$\ul{\DGCat} \overset{\CM}\longrightarrow \ul{\CR\mmod} \overset{\on{oblv}}\longrightarrow \ul{\DGCat},$$
where the second arrow is the forgetful map, i.e., given by $\fT_\CR$. The composite is
the map
$$\ul{\DGCat}\to \ul{\DGCat}$$
corresponding to $\on{oblv}(\CM)\in \DGCat$, i.e., $\CM$, viewed as a plain DG category.

\medskip

We have the following diagram of 2-morphisms
$$
\xy
(0,0)*+{\ul{\DGCat}}="A";
(30,0)*+{\ul{\DGCat}}="B";
(0,-20)*+{\ul{\CR\mmod}}="C";
(30,-20)*+{\ul{\CR\mmod}}="D";
(0,-40)*+{\ul{\DGCat}}="E";
(30,-40)*+{\ul{\DGCat}.}="F";
{\ar@{->}^{\on{Id}} "A";"B"};
{\ar@{->}_{\CM} "A";"C"};
{\ar@{->}^{\CM} "B";"D"};
{\ar@{->}_{F_\CR} "C";"D"};
{\ar@{=>}^{\alpha_{F_\CM}} "B";"C"};
{\ar@{->}^{\on{Id}} "E";"F"};
{\ar@{->}_{\on{oblv}} "C";"E"};
{\ar@{->}^{\on{oblv}} "D";"F"};
{\ar@{=>}^{\on{taut}} "D";"E"};
\endxy
$$

The composite 2-morphism identifies with
$$
\xy
(0,0)*+{\ul{\DGCat}}="A";
(20,0)*+{\ul{\DGCat}}="B";
(0,-20)*+{\ul{\DGCat}}="C";
(20,-20)*+{\ul{\DGCat}.}="D";
{\ar@{->}^{\on{Id}} "A";"B"};
{\ar@{->}_{\CM} "A";"C"};
{\ar@{->}^{\CM} "B";"D"};
{\ar@{->}_{\on{Id}} "C";"D"};
{\ar@{=>}^{F_\CM} "B";"C"};
\endxy
$$

By unwinding the definitions, it is easy to see that the resulting map
$$\Tr(\on{Id},\ul{\DGCat})\to\Tr(\on{Id},\ul{\DGCat}),$$
viewed as a functor
$$\DGCat\to \DGCat,$$
is given by $\Tr(F_\CM,\CM)\in \DGCat$.

\sssec{}  \label{sss:categ RR}

Hence, to prove the isomorphism \eqref{e:trace module} as vector spaces, it suffices to show that the map
$$\Tr(F_\CR,\ul{\CR\mmod}) \to \Tr(\on{Id},\ul{\DGCat}),$$
corresponding to the diagram
$$
\xy
(0,-20)*+{\ul{\CR\mmod}}="C";
(30,-20)*+{\ul{\CR\mmod}}="D";
(0,-40)*+{\ul{\DGCat}}="E";
(30,-40)*+{\ul{\DGCat},}="F";
{\ar@{->}_{F_\CR} "C";"D"};
{\ar@{->}^{\on{Id}} "E";"F"};
{\ar@{->}_{\on{oblv}} "C";"E"};
{\ar@{->}^{\on{oblv}} "D";"F"};
{\ar@{=>}^{\on{taut}} "D";"E"};
\endxy$$
viewed as a functor
\begin{equation} \label{e:calc LHS}
\on{HH}_\bullet(\CR,F_\CR)\to \Vect,
\end{equation}
is given by
\begin{equation} \label{e:calc RHS}
\CHom_{\on{HH}_\bullet(\CR,F_\CR)}(\one_{\on{HH}_\bullet(\CR,F_\CR)},-).
\end{equation}

\sssec{}

Consider the corresponding diagram
\begin{equation} \label{e:main diagram}
\xy
(0,0)*+{\ul{\DGCat}}="A";
(40,0)*+{\ul{\CR^{\otimes 2}\mmod}}="B";
(80,0)*+{\ul{\CR^{\otimes 2}\mmod}}="C";
(120,0)*+{\ul{\DGCat}}="D";
(0,-20)*+{\ul{\DGCat}}="A'";
(40,-20)*+{\ul{\DGCat}}="B'";
(80,-20)*+{\ul{\DGCat}}="C'";
(120,-20)*+{\ul{\DGCat}.}="D'";
{\ar@{->}^{\on{unit}} "A";"B"};
{\ar@{->}_{\on{unit}} "A'";"B'"};
{\ar@{->}^{\on{counit}} "C";"D"};
{\ar@{->}_{\on{counit}} "C'";"D'"};
{\ar@{->}^{F_\CR\otimes \on{Id}} "B";"C"};
{\ar@{->}_{\on{Id}} "B'";"C'"};
{\ar@{->}^{\on{Id}} "A";"A'"};
{\ar@{->}^{\on{Id}} "D";"D'"};
{\ar@{->} "B";"B'"};
{\ar@{->} "C";"C'"};
{\ar@{=>} "B";"A'"};
{\ar@{=>}_{\on{taut}} "C";"B'"};
{\ar@{=>} "D";"C'"};
\endxy
\end{equation}

We need to calculate the 2-morphism \emph{from the clockwise circuit}, which corresponds to
$$\on{HH}_\bullet(\CR,F_\CR)\in \DGCat\simeq \End(\ul{\DGCat}),$$
\emph{to the counterclockwise circuit}, which corresponds to
$$\Vect\in  \DGCat\simeq \End(\ul{\DGCat}).$$

\sssec{}  \label{sss:coind categ}

We now recall that for a rigid symmetric monoidal category $\CR'$, the \emph{right} adjoint to
\begin{equation} \label{e:oblv act}
\ul{\CR'\mmod} \overset{\on{oblv}}\longrightarrow \ul{\DGCat}
\end{equation}
identifies with
\begin{equation} \label{e:oblv act R}
\ul{\DGCat} \overset{\CR'}\longrightarrow \ul{\CR'\mmod},
\end{equation}
see \secref{sss:right adj rigid}.

%
%
%
%
%
%
%
\medskip

Thus, we obtain that the two middle vertical arrows in diagram \eqref{e:main diagram} are given by the forgetful map
$$\ul{(\CR\otimes \CR)\mmod} \overset{\on{oblv}\otimes \on{oblv}}\longrightarrow \ul{\DGCat}.$$

Furthermore, we obtain that the 2-morphism in the left square, viewed as a functor
$$\CR\to \Vect,$$
is given by
$$\CHom_\CR(\one_\CR,-).$$

The 2-morphism in the right square, evaluated on $\CQ\in (\CR\otimes \CR)\mmod$ (see \secref{sss:evaluate} for
what we mean by ``evaluate"), is the map
$$\on{HH}_\bullet(\CR,\CQ)=\CR\underset{\CR\otimes \CR}\otimes \CQ \overset{\on{mult}^R\otimes \on{Id}_\CQ }\longrightarrow
(\CR\otimes \CR)\underset{\CR\otimes \CR}\otimes \CQ\simeq \CQ,$$
which is the right adjoint to the map
$$\CQ\simeq (\CR\otimes \CR)\underset{\CR\otimes \CR}\otimes \CQ  \overset{\on{mult} \otimes \on{Id}_\CQ}\longrightarrow
\CR\underset{\CR\otimes \CR}\otimes \CQ=\on{HH}_\bullet(\CR,\CQ).$$

\sssec{}  \label{sss:categ RR bis}

We obtain that the functor in \eqref{e:calc LHS} equals the composite
$$\on{HH}_\bullet(\CR,F_\CR)=\CR\underset{\on{mult},\CR\otimes \CR,\on{mult}\circ (F_\CR\otimes \on{Id})}\otimes
\CR\overset{(\on{mult})^R\otimes \on{Id}}\longrightarrow   (\CR\otimes \CR) \underset{\CR\otimes \CR,\on{mult}\circ (F_\CR\otimes \on{Id})}\otimes \CR \simeq
\CR \overset{\CHom_\CR(\one_\CR,-)}\longrightarrow \Vect.$$

By adjunction, this composite is the same as
$$\on{HH}_\bullet(\CR,F_\CR) \overset{\CHom_{\on{HH}_\bullet(\CR,F_\CR)}((\on{mult}\otimes \on{Id})(\one_\CR),-)}\longrightarrow \Vect,$$
i.e., \eqref{e:calc RHS}.

\begin{rem}
Let us contrast the computation of the map \eqref{e:calc LHS} with the computation of the map
\begin{equation} \label{e:Todd}
\Tr(\on{Id},\QCoh(\CY))\to \sfe
\end{equation}
corresponding to the functor
$$\Gamma(\CY,-):\QCoh(\CY)\to \Vect,$$
where $\CY$ is a smooth proper scheme.

\medskip

We identity
$$\Tr(\on{Id},\QCoh(\CY))\simeq \underset{i}\oplus \,\Gamma(\CY,\Omega^i(\CY))[i],$$
see \secref{sss:Chern}.

\medskip

The computation performed in \cite{KP2} amounts to saying that the resulting map
$$\underset{i}\oplus \,\Gamma(\CY,\Omega^i(\CY))[i]\to \sfe$$
is the projection
$$\underset{i}\oplus \,\Gamma(\CY,\Omega^i(\CY))[i] \to
\Gamma(\CY,\Omega^{\on{top}}(\CY))[\on{top}]\overset{\text{Serre duality}}\longrightarrow \sfe,$$
\emph{precomposed} with the operation of multiplication by the Todd class. So, it is highly non-trivial.

\medskip

By contrast, in the setting of \secref{sss:categ RR}, the map \eqref{e:calc LHS} is something very
simple, namely, the map \eqref{e:calc RHS}.

\medskip

This my be viewed as an incarnation of the fact that for a rigid symmetric monoidal category, the 2-category
$\CR\mmod$
is 0-Calabi-Yau, in the sense that the left and right adjoints to the functor
$$\ul{\DGCat}\to \ul{\CR\mmod}, \quad \CC\mapsto \CR\otimes \CC$$
are canonically isomorphic, see Remark \ref{r:Ind and coInd}.

\end{rem}

\sssec{}  \label{sss:trace endo alg}

Thus, we have established the isomorphism between the two sides of \eqref{e:trace module} as objects of $\Vect$.
In particular, we obtain an isomorphism between the two sides of \eqref{e:trace endo}, also as objects of $\Vect$.

\medskip

We will now assume that $\CR$ is \emph{symmetric monoidal}, and upgrade these isomorphisms to isomorphisms
of algebras (resp., modules over them). This will be achieved by an Eckmann-Hilton argument.

\medskip

Let $\bQ$ denote the category, whose objects are quadruples
$$(\CR,\CM,F_\CR,F_\CM),$$
where $\CR$ is a rigid symmetric monoidal DG category, and $\CM$ is an $\CR$-module, dualizable as a plain DG category.

\medskip

For a pair of objects $(\CR,\CM,F_\CR,F_\CM)$ and $(\CR',\CM',F_{\CR'},F_{\CM'})$, the space of morphisms between them
consists of a symmetric monoidal functor $\varphi_\CR:\CR\to \CR'$, intertwining $F_\CR$ with $F_{\CR'}$, and a functor
of $\CR$-module categories $\varphi_\CM:\CM\to \CM'$, intertwining $F_\CM$ with $F_{\CM'}$, \emph{such that} the induced
functor
$$\CR'\underset{\CR}\otimes \CM\to \CM'$$
is an equivalence.

\medskip

The assignments
\begin{equation} \label{e:quad functor1}
(\CR,\CM,F_\CR,F_\CM)\mapsto \Tr(F_\CM,\CM)
\end{equation}
 and
\begin{equation} \label{e:quad functor2}
(\CR,\CM,F_\CR,F_\CM)\mapsto \CHom_{\on{HC}_\bullet(\CR,F_\CR)}(\one_{\on{HC}_\bullet(\CR,F_\CR)},\Tr^{\on{enh}}_\CR(F_\CM,\CM))
\end{equation}
are both functors $\bQ\to \Vect$.

\medskip

Moreover, the category $\bQ$ carries a naturally defined symmetric monoidal structure:
$$(\CR^1,\CM^1,F_{\CR^1},F_{\CM^1}) \otimes (\CR^2,\CM^2,F_{\CR^2},F_{\CM^2}):=
(\CR^1\otimes \CR^2,\CM^1\otimes \CM^2,F_{\CR^1}\otimes F_{\CR^2},F_{\CM^1}\otimes F_{\CM^2}),$$
and  functors \eqref{e:quad functor1} and \eqref{e:quad functor2} are symmetric monoidal.

\medskip

Furthermore, by the construction of the isomorphism of \eqref{e:trace module}, it
upgrades to an isomorphism of symmetric monoidal functors \eqref{e:quad functor1} and \eqref{e:quad functor2},
as symmetric monoidal functors.

\medskip

Note now that
$$(\CR,\CR,F_\CR,F_\CR)$$
is naturally a commutative algebra in $\bQ$. Hence,
both sides of \eqref{e:trace endo} have a structure
of associative algebras, and \eqref{e:trace endo} respects these structures.

\medskip

By construction, the above commutative algebra structure on $\Tr(F_\CR,\CR)$ is the same one as given
by the construction of \secref{sss:trace on sym mon dg categ}. Furthermore, by the Eckmann-Hilton argument,
the above commutative algebra structure on
$$\CEnd_{\on{HH}_\bullet(\CR,F_\CR)}(\one_{\on{HH}_\bullet(\CR,F_\CR)})$$
goes over under the forgetful functor
$$\on{ComAlg}(\Vect)\to \on{AssocAlg}(\Vect),$$
to the structure of associative algebra on $\CEnd$.

\medskip

This implies the assertion that \eqref{e:trace endo} is an algebra isomorphism.

\sssec{}  \label{sss:trace endo alg mod}

Finally, for $(\CM,F_\CM)$ as in \secref{sss:enhanced trace}, the object
$$(\CR,\CM,F_\CR,F_\CM)\in \bQ$$ is a module over the algebra object $(\CR,\CR,F_\CR,F_\CR)$.
The construction in \secref{sss:trace endo alg} gives each side of \eqref{e:trace module} a structure of module over the corresponding side of
\eqref{e:trace endo}, and \eqref{e:trace module} respects these structures.

\medskip

The resulting action of $\Tr(F_\CR,\CR)$ on $\Tr(F_\CM,\CM)$ is the same one as given
by the construction of \secref{sss:trace on sym mon dg categ}. Again, by the Eckmann-Hilton argument,
the resulting action of
$$\CEnd_{\on{HH}_\bullet(\CR,F_\CR)}(\one_{\on{HH}_\bullet(\CR,F_\CR)})$$
on
$$\CHom_{\on{HH}_\bullet(\CR,F_\CR)}(\one_{\on{HH}_\bullet(\CR,F_\CR)},\Tr^{\on{enh}}_\CR(F_\CM,\CM))$$
coincides with one coming from the action of $\CEnd$ on $\CHom$.

\medskip

This implies that \eqref{e:trace module} is an isomorphism of modules, as desired.

\section{A few mind-twisters}  \label{s:wild}

In this section we will study some particular cases and generalizations of \thmref{t:two notions of trace}.
We recommend the reader to skip this section on the first pass, because the assertions contained therein may
appear abstract and un-motivated, and return to it when necessary.

\medskip

That said, the results discussed in this section will all acquire a transparent meaning in the context
of shtukas, which will be introduced in \secref{s:sht}.

\ssec{The class of a class}

\sssec{}

Let $\CR$ be a rigid symmetric monoidal category, and let $F_\CR:\CR\to \CR$ be a symmetric monoidal endofunctor.

\medskip

Consider the category
\begin{equation} \label{e:HH as ten}
\on{HH}_\bullet(\CR,F_\CR):= \CR\underset{\on{mult},\CR\otimes \CR,\on{mult}\circ (F_\CR\otimes \on{Id})}\otimes \CR.
\end{equation}

Since the tensor product in the right-hand side of \eqref{e:HH as ten} involves symmetric monoidal categories and functors,
we obtain that $\on{HH}_\bullet(\CR,F_\CR)$ acquires a symmetric monoidal structure.

\medskip

On the other hand, by \eqref{e:Tr as HH}, we have
$$\on{HH}_\bullet(\CR,F_\CR)\simeq \Tr(\Res_{F_\CR},\ul{\CR\mmod}),$$
and hence it acquires a symmetric monoidal structure by \secref{sss:2-categ sym mon}.

\medskip

However, it is easy to see that these two ways of defining a symmetric monoidal structure on $\on{HH}_\bullet(\CR,F_\CR)$
are equivalent.

\sssec{}

Denote by $\iota$ the functor
$$\CR\to \CR\underset{\on{mult},\CR\otimes \CR,\on{mult}\circ (F_\CR\otimes \on{Id})}\otimes \CR \simeq \on{HH}_\bullet(\CR,F_\CR),$$
corresponding to the left copy of $\CR$. By \eqref{e:HH as ten}, the functor $\iota$ is symmetric monoidal.

\medskip

Note that by construction
\begin{equation} \label{e:iota inv}
\iota\circ F_\CR\simeq \iota.
\end{equation}

\sssec{}

Let $r\in \CR$ be a compact object equipped with a map
$$\alpha_r:r\to F_\CR(r).$$

\medskip

On the one hand, to the pair $(r,\alpha_r)$ we attach its class
$$\on{cl}(r,\alpha_r)\in \Tr(F_\CR,\CR).$$

\sssec{}  \label{sss:iota F}


\medskip

The data of $\alpha_r$ gives rise to a map
$$\iota(r) \overset{\alpha_r}\longrightarrow \iota\circ F_\CR (r)\overset{\text{\eqref{e:iota inv}}}\simeq \iota(r);$$
denote this map by $a_r^{F_\CR}$.

\medskip

Since $r\in \CR$ is compact and $\CR$ is rigid, we obtain that $r$ is dualizable
as an object of $\CR$ as a monoidal category. Since $\iota$ is symmetric monoidal, we obtain that $\iota(r)$ is dualizable as an object of
$\on{HH}_\bullet(\CR,F_\CR)$.

\medskip

So, on the other hand, we can consider the element
$$\Tr(a_r^{F_\CR},\iota(r))\in \CEnd_{\on{HH}_\bullet(\CR,F_\CR)}(\one_{\on{HH}_\bullet(\CR,F_\CR)}).$$

\sssec{}

We claim:

\begin{prop} \label{p:class R}
The elements $\on{cl}(r,\alpha_r)$ and $\Tr(a_r^{F_\CR},\iota(r))$ coincide under the identification
\begin{equation} \label{e:two notions of trace again}
\Tr(F_\CR,\CR)\simeq \CEnd_{\on{HH}_\bullet(\CR,F_\CR)}(\one_{\on{HH}_\bullet(\CR,F_\CR)})
\end{equation}
of \thmref{t:two notions of trace}.
\end{prop}

\begin{rem}
Note that in the particular case of $\CR=\QCoh(\CY)$ for a prestack $\CY$ as in \secref{sss:prestack},
and $F_\CR$ given by $\phi^*$ for an endomorphism $\phi$ of $\CY$, the assertion of \propref{p:class R}.
coincides with that of \propref{p:class QCoh}.
%
%
%
\end{rem}

\ssec{Proof of \propref{p:class R}}  \label{ss:KPgen}

The proof is a word-for-word repetition of the proof of \cite[Proposition 2.2.3]{KP1}. We include it for the sake of completeness.

\sssec{}

First, let us make the isomorphism \eqref{e:two notions of trace again} explicit (this will imitate the manipulation
in \secref{sss:Tr QCoh}).

\medskip

Recall that $\CR$ is self-dual as a DG category, with the duality datum given by
$$\Vect \overset{\one_\CR}\to \CR \overset{\on{mult}^R}\longrightarrow \CR\otimes \CR$$
and
$$\CR\otimes \CR \overset{\on{mult}}\longrightarrow \CR \overset{\CHom_\CR(\one_\CR,-)}\longrightarrow \Vect.$$

Hence, on the one hand, $\Tr(F_\CR,\CR)$ is the composition
\begin{equation} \label{e:calc HH}
\CD
& &  \CR @>{\CHom_\CR(\one_\CR,-)}>> \Vect  \\
& & @AA{\on{mult}}A \\
& &  \CR\otimes \CR  \\
& & @AA{F_\CR\otimes \on{Id}}A  \\
\CR @>{\on{mult}^R}>> \CR\otimes \CR  \\
@A{\one_\CR}AA  \\
\Vect.
\endCD
\end{equation}

Consider the commutative diagram
\begin{equation} \label{e:calc HH1}
\CD
\CR\underset{\on{mult},\CR\otimes \CR,\on{mult}\circ (F_\CR\otimes \on{Id})}\otimes \CR  @<{\iota}<< \CR  \\
@A{\iota}AA  @AA{\on{mult}\circ (F_\CR\otimes \on{Id})}A  \\
\CR  @<{\on{mult}}<<  \CR\otimes \CR.
\endCD
\end{equation}

By rigidity, the diagram obtained by passing to right adjoints along the horizontal arrows is also commutative:

\begin{equation} \label{e:calc HH2}
\CD
\CR\underset{\on{mult},\CR\otimes \CR,\on{mult}\circ (F_\CR\otimes \on{Id})}\otimes \CR  @>{\iota^R}>> \CR  \\
@A{\iota}AA  @AA{\on{mult}\circ (F_\CR\otimes \on{Id})}A  \\
\CR  @>{\on{mult}^R}>>  \CR\otimes \CR.
\endCD
\end{equation}

Hence, the composite in \eqref{e:calc HH} identifies with
\begin{equation} \label{e:calc HH3}
\CD
\CR\underset{\on{mult},\CR\otimes \CR,\on{mult}\circ (F_\CR\otimes \on{Id})}\otimes \CR @>{\iota^R}>> \CR @>{\CHom_\CR(\one_\CR,-)}>> \Vect \\
@A{\iota}AA \\
\CR  \\
@A{\one_\CR}AA  \\
\Vect.
\endCD
\end{equation}

In the latter diagram, the composite horizontal arrow is the right adjoint of the composite vertical arrow, and the latter is
$$\sfe\mapsto \one_{\on{HH}_\bullet(\CR,F_\CR)}\in \on{HH}_\bullet(\CR,F_\CR)\simeq
\CR\underset{\on{mult},\CR\otimes \CR,\on{mult}\circ (F_\CR\otimes \on{Id})}\otimes \CR.$$

Hence, the resulting functor $\Vect\to \Vect$ is given by $\CEnd_{\on{HH}_\bullet(\CR,F_\CR)}(\one_{\on{HH}_\bullet(\CR,F_\CR)})$.

\medskip

By unwinding the constructions (see \secref{ss:two notions of trace expl}), one shows that the identification
\begin{equation} \label{e:two notions of trace again again}
\Tr(F_\CR,\CR)\simeq \CEnd_{\on{HH}_\bullet(\CR,F_\CR)}(\one_{\on{HH}_\bullet(\CR,F_\CR)})
\end{equation}
just constructed is equivalent to one in \thmref{t:two notions of trace}.

\sssec{}

For a compact object $r\in \CR$, let $r^\vee\in \CR$ be its monoidal dual; this is also its formal dual with respect
to the identification $\CR^\vee\simeq \CR$.

\medskip

The class $\on{cl}(r,\alpha_r)\in \Tr(F_\CR,\CR)$ corresponds to the 2-morphism from the clockwise circuit to the
counter-clockwise circuit in the following diagram

$$
\xy
(0,0)*+{\Vect}="A";
(30,0)*+{\Vect}="B";
(60,0)*+{\Vect}="C";
(90,0)*+{\Vect}="D";
(0,-20)*+{\Vect}="A'";
(30,-20)*+{\CR\otimes \CR}="B'";
(60,-20)*+{\CR\otimes \CR}="C'";
(90,-20)*+{\Vect.}="D'";
{\ar@{->}^{\on{Id}} "A";"B"};
{\ar@{->}_{\on{unit}} "A'";"B'"};
{\ar@{->}^{\on{Id}} "C";"D"};
{\ar@{->}_{\on{counit}} "C'";"D'"};
{\ar@{->}^{\on{Id}} "B";"C"};
{\ar@{->}_{F_\CR\otimes \on{Id}} "B'";"C'"};
{\ar@{->}^{\on{Id}} "A";"A'"};
{\ar@{->}^{\on{Id}} "D";"D'"};
{\ar@{->}_{r\otimes r^\vee} "B";"B'"};
{\ar@{->}^{r\otimes r^\vee}  "C";"C'"};
{\ar@{=>} "B";"A'"};
{\ar@{=>}_{\on{taut}} "C";"B'"};
{\ar@{=>} "D";"C'"};
\endxy
$$

In this diagram, the 2-morphism in the left square is the map
$$r\boxtimes r^\vee \to \on{unit}(\sfe)=\on{mult}^R(\one_\CR)$$
equal to
$$r\boxtimes r^\vee \to \on{mult}^R\circ \on{mult}(r\boxtimes r^\vee)=
\on{mult}^R(r\otimes r^\vee) \to \on{mult}^R(\one_\CR).$$

\medskip

The 2-morphism in the middle square is obtained from the map
$$\alpha_r:r\to F_\CR(r).$$

\medskip

The 2-morphism in the right square is map
$$\sfe\to \on{counit}(r\boxtimes r^\vee)=\CHom_{\CR}(\one_\CR,\on{mult}(r\boxtimes r^\vee))=
\CHom_{\CR}(\one_\CR,r\otimes r^\vee),$$
corresponding to the canonical map
$$\one_\CR\to r\otimes r^\vee.$$

\sssec{}

Hence, $\on{cl}(r,\alpha_r)\in \Tr(F_\CR,\CR)$ is the composite
\begin{multline} \label{e:cl as composite}
\sfe \overset{\on{unit}}\to \CHom_{\CR}(\one_\CR,r\otimes r^\vee) \overset{\alpha_r\otimes \on{id}}\longrightarrow
\CHom_{\CR}(\one_\CR,F_\CR(r)\otimes r^\vee) \simeq \\
\simeq \CHom_{\CR}(\one_\CR, \on{mult} \circ (F_\CR\otimes \on{Id})(r\boxtimes r^\vee)) \to
\CHom_{\CR}(\one_\CR, \on{mult} \circ (F_\CR\otimes \on{Id})\circ \on{mult}^R(\one_\CR)).
\end{multline}

\medskip

Using the commutative diagram \eqref{e:calc HH2}, we identify
$$\on{mult} \circ (F_\CR\otimes \on{Id}) \circ \on{mult}^R \simeq \iota^R\circ \iota.$$

Hence, we can rewrite the composition in \eqref{e:cl as composite} as
\begin{multline} \label{e:cl as composite1}
\sfe \overset{\on{unit}}\to \CHom_{\CR}(\one_\CR,r\otimes r^\vee) \overset{\alpha_r\otimes \on{id}}\longrightarrow
\CHom_{\CR}(\one_\CR,F_\CR(r)\otimes r^\vee)  \simeq \\
\simeq \CHom_{\CR}(\one_\CR, \on{mult} \circ (F_\CR\otimes \on{Id})(r\boxtimes r^\vee))
\to \CHom_{\CR}(\one_\CR, \on{mult} \circ (F_\CR\otimes \on{Id})\circ \on{mult}^R(r\otimes r^\vee)) \simeq \\
\simeq \CHom_{\CR}(\one_\CR, \iota^R\circ \iota(r\otimes r^\vee)) \to
\CHom_{\CR}(\one_\CR, \iota^R\circ \iota(\one_\CR)).
\end{multline}

\sssec{}

We have a commutative diagram
$$
\CD
\CHom_{\CR}(\one_\CR,F_\CR(r)\otimes r^\vee)   @>{\sim}>> \CHom_{\CR}(\one_\CR, \on{mult} \circ (F_\CR\otimes \on{Id})(r\boxtimes r^\vee))   \\
@VVV   @VVV  \\
\CHom_{\on{HH}_\bullet(\CR,F_\CR)}(\one_{\on{HH}_\bullet(\CR,F_\CR)},\iota(F_\CR(r)\otimes r^\vee)) & &
\CHom_{\CR}(\one_\CR, \iota^R\circ \iota(r\otimes r^\vee))   \\
@V{\text{\eqref{e:iota inv}}}V{\sim}V  @VV{\sim}V    \\
\CHom_{\on{HH}_\bullet(\CR,F_\CR)}(\one_{\on{HH}_\bullet(\CR,F_\CR)},\iota(r\otimes r^\vee)) @>{\sim}>>
\CHom_{\on{HH}_\bullet(\CR,F_\CR)}(\one_{\on{HH}_\bullet(\CR,F_\CR)},\iota(r\otimes r^\vee)).
\endCD
$$

Hence, the composition in \eqref{e:cl as composite1} can be rewritten as
\begin{multline*}
\sfe \overset{\on{unit}}\to \CHom_{\CR}(\one_\CR,r\otimes r^\vee) \overset{\alpha_r\otimes \on{id}}\longrightarrow
\CHom_{\CR}(\one_\CR,F_\CR(r)\otimes r^\vee) \to  \\
\to \CHom_{\on{HH}_\bullet(\CR,F_\CR)}(\one_{\on{HH}_\bullet(\CR,F_\CR)},\iota(F_\CR(r)\otimes r^\vee)) \overset{\text{\eqref{e:iota inv}}}\simeq
\CHom_{\on{HH}_\bullet(\CR,F_\CR)}(\one_{\on{HH}_\bullet(\CR,F_\CR)},\iota(r\otimes r^\vee))  \overset{\on{counit}}\to \\
\to \CHom_{\on{HH}_\bullet(\CR,F_\CR)}(\one_{\on{HH}_\bullet(\CR,F_\CR)},\iota(\one_\CR))
\simeq \CHom_{\on{HH}_\bullet(\CR,F_\CR)}(\one_{\on{HH}_\bullet(\CR,F_\CR)},\one_{\on{HH}_\bullet(\CR,F_\CR)}),
\end{multline*}
and further as
\begin{multline*}
\sfe \overset{\on{unit}}\to
\CHom_{\on{HH}_\bullet(\CR,F_\CR)}(\one_{\on{HH}_\bullet(\CR,F_\CR)},\iota(r)\otimes \iota(r^\vee)) \overset{\alpha_r\otimes \on{id}}\longrightarrow
\CHom_{\on{HH}_\bullet(\CR,F_\CR)}(\one_{\on{HH}_\bullet(\CR,F_\CR)},\iota(F_\CR(r))\otimes \iota(r^\vee))
\overset{\text{\eqref{e:iota inv}}}\simeq  \\
\simeq \CHom_{\on{HH}_\bullet(\CR,F_\CR)}(\one_{\on{HH}_\bullet(\CR,F_\CR)},\iota(r)\otimes \iota(r^\vee))
\overset{\on{counit}}\to \CHom_{\on{HH}_\bullet(\CR,F_\CR)}(\one_{\on{HH}_\bullet(\CR,F_\CR)},\one_{\on{HH}_\bullet(\CR,F_\CR)}),
\end{multline*}
while the latter is the right-hand side is by definition $\Tr(a_r^{F_\CR},\iota(r))$.

\qed[\propref{p:class R}]

\ssec{The ``trivial" case and excursions}

We will now specialize further to the case when $F_\CR$ is the identity map.

\sssec{}

Let us take $\alpha$ to be the identity endomorphism of $r$. Consider the corresponding endomorphism
\begin{equation} \label{e:pre-chern}
\on{id}_r^{\on{Id}_\CR}\in \CEnd_{\on{HH}_\bullet(\CR)}(\iota(r)),
\end{equation}
see \secref{sss:iota F}.

\medskip

Denote
$$\on{ch}(r):=\Tr(\on{id}_r^{\on{Id}_\CR},\iota(r))\in \CEnd_{\on{HH}_\bullet(\CR)}(\one_{\on{HH}_\bullet(\CR)}).$$

Note that according to \propref{p:class R}, we have
\begin{equation} \label{e:chern}
\on{cl}(r,\on{id}_r)=\on{ch}(r)
\end{equation}
under the identification
$$\Tr(\on{Id}_\CR,\CR)\simeq \CEnd_{\on{HH}_\bullet(\CR)}(\one_{\on{HH}_\bullet(\CR)})$$
of \thmref{t:two notions of trace}.

\begin{rem}

We emphasize that despite the fact that we plugged in the identity map everywhere, the endomorphism
$\on{id}_r^{\on{Id}_\CR}$ of $\iota(r)$ is \emph{not} the identity map (for one thing, if it were the identity,
formula \eqref{e:chern} would fail).

\medskip

See also the last line in the proof of \propref{p:chern as excurs} for another interpretation of the element
$\on{id}_r^{\on{Id}_\CR}$.

\end{rem}

\sssec{An example} \label{sss:Chern}

Let $\CR=\QCoh(\CY)$, where $\CY$ is as in \secref{sss:prestack}. In \cite[Sect. 1.2]{KP2}
it is explained that for $r=\CF\in \QCoh(\CY)$, the map $\on{id}_r^{\on{Id}_\CR}$ can be interpreted as the action of
$$\CL(\CY):=\CY\underset{\CY\times \CY}\times \CY,$$
thought of as the group object over $\CY$, on $\CF$. Furthermore,
$$\on{ch}(r)\in \Gamma(\CL(\CY),\CO_{\CL(\CY)})$$
can be identified with the \emph{character} of this action.

\medskip

Let now $\CY$ be a smooth scheme over a field of characteristic zero. Consider the derivative of the above action of
$\CL(\CY)$ on $\CF$, which is an action of the Lie algebra
$$\on{Lie}(\CL(Y))\simeq T_\CY[-1]$$ on $\CF$. It is shown in \cite[Sect. 1.3]{KP2} that the resulting map
$$T_\CY[-1]\otimes \CF\to \CF$$
is the Atiyah class of $\CF$.

\medskip

Further, the Hochschild-Kostant-Rosenberg theorem identifies
$$\Gamma(\CL(\CY),\CO_{\CL(\CY)})\simeq \underset{i}\oplus \,\Gamma(\CY,\Omega^i(\CY))[i].$$

Under this identification, the element $\on{ch}(r)$ corresponds to the classical Chern character
$$\on{ch}(\CF)\in \underset{i}\oplus\, \Gamma(\CY,\Omega^i(\CY))[i].$$

\sssec{} \label{sss:taut loop}

We will now give one more interpretation of the above element
$$\on{ch}(r)\in \CEnd_{\on{HH}_\bullet(\CR)}(\one_{\on{HH}_\bullet(\CR)}),$$
in terms of excursion operators.

\medskip

Take
$$Y=S^1=\{*\}\underset{\{*\}\sqcup \{*\}}\sqcup \{*\}$$ so that
\begin{equation} \label{e:circle}
\CR^{\otimes S^1}\simeq \on{HH}_\bullet(\CR).
\end{equation}

Consider the tautological point $\gamma_{\on{taut}}\in \Omega(S^1,*)$. Let $\gamma^+_{\on{taut}}$
denote the pair $(\gamma_{\on{taut}},\gamma_{\on{triv}})$.

\medskip

Let $\xi_r$ be the map
$$\one_\CR \to \on{mult}\circ \on{mult}^R(\one_\CR)$$
equal to the composition
\begin{equation} \label{e:xi r}
\one_\CR\to r\otimes r^\vee\simeq \on{mult}(r\boxtimes r^\vee)\to \on{mult}\circ \on{mult}^R(\one_\CR),
\end{equation}
where:

\medskip

\noindent--the map $\one_\CR\to r\otimes r^\vee$ is the unit of the duality;

\medskip

\noindent--$r\boxtimes r^\vee$ denotes the corresponding object of $\CR\otimes \CR$;

\medskip

\noindent--the map $r\boxtimes r^\vee\to \on{mult}^R(\one_\CR)$ is obtained by adjunction from the map
$$\on{mult}(r\boxtimes r^\vee)=:r\otimes r^\vee\overset{\on{counit}}\longrightarrow \one_\CR.$$

\medskip

We claim:

\begin{prop} \label{p:chern as excurs}
The element
$$\on{ch}(r)\in  \CEnd_{\on{HH}_\bullet(\CR)}(\one_{\on{HH}_\bullet(\CR)})$$
corresponds to the universal excursion element (see \secref{sss:univ exc})
$$\on{Exc}_{\on{univ}}(\gamma^+_{\on{taut}},\xi_r)\in \CEnd_{\CR^{\otimes S^1}}(\one_{\CR^{\otimes S^1}})$$
under the identification \eqref{e:circle}.
\end{prop}

\begin{proof}

The proof is essentially an application of the definitions:

\medskip

By formula \eqref{e:exc univ}, the element $\on{Exc}_{\on{univ}}(\gamma^+_{\on{taut}},\xi_r)$
is the composition
$$\one_{\on{HH}_\bullet(\CR)}\overset{\on{unit}}\longrightarrow \iota(r)\otimes \iota(r^\vee)
\overset{\on{mon}_{\gamma_{\on{taut}}}\otimes \on{id}}\longrightarrow \iota(r)\otimes \iota(r^\vee)
\overset{\on{counit}}\longrightarrow \one_{\on{HH}_\bullet(\CR)},$$
where $\on{mon}_{\gamma_{\on{taut}}}$ denotes the automorphism of the functor $\iota$, i.e.,
$$\CR\simeq \CR^{\otimes \{*\}}\to \CR^{\otimes S^1},$$
corresponding to the loop $\gamma_{\on{taut}}$.

\medskip

However, by definition, the above automorphism $\on{mon}_{\gamma_{\on{taut}}}$ of $\iota$, when evaluated on $r\in \CR$,
identifies with $\on{id}^{\on{Id}_\CR}_r$.

\end{proof}

Combining Propositions \ref{p:class R} and \ref{p:chern as excurs}, we obtain:

\begin{cor} \label{c:chern as excurs}
Under the identification
$$\Tr(\on{Id},\CR)\simeq \CEnd_{\on{HH}_\bullet(\CR)}(\one_{\on{HH}_\bullet(\CR)})\simeq \CEnd_{\CR^{\otimes S^1}}(\one_{\CR^{\otimes S^1}})$$
of \thmref{t:two notions of trace}, the element
$$\on{cl}(r,\on{id})\in \Tr(\on{Id},\CR)$$
goes over to
$$\on{Exc}_{\on{univ}}(\gamma^+_{\on{taut}},\xi_r)\in \CEnd_{\CR^{\otimes S^1}}(\one_{\CR^{\otimes S^1}}).$$
\end{cor}

\begin{rem}
Note that on the one hand, $\Tr(\on{Id}_\CR,\CR)$ is explicitly given by the composition
$$\Vect \overset{\one_\CR}\to \CR \overset{\on{mult}^R}\longrightarrow \CR\otimes \CR \overset{\on{mult}}\longrightarrow \CR
\overset{\CHom_\CR(\one_\CR,-)}\to \Vect,$$
i.e., the resulting vector space is
$$\CHom_\CR(\one_\CR,\on{mult}\circ \on{mult}^R(\one_\CR)).$$

On the other hand, by \lemref{l:base change}, we have
\begin{multline*}
\CEnd_{\on{HH}_\bullet(\CR)}(\one_{\on{HH}_\bullet(\CR)}) \simeq
\CEnd_{\CR^{\otimes S^1}}(\one_{\CR^{\otimes S^1}})=\CHom_{\CR^{\otimes S^1}}(\iota(\one_\CR),\iota(\one_\CR))\simeq  \\
\simeq \CHom_{\CR}(\one_\CR,\iota^R\circ \iota(\one_\CR))\simeq
\CHom_{\CR}(\one_\CR,\on{mult}\circ \on{mult}^R(\one_\CR)).
\end{multline*}

Thus, we obtain an identification
$$\Tr(\on{Id}_\CR,\CR) \simeq \CHom_\CR(\one_\CR,\on{mult}\circ \on{mult}^R(\one_\CR)) \simeq
\CEnd_{\on{HH}_\bullet(\CR)}(\one_{\on{HH}_\bullet(\CR)}).$$

By unwinding the definitions, one can show that this is the same identification as one given by
\thmref{t:two notions of trace}.

\medskip

Assuming this, one can obtain the assertion of \corref{c:chern as excurs} by combining
\eqref{e:chern} and \thmref{t:excurs}. Indeed, this follows from the fact that the element $\on{cl}(r,\on{id}_r)$, thought of as an element in
$$\Tr(\on{Id}_\CR,\CR)\simeq \CHom_\CR(\one_\CR,\on{mult}\circ \on{mult}^R(\one_\CR)),$$
equals $\xi_r$.

\end{rem}

\ssec{Introducing observables}

In this subsection we will study a certain generalization of \thmref{t:two notions of trace}, where we modify both sides by
inserting an object $r\in \CR$.

\sssec{}

Let $\CR$ be a rigid symmetric monoidal category. Let $\CM$ be a dualizable $\CR$-module category.

\medskip

On the one hand, we consider the endofunctor of $\CM$, given by $H_r\circ F_\CM$,
where
$$H_r(m):=r\otimes m$$
denotes the action of the object $r$ on $\CM$ as an $\CR$-module category.

\sssec{}

On the other hand, consider
$$\iota(r)\in \on{HH}_\bullet(\CR,F_\CR),$$
where we recall that $\iota$ denotes the functor
$$\CR\to \CR\underset{\on{mult},\CR\otimes \CR,\on{mult}\circ (F_\CR\otimes \on{Id})}\otimes \CR=:\on{HH}_\bullet(\CR,F_\CR),$$
corresponding to the left copy of $\CR$ in the tensor product.

\medskip

Recall also that $\on{HH}_\bullet(\CR,F_\CR)$ is itself a symmetric monoidal category, so
for any $\CF\in \on{HH}_\bullet(\CR,F_\CR)$ it makes sense to consider
$$\iota(r)\otimes \CF \in  \on{HH}_\bullet(\CR,F_\CR).$$

\sssec{}

We claim:

\begin{thm} \label{t:two notions of trace obs}
There exists a canonical isomorphism
\begin{equation} \label{e:trace obs}
\Tr(H_r\circ F_\CM,\CM) \simeq
\CHom_{\on{HH}_\bullet(\CR,F_\CR)}(\one_{\on{HH}_\bullet(\CR,F_\CR)},\iota(r)\otimes \Tr^{\on{enh}}_\CR(F_\CM,\CM)),
\end{equation}
functorial in $r\in \CR$. This isomorphism is compatible with the actions of the two sides of \eqref{e:trace endo}.
\end{thm}

The rest of this subsection is devoted to the proof of this theorem.

\sssec{}

For $\CM$ as in the theorem, denote by $F_{\CM,r}$ the composite $H_r\circ F_\CM$. Note that since $\CR$
is \emph{symmetric} monoidal, $F_{\CM,r}$
is also compatible with the action of $F_\CR$ on $\CR$.  By \thmref{t:two notions of trace}, we have
$$\Tr(H_r\circ F_\CM,\CM)\simeq
\CHom_{\on{HH}_\bullet(\CR,F_\CR)}(\one_{\on{HH}_\bullet(\CR,F_\CR)},\Tr^{\on{enh}}_\CR(F_{\CM,r},\CM)).$$

Hence, in order to prove \thmref{t:two notions of trace obs}, it suffices to establish the following:

\begin{thm} \label{t:obs internal}
There exists a canonical isomorphism
\begin{equation} \label{e:obs internal}
\iota(r)\otimes \Tr^{\on{enh}}_\CR(F_\CM,\CM)\simeq \Tr^{\on{enh}}_\CR(F_{\CM,r},\CM)
\end{equation}
as objects of $\on{HH}_\bullet(\CR,F_\CR)$.
\end{thm}

\qed[\thmref{t:two notions of trace obs}]

%

\sssec{Proof of \thmref{t:obs internal}}

Consider the diagram
\begin{equation} \label{e:mult by r}
\xy
(0,0)*+{\ul{\CR\mmod}}="A";
(30,0)*+{\ul{\CR\mmod}}="B";
(0,-20)*+{\ul{\CR\mmod}}="C";
(30,-20)*+{\ul{\CR\mmod},}="D";
{\ar@{->}^{F_\CR} "A";"B"};
{\ar@{->}_{\on{Id}} "A";"C"};
{\ar@{->}^{\on{Id}} "B";"D"};
{\ar@{->}_{F_\CR} "C";"D"};
{\ar@{=>}^{\alpha_r} "B";"C"};
\endxy
\end{equation}
where $\alpha_r$ is the 2-morphism, which, when evaluated on $\CM'\in \CR\mmod$ (see \secref{sss:evaluate} for what evaluation
means), acts as
$$H_r:\CM'\to \CM',$$
viewed as a map of $\CR$-module categories.

\medskip

Concatenating with the diagram
$$
\xy
(0,0)*+{\ul{\DGCat}}="A";
(30,0)*+{\ul{\DGCat}}="B";
(0,-20)*+{\ul{\CR\mmod}}="C";
(30,-20)*+{\ul{\CR\mmod},}="D";
{\ar@{->}^{\on{Id}} "A";"B"};
{\ar@{->}_{\CM} "A";"C"};
{\ar@{->}^{\CM} "B";"D"};
{\ar@{->}_{F_\CR} "C";"D"};
{\ar@{=>}^{\alpha_{F_\CM}} "B";"C"};
\endxy
$$
which produces $\Tr^{\on{enh}}_\CR(F_\CM,\CM)$, we obtain the diagram
$$
\xy
(0,0)*+{\ul{\DGCat}}="A";
(30,0)*+{\ul{\DGCat}}="B";
(0,-20)*+{\ul{\CR\mmod}}="C";
(30,-20)*+{\ul{\CR\mmod},}="D";
{\ar@{->}^{\on{Id}} "A";"B"};
{\ar@{->}_{\CM} "A";"C"};
{\ar@{->}^{\CM} "B";"D"};
{\ar@{->}_{F_\CR} "C";"D"};
{\ar@{=>}^{\alpha_{F_{\CM,r}}} "B";"C"};
\endxy
$$
which produces $\Tr^{\on{enh}}_\CR(F_{\CM,r},\CM)$.

\medskip

Since the formation of trace is compatible with compositions, it suffices to prove the following:

\begin{prop} \label{p:Tr r calc}
The map
\begin{equation} \label{e:Tr r calc}
\on{HH}_\bullet(\CR,F_\CR)\to \on{HH}_\bullet(\CR,F_\CR),
\end{equation}
induced by \eqref{e:mult by r}, is given by $\iota(r)\otimes -$.
\end{prop}

\qed[\thmref{t:obs internal}]

\begin{proof}[Proof of \propref{p:Tr r calc}]

Consider the corresponding diagram
$$
\xy
(0,0)*+{\ul{\DGCat}}="A";
(30,0)*+{\ul{\CR^{\otimes 2}\mmod}}="B";
(60,0)*+{\ul{\CR^{\otimes 2}\mmod}}="C";
(90,0)*+{\ul{\DGCat}}="D";
(0,-20)*+{\ul{\DGCat}}="A'";
(30,-20)*+{\ul{\CR^{\otimes 2}\mmod}}="B'";
(60,-20)*+{\ul{\CR^{\otimes 2}\mmod}}="C'";
(90,-20)*+{\ul{\DGCat},}="D'";
{\ar@{->}^{\on{unit}} "A";"B"};
{\ar@{->}^{\on{unit}} "A'";"B'"};
{\ar@{->}^{\on{counit}} "C";"D"};
{\ar@{->}^{\on{counit}} "C'";"D'"};
{\ar@{->}^{F_\CR\otimes \on{Id}} "B";"C"};
{\ar@{->}^{F_\CR\otimes \on{Id}} "B'";"C'"};
{\ar@{->}_{\on{Id}} "A";"A'"};
{\ar@{->}^{\on{Id}} "D";"D'"};
{\ar@{->}_{\on{Id}} "B";"B'"};
{\ar@{->}_{\on{Id}} "C";"C'"};
{\ar@{=>}_{\on{id}} "B";"A'"};
{\ar@{=>}_{\alpha_r} "C";"B'"};
{\ar@{=>}_{\on{id}} "D";"C'"};
\endxy
$$
which gives rise to \eqref{e:Tr r calc}.

\medskip

The 2-morphism in the inner square, when evaluated on $\CQ\in (\CR\otimes \CR)\mmod$, acts as
$$H_r\otimes \on{Id}.$$

\medskip

This makes the assertion concerning \eqref{e:Tr r calc} manifest.

\end{proof}

\begin{rem}
The proof of \thmref{t:obs internal} can be reformulated as the combination of the following two assertions.
Consider the $\CR$-module which is $\CR$ itself, equipped with the endofunctor $F_{\CR,r}:=H_r\circ F_\CR$.

\medskip

The first assertion is that there is a canonical isomorphism in $\on{HH}_\bullet(\CR,F_\CR)$
$$\Tr^{\on{enh}}_\CR(F_{\CR,r},\CR)\simeq \iota(r).$$
This is a particular case of \propref{p:Tr r calc}.

\medskip

For the second assertion, consider $(\ul{\CR\mmod},\Res_{F_\CR})$ as a commutative algebra object in the symmetric monoidal
category $L(\tDGCat)_{\on{rgd}}$, and the monoid
\begin{equation} \label{e:enh modules}
\Maps_{L(\tDGCat)_{\on{rgd}}}((\ul{\DGCat},\on{Id}),(\ul{\CR\mmod},\Res_{F_\CR})),
\end{equation}
where we recall that $(\ul{\DGCat},\on{Id})$ is the unit in $L(\tDGCat)_{\on{rgd}}$.

\medskip

The second assertion is that $\Tr^{\on{enh}}_\CR$ is a symmetric monoidal functor
\begin{equation} \label{e:Tr enh as a functor}
\Maps_{L(\tDGCat)_{\on{rgd}}}((\ul{\DGCat},\on{Id}),(\ul{\CR\mmod},\Res_{F_\CR}))\to \on{HH}_\bullet(\CR,F_\CR).
\end{equation}
This follows from the fact that the functor $\Tr$ of \eqref{e:2-tr as functor} is symmetric monoidal.

\medskip

To deduce \thmref{t:obs internal}, we note that for any $(\CM,F_\CM)$ we have
$$(\CM,F_{\CM,r})\simeq (\CR,F_{\CR,r})\otimes (\CM,F_\CM)$$
as objects in \eqref{e:enh modules}.

\medskip

Note also that for this manipulation, we regarded $L(\tDGCat)_{\on{rgd}}$ as an $(\infty,1)$-category (i.e.,
we did not need to consider non-invertible 3-morphisms, as in \secref{ss:L tDGCat}). More generally, $\Tr^{\on{enh}}_\CR$
can be considered as a functor of  \emph{symmetric monoidal categories}
\begin{equation} \label{e:Tr enh as a functor 2}
\bMaps_{L(\tDGCat)_{\on{rgd}}}((\ul{\DGCat},\on{Id}),(\ul{\CR\mmod},\Res_{F_\CR}))\to \on{HH}_\bullet(\CR,F_\CR),
\end{equation}
whose source is the \emph{symmetric monoidal category}
\begin{equation} \label{e:Tr enh source}
\bMaps_{L(\tDGCat)_{\on{rgd}}}((\ul{\DGCat},\on{Id}),(\ul{\CR\mmod},\Res_{F_\CR})),
\end{equation}
where $\bMaps(-,-)$ denotes the $(\infty,1)$-category of maps between objects in a given $(\infty,2)$-category.

\end{rem}

\ssec{Cyclicity and observables}

We return to the setting of \thmref{t:two notions of trace obs}. Let us now fix two objects $r_1,r_2\in \CR$.

\sssec{}

On the one hand, consider the following two modules over $\Tr(F_\CR,\CR)$:
$$\Tr(H_{r_1}\circ H_{r_2}\circ F_\CM,\CM) \text{ and }
\Tr(H_{F_\CR(r_1)}\circ H_{r_2}\circ F_\CM,\CM).$$

We claim that there exists a canonical isomorphism
\begin{equation} \label{e:cyclicity one}
\Tr(H_{r_1}\circ H_{r_2}\circ F_\CM,\CM)  \simeq \Tr(H_{F_\CR(r_1)}\circ H_{r_2}\circ F_\CM,\CM)
\end{equation}

Indeed, it is obtained as the composition
\begin{equation} \label{e:cycle snake}
\CD
\Tr(H_{r_1}\circ H_{r_2}\circ F_\CM,\CM)   @>{\text{cyclicity of trace}}>{\sim}>   \Tr(H_{r_2}\circ F_\CM \circ H_{r_1},\CM) \\
& & @VV{\sim}V  \\
\Tr(H_{F_\CR(r_1)}\circ H_{r_2}\circ F_\CM,\CM) @<{\sim}<< \Tr(H_{r_2}\circ H_{F_\CR(r_1)}\circ F_\CM,\CM),\\
\endCD
\end{equation}
where the second arrow uses the isomorphism
\begin{equation} \label{e:twisting}
F_\CM \circ H_{r_1} \simeq H_{F_\CR(r_1)}\circ F_\CM,
\end{equation}
and the last arrow uses the fact that the product on $\CR$ is commutative.

\sssec{}

On the other hand, consider the following two modules over $\CEnd_{\on{HH}_\bullet(\CR,F_\CR)}(\one_{\on{HH}_\bullet(\CR,F_\CR)})$:
$$\CHom_{\on{HH}_\bullet(\CR,F_\CR)}(\one_{\on{HH}_\bullet(\CR,F_\CR)},\iota(r_1\otimes r_2)\otimes \Tr^{\on{enh}}_\CR(F_\CM,\CM))$$
and
$$\CHom_{\on{HH}_\bullet(\CR,F_\CR)}(\one_{\on{HH}_\bullet(\CR,F_\CR)},\iota(F_\CR(r_1)\otimes r_2)\otimes \Tr^{\on{enh}}_\CR(F_\CM,\CM)).$$

We claim that there is a canonical isomorphism
\begin{equation} \label{e:cyclicity two}
\CD
\CHom_{\on{HH}_\bullet(\CR,F_\CR)}(\one_{\on{HH}_\bullet(\CR,F_\CR)},\iota(r_1\otimes r_2)\otimes \Tr^{\on{enh}}_\CR(F_\CM,\CM)) \\
@V{\sim}VV \\
\CHom_{\on{HH}_\bullet(\CR,F_\CR)}(\one_{\on{HH}_\bullet(\CR,F_\CR)},\iota(F_\CR(r_1)\otimes r_2)\otimes \Tr^{\on{enh}}_\CR(F_\CM,\CM)).
\endCD
\end{equation}

Indeed, this follows from the fact that
\begin{equation} \label{e:iota twisting}
\iota(r)\simeq \iota(F_\CR(r)).
\end{equation}

\sssec{}

We now claim:

\begin{lem} \label{l:cyclicity and observables}
The isomorphisms \eqref{e:cyclicity one} and \eqref{e:cyclicity two} match up under the isomorphisms
$$\Tr(H_{r_1}\circ H_{r_2}\circ F_\CM,\CM) \simeq
\CHom_{\on{HH}_\bullet(\CR,F_\CR)}(\one_{\on{HH}_\bullet(\CR,F_\CR)},\iota(r_1\otimes r_2)\otimes \Tr^{\on{enh}}_\CR(F_\CM,\CM))$$
and
$$\Tr(H_{F_\CR(r_1)}\circ H_{r_2}\circ F_\CM,\CM) \simeq
\CHom_{\on{HH}_\bullet(\CR,F_\CR)}(\one_{\on{HH}_\bullet(\CR,F_\CR)},\iota(F_\CR(r_1)\otimes r_2)\otimes \Tr^{\on{enh}}_\CR(F_\CM,\CM))$$
of \thmref{t:two notions of trace obs}.
\end{lem}

\begin{proof}
Note first that it suffices to show the assertion in the particular case $r_2=\one_{\CR}$ and $r_1=r$. (Namely, the general
case would follow from this particular case by replacing $F_\CM$ by $F_{\CM,r_2}$).

\medskip

Thus, we have to show that
under the isomorphisms
$$\Tr(H_{r}\circ F_\CM,\CM) \simeq
\CHom_{\on{HH}_\bullet(\CR,F_\CR)}(\one_{\on{HH}_\bullet(\CR,F_\CR)},\iota(r)\otimes \Tr^{\on{enh}}_\CR(F_\CM,\CM))$$
and
$$\Tr(H_{F_\CR(r)}\circ F_\CM,\CM) \simeq
\CHom_{\on{HH}_\bullet(\CR,F_\CR)}(\one_{\on{HH}_\bullet(\CR,F_\CR)},\iota(F_\CR(r))\otimes \Tr^{\on{enh}}_\CR(F_\CM,\CM))$$
of \thmref{t:two notions of trace obs}, the composition
\begin{equation}
\begin{CD}
\Tr(H_{r}\circ F_\CM,\CM)@>\text{cyclicity of trace}>\sim>\Tr(F_\CM\circ H_r,\CM) @>\eqref{e:twisting}>{\sim}> \Tr(H_{F_{\CR}(r)}\circ F_\CM,\CM),
\end{CD}
\end{equation}
of the left hand sides corresponds to the isomorphism induced by \eqref{e:iota twisting} of the right hand side.

\medskip

Note that isomorphism \eqref{e:cycl isomorp}, corresponding to the pair of diagrams
\begin{equation} \label{e:fm and hr0}
\xy
(0,0)*+{\ul{\DGCat}}="A";
(30,0)*+{\ul{\DGCat}}="B";
(0,-20)*+{\ul{\CR\mmod}}="C";
(30,-20)*+{\ul{\CR\mmod}}="D";
{\ar@{->}^{\on{Id}} "A";"B"};
{\ar@{->}_{\CM} "A";"C"};
{\ar@{->}^{\CM} "B";"D"};
{\ar@{->}_{F_\CR} "C";"D"};
{\ar@{=>}^{\alpha_{F_\CM}} "B";"C"};
(60,0)*+{\ul{\DGCat}}="A'";
(90,0)*+{\ul{\DGCat}}="B'";
(60,-20)*+{\ul{\CR\mmod}}="C'";
(90,-20)*+{\ul{\CR\mmod}}="D'",;
{\ar@{->}^{\on{Id}} "A'";"B'"};
{\ar@{->}_{\CM} "A'";"C'"};
{\ar@{->}^{\CM} "B'";"D'"};
{\ar@{->}_{\Id} "C'";"D'"};
{\ar@{=>}^{\alpha_{\on{id}_\CM,r}} "B'";"C'"};
\endxy
\end{equation}
induces an isomorphism
\begin{equation} \label{e:cyclic enhanced}
\Tr_{\CR}^{\on{enh}}(H_{r}\circ F_\CM,\CM) \simeq\Tr_{\CR}^{\on{enh}}(F_\CM\circ H_{r} ,\CM)
\end{equation}
of objects of $\on{HH}_\bullet(\CR,F_\CR)$.

\medskip

To finish the proof, it suffices to show that diagrams
\begin{equation} \label{e:cyclic enh1}
\begin{CD}
\Tr(H_{r}\circ F_\CM,\CM) @>\eqref{e:trace module}>\sim> \CHom_{\on{HH}_\bullet(\CR,F_\CR)}(\one_{\on{HH}_\bullet(\CR,F_\CR)},\Tr^{\on{enh}}_\CR(H_r\circ F_\CM,\CM)) \\
@V\sim V \text{cyclicity of trace}V    @V\sim V\eqref{e:cyclic enhanced}V\\
\Tr(F_\CM\circ H_r,\CM) @>\eqref{e:trace module}>\sim> \CHom_{\on{HH}_\bullet(\CR,F_\CR)}(\one_{\on{HH}_\bullet(\CR,F_\CR)},\Tr^{\on{enh}}_\CR(F_\CM\circ H_r,\CM))
\end{CD}
\end{equation}
and
\begin{equation} \label{e:cyclic enh2}
\begin{CD}
\Tr^{\on{enh}}_\CR(H_r\circ F_\CM,\CM) @>\eqref{e:obs internal}>\sim> \iota(r)\otimes \Tr^{\on{enh}}_\CR(F_\CM,\CM))\\
@V\eqref{e:twisting}\circ \eqref{e:cyclic enhanced}V\sim V @V\sim V\eqref{e:iota twisting} V\\
\Tr^{\on{enh}}_\CR(H_{F_{\CR}(r)}\circ F_\CM,\CM)  @>\sim>\eqref{e:obs internal}> \iota(F_{\CR}(r))\otimes \Tr^{\on{enh}}_\CR(F_\CM,\CM))
\end{CD}
\end{equation}
are commutative.

\medskip

We will deduce both assertions from the commutativity of the diagram \eqref{e:cycl isomorp3}.

\medskip

To show the commutativity of \eqref{e:cyclic enh1}, consider the following pair of diagrams:

\begin{equation} \label{e:fm and hr1}
\xy
(0,0)*+{\ul{\CR\mmod}}="A";
(30,0)*+{\ul{\CR\mmod}}="B";
(0,-20)*+{\ul{\DGCat}}="C";
(30,-20)*+{\ul{\DGCat}}="D";
{\ar@{->}^{F_\CR} "A";"B"};
{\ar@{->}_{\on{oblv}} "A";"C"};
{\ar@{->}^{\on{oblv}} "B";"D"};
{\ar@{->}_{\on{Id}} "C";"D"};
{\ar@{=>}^{\on{taut}} "B";"C"};
(60,0)*+{\ul{\CR\mmod}}="A'";
(90,0)*+{\ul{\CR\mmod}}="B'";
(60,-20)*+{\ul{\DGCat}}="C'";
(90,-20)*+{\ul{\DGCat}.}="D'";
{\ar@{->}^{\Id} "A'";"B'"};
{\ar@{->}_{\on{oblv}} "A'";"C'"};
{\ar@{->}^{\on{oblv}} "B'";"D'"};
{\ar@{->}_{\on{Id}} "C'";"D'"};
{\ar@{=>}^{\on{taut}} "B'";"C'"};
\endxy
\end{equation}

Arguing as in \secref{sss:trace module again} or \secref{sss:categ RR}, the isomorphism \eqref{e:cycl isomorp}
corresponding to the diagrams in \eqref{e:fm and hr1} can be identified with the identity endomorphism on
$\CHom_{\on{HH}_\bullet(\CR,F_\CR)}(\one_{\on{HH}_\bullet(\CR,F_\CR)},-)$.

\medskip

Since the vertical compositions of the diagrams appearing in \eqref{e:fm and hr0} and \eqref{e:fm and hr1} are isomorphic,
respectively, to
\begin{equation} \label{e:fm and hr2}
\xy
(0,0)*+{\ul{\DGCat}}="A";
(30,0)*+{\ul{\DGCat}}="B";
(0,-20)*+{\ul{\DGCat}}="C";
(30,-20)*+{\ul{\DGCat}}="D";
{\ar@{->}^{\on{Id}} "A";"B"};
{\ar@{->}_{\CM} "A";"C"};
{\ar@{->}^{\CM} "B";"D"};
{\ar@{->}_{\Id} "C";"D"};
{\ar@{=>}^{\alpha_{F_\CM}} "B";"C"};
(60,0)*+{\ul{\DGCat}}="A'";
(90,0)*+{\ul{\DGCat}}="B'";
(60,-20)*+{\ul{\DGCat}}="C'";
(90,-20)*+{\ul{\DGCat},}="D'",;
{\ar@{->}^{\on{Id}} "A'";"B'"};
{\ar@{->}_{\CM} "A'";"C'"};
{\ar@{->}^{\CM} "B'";"D'"};
{\ar@{->}_{\Id} "C'";"D'"};
{\ar@{=>}^{\alpha_{\on{id}_\CM,r}} "B'";"C'"};
\endxy
\end{equation}
the commutative diagram \eqref{e:cycl isomorp3}, corresponding to \eqref{e:fm and hr0} and \eqref{e:fm and hr1}, naturally identifies with  \eqref{e:cyclic enh1}.

\medskip

To show the commutativity of \eqref{e:cyclic enh2}, note that the diagrams in \eqref{e:fm and hr0} can be written as vertical compositions of
the diagrams

\begin{equation} \label{e:fm and hr3}
\xy
(0,0)*+{\ul{\DGCat}}="A";
(30,0)*+{\ul{\DGCat}}="B";
(0,-20)*+{\ul{\CR\mmod}}="C";
(30,-20)*+{\ul{\CR\mmod}}="D";
(0,-40)*+{\ul{\CR\mmod}}="E";
(30,-40)*+{\ul{\CR\mmod}}="F";
{\ar@{->}^{\on{Id}} "A";"B"};
{\ar@{->}_{\CM} "A";"C"};
{\ar@{->}^{\CM} "B";"D"};
{\ar@{->}_{F_\CR} "C";"D"};
{\ar@{=>}^{\alpha_{F_\CM}} "B";"C"};
{\ar@{->}_{\Id} "C";"E"};
{\ar@{->}^{\Id} "D";"F"};
{\ar@{->}_{F_\CR} "E";"F"};
{\ar@{=>}^{\Id} "D";"E"};
(60,0)*+{\ul{\DGCat}}="A'";
(90,0)*+{\ul{\DGCat}}="B'";
(60,-20)*+{\ul{\CR\mmod}}="C'";
(90,-20)*+{\ul{\CR\mmod}}="D'";
(60,-40)*+{\ul{\CR\mmod}}="E'";
(90,-40)*+{\ul{\CR\mmod}.}="F'";
{\ar@{->}^{\on{Id}} "A'";"B'"};
{\ar@{->}_{\CM} "A'";"C'"};
{\ar@{->}^{\CM} "B'";"D'"};
{\ar@{->}_{\Id} "C'";"D'"};
{\ar@{=>}^{\Id} "B'";"C'"};
{\ar@{->}_{\Id} "C'";"E'"};
{\ar@{->}^{\Id} "D'";"F'"};
{\ar@{->}_{\Id} "E'";"F'"};
{\ar@{=>}^{\alpha_{r}} "D'";"E'"};
\endxy
\end{equation}

Arguing as in the proof of \propref{p:Tr r calc}, the isomorphism \eqref{e:cycl isomorp} corresponding to the bottom part of
\eqref{e:fm and hr3} can be identified with the isomorphism of functors
$$\iota(r)\otimes -\simeq \iota(F_{\CR}(r))\otimes -,$$
induced by isomorphism \eqref{e:iota twisting}.

\medskip

Since the isomorphism \eqref{e:cycl isomorp} corresponding to the top part of \eqref{e:fm and hr3} can be identified with the identity endomorphism of
$\Tr^{\on{enh}}_\CR(F_\CM,\CM)$, the commutative diagram \eqref{e:cycl isomorp3}, corresponding to the diagrams in \eqref{e:fm and hr3},
naturally identifies with \eqref{e:cyclic enh2}.
\end{proof}

\section{The ``shtuka" construction}  \label{s:sht}

In this section we combine all the ingredients developed in the previous sections to obtain
our 
the toy model for the shtuka construction.

\ssec{The universal shtuka}

\sssec{}  \label{sss:introduce shtuka}

Let $\CA$ be a rigid symmetric monoidal category and $Y$ an object of $\Spc$. Consider the symmetric monoidal
category $\CA^{\otimes Y}$.

\medskip

Let $\phi$ be an endomorphism of $Y$, and let $\CA^{\otimes \phi}$
be the induced (symmetric monoidal) endofunctor of $\CA^{\otimes Y}$.

\medskip

Let $\CM$ be a dualizable DG category, equipped with an action of $\CA^{\otimes Y}$. Let $F_\CM$ be an endofunctor of $\CM$
compatible with the action of $\CA^{\otimes \phi}$ on $\CA^{\otimes Y}$.

\sssec{}

According to \secref{sss:enhanced trace}, to this data we can attach an object
$$\Tr^{\on{enh}}_{\CA^{\otimes Y}}(F_\CM,\CM)\in  \on{HH}_\bullet(\CA^{\otimes Y},\CA^{\otimes \phi}).$$

We identify
$$\on{HH}_\bullet(\CA^{\otimes Y},\CA^{\otimes \phi})\simeq \CA^{\otimes Y/\phi}$$
by \propref{p:HC twist}.

\medskip

We will use yet another notation for $\Tr^{\on{enh}}_{\CA^{\otimes Y}}(F_\CM,\CM)$, namely
$$\on{Sht}_{\CM,F_\CM,\on{univ}}\in \CA^{\otimes Y/\phi},$$
and call it ``the universal shtuka".

\sssec{}

According to \thmref{t:two notions of trace}, we have an identification
\begin{equation} \label{e:glob funct shtuka}
\CEnd_{\CA^{\otimes Y/\phi}}(\one_{\CA^{\otimes Y/\phi}})\simeq \Tr(\CA^{\otimes \phi},\CA^{\otimes Y})
\end{equation}
and
\begin{equation} \label{e:glob sect shtuka}
\CHom_{\CA^{\otimes Y/\phi}}(\one_{\CA^{\otimes Y/\phi}},\on{Sht}_{\CM,F_\CM,\on{univ}})\simeq
\Tr(F_\CM,\CM).
\end{equation}

\medskip

The isomorphisms \eqref{e:glob funct shtuka} and \eqref{e:glob sect shtuka} are compatible with the
action of
$\CEnd_{\CA^{\otimes Y/\phi}}(\one_{\CA^{\otimes Y/\phi}})$ on
$\CHom_{\CA^{\otimes Y/\phi}}(\one_{\CA^{\otimes Y/\phi}},\on{Sht}_{\CM,F_\CM,\on{univ}})$
and the action of $\Tr(\CA^{\otimes \phi},\CA^{\otimes Y})$ on $\Tr(F_\CM,\CM)$.

\medskip

In particular, $\Tr(F_\CM,\CM)$ carries an action of $\CEnd_{\CA^{\otimes Y/\phi}}(\one_{\CA^{\otimes Y/\phi}})$.

\sssec{Example}  \label{sss:Sht LocSys}

Let $\sfe$ be a field of characteristic zero. Let $\CA=\Rep(\sG)$ and assume that $Y$ has finitely many connected components.
Then according to \thmref{t:integral to LocSys},
$$\CA^{\otimes Y/\phi}\simeq \QCoh(\LocSys_\sG(Y/\phi)),$$
and hence
$$\CEnd_{\CA^{\otimes Y/\phi}}(\one_{\CA^{\otimes Y/\phi}})\simeq \Gamma(\LocSys_\sG(Y/\phi),\CO_{\LocSys_\sG(Y/\phi)}).$$

We obtain that in this case
$$\on{Sht}_{\CM,F_\CM,\on{univ}}\in \QCoh(\LocSys_\sG(Y/\phi)),$$
and the vector space
$$\Tr(F_\CM,\CM)\simeq \Gamma(\LocSys_\sG(Y/\phi),\on{Sht}_{\CM,F_\CM,\on{univ}})$$
carries an action of the algebra
$$\Gamma(\LocSys_\sG(Y/\phi),\CO_{\LocSys_\sG(Y/\phi)}).$$

\sssec{}

According to \secref{sss:objects in integral}, we can think of an object of $\CA^{\otimes Y/\phi}$ as a compatible family of functors
$$\CA^{\otimes I}\to \on{LS}((Y/\phi)^I), \quad I\in \on{fSet}.$$

Applying this to the object
$$\on{Sht}_{\CM,F_\CM,\on{univ}}\in \CA^{\otimes Y/\phi}$$
we obtain a family of functors denoted
$$\on{Sht}_{\CM,F_\CM,Y/\phi,I}:\CA^{\otimes I}\to \on{LS}((Y/\phi)^I).$$

We will denote by $\on{Sht}_{\CM,F_\CM,Y,I}$ the composite of $\on{Sht}_{\CM,F_\CM,Y/\phi,I}$ with the pullback functor
$$\on{LS}((Y/\phi)^I)\to \on{LS}(Y^I),$$
corresponding to the projection $Y\to Y/\phi$.

\medskip

By \secref{sss:objects in integral vac}, the object
$$\on{Sht}_{\CM,F_\CM,Y/\phi,\emptyset}\in \Vect$$
can be canonically identified with
$$\CHom_{\CA^{\otimes Y/\phi}}(\one_{\CA^{\otimes Y/\phi}},\on{Sht}_{\CM,F_\CM,\on{univ}})\simeq
\Tr(F_\CM,\CM).$$

\sssec{}

The goals of the present section are the following:

\medskip

\noindent--Describe the functors $\on{Sht}_{\CM,F_\CM,Y,I}$ explicitly via the usual (i.e., 1-categorical) trace construction;

\medskip

\noindent--Describe the descent of $\on{Sht}_{\CM,F_\CM,Y,I}$ to $\on{Sht}_{\CM,F_\CM,Y/\phi,I}$ via the action of ``partial Frobeniuses";

\medskip

\noindent--Describe the action of $\CEnd_{\CA^{\otimes Y/\phi}}(\one_{\CA^{\otimes Y/\phi}})$ on
$\Tr(F_\CM,\CM)$ in terms of the functors $\on{Sht}_{\CM,F_\CM,Y/\phi,I}$ via the ``excursion operators";

\medskip

\noindent--Prove the ``S=T" identity (see \secref{ss:S=T} for what this means).

\medskip

All of the above will amount to an application of the constructions of the previous subsections.

\sssec{}   \label{sss:Sht rel}

Before we proceed further, let us note the following functoriality property of the shtuka construction in $Y$:

\medskip

Let us be given another space $Y'$, equipped with an endomorphism $\phi'$, and a map of spaces
$\psi:Y'\to Y$ that intertwines $\phi'$ and $\phi$.  Let $\Res_\psi(\CM)$ denote the
$\CA^{\otimes Y'}$-module category, obtained from $\CM$ by restricting along $\psi$.

\medskip

Consider the resulting objects
$$\on{Sht}_{\CM,F_\CM,\on{univ}}\in \CA^{\otimes Y/\phi} \text{ and } \on{Sht}_{\Res_\psi(\CM),F_\CM,\on{univ}}\in \CA^{\otimes Y'/\phi'}.$$

From \eqref{e:trace module gen R} and we obtain:

\begin{cor} \label{c:rel shtuka}
The object $\on{Sht}_{\Res_\psi(\CM),F_\CM,\on{univ}}$ is obtained from $\on{Sht}_{\CM,F_\CM,\on{univ}}$ by applying
the right adjoint of the functor
$$\CA^{\otimes \psi}:\CA^{\otimes Y'/\phi'}\to \CA^{\otimes Y/\phi}.$$
%
%
%
%
\end{cor}

\begin{rem} \label{r:dual of cor}
Note that the equivalences
$$\CA^{\otimes Y/\phi}\simeq (\CA^{\otimes Y/\phi})^{\vee}\text{ and }\CA^{\otimes Y'/\phi'}\simeq (\CA^{\otimes Y'/\phi'})^{\vee}$$
from \secref{sss:objects in integral} identify the right adjoint
$(\CA^{\otimes \psi})^R$ of $\CA^{\otimes \psi}$ with the dual $(\CA^{\otimes \psi})^\vee$
(see \cite[Section 1, Lemma 9.2.6]{GR1}).

\medskip

Therefore it follows from Corollary \ref{c:rel shtuka} that the functor
$$\CS_{Y'/\phi'}:\CA^{\otimes Y'/\phi'}\to \Vect,$$
corresponding to $\on{Sht}_{\Res_\psi(\CM),F_\CM,\on{univ}}\in \CA^{\otimes Y'/\phi'}$, is obtained from the functor
$$\CS_{Y/\phi}:\CA^{\otimes Y/\phi}\to \Vect,$$
corresponding to $\on{Sht}_{\CM,F_\CM,\on{univ}}\in \CA^{\otimes Y/\phi}$, by precomposition with $\CA^{\otimes\psi}$.
\end{rem}

\begin{rem}
Note that in the example of \secref{sss:Sht LocSys}, the map $\psi$ induces a map
$$\LocSys_\sG(\psi):\LocSys_\sG(Y/\phi)\to \LocSys_\sG(Y'/\phi'),$$
and the functors $\CA^{\otimes \psi}$ and $(\CA^{\otimes \psi})^R$ identify with $\LocSys_\sG(\psi)^*$
and $\LocSys_\sG(\psi)_*$, respectively.
\end{rem}

\ssec{Explicit description of $I$-legged shtukas}   \label{ss:I-legged}

We retain the setting of \secref{sss:introduce shtuka}.

\sssec{}

Fix a finite set $I$. The evaluation map
$$Y^I\times I\to Y$$
defines a map from $Y^I$ to the space of symmetric monoidal functors $\CA^{\otimes I}\to \CA^{\otimes Y}$,
and hence to the space of actions of $\CA^{\otimes I}$ on $\CM$.

\medskip

Define the functor
$$\on{Sht}'_{\CM,F_\CM,Y,I}: Y^I\times \CA^{\otimes I}\to \Vect$$
as follows:

\medskip

It sends
$$\ul{y}\in Y^I, r\in  \CA^{\otimes I} \mapsto \Tr(H_{r_{\ul{y}}}\circ F_\CM,\CM),$$
where $r_{\ul{y}} $ denotes the image of $r$ along the functor
$$\CA^{\otimes I}\overset{\ul{y}}\longrightarrow \CA^{\otimes Y}.$$
See \secref{sss:geometric shtukas} for an explanation of how this relates to the usual notion of shtukas.

\begin{rem}
We can tautologically rewrite the functor
\begin{equation} \label{e:Hecke differently}
Y^I\times \CA^{\otimes I}\to \End(\CM), \quad (\ul{y},r)\mapsto H_{r_{\ul{y}}},
\end{equation}
which appears in the definition of $\on{Sht}'_{\CM,F_\CM,Y,I}$, as follows:

\medskip

Recall that according to \secref{sss:monoidal actions}, a datum of action of $\CA^{\otimes Y}$ on $\CM$
gives rise to a map
$$\CA^{\otimes I}\to \End(\CM)\otimes \LS(Y^I),$$
or, equivalently by \propref{p:LocSys as right adj}, to a map
\begin{equation} \label{e:Hecke differently bis}
Y^I \to \Maps(\CA^{\otimes I},\End(\CM)).
\end{equation}

The map \eqref{e:Hecke differently} is the one corresponding to \eqref{e:Hecke differently bis}.

\end{rem}

\sssec{}

We claim:

\medskip

\begin{prop}  \label{p:identify I-legged}
The functor
$$\CA^{\otimes I}\to \on{LS}(Y^I),$$ corresponding to $\on{Sht}'_{\CM,F_\CM,Y,I}$, identifies canonically with
$\on{Sht}_{\CM,F_\CM,Y,I}$.
\end{prop}

\begin{proof}

Taking into account \secref{sss:objects in integral}, we have to establish the isomorphism
$$\CHom_{\CA^{\otimes Y/\phi}}(\one_{\CA^{\otimes Y/\phi}}, r_{\ul{\bar{y}}}\otimes \on{Sht}_{\CM,F_\CM,\on{univ}})
\simeq \Tr(H_{r_{\ul{y}}}\circ F_\CM,\CM),$$
functorial in $\ul{y}\in Y^I$,
where $\ul{\bar{y}}$ is the image of $\ul{y}$ under the projection
$$Y\to Y/\phi.$$

Now the assertion follows from \thmref{t:two notions of trace obs} using the fact that
$$r_{\ul{\bar{y}}}\simeq \iota(r_{\ul{y}}).$$

\end{proof}

\ssec{Partial Frobeniuses}

\sssec{}

Since the object
$$\on{Sht}_{\CM,F_\CM,Y,I}\in \on{Funct}\left(\CA^{\otimes I}, \LS(Y^I)\right)$$ is the image of the object
$$\on{Sht}_{\CM,F_\CM,Y/\phi,I}\in  \on{Funct}\left(\CA^{\otimes I}, \LS((Y/\phi)^I)\right)$$
under the pullback functor
$$\LS((Y/\phi)^I)\to \LS(Y^I),$$
the former should carry the structure of equivariance with respect to the endomorphisms
that act as $\phi$ along \emph{each} of the $Y$ factors in $Y^I$.

\medskip

Fix an element $i^1\in I$ and write $I=\{i^1\}\sqcup I'$. We will now write down explicitly the structure
of equivariance corresponding to this endomorphism of $Y^I$, to be denoted $\phi_{i^1}$.

\medskip

Using the identification of \propref{p:identify I-legged}, this amounts to describing the isomorphism
\begin{equation} \label{e:partial Frob one}
\on{Sht}'_{\CM,F_\CM,Y,I}\circ (\phi_{i^1}\times \on{Id}_{\CA^{\otimes I}})\simeq \on{Sht}'_{\CM,F_\CM,Y,I}
\end{equation}
as functors
$$Y^I\times \CA^{\otimes I}\to \Vect.$$

\sssec{}

By definition, the datum of \eqref{e:partial Frob one} amounts to a system of isomorphisms
\begin{equation} \label{e:partial Frob two}
\Tr(H_{r_{\phi_{i^1}(\ul{y})}}\circ F_\CM,\CM)\simeq \Tr(H_{r_{\ul{y}}}\circ F_\CM,\CM), \quad r\in \CA^{\otimes I}.
\end{equation}

Write $\ul{y}$ as
$$\ul{y}=y^1\sqcup \ul{y'}, \quad y^1\in Y,\,\, \ul{y'}\in Y^{I'}.$$

It is enough to establish \eqref{e:partial Frob two} for $r$ of the form
$$r^1\otimes r', \quad r^1\in \CA,\,\, r'\in \CA^{\otimes I'}.$$

\begin{prop} \label{p:partial Frob}
The map \eqref{e:partial Frob two} is given by the
composition
\begin{equation} \label{e:partial Frob shtuka}
\CD
\Tr(H_{r^1_{y^1}}\circ H_{r'_{\ul{y'}}}\circ F_\CM,\CM)     @<{\sim}<<   \Tr(H_{r_{\ul{y}}}\circ F_\CM,\CM)  \\
@V{\sim}V{\text{\emph{cyclicity of trace}}}V \\
\Tr(H_{r'_{\ul{y'}}}\circ F_\CM\circ H_{r^1_{y^1}},\CM) @>{\sim}>> \Tr(H_{r'_{\ul{y'}}}\circ H_{r^1_{\phi(y^1)}}\circ F_\CM,\CM)\\
& & @VV{\sim}V \\
\Tr(H_{r_{\phi_{i^1}(\ul{y})}}\circ F_\CM,\CM) @<{\sim}<<  \Tr(H_{r^1_{\phi(y^1)}}\circ H_{r'_{\ul{y'}}}\circ F_\CM,\CM).
\endCD
\end{equation}
\end{prop}

The proof follows from \lemref{l:cyclicity and observables}.

\ssec{Description of the action via excursions}

\sssec{}

Choose a pair of points $\bar{y}_1,\bar{y}_2$ in $Y/\phi$.
Fix finite set $J$ and a $J$-tuple $\gamma^J$ of paths $\gamma^i$ from $\bar{y}_1$ to $\bar{y}_2$.
Choose an element
$$\xi\in \CHom(\one_\CA,\on{mult}_J\circ \on{mult}_J^R(\one_\CA)).$$

Let $\on{Exc}_{\on{univ}}(\gamma^J,\xi)$ denote the resulting endomorphism of $\one_{\CA^{\otimes Y/\phi}}$, see \secref{sss:univ exc}.

\sssec{}

The next assertion follows from the definition of the excursion operators in \secref{sss:excurs}:

\begin{prop} \label{p:excursion shtukas}
The action of $\on{Exc}_{\on{univ}}(\gamma^J,\xi)$ on
$$\Tr(F_\CM,\CM)\simeq
\CHom_{\CA^{\otimes Y/\phi}}(\one_{\CA^{\otimes Y/\phi}},\on{Sht}_{\CM,F_\CM,\on{univ}})\simeq \on{Sht}_{\CM,F_\CM,Y/\phi,\emptyset}$$
is given by the excursion

\smallskip

\begin{equation} \label{e:exc shtuka}
\CD
\on{Sht}_{\CM,F_\CM,Y/\phi,\emptyset} \\
@V{\sim}VV \\
\on{ev}_{\bar{y}_1}\left(\on{Sht}_{\CM,F_\CM,Y/\phi,\{*\}}(\one_\CA)\right) @>{\xi}>>
\on{ev}_{\bar{y}_1}\left(\on{Sht}_{\CM,F_\CM,Y/\phi,\{*\}}(\on{mult}_J\circ \on{mult}_J^R(\one_\CA))\right)  \\
& & @VV{\sim}V   \\
& & \on{ev}_{\bar{y}_1^J}\left(\on{Sht}_{\CM,F_\CM,Y/\phi,J}(\on{mult}_J^R(\one_\CA))\right)  \\
& & @VV{\on{mon}_{\gamma^J}}V  \\
& & \on{ev}_{\bar{y}_2^J}\left(\on{Sht}_{\CM,F_\CM,Y/\phi,J}(\on{mult}_J^R(\one_\CA))\right)  \\
& & @VV{\sim}V   \\
\on{ev}_{\bar{y}_2}\left(\on{Sht}_{\CM,F_\CM,Y/\phi,\{*\}}(\one_\CA)\right) @<{\on{counit}}<<
\on{ev}_{\bar{y}_2}\left(\on{Sht}_{\CM,F_\CM,Y/\phi,\{*\}}(\on{mult}_J\circ \on{mult}_J^R(\one_\CA))\right)  \\
@V{\sim}VV \\
\on{Sht}_{\CM,F_\CM,Y/\phi,\emptyset}.
\endCD
\end{equation}
\end{prop}

%
%
%
%

\ssec{The ``S=T" identity, vacuum case}  \label{ss:S=T}

We are now coming to what is perhaps in the most interesting part in the entire story.

\sssec{} \label{sss:y_0}

Let $y_0\in Y$ be a $\phi$-fixed point. Fix a compact object $a\in \CA$. Note that the object
$$a_{y_0}\in \CA^{\otimes Y}$$
is equipped with a natural isomorphism
$$\alpha_{\on{taut}}:a_{y_0}\overset{\sim}\to \phi_*(a_{y_0}).$$

Consider the corresponding natural transformation (in fact, an isomorphism)
\begin{equation} \label{e:alpha taut}
\alpha_{a_{y_0},\CM,F_\CM,\on{taut}}:H_{a_{y_0}}\circ F_\CM\overset{\sim}\to F_\CM\circ H_{a_{y_0}},
\end{equation}
see \eqref{e:Hecke a}.

\medskip

Hence we can consider the endomorphism
$$\Tr(H_{a_{y_0}},\alpha_{a_{y_0},\CM,F_\CM,\on{taut}}): \Tr(F_\CM,\CM)\to \Tr(F_\CM,\CM).$$

\begin{rem}
We should think of $$\Tr(H_{a_{y_0}},\alpha_{a_{y_0},\CM,F_\CM,\on{taut}})$$ as an analogue of the Hecke operator acting on the space of automorphic functions,
corresponding to a finite-dimensional representation of $\cG$ (thought of as $a\in \Rep(\cG)$) applied at a rational point of the curve
(thought of as $y_0$). So, this is V.~Lafforgue's ``T" operator.
\end{rem}

\sssec{}

Let $\bar{y}_0$ be the projection of $y_0$ to $Y/\phi$. The fact that $y_0$ was $\phi$-invariant defines a tautological
point
$$\gamma_{\on{taut}}\in \Omega(Y/\phi,\bar{y}_0).$$

Let $\xi_a$ be the map
$$\one_\CA \to \on{mult}\circ \on{mult}^R(\one_\CA)$$
defined by $a$, see \eqref{e:xi r}.

\sssec{}

We claim:

\begin{thm} \label{t:S=T}
The endomorphism $\Tr(H_{a_{y_0}},\alpha_{a_{y_0},\CM,F_\CM,\on{taut}})$ of
$$\Tr(F_\CM,\CM)\simeq \on{Sht}_{\CM,F_\CM,Y/\phi,\emptyset}$$
equals the operator \eqref{e:exc shtuka} for $(\bar{y}_0,J=\{*\}\sqcup \{*\},\gamma^+_{\on{taut}},\xi_a)$, where
$$\gamma^+_{\on{taut}}=(\gamma_{\on{taut}},\gamma_{\on{triv}}).$$
\end{thm}

\begin{rem} \label{r:partfrob}
We observe that by \propref{p:partial Frob}, the operator \eqref{e:exc shtuka} for $(\bar{y}_0,J=\{*\}\sqcup \{*\},\gamma^+_{\on{taut}},\xi_a)$,
appearing in \thmref{t:S=T}, is given explicitly by
\begin{equation} \label{e:exc shtuka expl}
\CD
\Tr(F_\CM,\CM) @>{\xi_a}>> \Tr(H_{(\on{mult}\circ \on{mult}^R(\one_\CA))_{y_0}}\circ F_\CM,\CM) \\
& & @VV{\sim}V   \\
& &  \Tr(H_{(\on{mult}^R(\one_\CA))_{y_0,y_0}}\circ F_\CM,\CM)   \\
& & @VV{\text{partial Frobenius, i.e., isomorphism \eqref{e:partial Frob shtuka}}}V  \\
& &  \Tr(H_{(\on{mult}^R(\one_\CA))_{y_0,y_0}}\circ F_\CM,\CM)  \\
& & @VV{\sim}V   \\
\Tr(F_\CM,\CM)  @<{\on{counit}}<< \Tr(H_{(\on{mult}\circ \on{mult}^R(\one_\CA))_{y_0}}\circ F_\CM,\CM).
\endCD
\end{equation}
\end{rem}

\begin{rem} \label{r:S=T}
Note that the operator \eqref{e:exc shtuka expl} is an analog of V.~Lafforgue's ``S" operator. For this reason, we view
\thmref{t:S=T} as the toy model for V.~Lafforgue's ``S=T" statement, i.e., \cite[Proposition 6.2]{Laf}.
\end{rem}

\begin{proof}[Proof of \thmref{t:S=T}]
First, we note that when considering the endomorphism $$\Tr(H_{a_{y_0}},\alpha_{a_{y_0},\CM,F_\CM,\on{taut}})$$ of $\Tr(F_\CM,\CM)$,
we can replace the original $Y$ by $\{*\}$ (with the necessarily trivial endomorphism $\phi$) via
$$\{*\} \overset{y_0}\longrightarrow Y.$$

\medskip

Next, we claim that when considering the operator \eqref{e:exc shtuka} for $$(\bar{y}_0,J=\{*\}\sqcup \{*\},\gamma^+_{\on{taut}},\xi_a),$$
we can also replace $Y$ by $\{*\}$. Indeed, this  immediately follows either from Remark \ref{r:dual of cor} or from
Remark \ref{r:partfrob}.



%

\medskip

Hence, in proving the theorem, we can assume that $Y=\{*\}$. Note that
$$\{*\}/\on{id}\simeq S^1.$$

\medskip

On the one hand, by \propref{p:S=T, one}, the endomorphism $\Tr(H_{a_{y_0}},\alpha_{a_{y_0},\CM,F_\CM,\on{taut}})$ of $\Tr(F_\CM,\CM)$
equals the action of the element
$$\on{cl}(a_{y_0}, \alpha_{a,\on{taut}})\in \Tr(\CA^{\otimes \phi},\CA^{\otimes Y}),$$
which in the case $Y=\{*\}$ amounts to the element
$$\on{cl}(a,\on{id})\in \Tr(\on{Id},\CA).$$

\medskip

On the other hand, by \propref{p:excursion shtukas}, the operator \eqref{e:exc shtuka} for $(\bar{y}_0,J=\{*\}\sqcup \{*\},\gamma^+_{\on{taut}},\xi_a)$
is given by the action of
$$\on{Exc}_{\on{univ}}(\gamma^+_{\on{taut}},\xi_a)\in \CEnd_{\CA^{\otimes Y/\phi}}(\one_{\CA^{\otimes Y/\phi}}),$$
which in our case is the element
$$\on{Exc}_{\on{univ}}(\gamma^+_{\on{taut}},\xi_a)\in \CEnd_{\CA^{\otimes S^1}}(\one_{\CA^{\otimes S^1}})$$
in the notations of \propref{p:chern as excurs}.

\medskip

Hence, we need to show to show that the above two elements match under the isomorphism
\begin{equation} \label{e:two notions of trace triv}
\Tr(\CA^{\otimes \phi},\CA^{\otimes Y})\simeq \CEnd_{\CA^{\otimes Y/\phi}}(\one_{\CA^{\otimes Y/\phi}})
\end{equation}
of \thmref{t:two notions of trace}, where we identify
$$\CA^{\otimes Y/\phi}\simeq \on{HH}_\bullet(\CA^{\otimes Y},\CA^{\otimes \phi}).$$

Thus, in our case, the identification \eqref{e:two notions of trace triv} amounts to
$$\Tr(\on{Id},\CA) \simeq \CEnd_{\CA^{\otimes S^1}}(\one_{\CA^{\otimes S^1}}).$$

\medskip

Now, the desired identity of elements follows from
\corref{c:chern as excurs}.

\end{proof}

\ssec{The ``S=T" identity, general case}   \label{ss:S=T gen}

As was explained in Remark \ref{r:S=T}, the assertion of \thmref{t:S=T} is an analog of V.~Lafforgue's $S=T$ identity as
operators acting on the space of automorphic functions, i.e., empty-legged shtukas. We will now discuss its generalization,
which extends the $S=T$ identity as operators on $I$-legged shtukas.

\sssec{}

Let $y_0$ and $a$ be as in \secref{sss:y_0} above. Fix a finite set $J$, a point $\ul{y}\in Y^J$ and an object $r\in \CA^{\otimes J}$.
Denote
$$J_+:=\{*\}\sqcup J,\,\,\ul{y}{}_+:=y_0\sqcup \ul{y},\,\, J_{++}:=\{*\}\sqcup \{*\}\sqcup J,\,\,\ul{y}{}_{++}:=y_0\sqcup y_0\sqcup \ul{y}.$$


\medskip

The natural transformation \eqref{e:alpha taut} induces a natural transformation
$$H_{a_{y_0}}\circ H_{r_{\ul{y}}}\circ F_\CM  \simeq H_{r_{\ul{y}}} \circ H_{a_{y_0}}\circ F_\CM  \to
H_{r_{\ul{y}}} \circ F_\CM \circ H_{a_{y_0}},$$
to be denoted $\alpha_{a,\CM,F_\CM,\on{taut},r_{\ul{y}}}$.

\medskip

Then we can consider the resulting endomorphism
\begin{equation} \label{e:T obs}
\Tr(H_{a_{y_0}},\alpha_{a,\CM,F_\CM,\on{taut},r_{\ul{y}}}):\Tr(H_{r_{\ul{y}}}\circ F_\CM,\CM)\to \Tr(H_{r_{\ul{y}}}\circ F_\CM,\CM).
\end{equation}

\begin{rem}
The map \eqref{e:T obs} is an analog of V.~Lafforgue's ``T" operator acting on the cohomology of shtukas.
\end{rem}

\sssec{}
Let $\bar{\ul{y}}$ be the projection of $\ul{y}$ to $(Y/\phi)^J$. By \propref{p:identify I-legged}, we have an identification
$$\Tr(H_{r_{\ul{y}}}\circ F_\CM,\CM)\simeq \on{ev}_{\ul{y}}\left(\on{Sht}'_{\CM,F_\CM,Y,J}(r)\right)\simeq
\on{ev}_{\bar{\ul{y}}}\left(\on{Sht}_{\CM,F_\CM,Y/\phi,J}(r)\right).$$

\sssec{}
Set $\bar{\ul{y}}_+:=\bar{y}_0\sqcup \bar{\ul{y}}$ and $\bar{\ul{y}}_{++}:=\bar{y}_0\sqcup \bar{y}_0\sqcup \bar{\ul{y}}$.
Then we can consider the endomorphism of $\on{ev}_{\bar{\ul{y}}}\left(\on{Sht}_{\CM,F_\CM,Y,J}(r)\right)$
equal to the composite
\begin{equation} \label{e:exc shtuka general}
\CD
\on{ev}_{\bar{\ul{y}}}\left(\on{Sht}_{\CM,F_\CM,Y/\phi,J}(r)\right) @>{\sim}>> \on{ev}_{\bar{\ul{y}}{}_+}\left(\on{Sht}_{\CM,F_\CM,Y/\phi,J_+}(r\otimes \one_\CA)\right)  \\
& & @V{\xi_a}VV  \\
& & \on{ev}_{\bar{\ul{y}}{}_+}\left(\on{Sht}_{\CM,F_\CM,Y/\phi,J_+}(r\otimes \on{mult}\circ \on{mult}^R(\one_\CA))\right)  \\
& & @VV{\sim}V   \\
& & \on{ev}_{\bar{\ul{y}}{}_{++}}\left(\on{Sht}_{\CM,F_\CM,Y/\phi,J_{++}}(r\otimes\on{mult}^R(\one_\CA))\right)  \\
& & @VV{\on{mon}_{\gamma_{\on{taut}},\gamma_{\on{triv}},\gamma^J_{\on{triv}}}}V  \\
& &  \on{ev}_{\bar{\ul{y}}{}_{++}}\left(\on{Sht}_{\CM,F_\CM,Y/\phi\,J_{++}}(r\otimes\on{mult}^R(\one_\CA))\right) \\
& & @VV{\sim}V   \\
& & \on{ev}_{\bar{\ul{y}}{}_+}\left(\on{Sht}_{\CM,F_\CM,Y/\phi,J_+}(r\otimes \on{mult}\circ \on{mult}^R(\one_\CA))\right)   \\
& & @V{\on{counit}}VV  \\
\on{ev}_{\bar{\ul{y}}}\left(\on{Sht}_{\CM,F_\CM,Y/\phi,J}(r)\right) @<{\sim}<<  \on{ev}_{\bar{\ul{y}}{}_+}\left(\on{Sht}_{\CM,F_\CM,Y/\phi,J_+}(r\otimes \one_\CA)\right)
\endCD
\end{equation}

\begin{rem}

Note that by Propositions \ref{p:identify I-legged} and \ref{p:partial Frob}, the operator \eqref{e:exc shtuka general} identifies with
\begin{equation} \label{e:exc shtuka general expl}
\CD
\Tr(H_{r_{\ul{y}}}\circ F_\CM,\CM)  @>{\xi_a}>> \Tr(H_{(\on{mult}\circ \on{mult}^R(\one_\CA))_{y_0}}\circ H_{r_{\ul{y}}}\circ F_\CM,\CM) \\
& & @VV{\sim}V \\
& & \Tr(H_{(\on{mult}^R(\one_\CA))_{y_0,y_0}}\circ H_{r_{\ul{y}}}\circ F_\CM,\CM)  \\
& & @VV{\text{partial Frobenius, i.e., isomorphism \eqref{e:partial Frob shtuka}}}V  \\
& & \Tr(H_{(\on{mult}^R(\one_\CA))_{y_0,y_0}}\circ H_{r_{\ul{y}}}\circ F_\CM,\CM)  \\
& & @VV{\sim}V \\
\Tr(H_{r_{\ul{y}}}\circ F_\CM,\CM)  @<{\on{counit}}<< \Tr(H_{(\on{mult}\circ \on{mult}^R(\one_\CA))_{y_0}}\circ H_{r_{\ul{y}}}\circ F_\CM,\CM).
\endCD
\end{equation}

So the operator \eqref{e:exc shtuka general} is indeed an analog of V.~Lafforgue's ``S" operator in the presence of other Hecke functors.

\end{rem}

\sssec{}

We claim:

\begin{thm} \label{t:S=T gen}
The endomorphism $\Tr(H_{a_{y_0}},\alpha_{a,\CM,F_\CM,\on{taut},r_{\ul{y}}})$ of
$$\Tr(H_{r_{\ul{y}}}\circ F_\CM,\CM)\simeq \on{ev}_{\bar{\ul{y}}}\left(\on{Sht}_{\CM,F_\CM,Y/\phi,J}(r)\right)$$
equals the operator \eqref{e:exc shtuka general}.
\end{thm}

\begin{proof}

We claim that \thmref{t:S=T gen} is in fact a formal corollary of \thmref{t:S=T}. Namely, for our given $\CM$ set
$$F_{\CM,r_{\ul{y}}}:=H_{r_{\ul{y}}}\circ F_\CM.$$
We claim that the identity stated in \thmref{t:S=T gen} specializes to the identity of \thmref{t:S=T} for the endomorphism
$F_{\CM,r_{\ul{y}}}$.

\medskip

Indeed, we have an isomorphism
\begin{equation*} 
\Tr(H_{r_{\ul{y}}}\circ F_\CM,\CM)\simeq \Tr(F_{\CM,r_{\ul{y}}},\CM),
\end{equation*}
and under this identification, the endomorphism $\Tr(H_{a_{y_0}},\alpha_{a,\CM,F_\CM,\on{taut},r_{\ul{y}}})$ of the LHS corresponds to the
endomorphism $\Tr(H_{a_{y_0}},\alpha_{a,\CM,F_{\CM,r_{\ul{y}}},\on{taut}})$ of the RHS (see \secref{sss:y_0} for the notation).

\medskip
It therefore suffices to show that there exists an isomorphism
\begin{equation} \label{e:insert r bis}
\on{ev}_{\bar{\ul{y}}}\left(\on{Sht}_{\CM,F_\CM,Y/\phi,J}(r)\right)\simeq\on{Sht}_{\CM,F_{\CM,r_{\ul{y}}},Y/\phi,\emptyset}
\end{equation}
such that the endomorphism \eqref{e:exc shtuka general} of the LHS corresponds to the endomorphism \eqref{e:exc shtuka} with
$$(\bar{y}_0,J=\{*\}\sqcup \{*\},\gamma^+_{\on{taut}},\xi_a)$$ of the RHS.

\medskip

It follows from \thmref{t:obs internal} that there exists a canonical isomorphism
$$\iota(r_{\ul{y}})\otimes \on{Sht}_{\CM,F_{\CM},\on{univ}}\simeq \on{Sht}_{\CM,F_{\CM,r_{\ul{y}}},\on{univ}}$$
of objects of $\CA^{\otimes Y/\phi}$. Therefore using explicit formulas of \secref{sss:objects in integral}, for every finite set $I$, a point
$\bar{\ul{y}}'\in Y^I$ and an object $r'\in \CA^{\otimes I}$, we have an isomorphism
\[
\on{ev}_{\bar{\ul{y}}\sqcup\bar{\ul{y}}'}\left(\on{Sht}_{\CM,F_\CM,Y/\phi,J\sqcup I}(r\otimes r')\right)\simeq\on{ev}_{\bar{\ul{y}}'}\left(\on{Sht}_{\CM,F_{\CM,r_{\ul{y}}},Y/\phi,I}(r')\right),
\]
functorial in $\bar{\ul{y}}'$ and $r'$.

\medskip

Applying this isomorphism in the particular case $I=\emptyset$ and $r'=\one_{\Vect}$, we
get the isomorphism  \eqref{e:insert r bis} we need.

\end{proof}

\appendix

\section{Sheaves and singular support}  \label{s:sheaves}

This appendix is included for the sake of completeness. We will review the notion of singular support in
different sheaf-theoretic contexts.

\medskip

The main result of this section is \thmref{t:vert sing supp}, which says the following: In the product situation $\CY\times X$,
where $\CY$ is an algebraic stack and $X$ a proper scheme, the category of sheaves on $\CY\times X$ whose singular support is of the form
$$\CN':=\CN\times \{\text{zero section}\}\subset T^*(\CY)\times T^*(X)=T^*(\CY\times X),$$
is equivalent to the tensor product category
$$\Shv_{\CN}(\CY)\otimes \Shv_{\on{lisse}}(X).$$

\ssec{Sheaf-theoretic contexts}   \label{ss:sheaves}

\sssec{}  \label{sss:sheaves}

In this section and the next, we will take $\Shv(-)$ to be any of the following sheaf-theoretic contexts:

\medskip

\noindent(a) $\Shv_{\on{cl}}(S)$, the category of all sheaves in the classical topology with coefficients in $\sfe$,
for $S$ an affine scheme over $\BC$;

\medskip

\noindent(a') $\Shv(S)=\on{Ind}(\Shv_{\on{cl,constr}}(S))$, where $\Shv_{\on{cl,constr}}(S)\subset \Shv_{\on{cl}}(S)$ is
the (small) subcategory consisting of constructible sheaves;

\medskip

\noindent(b) $\Shv(S)=\Dmod(S)$, for $S$ an affine scheme over a ground field $k$ of characteristic $0$;

\medskip

\noindent(b') $\Shv(S)=\on{Ind}(\Dmod_{\on{hol}}(S))$, where $\Dmod_{\on{hol}}(S)\subset \Dmod(S)$ is
the (small) subcategory consisting of holonomic D-modules;

\medskip

\noindent(b'') $\Shv(S)=\on{Ind}(\Dmod_{\on{hol,RS}}(S))$, where $\Dmod_{\on{hol,RS}}(S)\subset \Dmod(S)$ is
the (small) subcategory consisting of holonomic D-modules with regular singularities;

\medskip

\noindent(c) $\Shv(S):=\Shv_{\BZ/\ell^n,\on{et}}(S)$, the category of all \'etale sheaves on $S$ with coefficients
in $\BZ/\ell^n\BZ$, for $S$ an affine scheme over any ground field of characteristic prime to $\ell$. Note that
we have
$$\Shv_{\BZ/\ell^n,\on{et}}(S)\simeq \on{Ind}(\Shv_{\BZ/\ell^n,\on{et},\on{constr}}(S)),$$
where $\Shv_{\BZ/\ell^n,\on{et},\on{constr}}(S)\subset \Shv_{\BZ/\ell^n,\on{et}}(S)$ is the
full subcategory of constructible sheaves, see \cite[Proposition 2.2.6.2]{GaLu};

\medskip

\noindent(d) $\Shv(S):=\on{Ind}(\Shv_{\BZ_\ell,\on{et},\on{constr}}(S))$, where $\Shv_{\BZ_\ell,\on{et},\on{constr}}(S)$ is
the category of constructible $\ell$-adic sheaves on $S$
(see \cite[Defn. 2.3.2.1]{GaLu}), which is equivalent to
$$\underset{n}{\on{lim}}\, \Shv_{\BZ/\ell^n,\on{et},\on{constr}}(S);$$

\medskip

\noindent(d') $\Shv(S):=\on{Ind}(\Shv_{\BQ_\ell,\on{et},\on{constr}}(S))$, where
$\Shv_{\BQ_\ell,\on{et},\on{constr}}(S)$ is obtained from
$\Shv_{\BZ_\ell,\on{et},\on{constr}}(S)$ by inverting $\ell$. Note that
$$\on{Ind}(\Shv_{\BQ_\ell,\on{et},\on{constr}}(S))\simeq
\on{Ind}(\Shv_{\BZ_\ell,\on{et},\on{constr}}(S))\underset{\BZ_\ell}\otimes \BQ_\ell.$$

\sssec{}

In what follows, we will refer to the cases (a') and (b')-(d') as ``ind-constructible".

\medskip

Note that, by definition,
in these cases, the category $\Shv(S)$ is compactly generated. In particular, it is dualizable.

\sssec{}  \label{sss:on prestacks}

Let $\CY$ be a prestack. In all cases apart from (a), we define $\Shv(\CY)$ as
\begin{equation} \label{e:on prestack upper!}
\underset{S\in (\affSch_{/\CY})^{\on{op}}}{\on{lim}}\, \Shv(S),
\end{equation}
where for $S_1\overset{f}\to S_2$, the corresponding functor $\Shv(S_2)\to \Shv(S_1)$
is $f^!$. In the above formula, $\affSch_{/\CY}$ is the category of affine schemes over $\CY$.

\medskip

Note that in all cases apart from (b), the functor $f^!$ admits a left adjoint, namely $f_!$. Hence, from \lemref{l:lim colim}(b),
we obtain that in all of these cases, we can write $\Shv(\CY)$ also as
\begin{equation} \label{e:Shv as colim}
\underset{S\in \affSch_{/\CY}}{\on{colim}}\, \Shv(S),
\end{equation}
where for $S_1\overset{f}\to S_2$,
the corresponding functor
$\Shv(S_1)\to \Shv(S_2)$
is $f_!$.

\medskip

Note also that in the ind-constructible contexts (i.e., cases (a') and (b')-(d')), we obtain from \eqref{e:Shv as colim} that the category
$\Shv(\CY)$ is compactly generated (since each $\Shv(S)$ is). In particular, $\Shv(\CY)$ is dualizable.

\sssec{}

Let us now consider case (a). (The slight glitch here is that in this case the functors $f^!$ are no longer continuous.) We define
$\Shv(\CY)$ by formula \eqref{e:Shv as colim}, where the colimit is taken $\DGCat$.

\medskip

Let $\DGCat^{\on{discont}}$ denote the category whose objects are not necessarily cocomplete DG categories,
and we allow all (exact) functors between such categories.  Specifically, $\DGCat^{\on{discont}}$ is the category of stable
idempotent-complete $\infty$-categories tensored over $\on{Vect}^{\on{f.d.}}$ and exact $\on{Vect}^{\on{f.d.}}$-linear functors. We have a tautological functor
\begin{equation} \label{e:forget discont}
\DGCat\to \DGCat^{\on{discont}}
\end{equation}

We can also form the limit \eqref{e:on prestack upper!}, taking place in $\DGCat^{\on{discont}}$. Now, according to
\cite[Corollary 5.3.4(b)]{GR1} (which is a generalization of \lemref{l:lim colim} to the case when the right adjoints are not
necessarily continuous), the image of $\Shv(\CY)$, which is defined by formula \eqref{e:Shv as colim}, under
the forgetful functor \eqref{e:forget discont} identifies with the limit \eqref{e:on prestack upper!}.

\medskip

In other words, we have a canonical isomorphism
$$\Shv_{\on{cl}}(\CY)\simeq \underset{S\in \affSch_{/\CY}}{\on{lim}}\, \Shv_{\on{cl}}(S)$$
as objects of $\DGCat^{\on{discont}}$.

\sssec{}  \label{sss:Artin}

Let now $\CY$ be an algebraic stack. In this case, we can consider the category $\affSch_{/\CY,\on{sm}}$,
consisting of affine schemes equipped with a \emph{smooth} map to $\CY$, and whose morphisms are \emph{smooth}
maps between affine schemes over $\CY$.

\medskip

A smooth descent argument shows that the restriction functor
$$\underset{S\in \affSch_{/\CY}}{\on{lim}}\,\Shv(S) \to \underset{S\in \affSch_{/\CY,\on{sm}}}{\on{lim}}\, \Shv(S)$$
is an equivalence.

\medskip

Hence, for an algebraic stack $\CY$, we have
\begin{equation} \label{e:on prestack upper! via smooth}
\Shv(\CY)\simeq \underset{S\in \affSch_{/\CY,\on{sm}}}{\on{lim}}\, \Shv(S)
\end{equation}
where this isomorphism takes place in $\DGCat^{\on{discont}}$ in context (a) and in  $\DGCat$ in other contexts.

\begin{rem}  \label{r:a for Artin}
The presentation of $\Shv(\CY)$ given by \eqref{e:on prestack upper! via smooth} as an object of $\DGCat$ is valid
also in case (a): indeed, since smooth pullbacks $f^!$ are continuous also in case (a), the assertion follows from
\eqref{e:on prestack upper! via smooth} and the fact that the tautological functor $\DGCat\to \DGCat^{\on{discont}}$ preserves limits,
and for $\CC_1,\CC_2\in \DGCat$, the map induced by forgetful functor
$$\Maps_{\DGCat}(\CC_1,\CC_2)\to \Maps_{\DGCat^{\on{discont}}}(\CC_1,\CC_2)$$
induces an isomorphism on the connected components corresponding to equivalences.


\end{rem}

\sssec{}  \label{sss:corr}

Let $\on{Corr}(\on{PreStk})_{\on{ind-sch,all}}$ be the category of correspondences as in \cite[Chapter 3, Sect. 5.4]{GR2},
whose objects are prestacks $\CY$, and where morphisms from $\CY_1$ to $\CY_2$ are diagrams
\begin{equation} \label{e:corr}
\CD
\CY_{1,2}  @>{g}>> \CY_1 \\
@V{f}VV  \\
\CY_2,
\endCD
\end{equation}
where $g$ any map, and $f$ is required to be ind-schematic. The composition of \eqref{e:corr} and
$$
\CD
\CY_{2,3}  @>>> \CY_2 \\
@VVV  \\
\CY_3
\endCD
$$
is given by
$$
\CD
\CY_{2,3}\underset{\CY_2}\times \CY_{1,2}  @>>> \CY_1 \\
@VVV  \\
\CY_3.
\endCD
$$

\medskip

Let us first exclude the context (a). Then the construction of {\it loc.cit.} applies, and we can extend $\Shv(-)$ to
a functor
$$\Shv_{\on{Corr}}:\on{Corr}(\on{PreStk})_{\on{ind-sch,all}}\to \DGCat.$$

At the level of objects this functor sends $\CY\mapsto \Shv(\CY)$. At the level of morphisms, this functor sends
a morphism \eqref{e:corr} to the functor
$$f_*\circ g^!:\Shv(\CY_1)\to \Shv(\CY_2).$$

Compatibility with compositions is insured by base change.

\medskip

Furthermore, the functor $\Shv_{\on{Corr}}$ possesses a natural right-lax symmetric monoidal structure, see \cite[Chapter 3, Sect. 6.1]{GR2},
where $\on{Corr}(\on{PreStk})_{\on{ind-sch,all}}$ is a symmetric monoidal category with respect to the level-wise product.

\sssec{}  \label{sss:corr a}

Let us now consider the context (a). In this case we will consider $\Shv_{\on{Corr}}$ as taking values in $\DGCat^{\on{discont}}$.
We will regard $\Shv_{\on{Corr}}$ as equipped with the right-lax symmetric monoidal structure, with respect to the following
symmetric monoidal structure on $\DGCat^{\on{discont}}$:

\medskip

For $\CC\in \DGCat^{\on{discont}}$, we let
$$\on{Funct}_{\DGCat^{\on{discont}}}(\CC_1\overset{\on{discont}}\otimes \CC_2,\CC)$$
consist of all bi-exact bi-$\sfe$-linear (but not necessarily bi-continuous) functors
$$\CC_1\times \CC_2\to \CC.$$

More formally, we apply the construction of \cite[Chapter 1, Sect. 6.1.1]{GR1}, but dropping the continuity
condition.

\begin{rem}  \label{r:recover cont}
Note that given an object $\CC\in \DGCat$, the space of structures of associative (resp., commutative)
algebras on it within $\DGCat$ embeds fully faithfully into the space of such structures within $\DGCat^{\on{discont}}$.

\medskip

The same applies to actions of a given monoidal DG category on another DG category.

\end{rem}

\ssec{Sheaves on a product}

\sssec{}

Note that for a pair of affine schemes, we have a naturally defined functor, given by \emph{external tensor product}
\begin{equation} \label{e:external product}
\Shv(S_1)\otimes \Shv(S_2)\to \Shv(S_1\times S_2), \quad \CF_1,\CF_2\mapsto \CF_1\boxtimes \CF_2.
\end{equation}

This functor is an equivalence in case (b): for a pair of associative algebras $A_1$ and $A_2$, the functor
$$A_1\mod\otimes A_2\mod\to (A_1\otimes A_2)\mod, \quad M_1,M_2\mapsto M_1\otimes M_2$$
is an equivalence (see \cite[Theorem 4.8.5.16]{Lu2}).

\sssec{}

The following is known (the assertion is valid for any pair of locally compact Hausdorff topological spaces, see \cite[Theorem 7.3.3.9, Prop. 7.3.1.11]{Lu1} and \cite[Prop. 4.8.1.17]{Lu2}):

\begin{thm} \label{t:external prod top}
The functor \eqref{e:external product} is an equivalence in case (a).
\end{thm}

\begin{rem}

It follows from \thmref{t:external prod top} that for a locally compact topological space $M$, the functors
$$\Vect\overset{\sfe_M}\to \Shv(M) \overset{\Delta_!}\to \Shv(M\times M)\simeq \Shv(M)\otimes \Shv(M)$$
and
$$\Shv(M)\otimes \Shv(M)\to \Shv(M\times M)\overset{\Delta^*}\to \Shv(M) \overset{\on{C}_c^\bullet(M,-)}\longrightarrow \Vect$$
define an identification
$$\Shv(M)^\vee \simeq \Shv(M).$$

Note, however, that as was shown by A.~Neeman (see \cite{Ne1}), for a topological manifold $M$,
the category $\Shv(M)$ is \emph{not} compactly generated, unless $M$ is discrete. So, $\Shv(M)$ is an example
of a dualizable but not compactly generated category.
\end{rem}

\sssec{}

In the ind-constructible contexts, the functor \eqref{e:external product} fails to be an equivalence.
However, we have:

\begin{lem} \label{l:external prod ff}
In the constructible contexts the functor \eqref{e:external product} is fully faithful.
\end{lem}

\begin{proof}

For a pair of DG categories $\CC_1,\CC_2$ and
$$c'_i\in \CC_i^c,\,\, c''_i\in \CC_i,\,\, i=1,2,$$
the map
$$\CHom_{\CC_1}(c'_1,c''_1)\otimes \CHom_{\CC_2}(c'_2,c''_2)\to
\CHom_{\CC_1\otimes \CC_2}(c'_1\otimes c'_2,c''_1\otimes c''_2)$$
is an isomorphism (see \cite[Chapter 1, Proposition 7.4.2]{GR1}).

\medskip

Since the tensor product of compactly generated categories is compactly generated by tensor products of compact objects (see \cite[Chapter 1, Proposition 7.4.2]{GR1}), in order to prove the lemma, it suffices to show that for
$$\CF'_i,\CF''_i\in \Shv(S_i)^c, \quad i=1,2$$
the map
$$\CHom_{\Shv(S_1)}(\CF'_1,\CF''_1)\otimes \CHom_{\Shv(S_2)}(\CF'_2,\CF''_2)\to
\CHom_{\Shv(S_1\times S_2)}(\CF'_1\boxtimes \CF_2,\CF''_1\boxtimes \CF''_2)$$
is an isomorphism. However, this follows from Kunneth's formula.
\end{proof}

\sssec{}

We now consider the case of prestacks. Again, external tensor product gives rise to a functor
\begin{equation} \label{e:external product stacks}
\Shv(\CY_1)\otimes \Shv(\CY_2)\to \Shv(\CY_1\times \CY_2), \quad \CF_1,\CF_2\mapsto \CF_1\boxtimes \CF_2.
\end{equation}

The argument in \cite[Chapter 3, Theorem 3.1.7]{GR1} shows:

\begin{lem}
In case (b), if one of the categories $\Shv(\CY_i)$ is dualizable, then the functor \eqref{e:external product stacks}
is an equivalence.
\end{lem}

In addition, from \thmref{t:external prod top} and  \cite[Chapter 3, Equation (3.4)]{GR1}, we obtain:

\begin{cor}
The functor \eqref{e:external product stacks} is an equivalence in case (a).
\end{cor}

Finally, we claim:

\begin{prop} \label{l:external prod ff stacks}
The functor \eqref{e:external product stacks} is fully faithful in the ind-constructible contexts.
\end{prop}

\begin{proof}

Since $\Shv(\CY_i)$ are dualizable (see \secref{sss:on prestacks}), the argument in \cite[Chapter 3, Theorem 3.1.7]{GR1} shows
that we can write
$$\Shv(\CY_1)\otimes \Shv(\CY_2)$$ as
$$\underset{S_1\to \CY_1,S_2\to \CY_2}{\on{lim}}\, \Shv(S_1)\otimes \Shv(S_2),$$
and by \cite[Chapter 3, Equation (3.4)]{GR1}, we have
$$\Shv(\CY_1\times \CY_2)\simeq \underset{S_1\to \CY_1,S_2\to \CY_2}{\on{lim}}\, \Shv(S_1\times S_2).$$

Hence, the assertion follows from \lemref{l:external prod ff}.

\end{proof}

\ssec{Singular support}

\sssec{}

Let $S$ be an affine scheme. First, we assume that $S$ is smooth. Let $\CN\subset T^*(S)$ be a conical
Zariski-closed subset. In each of our sheaf-theoretic
contexts we can single out a full subcategory
$$\Shv_{\CN}(S)\subset \Shv(S),$$
consisting of objects with singular support contained in $\CN$:

\medskip

\noindent--In case (a), we require that each cohomology sheaf belongs to $\Shv_{\on{cl},\CN}(S)$, where the latter
is defined in \cite[Sect. 8]{KS};

\medskip

\noindent--In case (a'), we inherit the definition from case (a) for $\Shv_{\on{cl,constr}}(S)$, and then ind-extend
(alternatively, transfer the definition via the Riemann-Hilbert from case (b'')).

\medskip

\noindent--In case (b), we take the ind-completion of the category $\Dmod_\CN(S)^{\on{f.g}}\subset \Dmod(S)^{\on{f.g}}$,
obtained by requiring that each cohomology belong to $\Dmod_{\CN}(S)^{\on{f.g.,\heartsuit}}$, where the latter is
the standard D-module notion;

\medskip

\noindent--In case (b') (resp., (b'')), we inherit the definition for $\Dmod_{\on{hol}}(S)$
(resp., $\Dmod_{\on{hol,RS}}(S)$), then ind-extend.

\medskip

\noindent--In cases (c), we give the definition for $\Shv_{\BZ/\ell^n,\on{et},\on{constr}}(S)$ following \cite{Be},
then ind-extend.

\medskip

\noindent--In case (d) we will say that an object of $\Shv_{\BZ_\ell,\on{et},\on{constr}}(S)$
has singular support in $\CN$ if its projection mod $\ell$ does (alternatively, we apply the definition of
\cite{Be} directly to $\Shv_{\BZ_\ell,\on{et},\on{constr}}(S)$), then ind-extend;

\medskip

\noindent--In case (d') we apply the definition of \cite{Be} to $\Shv_{\BQ_\ell,\on{et},\on{constr}}(S))$,
then ind-extend.

\begin{rem}

Note that in all cases apart from (a), there exists another (a priori, larger) full subcategory of $\Shv(S)$, to be denoted
$\Shv^\wedge_{\CN}(S)$, that one could call ``sheaves with singular support in $\CN$". Namely, an
object belongs to $\Shv^\wedge_{\CN}(S)$ if its cohomology sheaves (with respect to the perverse t-structure)
belong to $\Shv_{\CN}(S)$.

\medskip

One can show that $\Shv^\wedge_{\CN}(S)$ identifies with the left completion of $\Shv_{\CN}(S)$ with
respect to its t-structure. What is not clear, however, is whether the category $\Shv^\wedge_{\CN}(S)$ is compactly
generated.

\medskip

Note that the embedding $\Shv_{\CN}(S)\hookrightarrow \Shv^\wedge_{\CN}(S)$ is \emph{not} always an equivalence
(equivalently, $\Shv^\wedge_{\CN}(S)$ is \emph{not} always generated by objects that are compact in $\Shv(S)$).
For example, this occurs in the example of $S=\BP^1$ and $\CN$ being the zero-section.

\end{rem}


\sssec{}

We will now show how to extend the above definition to the case when $S$ is not necessarily smooth.
Let $\CF$ be a coherent sheaf on $S$ (in cohomological degree $0$). Represent $\CF$ as
\begin{equation} \label{e:coker}
\on{coker}(\CE_1\to \CE_0),
\end{equation}
where $\CE_1$ and $\CE_0$ are locally free. Consider the total spaces of $\CE_i$ as group-schemes over $S$,
\begin{equation} \label{e:Tot}
\on{Tot}(\CE_i):=\Spec_S(\on{Sym}_{\CO_S}(\CE^\vee_i)).
\end{equation}

Consider the algebraic stack
$$\on{Tot}(\CF):=\on{Tot}(\CE_0)/\on{Tot}(\CE_1).$$

The object $\on{Tot}(\CF)$, viewed as an algebraic stack, depends on the presentation \eqref{e:coker}. But it is well-defined
in the localization of the category of algebraic stacks, where we invert morphisms
that are smooth, surjective and whose fibers are of the form $\on{pt}/H$ where $H$ is a vector group.

\medskip

We have a well-defined notion of a Zariski-closed subset of $\on{Tot}(\CF)$. For a choice of a presentation \eqref{e:coker},
they bijectively correspond to Zariski-closed subsets of $\on{Tot}(\CE_0)$ that are invariant under the action of $\on{Tot}(\CE_1)$.
This notion \emph{does not} depend on the presentation \eqref{e:coker}.

\sssec{}

Taking $\CF:=\Omega^1(S)$, we thus obtain an object $T^*(S):=\on{Tot}(\Omega^1(S))$ in the above localization of the category of algebraic stacks.
For a closed embedding
$$f:S\hookrightarrow S'$$
with $S'$ smooth, the codifferential map
\begin{equation} \label{e:codiff S'}
S\underset{S'}\times T^*(S') \to T^*(S)
\end{equation}
realizes $T^*(S)$ as a quotient of $S\underset{S'}\times T^*(S')$ by a vector group.

\sssec{}

Thus, we have a \emph{well-defined} notion of (conical) Zariski-closed subset of $T^*(S)$. We emphasize
that although $T^*(S)$ as an algebraic stack depends on some choices, the set of its Zariski-closed subsets
does not.

\medskip

Note also that the \emph{cotangent fibers}, i.e., the sets $T^*_s(S)$ for $s\in S(k)$, underlying the fibers of 
$T^*(S)$, are the classical
cotangent spaces, and as such do not depend on any choices.

\medskip

For $\CN\subset T^*(S)$ let
$$\CN'\subset S\underset{S'}\times T^*(S')\subset T^*(S')$$
be the preimage of $\CN$ under the map \eqref{e:codiff S'}. A subset $\CN$ as above is completely determined
by its fibers
$$\CN_s:=\CN\cap T^*_s(S)\subset T^*_s(S), \quad s\in S(k).$$

\medskip

We say that an object $\CF\in \Shv(S)$ has singular support in $\CN$ if $f_*(\CF)$ has singular support
in $\CN'$. It is not difficult to verify that this definition does not depend in the choice of the
embedding $f:S\to S'$.

\medskip

We shall say that $\CN\subset T^*(S)$ is half-dimensional if $\CN'\subset T^*(S')$
is such for some/any smooth $S'$.

\sssec{}

Let $\CY$ be an algebraic stack. For a coherent sheaf $\CF$ on $\CY$ we can talk about Zariski-closed subsets
of $\on{Tot}(\CF)$. Namely, they correspond bijectively to compatible families of Zariski-closed subsets of
$\on{Tot}(\CF|_S)$ for $S\in \on{Sch}^{\on{aff}}_{/\CY}$ (or, equivalently, $S\in \on{Sch}^{\on{aff}}_{/\CY,\on{sm}}$).

\medskip

Taking $\CF:=\Omega^1(\CY)$, we thus obtain a well-defined notion of (conical) Zariski-closed subset of
$T^*(\CY)$ (note that we are not even trying to define $T^*(\CY)$ itself; that said, as in the
case of schemes, the cotangent fibers $T^*_y(\CY)$, $y\in \CY(k)$, are the classical cotangent spaces,
and thus are well-defined).

\medskip

Thus, we can talk about Zariski-closed subsets $\CN\subset T^*(\CY)$. As in
the case of schemes, such $\CN$ is completely determined by the subsets
$$\CN_y:=\CN\cap T^*_y(\CY)\subset T^*_y(\CY), \quad y\in \CY(k).$$

\medskip

To a conical Zariski-closed subset $\CN\subset T^*(\CY)$ we associate a full subcategory
$$\Shv_\CN(\CY)\subset \Shv(\CY).$$

Namely, an object belongs to $\Shv_\CN(\CY)$ if for any smooth map $S\to \CY$
(for $S$ an affine scheme), its pullback to $S$ belongs to
$$\Shv_{\CN_S}(S)\subset \Shv(S),$$
where $\CN_S$ is the image of $\CN\underset{\CY}\times S$ under the co-differential
\begin{equation} \label{e:codiff Y}
T^*(\CY)\underset{\CY}\times S\to T^*(S).
\end{equation}

Here we are using the presentation of $\Shv(\CY)$ as in \secref{sss:Artin} (see also Remark \ref{r:a for Artin} in case (a)).

\medskip

In what follows we shall say that $\CN\subset T^*(\CY)$ is half-dimensional if its image is such under
\eqref{e:codiff Y} for some/any smooth cover $S\to \CY$.

\sssec{} \label{sss:defn lisse A}

Let $\CY$ be a smooth. In this case we define a full subcategory
$$\Shv_{\on{lisse}}(\CY)\subset \Shv(\CY).$$

\medskip

We first give the definition for affine schemes; for stacks, lisse would mean that the
pullback to affine schemes under smooth maps is lisse.

\medskip

For an affine scheme, we set:

\medskip

\noindent--In case (a), lisse means that each cohomology sheaf is locally constant;

\medskip

\noindent--In case (a'), lisse means a colimit of constructible locally constant objects;

\medskip

\noindent--In cases (b), (b') and (b''), lisse means a colimit of $\CO$-coherent objects;

\medskip

\noindent--In case (c), lisse means a colimit of constructible locally constant objects;

\medskip

\noindent--In case (d), lisse means a colimit of objects that are
constructible and locally constant (i.e., ones whose reduction mod $\ell$ is constructible and locally constant);

\medskip

\noindent--In case (d'), lisse means a colimit of objects that are
constructible and locally constant.


\medskip

We note that in all of the above case, being lisse is equivalent to belonging to $\Shv_{\CN}(\CY)$, where
$\CN$ is the zero-section.

\sssec{}

We are going to prove:

\begin{thm}  \label{t:vert sing supp}
Let $\CN\subset T^*(\CY)$ be half-dimensional. Let $X$ be a smooth scheme, assumed proper in all
cases apart from (a), (a') and (b''). Set
$$\CN':=\CN\times \{\text{zero-section}\}\subset T^*(\CY\times X).$$
Then the functor
$$\Shv_{\CN}(\CY)\otimes \Shv_{\on{lisse}}(X)\to \Shv_{\CN'}(\CY\times X)$$
is an equivalence.
\end{thm}

\begin{rem}
Note that if our sheaf-theoretic context is (b), the assumption that $\CN\subset T^*(\CY)$ is half-dimensional
implies that objects from $\Dmod_\CN(\CY)$ are automatically holonomic, i.e., this puts us into context (b').
\end{rem}

\ssec{Proof of \thmref{t:vert sing supp} in case (a)}

\sssec{}

The initial observation is the following:

\medskip

Let
$$Z\mapsto Z^{\on{top}}, \quad \on{Sch}\to \on{Top}$$
denote the functor that associates to a scheme over $\BC$ the topological space
underlying the corresponding analytic space.

\medskip

Let
$${\mathsf Y}\to {\mathsf Y}^{\on{sing}}, \quad \on{Top}\to \Spc$$
denote the functor of singular chains.

\medskip

We will denote the composite functor $\on{Sch}\to \Spc$ by
$$Z\mapsto Z^{\on{top,sing}}.$$

\medskip

Taking fibers at the points of $Z^{\on{top}}$ defines a functor
$$Z^{\on{top,sing}}\times \Shv_{\on{lisse}}(Z)\to \Vect,$$
i.e., a functor
\begin{equation} \label{e:LS on top}
{\mathsf r}:\Shv_{\on{lisse}}(Z)\to \LS(Z^{\on{top,sing}}).
\end{equation}

We have:

\begin{lem} \label{l:LS on top}
The functor \eqref{e:LS on top} is an equivalence.
\end{lem}

\begin{proof}

Both categories are equipped with t-structures, in which they are left and right complete, and
the functor ${\mathsf r}$ is t-exact. Hence, it is enough to show that the functor \eqref{e:LS on top} is
fully faithful on the bounded subcategories and essentially surjective on the hearts.

\medskip

The former is the expression of the fact that sheaf cohomology
of a complex with locally constant cohomology sheaves can be computed via singular cochains.

\medskip

The latter follows from the fact that both abelian categories in question identify with modules over
the fundamental groupoid of $Z^{\on{top}}$.

\end{proof}

\sssec{}  \label{sss:reduce sing to aff}

We proceed with the proof of \thmref{t:vert sing supp}.

\medskip

Since $\Shv_{\on{lisse}}(X)$ is dualizable (e.g., by \lemref{l:LS on top} and \propref{p:LS}(b)), by the argument of \cite[Chapter 3, Theorem 3.1.7]{GR1},
we have:
$$\Shv_{\CN}(\CY)\otimes \Shv_{\on{lisse}}(X)\simeq \underset{S\in (\affSch_{/\CY,\on{sm}})^{\on{op}}}{\on{lim}}\,
\Shv_{\CN_S}(S)\otimes \Shv_{\on{lisse}}(X).$$

\medskip

Similarly, a smooth descent argument shows that the functor
$$\Shv_{\CN'}(\CY\times X)\to \underset{S\in (\affSch_{/\CY,\on{sm}})^{\on{op}}}{\on{lim}}\,
\Shv_{\CN'_S}(S\times X)$$
is an equivalence.

\medskip

Hence, the assertion of the theorem reduces to the case when $\CY$ is an (affine) scheme.

\sssec{} \label{sss:ff ten prod}

Note that if $\CC_1$ is dualizable, and $\CC'_2\hookrightarrow \CC_2$ is a fully faithful embedding,
then
$$\CC_1\otimes \CC'_2\to \CC_1\otimes \CC_2$$
is also fully faithful (indeed, interpret $\CC_1\otimes -$ as $\on{Funct}_{\on{cont}}(\CC^\vee_1,-)$).

\medskip

Hence, the functors
$$\Shv_{\CN}(\CY)\otimes \Shv_{\on{lisse}}(X)\to
\Shv(\CY)\otimes \Shv_{\on{lisse}}(X)\to \Shv(\CY)\otimes \Shv(X)$$
are both fully faithful. Combined with \thmref{t:external prod top}, we obtain that the functor
\begin{equation} \label{e:ten prod top}
\Shv_{\CN}(\CY)\otimes \Shv_{\on{lisse}}(X)\to \Shv_{\CN'}(\CY\times X)
\end{equation}
is also fully faithful. Hence, to prove \thmref{t:vert sing supp}, it remains to prove that \eqref{e:ten prod top} is
essentially surjective.

%
%
%
%
%
%
%
%

\sssec{}

From this point, the proof is essentially borrowed from \cite[Page 20]{NY}.

%

\medskip

According to \cite[Corollary 8.3.22]{KS}, we can choose a
$\mu$-stratification of $\CY=\underset{\alpha}\cup\, \CY_\alpha$, such that $\CN$ is contained in the union of the
conormals to the strata. Consider the corresponding stratification $\CY_\alpha\times X$ of $\CY\times X$. This is still
a $\mu$-stratification.

\medskip

Let
$$\Shv_{\on{str}}(\CY)\subset \Shv(\CY) \text{ and } \Shv_{\on{str}'}(\CY\times X)\subset \Shv(\CY\times X)$$
denote the full subcategories consisting of objects locally constant along the strata. By \cite[Proposition 8.4.1]{KS},
we have
$$\Shv_\CN(\CY)\subset \Shv_{\on{str}}(\CY) \text{ and }
\Shv_{\CN'}(\CY\times X)\subset \Shv_{\on{str}}(\CY\times X).$$

\sssec{}

First, we claim that the functor
\begin{equation} \label{e:str ten prod}
\Shv_{\on{str}}(\CY)\otimes \Shv_{\on{lisse}}(X) \to \Shv_{\on{str}'}(\CY\times X)
\end{equation}
is an equivalence.

\medskip

By \secref{sss:ff ten prod}, this functor is fully faithful. To prove that it is essentially surjective,
by a Cousin argument, it reduces to the assertion that for each $\alpha$, the functor
$$\Shv_{\on{lisse}}(\CY_\alpha)\otimes \Shv_{\on{lisse}}(X)\to \Shv_{\on{lisse}}(\CY_\alpha\times X)$$
is an equivalence.

\medskip

However, this follows from \lemref{l:LS on top}, and the fact that for any two $Y_1,Y_2\in \Spc$, the functor
\begin{equation} \label{e:loc sys prod}
\LS(Y_1)\otimes \LS(Y_2)\to \LS(Y_1\times Y_2)
\end{equation}
is an equivalence (say, by \propref{p:LocSys as right adj}).

\sssec{}

Thus, it remains to show that if $\CF\in \Shv_{\on{str}}(\CY)\otimes \Shv_{\on{lisse}}(X)$ is an object whose image along
\eqref{e:str ten prod} belongs to $\Shv_{\CN'}(\CY\times X)$, then $\CF$ itself belongs to
$\Shv_\CN(\CY)\otimes \Shv_{\on{lisse}}(X)$.

\medskip

With no restriction of generality, we can assume that $X$ is connected, and let us choose a base point $x\in X$.
Interpreting $\Shv_{\on{lisse}}(X)$ via \lemref{l:LS on top}, we obtain that if
$$\CC'\hookrightarrow \CC$$
is a fully faithful map in $\DGCat$, then
$$\CC'\otimes \Shv_{\on{lisse}}(X)\to \CC\otimes \Shv_{\on{lisse}}(X)$$
is also fully faithful, and an object $c_X\in \CC\otimes \Shv_{\on{lisse}}(X)$ belongs to $\CC'\otimes \Shv_{\on{lisse}}(X)$
if and only if its essential image under the evaluation functor
$$\CC\otimes \Shv_{\on{lisse}}(X) \overset{\on{Id}\otimes \on{ev}_x}\longrightarrow \CC$$
belongs to $\CC'$.

\medskip

Consider the commutative diagram
$$
\CD
\Shv_\CN(\CY)\otimes \Shv_{\on{lisse}}(X) @>>> \Shv_{\on{str}}(\CY)\otimes \Shv_{\on{lisse}}(X) @>>>  \Shv_{\on{str}'}(\CY\times X)  \\
@V{\on{Id}\otimes \on{ev}_x}VV @VV{\on{Id}\otimes \on{ev}_x}V  @VVV  \\
\Shv_\CN(\CY) @>>> \Shv_{\on{str}}(\CY) @>{\on{Id}}>>  \Shv_{\on{str}}(\CY),
\endCD
$$
where the right vertical arrow is given by *-restriction along $\{x\}\hookrightarrow X$.

\medskip

We obtain that it suffices to show that the functor
$$
\CD
\Shv_{\on{str}'}(\CY\times X)  @<<< \Shv_{\CN'}(\CY\times X) \\
@VVV \\
\Shv_{\on{str}}(\CY)
\endCD
$$
takes values in $\Shv_\CN(\CY)$. This is a standard fact, but let us prove it for completeness.

\sssec{}

Let $\cD_x$ be an open disc around $x$. We have a commutative diagram
$$
\CD
\Shv_{\on{str}'}(\CY\times X)  @<<< \Shv_{\CN'}(\CY\times X) \\
@VVV @VVV  \\
\Shv_{\on{str}'}(\CY\times \cD_x)  @<<< \Shv_{\CN'}(\CY\times \cD_x),
\endCD
$$
so it suffices to show that the functor
$$
\CD
\Shv_{\on{str}'}(\CY\times \cD_x)  @<<< \Shv_{\CN'}(\CY\times \cD_x) \\
@VVV \\
\Shv_{\on{str}}(\CY)
\endCD
$$
takes values in $\Shv_\CN(\CY)$.

\medskip

However, by \eqref{e:str ten prod}, the functor
$\Shv_{\on{str}'}(\CY\times \cD_x) \to \Shv_{\on{str}}(\CY)$ is an equivalence, whose inverse is given by pullback.
In particular, this inverse functor preserves singular support and hence defines an equivalence
$$\Shv_\CN(\CY)\to \Shv_{\CN'}(\CY\times \cD_x),$$
as desired.

%
%
%

\ssec{Proof of \thmref{t:vert sing supp} in the ind-constructible contexts}

\sssec{}

The proof of \thmref{t:vert sing supp} in case (a') follows verbatim the argument in case (a).
Case (b'') follows from case (a') by Lefschetz principle and Riemann-Hilbert.

\medskip

We will now prove \thmref{t:vert sing supp} in cases (a') and (b')-(d') assuming that $X$ is
proper. 

\medskip

With no restriction of generality, we can assume that $X$ is connected.

\sssec{}  \label{sss:set up ind}

%

As in \secref{sss:reduce sing to aff}, we reduce the assertion to the case when $\CY$ is an affine
scheme.

\medskip

By \propref{l:external prod ff stacks} (e.g., using the argument in \secref{sss:ff ten prod}),
we know that the functor in question is fully faithful. Hence,
it remains to show that it is essentially surjective.

\medskip

Let $\CF$ be an object in $\Shv_{\CN'}(\CY\times X)$. Since our functor preserves colimits,
we can assume that $\CF$ is constructible.

\medskip

We will argue by Noetherian induction on $\CY$,
so we will assume that the support of $\CF$ is dominant over $\CY$
(otherwise, replace $\CY$ by the closure of the image of the support of $\CF$).

\sssec{}

We claim that in the constructible case, $\CF$ belongs to $\Shv_{\CN'}(\CY\times X)$ if and only if this holds
for all of its \emph{perverse} cohomology sheaves. Indeed:

\medskip

\noindent--In case (a') this follows from the fact that $\on{SingSupp}$ is measured by the functor
of vanishing cycles, which is t-exact for the perverse t-structure.

\medskip

\noindent--In cases (b'), (b'') this follows from the definition.

\medskip

\noindent--In cases (c), (d) and (d'), this follows from the definition of $\on{SingSupp}$ in \cite{Be} and the corresponding
fact for the ULA property, see \cite{Ga6}.

\medskip

For the same reason, for a short exact sequence of perverse sheaves
$$0\to \CF_1\to \CF\to \CF_2\to 0,$$
we have
$$\on{SingSupp}(\CF)=\on{SingSupp}(\CF_1)\cup \on{SingSupp}(\CF_2).$$

\sssec{} \label{sss:red open 0}

Hence, we can assume that $\CF$ is of the form
$$j'_{!*}(\CF_\CU),$$
for a smooth locally closed
$$\CU\overset{j'}\hookrightarrow \CY\times X$$
and $\CF_\CU\in \Shv_{\on{lisse}}(\CU)$ is \emph{perverse and lisse}.
With no restriction of generality we can assume that $\CU$ is connected.

\medskip

By the assumption in \secref{sss:set up ind}, our $\CU$ is dominant over $\CY$. We claim that we can assume that $\CU$ is of the form
$\oCY\times X$, where $\oCY\overset{j}\hookrightarrow \CY$ is an open subset. This reduction will be carried
out in Sects. \ref{sss:red open 1}-\ref{sss:red open 2}.

\sssec{}  \label{sss:red open 1}

By the transitivity property of $j_{!*}$, we can replace $\CU$ by any of its non-empty open subsets.
Let $\CY_1\subset \CY$ be a non-empty smooth open subset contained in the image of $\CU$.
Let $\CU_1$ be the preimage of $\CY_1$ in $\CU$.

\medskip

We claim that $\CU_1$ is dense in $\CY_1\times X$.
Indeed, let $\ol\CU_1$ denote the closure of $\CU_1$ in $\CY_1\times X$.
If $\CU_1$ were not dense in $\CY_1\times X$, the object $\CF|_{\CY_1\times X}$ would be the direct image under the closed embedding
$$\ol\CU_1\hookrightarrow \CY_1\times X,$$
and hence $\on{SingSupp}(\CF)$ would contain the conormals to $\ol\CU_1$ at each of its generic points. However, since
$\CU_1\to \CY_1$ is surjective, these
conormals are \emph{not} contained in
$$T^*(\CY_1)\times \{\text{zero-section}\},$$
contradicting the assumption on $\on{SingSupp}(\CF)$.

\sssec{} \label{sss:red open 2}

Let $\oCU$ be an open subset
$$\CU_1\subseteq \oCU\subseteq \CY_1\times X,$$
maximal with respect to the property that the restriction to it of $\CF|_{\CY_1\times X}$ is lisse.
We claim that $\oCU_2$ is all of $\CY_1\times X$. Once we prove this, we will be able to take $\oCY:=\CY_1$, and thus achieve the reduction claimed in
\secref{sss:red open 0}.

\medskip

To prove the desired form of $\CU_2$ we argue as follows. By purity,
$$D':=\CY_1\times X-\oCU$$
is a divisor. We want to show that each irreducible component $D'_\alpha$ of $D'$
is the preimage of a divisor in $\CY_1$. Since $\oCU\to \CY_1$ is surjective, this would
imply that $D'$ is empty.

\medskip

Since $\CF|_{\CY_1\times X}$ is ramified around $D'$, its singular support is \emph{not}
contained in the zero-section of $T^*(\CY_1\times X)$ near the generic point of
each irreducible component $D'_\alpha$ of $D'$. Hence, for every such
$D'_\alpha$, there exists an irreducible component
of $\CN'_\alpha\subset \on{SingSupp}(\CF)$ such that the projection
$$\CN'_\alpha\to T^*(\CY_1\times X)\to \CY_1\times X$$
maps to $D'_\alpha$ with positive-dimensional fibers.

\medskip

We now use the assumption that
$$\on{SingSupp}(\CF)\subset \CN\times \{\text{zero-section}\}.$$

We obtain that for each $\alpha$ there exists an irreducible component $\CN_\alpha$ of $\CN$ such that
$$\CN'_\alpha\subset \CN_\alpha\times X.$$

Let $D_\alpha$ be the (closure of the) image of $\CN_\alpha$ along the map
$$\CN_\alpha\hookrightarrow T^*(\CY_1)\to \CY_1.$$

We obtain a commutative diagram
$$
\CD
\CN'_\alpha @>>>  \CN_\alpha\times X \\
@VVV  @VVV  \\
D'_\alpha @>>> D_\alpha\times X.
\endCD
$$

In particular, we have an inclusion
\begin{equation} \label{e:divisors}
D'_\alpha \subseteq D_\alpha\times X.
\end{equation}

Thus, it suffices to show that $D_\alpha$ is a divisor.

\medskip

Suppose not. Then $D_\alpha$ would be all of $\CY_1$. Since
$\CN$ was assumed half-dimensional, we would obtain that $\CN_\alpha$ is the zero-section.
However, this would imply that $\CN'_\alpha$ is contained in the zero-section of $\CY_1\times X$.
However, this contradicts
the fact that the fibers of $\CN'_\alpha \to D'_\alpha$ are positive-dimensional.

\sssec{}

We will now explicitly exhibit $\oCF:=\CF|_{\oCY\times X}$ as lying in the essential image of
the functor \begin{equation} \label{e:lisse ten prod}
\Shv_{\on{lisse}}(\oCY)\otimes \Shv_{\on{lisse}}(X)\to \Shv_{\on{lisse}}(\oCY\times X).
\end{equation}

\medskip

Pick a point $y\in \oCY$ and let $\CF_X$ be the !-restriction of
$\oCF$ to $y\times X\subset \CY\times X$.
Let $\pi_\CY$ and $\pi_X$ denote the projections from $\CY\times X$ to $\CY$ and $X$;
let $\opi_\CY$ and $\opi_X$ denote their respective restrictions to $\oCY\times X$.

\medskip

Consider the object
$$\oCF_\CY:=(\opi_\CY)_*(\oCF\sotimes \opi{}^!_X(\BD(\CF_X)))\in \Shv(\oCY).$$

Since $X$ is proper, $\oCF_\CY$ is also lisse. 
The object $\oCF_\CY$ is acted on the left by the associative algebra
$$\CEnd(\CF^\vee_X)\simeq \CEnd(\CF_X)^{\on{rev}}.$$

By adjunction, we have a map 
$$\opi{}_\CY^!(\oCF_\CY)\otimes \opi{}_X^!(\CF_X)\simeq \oCF_Y\boxtimes \CF_X\simeq
\opi{}_\CY^*(\oCF_\CY)\otimes \opi{}_X^*(\CF_X)\to \oCF.$$

Moreover, this map factors via a map
\begin{equation} \label{e:map from ten prod}
\oCF_Y\underset{\CEnd(\CF_X)}\boxtimes \CF_X:=\opi{}_\CY^!(\oCF_\CY) \underset{\CEnd(\CF_X)}\otimes \opi{}_X^!(\CF_X)\to \oCF.
\end{equation}

\medskip

We claim that \eqref{e:map from ten prod} is an isomorphism. Indeed,
the !-fiber the map \eqref{e:map from ten prod} over $y$ identifies with
$$\on{C}^\bullet(X,\CF_X\sotimes \BD(\CF_X)) \underset{\CEnd(\CF_X)}\otimes  \CF_X\simeq
\CEnd(\CF_X) \underset{\CEnd(\CF_X)}\otimes  \CF_X\to \CF_X,$$
and hence is an isomorphism (in the above formula $\on{C}^\bullet(X,-)$ denotes the functor
of global sections, i.e., sheaf cohomology at the cochain level).

\medskip

Since both sides in \eqref{e:map from ten prod} are lisse sheaves and $\oCY\times X$ is connected, we obtain that
\eqref{e:map from ten prod} is an isomorphism.

\sssec{}

Set
$$\CF_\CY:=(\pi_\CY)_*(\CF\sotimes \pi^!_X(\BD(\CF_X)))\in \Shv(\CY).$$

Since singular support is preserved under direct images along proper maps, we have:
$$\CF_\CY\in \Shv_\CN(\CY).$$

Consider the object
$$\CF':=\CF_\CY \underset{\CEnd(\CF_X)}\boxtimes \CF_X\in \Shv_{\CN'}(\CY\times X):=
\pi_\CY^*(\CF_\CY) \underset{\CEnd(\CF_X)}\otimes \pi_X^*(\CF_X)\in \Shv_{\CN'}(\CY\times X).$$

By construction, it belongs to the essential image of the functor
$$\Shv_{\CN}(\CY)\otimes \Shv_{\on{lisse}}(X)\to \Shv_{\CN'}(\CY\times X).$$

Moreover, by adjunction, we have a map
$$\CF'\to \CF,$$
which becomes an isomorphism when restricted to $\oCY\times X$.
Passing to the fiber of this map we accomplish the induction step (in our Noetherian induction on $\CY$).

%
%
%
%
%
%
%
%

%

\section{Spectral action in the context of Geometric Langlands (after \cite{NY})} \label{s:NY}

In this section we will reprove a result from \cite{NY} that says that the subcategory of $\Shv(\Bun_G)$
consisting of objects whose singular support belongs to the nilpotent cone carries a canonical action of
$\Rep(\cG)^{\otimes X^{\on{top,sing}}}$, where $X^{\on{top,sing}}$ is the object of $\Spc$ corresponding to $X$.

\medskip

The proof we present will apply to any sheaf-theoretic context (see \thmref{t:NY}).

\ssec{The players}

\sssec{}

For the duration of this section we let $X$ be a smooth projective curve over $k$. Let $G$ be a reductive
group (over $k$).

\medskip

Let $\Bun_G$ denote the moduli stack of $G$-bundles on $X$.

\sssec{}

Recall that $T^*(\Bun_G)$ is the moduli space of pairs $(\CP,\xi)$, where $\CP$ is a $G$-bundle on $X$,
and $\xi$ is an element of $\Gamma(X,\fg^*_\CP\otimes \omega_X)$ where $\fg^*_\CP$ is the vector bundle
on $X$ associated to $\CP$ and the co-adjoint representation $\fg^*$ of $G$. 

\medskip 

Let
$$\on{Nilp}\subset T^*(\Bun_G)$$
be the nilpotent cone, defined to be the locus of pairs $(\CP,\xi)$, where $\xi$ is \emph{nilpotent}, i.e., maps to zero under the Chevalley map
$$\fg^*\to \fg^*/\!/\on{Ad}(G)\simeq \ft^*/\!/W.$$

\sssec{}

Let $\Shv(-)$ be any of the sheaf-theoretic contexts from \secref{sss:sheaves} in which the ring
$\sfe$ of coefficients is a field of characteristic $0$\footnote{The results of this section can be applied
to \emph{any} sheaf theoretic from \secref{sss:sheaves}, but one will need to modify the category $\Rep(\cG)$
appropriately.}. Our interest in this section is the category $\Shv(\Bun_G)$ and its full subcategory
$$\Shv_{\on{Nilp}}(\Bun_G)\subset \Shv(\Bun_G).$$

\medskip

Let $\cG$ be the Langlands dual of $G$, thought of as an algebraic group over $\sfe$. We let
$\Rep(\cG)$ denote the symmetric monoidal category of representations of $\cG$.
%
%
%

\ssec{The Hecke action}   \label{ss:Hecke}

In this subsection we will discuss the general formalism of Hecke action.

\medskip

\noindent {\it Convention}: for the duration of this subsection, when working in the sheaf-theoretic context (a), when we write
$\DGCat$ we actually mean $\DGCat^{\on{discont}}$, and when we write $\otimes$, we actually mean
$\overset{\on{discont}}\otimes$.

\sssec{}  \label{sss:abs action}

Let $A$ be an index category. Let
$$a\mapsto \CM_a \text{ and } a\mapsto \CC_a$$
be functors
$$\CM_A:A\to \DGCat \text{ and } \CC_A:A\to \DGCat^{\on{Mon}},$$
respectively.

\medskip

Then we can talk about an action of $\CC_A$ on $\CM_A$. Indeed, we can view $\CC_A$
as an associative algebra object in the category $\on{Funct}(A,\DGCat)$, equipped with the level-wise
(symmetric) monoidal structure.

\medskip

Let $\on{Act}(\CC_A,\CM_A)$
denote the space of such actions. In other words,
$$\on{Act}(\CC_A,\CM_A)=\Maps_{\BE_1(\on{Funct}(A,\DGCat))}(\CC_A,\ul\End(\CM_A)).$$

\medskip

Suppose now that we are given yet another functor
$$\CC'_A:A\to \DGCat^{\on{Mon}}, \quad a\mapsto \CC'_a,$$
equipped with a natural transformation
$$\CC'_A\to \CC_A.$$

Given an action of $\CC'_A$ on $\CM_A$ we can talk about an extension of this action to an action of $\CC_A$.
The space of such extensions is by definition
$$\on{Act}(\CC_A,\CM_A)\underset{\on{Act}(\CC'_A,\CM_A)}\times \{*\},$$
where $\{*\}\to \on{Act}(\CC'_A,\CM_A)$ is the initial action.

\sssec{}

We take $A:=\on{fSet}$. Take $\CC'_A$ to be the functor
\begin{equation} \label{e:Shv I functor}
I\mapsto \Shv(X^I),
\end{equation}
where $\Shv(X^I)$ is viewed as a (symmetric) monoidal category with respect to the $\sotimes$ operation.

\medskip

Take $\CC_A$ to be the functor
\begin{equation} \label{e:Rep functor}
I\mapsto \Rep(\cG)^{\otimes I}\otimes \Shv(X^I).
\end{equation}

\medskip

Take $\CM_A$ to be the functor
\begin{equation} \label{e:Shv BunG functor}
I\mapsto \Shv(\Bun_G\times X^I),
\end{equation}
equipped with a natural action of $\CC'_A$ (see \secref{sss:geom actions}).

\medskip

We will prove the following:

\begin{propconstr}  \label{p:Hecke}
There exists a canonical extension of the action of \eqref{e:Shv I functor} on \eqref{e:Shv BunG functor}
to an action of \eqref{e:Rep functor}.
\end{propconstr}

\begin{rem}
It would follow from the construction, that in the sheaf-theoretic context (a), the restriction of the action in \propref{p:Hecke}
along
$$\Rep(\cG)^{\otimes I}\otimes \Shv_{\on{lisse}}(X^I)\hookrightarrow \Rep(\cG)^{\otimes I}\otimes \Shv(X^I)$$
is given by \emph{continuous} functors, see Remark \ref{r:recover cont}.
\end{rem}

%
%
%
%
%
%
%
%
%
%
%

The rest of this subsection is devoted to the proof of \propref{p:Hecke}. It will be carried out using a certain
formalism explained below.

\sssec{}  \label{sss:geom actions}

Let $\on{PreStk}/\on{Sch}$ be the category of pairs
$$(\CY,Z),$$
where $Z$ is a scheme, $\CY$ is prestack over $Z$.

\medskip

Take $A=\on{PreStk}^{\on{op}}/\on{Sch}^{\on{op}}$. Take $\CM_A$ to be the functor
\begin{equation} \label{e:M geom}
(\CY,Z)\mapsto \Shv(\CY).
\end{equation}

Take $\CC'_A$ to be the functor
\begin{equation} \label{e:C' geom}
(\CY,Z)\mapsto \Shv(Z).
\end{equation}

Then we have a natural action of $\CC'_A$ on $\CM_A$, given by !-pullback along $\CY\to Z$.

\sssec{}  \label{sss:geom actions grpds}

Let $\on{Grpds}/\on{PreStk}/\on{Sch}$ be the category of triples
$$(\CH,\CY,Z),$$
where $(\CY,Z)$ are as above, and where $\CH\to \CY\underset{Z}\times \CY$ is a groupoid
acting on $\CY$ over $Z$, such that the projections $\CH\rightrightarrows \CY$ are ind-schematic.

\medskip
Note that $\Shv(\CH)$ is a monoidal category with respect to a convolution, defined to be the ind-extension 
of the $*$-convolution from compact objects.

\medskip

Take $A=\on{Grpds}^{\on{op}}/\on{PreStk}^{\on{op}}/\on{Sch}^{\on{op}}$.  We take $\CM_A$ to be the composition of the forgetful functor
\begin{equation} \label{e:forget groupoid}
\on{Grpds}/\on{PreStk}/\on{Sch}\to \on{PreStk}/\on{Sch}, \quad (\CH,\CY,Z)\mapsto (\CY,Z)
\end{equation}
and the functor \eqref{e:M geom}

\medskip

We take $\CC'_A$ to be the composition of the forgetful functor \eqref{e:forget groupoid} with \eqref{e:C' geom}.

\medskip

We take $\CC_A$ to be the functor that sends
$$(\CH,\CY,Z)\mapsto \Shv(\CH).$$

\medskip

Note that we have a canonical action of $\CC_A$ on $\CM_A$. We also have a canonical natural transformation
$$\CC'_A\to \CC_A,$$
given, level-wise, by the unit section of $\CH$.

\medskip

The induced action of $\CC'_A$ on $\CM_A$ is one from \secref{sss:geom actions}.

\begin{rem}
More formally, the above constructions should be spelled out as follows: we have the functors
$$(\CH,\CY,Z) \mapsto \CY,\,\, (\CH,\CY,Z) \mapsto Z,\,\, (\CH,\CY,Z) \mapsto \CH$$
that map $\on{Grpds}/\on{PreStk}/\on{Sch}$ to
$$\on{Corr}(\on{PreStk})_{\on{ind-sch,all}}, \,\, \on{ComAlg}(\on{Corr}(\on{PreStk})_{\on{ind-sch,all}}),\,\,
\on{AssocAlg}(\on{Corr}(\on{PreStk})_{\on{ind-sch,all}}),$$
respectively, and
$$(\CM_A,\CC'_A,\CC_A)$$
are obtained by composing these functors with the functor $\Shv_{\on{Corr}}$ of \secref{sss:corr}.
\end{rem}

\sssec{}

We have a canonically defined functor
\begin{equation} \label{e:Hecke groupoid}
\on{fSet}\to \on{Grpds}/\on{PreStk}/\on{Sch}, \quad I \mapsto \on{Hecke}_I/\Bun_G\times X^I/X^I,
\end{equation}
where $\on{Hecke}_I$ is the $I$-legged Hecke stack. (Note that $\on{Hecke}_I$ and $\Bun_G$ are \emph{ordinary}
prestacks, so the construction of the functor \eqref{e:Hecke groupoid} takes place in $(2,1)$-categories, i.e., involves
finitely many pieces of data.)

\medskip

Hence, in order to perform the construction in \propref{p:Hecke}, it suffices to construct a map (in the category $\on{Funct}(\on{fSet},\DGCat^{\on{Mon}})$)
from \eqref{e:Rep functor} to \begin{equation} \label{e:Hecke category}
I\mapsto \Shv(\on{Hecke}_I),
\end{equation}
extending the map from \eqref{e:Shv I functor}.

\sssec{}

Note, however, that in the context of \secref{sss:geom actions grpds}, the natural transformation
$$(\CH,\CY,Z)\mapsto  (\Shv(Z)\to \Shv(\CH))$$
as functors $\on{Grpds}/\on{PreStk}/\on{Sch}\to \DGCat^{\on{Mon}}$,
extends naturally to a natural transformation
$$(\CH,\CY,Z)\mapsto (\Shv(\CH)\otimes \Shv(Z)\to \Shv(\CH)),$$
expressing the fact that $\Shv(Z)$ maps to the center of $\Shv(\CH)$.

\medskip

Hence, it suffices to construct a map from the functor
\begin{equation} \label{e:Rep functor bis}
I\mapsto \Rep(\cG)^{\otimes I}.
\end{equation}
to \eqref{e:Hecke category}.

\sssec{}

The required natural transformation
\begin{equation} \label{e:geom Satake}
\Rep(\cG)^{\otimes I}\to \Shv(\on{Hecke}_I), \quad I\in \on{fSet}
\end{equation}
is given by what is known as the \emph{naive} geometric Satake functor. For completeness, we
will recall its construction.

\ssec{Digression: naive geometric Satake}

\sssec{}

Recall that $S$-points of the Hecke stack $\on{Hecke}_I$ are quadruples
$$(x^I,\CP',\CP'',\alpha),$$
where

\medskip

\noindent--$x^I=\{x^i,i\in I\}$ is an $I$-tuple of $S$-points of $X$;

\medskip

\noindent--$\CP'$ and $\CP''$ are $G$-bundles on $S\times X$;

\medskip

\noindent--$\alpha$ is an identification
between $\CP'$ and $\CP''$ on $S\times X-\underset{i}\cup\, \on{Graph}_{x^i}$.

\sssec{}  \label{sss:local Hecke}

We introduce the \emph{local} Hecke stack $\on{Hecke}^{\on{loc}}_I$ as follows. Its $S$-points are
quadruples
$$(x^I,\CP',\CP'',\alpha),$$
where:

\medskip

\noindent--$x^I$ is an $I$-tuple of $S$-points of $X$;

\medskip

\noindent--$\CP'$ and $\CP''$ are $G$-bundles on $\cD_{x^I}$--the parameterized formal disc
around $x^I$ (i.e., the completion of the graph of $S\times X$ along $\underset{i}\cup\, \on{Graph}_{x^i}$);

\medskip

\noindent--$\alpha$ is an identification
between $\CP'$ and $\CP''$ on $\ocD_{x^I}$--the parameterized formal punctured disc around $x^I$ (see \cite[Sect. 6.4.3]{Ga4}).

\sssec{}

Convolution defines on $\Shv(\on{Hecke}^{\on{loc}}_I)$ a structure of monoidal category, and the assignment
\begin{equation} \label{e:Hecke category loc}
I \mapsto \Shv(\on{Hecke}^{\on{loc}}_I)
\end{equation}
is a functor $\on{fSet}\to \DGCat^{\on{Mon}}$.

\medskip

Restriction along $\cD_{x^I}\to S\times X$ defines a map
$$\fr_I:\on{Hecke}_I\to \on{Hecke}^{\on{loc}}_I.$$

The functors $\fr_I^!$ give rise to a natural transformation from \eqref{e:Hecke category loc} to \eqref{e:Hecke category}.
Thus, it is sufficient to construct a natural transformation
\begin{equation} \label{e:geom Satake loc}
\Rep(\cG)^{\otimes I}\to \Shv(\on{Hecke}^{\on{loc}}_I), \quad I\in \on{fSet}.
\end{equation}

\begin{rem}
For the reader unwilling to consider the category $\Shv(\on{Hecke}^{\on{loc}}_I)$ ``as-is"
(for reasons that the prestack $\on{Hecke}^{\on{loc}}_I$ is not locally of finite type), it can
be equivalently defined as
$$\Shv(\Gr_{G,I})^{\fL^+(G)_I},$$
where

\medskip

\noindent--$\Gr_{G,I}$ is the $I$-legged version of the affine Grassmannian;

\medskip

\noindent--$\fL^+(G)_I$ is the group-scheme over $X^I$ of arcs into $G$.

\end{rem}

\sssec{}

Each of the categories $\Rep(\cG)^{\otimes I}$ is endowed with a t-structure, for which the monoidal structure is t-exact.
For a map $I_1\to I_2$ in $\on{fSet}$, the corresponding functor $\Rep(\cG)^{\otimes I_1}\to \Rep(\cG)^{\otimes I_2}$ is t-exact.

\medskip

Each of the categories $\Shv(\on{Hecke}^{\on{loc}}_I)$ is also endowed with a t-structure: this is the perverse t-structure shifted
so that the dualizaing sheaf on the unit section
$$X^I\to \on{Hecke}^{\on{loc}}_I$$
lies in the heart.

\medskip

The monoidal operation on $\Shv(\on{Hecke}^{\on{loc}}_I)$ is right t-exact, making $(\Shv(\on{Hecke}^{\on{loc}}_I))^\heartsuit$
into a monoidal abelian category. For a map $I_1\to I_2$ in $\on{fSet}$, the corresponding functor
$$\Shv(\on{Hecke}^{\on{loc}}_{I_1})\to \Shv(\on{Hecke}^{\on{loc}}_{I_2}),$$
which is given by !-pullback, is right t-exact.

\sssec{}

The starting point for the construction of the natural transformation \eqref{e:geom Satake loc}
is a natural transformation
\begin{equation} \label{e:I Satake}
(\Rep(\cG)^{\otimes I})^\heartsuit \simeq \Rep(\cG^I)^\heartsuit \to (\Shv(\on{Hecke}^{\on{loc}}_I))^\heartsuit,
\end{equation}
as functors $\on{fSet}\to \on{AbCat}^{\on{Mon}}$.

\medskip

The natural transformation \eqref{e:I Satake} is given by the geometric Satake
functor of \cite{MV}.

\begin{rem}
Although \cite{MV} was stated in the context of (a'), it applies in any (ind)-constructible context.
The construction in cases (a) and (b) follows from that in cases (a') and (b'), respectively.
\end{rem}

\sssec{}

Let us now explain how to use \eqref{e:I Satake} to obtain the desired natural transformation \eqref{e:geom Satake}.

\medskip

The first key property of \eqref{e:I Satake} is that for each $I\in \on{fSet}$, the corresponding functor is t-exact. Hence,
\eqref{e:I Satake} gives rise to a system of t-exact functors
\begin{equation} \label{e:I Satake D}
(\Rep(\cG)^{\otimes I})^b\simeq
D((\Rep(\cG)^{\otimes I})^\heartsuit)^b \to \Shv(\on{Hecke}^{\on{loc}}_I),
\end{equation}
equipped with a right-lax monoidal structure (here $D(-)^b$ stands for the \emph{bounded derived category of a given abelian category},
equipped with its natural DG structure, see \cite[Sect. 1.3.3]{Lu2}; we are using its universal property, which is a variant
of \cite[Theorem 1.3.3.2]{Lu2} for the bounded derived category).

\medskip

The second key property of \eqref{e:I Satake} is that the right-lax monoidal structure on \eqref{e:I Satake D} is strict.
Hence, precomposing with
$$(\Rep(\cG)^{\otimes I})^c\hookrightarrow (\Rep(\cG)^{\otimes I})^b$$
and ind-extending, we obtain the desired natural transformation \eqref{e:geom Satake loc}.

\ssec{Hecke action on the subcategory with nilpotent singular support}



\sssec{}

Let $\on{Nilp}\subset T^*(\Bun_G)$ be the nilpotent cone. Consider the full subcategory
$$\Shv_{\on{Nilp}}(\Bun_G)\subset \Shv(\Bun_G).$$

Our goal is to prove the following:

\begin{thmconstr}  \label{t:NY}
There exists a natural transformation between the following two functors
$$\on{fSet} \to \DGCat^{\on{Mon}}:$$
from the functor
$$I\mapsto \Rep(\cG)^{\otimes I}$$
to the functor
$$I\mapsto  \End(\Shv_{\on{Nilp}}(\Bun_G))\otimes \on{Shv}_{\on{lisse}}(X^I).$$
\end{thmconstr}

\sssec{}

Before we prove \thmref{t:NY}, let us explain how it recovers the result of \cite{NY}.

\medskip

Let us specialize to the sheaf-theoretic context (a) from \secref{sss:sheaves}.

\medskip

Recall (see \lemref{l:LS on top}) that in this case, we have a functorial identification
$$\on{Shv}_{\on{lisse}}(X^I)\simeq \LS((X^{\on{top,sing}})^I).$$

\medskip

Hence, combining with \secref{sss:monoidal actions}, we obtain:

\begin{cor}  \label{c:NY1}
We have a canonically defined action of $\Rep(\cG)^{\otimes X^{\on{top,sing}}}$ on $\Shv_{\on{Nilp}}(\Bun_G)$.
\end{cor}

Finally, combining with \thmref{t:integral to LocSys}, we obtain:

\begin{cor}  \label{c:NY2}
There is a canonically defined action of the (symmetric) monoidal category
$\QCoh(\LocSys_\cG(X^{\on{top,sing}}))$ on $\Shv_{\on{Nilp}}(\Bun_G)$.
\end{cor}

\sssec{}

Note that by \lemref{l:LS on top}, the stack $\LocSys_\cG(X^{\on{top,sing}})$ that appears in \corref{c:NY2} identifies canonically with the Betti
version of the stack of $\cG$-local systems on $X$.

\ssec{Proof of \thmref{t:NY}}

\sssec{}

Consider the full subcategory
$$\Shv_{\on{Nilp}'}(\Bun_G\times X^I)\subset \Shv(\Bun_G\times X^I),$$
where
$$\on{Nilp}':=\on{Nilp}\times \{\text{zero-section}\}\subset T^*(\Bun_G\times X^I).$$

\medskip


The crucial ingredient for the proof is the following
geometric assertion:

\begin{thm}[\cite{NY}] \label{t:preserve nilp}
For every $I\in \on{fSet}$, the action of $\Rep(\cG)^{\otimes I}$ on $\Shv(\Bun_G\times X^I)$ preserves
the full subcategory
$$\Shv_{\on{Nilp}'}(\Bun_G\times X^I)\subset \Shv(\Bun_G\times X^I).$$
\end{thm}

We will prove \thmref{t:preserve nilp} in \secref{ss:preserve nilp} below.

\medskip

Let us show how \thmref{t:preserve nilp} leads to the construction of the natural transformation in \thmref{t:NY}.

\sssec{}  \label{sss:abs action restricted}

Let us return to the situation of \secref{sss:abs action}. Let us be given an action of $\CC_A$ on $\CM_A$, and let
$$\CM'_A:A\to \DGCat, \quad a\mapsto \CM'_a$$
be another functor. Let $\CM'_A$ be equipped with a natural transformation to $\CM_A$, such that for every
$a\in A$, the corresponding functor
$$\CM'_a\to \CM_a$$
is fully faithful.

\medskip

Assume being given an action of $\CC_A$ on $\CM_A$ so that for every $a\in A$, the action of $\CC_a$ on $\CM_a$
preserves the subcategory $\CM'_a$. In this case we obtain an action of $\CC_A$ on $\CM'_A$.

\sssec{}

Let us take $A=\on{fSet}$, and $\CC_A$ be the functor
\begin{equation} \label{e:Rep functor bis bis}
I\mapsto \Rep(\cG)^{\otimes I}\otimes \Shv_{\on{lisse}}(X^I).
\end{equation}

We take $\CM_A$ to be the functor \eqref{e:Shv BunG functor}. We let $\CC_A$ act on
$\CM_A$ by precomposing the action of \propref{p:Hecke} with the natural transformation
$$\Rep(\cG)^{\otimes I}\otimes \Shv_{\on{lisse}}(X^I)\to \Rep(\cG)^{\otimes I}\otimes \Shv(X^I), \quad I\in \on{fSet}.$$

\medskip

We take $\CM'_A$ to be the family of full subcategories
$$\Shv_{\on{Nilp}'}(\Bun_G\times X^I)\subset \Shv(\Bun_G\times X^I),\quad I\in \on{fSet}.$$

Thus, by \secref{sss:abs action restricted} and \thmref{t:preserve nilp} we obtain an action of
\eqref{e:Rep functor bis bis} on
\begin{equation} \label{e:Shv BunG functor nilp}
I\mapsto \Shv_{\on{Nilp}'}(\Bun_G\times X^I).
\end{equation}

\sssec{}  \label{sss:act decomp}

Let us be again in the situation of \secref{sss:abs action}, and let us be given an action of $\CC_A$ on $\CM_A$.
Assume now that $\CC_A$ arises as a tensor product
$$\CC^1_A\otimes \CC^2_A,$$
Assume also that $\CM_A$ arises as as
$$\CM^1\otimes \CC^2_A,$$
for some fixed DG category $\CM^1$.

\medskip

Consider the action of $\CC^2_A$ on $\CM$ given by right multiplication along the second factor.
Then the datum of extension of this action of $\CC^2_A$ to an action of $\CC_A$
is equivalent to that of a natural transformation
$$\CC^1_A\to \End(\CM^1)\otimes (\CC^2_A)^{\on{rev}}$$
of functors $A\to \DGCat^{\on{Mon}}$.

\sssec{}

We apply this to $\CC^1_A$ being the functor \eqref{e:Rep functor bis} and $\CC^2_A$ being the functor
$$I\mapsto \Shv_{\on{lisse}}(X^I).$$

Take $\CM^1=\Shv_{\on{Nilp}}(\Bun_G)$. Now, by \thmref{t:vert sing supp}, the natural transformation
$$\Shv_{\on{Nilp}}(\Bun_G)\otimes \Shv_{\on{lisse}}(X^I)\to \Shv_{\on{Nilp}'}(\Bun_G\times X^I)$$
is an isomorphism.

\medskip

Hence, by \secref{sss:act decomp}, the action of \eqref{e:Rep functor bis bis} on \eqref{e:Shv BunG functor nilp}
gives rise to a sought-for natural transformation
$$\Rep(\cG)^{\otimes I}\to  \End(\Shv_{\on{Nilp}}(\Bun_G))\otimes \on{Shv}_{\on{lisse}}(X^I), \quad I\in \on{fSet}.$$

\qed[\thmref{t:NY}]

\ssec{Preservation of nilpotence of singular support} \label{ss:preserve nilp}

The goal of this subsection is to prove \thmref{t:preserve nilp}. The proof is essentially a paraphrase of \cite{NY}.

\sssec{}

By \thmref{t:vert sing supp}, we have to show that for $\CF\in \Shv_{\on{Nilp}}(\Bun_G)$ and any
$V\in \Rep(\cG)^{\otimes I}$, we have
\begin{equation} \label{e:sing supp belongs}
\on{H}(V,\CF)\in 
\Shv_{\on{Nilp}'}(\Bun_G\times X^I)\subset \Shv(\Bun_G\times X^I),
\end{equation}
where $H(-,-)$ denotes the Hecke action
$$\Rep(\cG)^{\otimes I}\otimes \Shv(\Bun_G)\to \Shv(\Bun_G\times X^I).$$

\medskip

By the associativity of the Hecke action, we can assume that $I=\{*\}$.

\sssec{}


Let $\on{Hecke}_X$ denote the 1-legged \emph{global} Hecke stack.
Let $\on{Hecke}^{\on{loc}}_X$ denote the 1-legged \emph{local} Hecke stack, defined in \secref{sss:local Hecke}.
Let $\fr$ denote the restriction map
$$\on{Hecke}_X\to \on{Hecke}^{\on{loc}}_X.$$

\medskip

We have the diagram
\begin{equation} \label{e:Hecke diagram}
\CD
& & X \\
& & @AA{\pi}A  \\
& & \on{Hecke}^{\on{loc}}_X \\
& & @AA{\fr}A  \\
\Bun_G  @<{\hl}<<  \on{Hecke}_X  @>{\hr}>>  \Bun_G.
\endCD
\end{equation}

\medskip

Denote also
$$\fs:=\pi\circ \fr:\on{Hecke}_X\to X.$$

\medskip

With these notations, we have:
$$\on{H}(V,\CF)=(\hr\times \fs)_* \circ (\hl\times \fr)^!(\CF\boxtimes \CS_V),$$
where $\CS_V\in \Shv(\on{Hecke}^{\on{loc}}_X)$ corresponds to $V\in \Rep(\cG)$
by geometric Satake.

\sssec{}

Note that the map
$$\hl\times \fs:  \on{Hecke}_X\to \Bun_G\times  \on{Hecke}^{\on{loc}}_X$$ is (pro)-smooth,
while the map
$$\hr\times \fs: \on{Hecke}_X\to \Bun_G\times X$$
is (ind)-proper.

\medskip

For a given point $(x,\CP',\CP'',\alpha)\in \on{Hecke}_X$ consider the corresponding diagram of cotangent spaces
\begin{equation} \label{e:Hecke diagram cotan}
\CD
& & T^*_x(X) \\
& & @VV{(d\pi)^*}V \\
& & T^*_{(x,\CP',\CP'',\alpha)|_{\cD_x}}(\on{Hecke}^{\on{loc}}_X)  \\
& & @VV{(d\fr)^*}V  \\
T^*_{\CP'}(\Bun_G)  @>{(d\hl)^*}>>  T^*_{(x,\CP',\CP'',\alpha)}(\on{Hecke}_X)  @<{(d\hr)^*}<< T^*_{\CP''}(\Bun_G).
\endCD
\end{equation}

Hence, in order to prove \eqref{e:sing supp belongs}, it suffices to show the following.
Consider a quadruple of elements
$$\xi'\in T^*_{\CP'}(\Bun_G),\,\, \xi''\in T^*_{\CP''}(\Bun_G),\,\, \xi_H\in T^*_{(x,\CP',\CP'',\alpha)|_{\cD_x}}(\on{Hecke}^{\on{loc}}_X),\,\, \xi_X\in T^*_x(X),$$
so that

\begin{itemize}

\item $\xi'$ is nilpotent;

\medskip

\item $\xi_H\in \on{SingSupp}(\CS_V)$;

\medskip

\item $(d\hl)^*(\xi')+(d\fr)^*(\xi_H)=(d\hr)^*(\xi'')+(d\fs)^*(\xi_X)$.

\end{itemize}

\medskip

We need to prove that in this case:

\begin{itemize}

\item $\xi''$ is also nilpotent;

\medskip

\item $\xi_X=0$.

\end{itemize}

Indeed, once we show this, \eqref{e:sing supp belongs} would follow by combining the following two assertions:

\medskip

\noindent--For a smooth map $f:\CY_1\to \CY_2$ and $\CF\in \Shv(\CY_2)$, the subset
$\on{SingSupp}(f^!(\CF))\subset T^*(\CY_1)$ equals the image of
$$\on{SingSupp}(\CF)\underset{\CY_2}\times \CY_1\subset T^*(\CY_2)\underset{\CY_2}\times \CY_1$$
along the codifferential map
\begin{equation} \label{e:codiff glob}
T^*(\CY_2)\underset{\CY_2}\times \CY_1\to T^*(\CY_1).
\end{equation}

\medskip

\noindent--For a proper map $f:\CY_1\to \CY_2$ and $\CF\in \Shv(\CY_1)$, the subset
$\on{SingSupp}(f_*(\CF))\subset T^*(\CY_2)$ is contained in the image along the projection
$$T^*(\CY_2)\underset{\CY_2}\times \CY_1\to T^*(\CY_2)$$
of the preimage of $\on{SingSupp}(\CF)$ along \eqref{e:codiff glob}.

\sssec{}

Let $\on{Hecke}_x$ (resp., $\on{Hecke}^{\on{loc}}_x$) be the fiber of $\on{Hecke}_X$ (resp., $\on{Hecke}^{\on{loc}}_X$) over $x$.
Along with the diagrams \eqref{e:Hecke diagram} and \eqref{e:Hecke diagram cotan} consider the diagrams
$$
\CD
& & \on{Hecke}^{\on{loc}}_x \\
& & @AA{\fr_x}A  \\
\Bun_G  @<{\hl_x}<<  \on{Hecke}_x  @>{\hr_x}>>  \Bun_G
\endCD
$$
and
$$
\CD
& & T^*_{(\CP',\CP'',\alpha)|_{\cD_x}}(\on{Hecke}^{\on{loc}}_x)  \\
& & @VV{(d\fr_x)^*}V  \\
T^*_{\CP'}(\Bun_G)  @>{(d\hl_x)^*}>>  T^*_{(\CP',\CP'',\alpha)}(\on{Hecke}_x)  @<{(d\hr_x)^*}<< T^*_{\CP''}(\Bun_G).
\endCD
$$

\medskip

Let $\xi_{H_x}$ denote the image of $\xi_H$ under the restriction map
$$T^*_{(x,\CP',\CP'',\alpha)}(\on{Hecke}^{\on{loc}}_X)\to T^*_{(\CP',\CP'',\alpha)}(\on{Hecke}^{\on{loc}}_x).$$

\medskip

We have a commutative diagram
$$
\CD
T^*_{(\CP',\CP'',\alpha)|_{\cD_x}}(\on{Hecke}^{\on{loc}}_x) @>{(d\fr_x)^*}>> T^*_{(\CP',\CP'',\alpha)}(\on{Hecke}_x)  \\
@AAA   @AAA  \\
T^*_{(x,\CP',\CP'',\alpha)|_{\cD_x}}(\on{Hecke}^{\on{loc}}_X)   @>{(d\fr)^*}>>  T^*_{(x,\CP',\CP'',\alpha)}(\on{Hecke}_X) .
\endCD
$$

\medskip

Hence, the assumption on $(\xi',\xi'',\xi_H)$ implies
\begin{equation} \label{e:rel cot x}
(d\hl_x)^*(\xi')+(d\fr_x)^*(\xi_{H_x})=(d\hr_x)^*(\xi'').
\end{equation}

\sssec{}  \label{sss:identify cotan Hecke}

We identify
$$T^*_{\CP'}(\Bun_G) :=\Gamma(X,\fg^*_{\CP'}\otimes \omega_X) \text{ and } T^*_{\CP''}(\Bun_G) :=\Gamma(X,\fg^*_{\CP''}\otimes \omega_X).$$

\medskip

Further, we identify $T^*_{(\CP',\CP'',\alpha)|_{\cD_x}}(\on{Hecke}_x)$ with the dual of $R^1\Gamma(X,\sK)$
where $\sK$ is the complex
\begin{equation} \label{e:sK}
\sK:=\on{Cone}\left(\fg_{\CP'} \oplus \fg_{\CP''} \to j_*(j^*(\fg_{\CP'})\overset{\alpha}\simeq j^*(\fg_{\CP''}))\right),
\end{equation}
where $j$ denotes the embedding $X-x\hookrightarrow X$.

\medskip

Finally, we identify $T^*_{(\CP',\CP'',\alpha)|_{\cD_x}}(\on{Hecke}^{\on{loc}}_x)$ with the dual of $R^1\Gamma(\cD_x,\sK)$
where $\cD_x$ (resp., $\ocD_x$) is the formal (resp., formal punctured disc) around $x$ (see \secref{sss:local Hecke}).

\medskip

Since $\cD_x$ is affine, the above $R^1\Gamma(\cD_x,\sK)$ identifies with
$$\Gamma(\cD_x,H^1(\sK)),$$
which is the same as $\Gamma(X,H^1(\sK))$, since $H^1(\sK)$ is set-theoretically supported at $x\in X$.

\medskip

From here we obtain that $T^*_{(\CP',\CP'',\alpha)|_{\cD_x}}(\on{Hecke}^{\on{loc}}_x)$ identifies with the
set of pairs
\begin{equation} \label{e:identify cotan Hecke}
\left(\xi'_{\on{loc}}\in \Gamma(\cD_x,\fg^*_{\CP'}\otimes \omega_X),\,\,\xi''_{\on{loc}}\in \Gamma(\cD_x,\fg^*_{\CP''}\otimes \omega_X)\right),
\end{equation}
such that
$$\alpha(\xi'_{\on{loc}})=\xi''_{\on{loc}}$$ as elements of
$\Gamma(\ocD_x,\fg^*_{\CP''}\otimes \omega_X)$.

\sssec{}

Equation \eqref{e:rel cot x} translates into
$$\alpha(\xi'|_{\ocD_x})=\xi''|_{\ocD_x}.$$

In particular, we obtain:

\medskip

\hskip1cm $\xi'$ is nilpotent $\Rightarrow$ $\xi'|_{\ocD_x}$ is nilpotent $\Rightarrow$ $\xi''|_{\ocD_x}$ is nilpotent $\Rightarrow$ $\xi''$ is nilpotent.

\medskip

Thus, it remains to show that $\xi_X=0$.

\sssec{}

Note that the short exact sequences
\begin{equation} \label{e:SES cotan glob}
0\to T^*_x(X) \to T^*_{(x,\CP',\CP'',\alpha)}(\on{Hecke}_X)  \to T^*_{(\CP',\CP'',\alpha)}(\on{Hecke}) \to 0
\end{equation}
\begin{equation} \label{e:SES cotan}
0\to T^*_x(X) \to T^*_{(x,\CP',\CP'',\alpha)|_{\cD_x}}(\on{Hecke}^{\on{loc}}_X)  \to T^*_{(\CP',\CP'',\alpha)|_{\cD_x}}(\on{Hecke}^{\on{loc}}_x) \to 0
\end{equation}
admit a \emph{canonical} splittings (compatible with the map between these short exact sequences induced by $\fr$):

\medskip

These splittings are a consequence of the fact that the prestacks
$$\on{Hecke}_X \text{ and } \on{Hecke}^{\on{loc}}_X$$
each carries a canonical structure of \emph{crystal} along $X$. Indeed, for a test-scheme $S$ and two
points
$$x_1,x_2:S\rightrightarrows X$$
such that that
\begin{equation} \label{e:points inf}
x_1|_{S_{\on{red}}}=x_2|_{S_{\on{red}}},
\end{equation}
the data of lifting of these points to points of
$\on{Hecke}_X$ (resp.,  $\on{Hecke}^{\on{loc}}_X$) coincide. This is because \eqref{e:points inf} implies
$$S\times X-\on{Graph}_{x_1}=S\times X-\on{Graph}_{x_2} \text{ and } \cD_{x_1}=\cD_{x_2},\,\, \ocD_{x_1}=\ocD_{x_2}.$$

\sssec{}

Consider the resulting maps
\begin{equation} \label{e:split cotan glob X}
T^*_{(x,\CP',\CP'',\alpha)}(\on{Hecke}_X)  \to T^*_x(X)
\end{equation}
and
\begin{equation} \label{e:split cotan X}
T^*_{(x,\CP',\CP'',\alpha)|_{\cD_x}}(\on{Hecke}^{\on{loc}}_X)  \to T^*_x(X).
\end{equation}

It follows from the definition of the above crystal structure that the composition of \eqref{e:split cotan glob X} with the maps
$$T^*_{\CP'}(\Bun_G) \overset{(d\hl)^*}\longrightarrow T^*_{(x,\CP',\CP'',\alpha)}(\on{Hecke}_X)
\overset{(d\hr)^*}\longleftarrow T^*_{\CP''}(\Bun_G)$$
vanishes.

\medskip

Therefore, to prove that $\xi_X=0$, it suffices to show the following.
Let $\xi_H\in T^*_{(x,\CP',\CP'',\alpha)|_{\cD_x}}(\on{Hecke}^{\on{loc}}_X)$
be a vector that belongs to $\on{SingSupp}(\CS_V)$. Suppose that
$$\xi_{H_x}\in T^*_{(\CP',\CP'',\alpha)|_{\cD_x}}(\on{Hecke}^{\on{loc}}_x)$$
is \emph{nilpotent}, i.e., the corresponding pair $(\xi'_{\on{loc}},\xi''_{\on{loc}})$ in \eqref{e:identify cotan Hecke} consists of nilpotent elements.

\medskip

Then we need to prove that the image of $\xi_H$ along \eqref{e:split cotan X} is zero.

\sssec{}

To prove the latter assertion, we can assume that $X=\BA^1$ and $x=0$.  In this case, we identify
\begin{equation} \label{e:Hecke prod}
\on{Hecke}^{\on{loc}}_{\BA^1}\simeq \on{Hecke}^{\on{loc}}_0\times \BA^1.
\end{equation}

With respect to \eqref{e:Hecke prod}, the
object $\CS_V$ is the pullback along the projection
$$\on{Hecke}^{\on{loc}}_{\BA^1}\to \on{Hecke}^{\on{loc}}_0.$$

\medskip

In addition, from \eqref{e:Hecke prod}, we obtain a \emph{different} splitting of \eqref{e:SES cotan}.
The discrepancy between the two different splittings of \eqref{e:SES cotan} is a map
\begin{equation} \label{e:discr}
T^*_{(\CP',\CP'',\alpha)}(\on{Hecke}^{\on{loc}}_0) \to T^*_0(\BA^1)\simeq k.
\end{equation}
It suffices to show that this map vanishes on nilpotent elements (we refer to the description of
$T^*_{(\CP',\CP'',\alpha)}(\on{Hecke}^{\on{loc}}_x)$ given by \eqref{e:identify cotan Hecke}; when we say ``nilpotent",
we mean that the corresponding elements $\xi'_{\on{loc}}$ and $\xi''_{\on{loc}}$ are nilpotent).

\sssec{}

Let us calculate the map \eqref{e:discr}. Let $\fL(G)_X$ be the version of the loop group spread over $X$,
and let $\fL(G)_x$ be its fiber at $x\in X$. The group ind-scheme $\fL(G)_X$ also has a crystal structure along $X$,
which gives rise to a splitting of the short exact sequence
\begin{equation} \label{e:SES tan loops}
0\to T_g(\fL(G)_x) \to T_g(\fL(G)_X) \to T_x(X)\to 0, \quad g\in \fL(G)_x.
\end{equation}

For $X=\BA^1$, we have an identification
$$\fL(G)_{\BA^1}\simeq \fL(G)_0\times \BA^1,$$
and hence a \emph{different} splitting of the short exact sequence \eqref{e:SES tan loops}. The discrepancy between these two splittings
is a map
\begin{equation} \label{e:SES tan loops 1}
k\simeq T_0(\BA^1)\to T_g(\fL(G)_0).
\end{equation}

\medskip

Let $u$ denote the coordinate on $\BA^1$. We will identify the tangent space $T_g(\fL(G)_0)$
with $\fg\ppart$ by means of the left translation by $g$. Under this identification the map \eqref{e:SES tan loops 1},
thought of as an element of $\fg\ppart$ is
\begin{equation} \label{e:loga}
g^{-1}\cdot \frac{dg}{du},
\end{equation}
i.e., the (left) logarithmic derivative of $g$ with respect to the coordinate $u$.

\sssec{}

Let $(\CP',\CP'',\alpha)$ be a point of
$$\on{Hecke}^{\on{loc}}_0\simeq \fL(G)_0\backslash \fL(G)/\fL(G)_0,$$
represented by an element $g\in \fL(G)_0$. By the Cartan decomposition, we can assume that
$g=u^\lambda$ for a dominant coweight $\lambda$.

\medskip

Note that the corresponding element \eqref{e:loga} equals
$$\lambda\cdot u^{-1}\subset \ft\ppart\subset \fg\ppart.$$

\medskip

We need to show that for a \emph{nilpotent} element
$$\xi\in T^*_{(x,\CP',\CP'',\alpha)}(\on{Hecke}^{\on{loc}}_0)\simeq \fg^*\qqart du \cap \on{Ad}_{u^\lambda}(\fg^*\qqart du)\subset
\fg^*\ppart du,$$
the residue pairing
$$\fg\ppart \times \fg^*\ppart du\to k$$
against $\lambda\cdot u^{-1}$ evaluates to $0$.

\medskip

Set
$$\xi_0:=\xi\, \on{mod}\, u\in \fg^*du\simeq \fg^*.$$
This is a nilpotent element of $\fg^*$. We need to show that its pairing with $\lambda\in \ft\subset \fg$ gives zero.

\sssec{}

Let
$$\fg=\fn^+_\lambda \oplus \fm_\lambda \oplus \fn^-_\lambda$$
be the triangular decomposition of $\fg$ corresponding to $\lambda$, i.e., $\fm_\lambda$ is the centralizer of $\lambda$ and
$\fn^+_\lambda$ (resp., $\fn^-_\lambda$) corresponds to positive (resp., negative) eigenvalues of $\lambda$ for the adjoint
action.

\medskip

Consider the corresponding decomposition
$$\xi_0=\xi_0^++\xi_0^0+\xi_0^-$$
of $\xi_0$.

\medskip

Note, however, that the condition that $\xi\in \on{Ad}_{u^\lambda}(\fg^*\qqart du)$ implies that $\xi_0^-$ is zero.
Knowing that, the fact that $\xi_0$ is nilpotent implies that $\xi^0_0$ is a nilpotent element of $\fm_\lambda$.

\medskip

We have
$$\langle \lambda,\xi_0\rangle=\langle \lambda,\xi^+_0\rangle+\langle \lambda,\xi^0_0\rangle=
\langle \lambda,\xi^0_0\rangle.$$

However, the latter is zero, being the pairing of a central element (in the reductive Lie algebra $\fm_\lambda$) with a nilpotent one.

\qed

\section{Integrated actions in the context of D-modules}  \label{s:Dmod}

In \propref{p:describe functors} we described what it takes to have an action of
$\CA^{\otimes Y}$ on a DG category $\CM$ for $Y\in \Spc$. In this section we will
use the RHS of \propref{p:describe functors} to give a definition of \emph{integrated action}
in the context of D-modules.

\medskip

We will then study how our definition plays out in the context of geometric Langlands.

\ssec{Definition of integrated action}

\sssec{}

Let us place ourselves in the sheaf-theoretic context (b) from \secref{sss:sheaves}. In this case our ring of
coefficients $\sfe$ is the same as the ground field $k$. Thus, we will work with $k$-linear DG categories.

\sssec{}

Let $X$ be a scheme over $k$ (in the applications to geometric Langlands, $X$ will
be a smooth projective curve).

\medskip

For a symmetric monoidal DG category $\CA$ and a DG category $\CM$, an action of
$\CA^{\otimes X}$ on $\CM$ is by definition a natural transformation between the following two functors
$\on{fSet}\to \DGCat^{\on{mon}}$:

\medskip

From the functor
\begin{equation} \label{e:functor one Dmod}
I\mapsto \CA^{\otimes I},
\end{equation}
to the functor
\begin{equation} \label{e:functor two Dmod}
I\mapsto \End(\CM) \otimes \Dmod(X^I).
\end{equation}

The totality of categories equipped with an action of $\CA^{\otimes X}$ forms a 2-category, which we will
denote
$$\CA^{\otimes X}\mmod.$$

\begin{rem}
It is easy to see that the datum of action of $\CA^{\otimes X}$ on $\CM$ is equivalent to that of action on $\CM$
of the (non-unital) symmetric monoidal category $\on{Fact}(\CA)_{\Ran(X)}$, defined as in \cite[Sect. 2.5]{Ga3}.
\end{rem}

\sssec{Example}  \label{sss:geom Langlands}

Let $X$ be a curve and let $G$ be a reductive group.

\medskip

Take $\CA=\Rep(\cG)$ and $\CM=\Dmod(\Bun_G)$.
Then \propref{p:Hecke} says that we have an action of $\Rep(\cG)^{\otimes X}$ on $\Dmod(\Bun_G)$.

\sssec{}

Let us be given an action of $\CA^{\otimes X}$ on $\CM$.
For a finite set $I$ and $r\in \CA^{\otimes I}$,
let $\CS_I(r)$ denote the resulting functor
$$\CM\to \CM\otimes \Dmod(X^I).$$

\begin{lem} \label{l:Hecke prese comp}
Let $r$ be dualizable. Then the functor $\CS_I(r)$ admits a continuous right adjoint.
\end{lem}

\begin{proof}
The adjoint in question is explicitly given by
\begin{multline*}
\CM\otimes \Dmod(X^I) \overset{\CS_I(r^\vee)\otimes \on{Id}_{\Dmod(X^I)}}\longrightarrow
\CM\otimes \Dmod(X^I)\otimes \Dmod(X^I) \overset{\on{Id}_\CM\otimes \Delta^!}\longrightarrow  \\
\to \CM\otimes \Dmod(X^I) \overset{\on{Id}_\CM\otimes \Gamma_{\on{dR}}(X^I,-)}\longrightarrow \CM.
\end{multline*}
\end{proof}

\sssec{The ULA property}

Let us assume that $\CA$ is compactly generated and rigid, and
that $\CM$ is compactly generated.

\medskip

For a pair of \emph{compact} objects $r\in \CA^{\otimes I}$ and $m\in \CM$ consider again the object
$$\CS_I(r)(m)\in \CM\otimes \Dmod(X^I)$$

Note that by \lemref{l:Hecke prese comp}, this object is automatically compact. We claim:

\begin{prop}
The object $\CS_I(r)(m)$ is ULA over $X^I$.
\end{prop}

We refer the reader to \secref{s:ULA}, where the notion of ULA is reviewed.

\begin{proof}

Note that the dual of $\CM\otimes \Dmod(X^I)$ as a $\Dmod(X^I)$-module category
identifies with $$\CM^\vee\otimes \Dmod(X^I).$$

\medskip

Let $m^\vee\in \CM^\vee$ be the abstract Verdier dual of $m\in \CM$. Let $r^\vee$
be the monoidal dual of $r$. Then the
object
$$\CS_I(r^\vee)(m^\vee)\in \CM^\vee\otimes \Dmod(X^I)$$
satisfies the requirements from \secref{sss:abs ULA}.

\end{proof}

\ssec{Lisse actions}


\sssec{}  \label{sss:lisse rem}

Note that for a smooth scheme $Y$, the functor
\begin{equation} \label{e:embed lisse}
\Dmod_{\on{lisse}}(Y)\to \Dmod(Y)
\end{equation}
is a fully faithful embedding that admits a continuous right adjoint.

\medskip

Hence, for another DG category $\CM$, the functor
$$\CM\otimes \Dmod_{\on{lisse}}(Y)\to \CM\otimes \Dmod(Y)$$
is also fully faithful and admits a continuous right adjoint given by tensoring the right adjoint to
\eqref{e:embed lisse} with $\on{Id}_\CM$.

\medskip

Further, if $\CM\to \CM'$ is conservative, an object of $\CM\otimes \Dmod(Y)$
belongs to $\CM\otimes \Dmod_{\on{lisse}}(Y)$ if and only if its image in
$\CM'\otimes \Dmod(Y)$ belongs to $\CM'\otimes \Dmod_{\on{lisse}}(Y)$.

\sssec{}

For a pair of objects $r\in \CA^{\otimes I}$ and $m\in \CM$ consider the object
\begin{equation} \label{e:act Dmod}
\CS_I(r)(m)\in \CM\otimes \Dmod(X^I).
\end{equation}

\medskip

We shall say that the action of $\CA^{\otimes X}$ on $\CM$ is \emph{lisse} if for every $r$ and $m$ as above, the object
\eqref{e:act Dmod} belongs to the full subcategory.
$$\CM\otimes \Dmod_{\on{lisse}}(X^I)\subset \CM\otimes \Dmod(X^I).$$

Note also that by associativity,
it is enough to check this condition for $I=\{*\}$.

\medskip

Let
$$\CA^{\otimes X}\mmod^{\on{lisse}}\subset \CA^{\otimes X}\mmod$$
denote the corresponding full subcategory.

\sssec{Example}

The statement of \thmref{t:preserve nilp} is equivalent to the assertion that the action of
$\Rep(\cG)^{\otimes X}$ on $\Dmod_{\on{Nilp}}(\Bun_G)$ is lisse.

\sssec{}  \label{sss:defn lisse}

Given an action of $\CA^{\otimes X}$ on $\CM$, let
$$\CM^{\on{lisse}}\subset \CM$$
be the full subcategory consisting of objects $m$, for which
\eqref{e:act Dmod} with $I=\{*\}$ belongs to the subcategory
$\CM\otimes \Dmod_{\on{lisse}}(X)$ for all 
$r$.

\begin{prop}  \label{p:lisse pres}
For an object $m\in \CM^{\on{lisse}}$, the object \eqref{e:act Dmod} with $I=\{*\}$ belongs to
$$\CM^{\on{lisse}}\otimes \Dmod_{\on{lisse}}(X)\subset \CM\otimes \Dmod_{\on{lisse}}(X).$$
\end{prop}

\begin{proof}

By the last remark in \secref{sss:lisse rem}, it suffices to show that
the functors $\CS_{\{*\}}(r)$ send $\CM^{\on{lisse}}$ to $\CM^{\on{lisse}}\otimes \Dmod(X)$.

\medskip

For any DG category $\CC$, consider the full subcategory
$$(\CM\otimes \CC)^{\on{lisse}}\subset \CM\otimes \CC.$$
I.e., we apply the definition of $(-)^{\on{lisse}}$ to $\CM\otimes \CC$ on which $\CA^{\otimes X}$ acts
via the first factor.

\medskip

Explicitly, this subcategory consists of objects $m'$ for which
$$(\CS_{\{*\}}(r')\otimes \on{Id}_\CC)(m')\in \CM\otimes \Dmod_{\on{lisse}}(X)\otimes \CC\subset
\CM\otimes \Dmod(X)\otimes \CC \text{ for all } r'\in \CA.$$

The functor
$$\CM^{\on{lisse}}\otimes \CC\to \CM\otimes \CC$$ factors via a functor
\begin{equation} \label{e:incl lisse}
\CM^{\on{lisse}}\otimes \CC\to (\CM\otimes \CC)^{\on{lisse}}.
\end{equation}

We claim that if $\CC$ is dualizable, then \eqref{e:incl lisse} is an equivalence. Indeed, this follows
by interpreting $\CM\otimes \CC$ as
$$\on{Funct}_{\on{cont}}(\CC^\vee,\CM).$$

Hence, applying this to $\CC=\Dmod(X)$, it suffices to show that for any $r$ and $r'$, we have
$$(\CS_{\{*\}}(r')\otimes \on{Id}_\CC)\circ \CS_{\{*\}}(r)(m)\in \CM\otimes \Dmod_{\on{lisse}}(X)\otimes \Dmod(X).$$

We have
$$(\CS_{\{*\}}(r')\otimes \on{Id}_\CC)\circ \CS_{\{*\}}(r)(m)=
\CS_{\{*\}\sqcup \{*\}}(r'\otimes r)(m)\simeq \sigma(\CS_{\{*\}\sqcup \{*\}}(r\otimes r')(m)),$$
where $\sigma$ denotes the transposition acting on $X\times X$.

\medskip

Hence, it suffices to show that
$$\CS_{\{*\}\sqcup \{*\}}(r\otimes r')(m)\subset \CM\otimes \Dmod(X)\otimes \Dmod_{\on{lisse}}(X).$$

However,
$$\CS_{\{*\}\sqcup \{*\}}(r\otimes r')(m)\simeq
(\CS_{\{*\}}(r)\otimes \on{Id}_\CC)\circ \CS_{\{*\}}(r')(m),$$
and the assertion follows from the fact that
$$\CS_{\{*\}}(r')(m)\subset \CM\otimes  \Dmod_{\on{lisse}}(X),$$
by the assumption on $m$.

\end{proof}

\begin{cor}
The subcategory
$\CM^{\on{lisse}}$ carries an action of $\CA^{\otimes X}$, and this action is lisse.
\end{cor}

\sssec{}

Thus, \thmref{t:preserve nilp} says that we have an inclusion:

\begin{equation}  \label{e:lisse nilp}
\Dmod_{\on{Nilp}}(\Bun_G)\subset (\Dmod(\Bun_G))^{\on{lisse}}.
\end{equation}

We propose the following:

\begin{conj} \label{c:lisse nilp}
The inclusion \eqref{e:lisse nilp} is an equality.
\end{conj}

\sssec{Example}

Let us call an object $\CF\in \Dmod(\Bun_G)$ a \emph{loose} Hecke eigensheaf, if for every $V\in \Rep(\cG)$ we have
\emph{an} isomorphism
$$H(V,\CF)\simeq \CF'\otimes E_V,$$
where $E_V$ is some object of $\Dmod_{\on{lisse}}(X)$ and $\CF'\in \Dmod(\Bun_G)$.

\medskip

Clearly, a loose Hecke eigensheaf belongs to $(\Dmod(\Bun_G))^{\on{lisse}}$. Hence, \conjref{c:lisse nilp} contains as a special
case the following conjecture, first proposed by G.~Laumon (for actual Hecke eigensheaves, rather than loose ones):

\begin{conj}
A loose Hecke eigensheaf has a nilpotent singular support.
\end{conj}

\ssec{A spectral characterization of lisse actions}

In this subsection we specialize to the case when $\CA=\Rep(\sG)$ for an algebraic group $\sG$, and $X$ is a smooth
and proper curve.

\sssec{}

Let $\LocSys_\sG(X)$ be the stack of de Rham $\sG$-local systems on $X$, defined as in \cite[Sect. 10.1]{AG}.

\medskip

It is easy to see that one can write $\LocSys_\sG(X)$ as a quotient of a quasi-compact (derived) scheme by an
action of the algebraic group. From this it follows that $\QCoh(\LocSys_\sG(X))$ is compactly generated
and rigid.

\medskip

In particular, the object $\CO_{\LocSys_\sG(X)}\in \QCoh(\LocSys_\sG(X))$ is compact.

\sssec{}

Recall (see \cite[Sect. 4.3]{Ga1}) that there exists a canonically defined symmetric monoidal functor
\begin{equation} \label{e:Loc}
\on{Fact}(\Rep(\sG))_{\Ran(X)}\to \QCoh(\LocSys_\sG(X)),
\end{equation}
which admits a continuous and fully faithful right adjoint.

\medskip

The functor \eqref{e:Loc} defines a fully faithful embedding
\begin{equation} \label{e:Loc mod}
\QCoh(\LocSys_\sG(X))\mmod \to \CA^{\otimes X}\mmod.
\end{equation}

\sssec{}

Let $\CM$ be a compactly generated category equipped with an action of $\QCoh(\LocSys_\sG(X))$.

\medskip

To a compact object $m\in \CM$ one can attach its set-theoretic support
$$\on{supp}(m)\subset \LocSys_\sG(X).$$
This is the smallest among closed subsets $Y\subset \LocSys_\sG(X)$ such that the image of $m$ under the functor
$$\CM\simeq \QCoh(\LocSys_\sG(X))\underset{\QCoh(\LocSys_\sG(X))}\otimes \CM \to \QCoh(\LocSys_\sG(X)-Y)\underset{\QCoh(\LocSys_\sG(X))}\otimes \CM$$
is zero.

\sssec{}

Let $\sP$ be a parabolic in $\sG$ with Levi quotient $\sM$, and consider the diagram
$$\LocSys_\sG(X) \overset{\sfp_\sP}\longleftarrow \LocSys_\sP(X)  \overset{\sfq_\sP}\longrightarrow \LocSys_\sM(X).$$

Let $\sigma$ be an irreducible $\sM$-local system on $X$. Note that we have a closed embedding
$$i_\sigma:\on{pt}/\on{Aut}(\sigma)\hookrightarrow \LocSys_\sM(X).$$

In the above formula, $\on{Aut}(\sigma)$ is the group of automorphisms of $\sigma$. It contains the $Z(\sM)$
(the center of $\sM$), and  by the irreducibility assumption, $\on{Aut}(\sigma)/\sM$ is finite.

\medskip

Denote
$$\LocSys_{\sP,\sigma}(X):=\LocSys_\sP(X)\underset{\LocSys_\sM(X)}\times \on{pt}/\on{Aut}(\sigma).$$

Let $\sfp_{\sP,\sigma}$ denote the projection
$$\LocSys_{\sP,\sigma}(X)\to \LocSys_\sG(X).$$

Let
$$Y_{\sP,\sigma}\subset \LocSys_\sG(X)$$
be a closed subset equal to the image of $\sfp_{\sP,\sigma}$.

\begin{prop}
Let $m\in \CM$ be a compact object such that $\on{supp}(m)$ is contained in some $Y_{\sP,\sigma}$.
Then $m\in \CM^{\on{lisse}}$.
\end{prop}

\begin{proof}

It is enough to prove the proposition for $$\CM:=\QCoh(\LocSys_\sG(X))_{Y_{\sP,\sigma}},$$
where the latter is the full subcategory of $\QCoh(\LocSys_\sG(X))$ consisting of objects with set-theoretic support on $Y_{\sP,\sigma}$.

\medskip

Note that the functor
$$\sfp_{\sP,\sigma}^*:\QCoh(\LocSys_\sG(X))_{Y_{\sP,\sigma}}\to \QCoh(\LocSys_{\sP,\sigma}(X))$$
is conservative. Hence, by the last remark in \secref{sss:lisse rem}, we can further replace the category
$\QCoh(\LocSys_\sG(X))_{\sP,Y_\sigma}$ by $\QCoh(\LocSys_{\sP,\sigma}(X))$.

\medskip

The action of $\Rep(\sG)^{\otimes X}$ on $\QCoh(\LocSys_{\sP,\sigma}(X))$ factors through the restriction
$$\Rep(\sG)^{\otimes X}\to \Rep(\sP)^{\otimes X}.$$
Hence, it is enough to show that the action of $\Rep(\sP)^{\otimes X}$ on $\QCoh(\LocSys_{\sP,\sigma}(X))$ is lisse.

\medskip

Next, we note that the essential image of the restriction functor
$$\Rep(\sM)\to \Rep(\sP)$$
generates $\Rep(\sP)$. Hence, it is enough to show that the action of $\Rep(\sM)^{\otimes X}$ on
$\QCoh(\LocSys_{\sP,\sigma}(X))$ is lisse.

\medskip

Next, we note that the map
$$\sfq_{\sP,\sigma}:\LocSys_{\sP,\sigma}(X)\to \on{pt}/\on{Aut}(\sigma)$$
has the property that its base change with respect to $\on{pt}\to \on{pt}/\on{Aut}(\sigma)$ yields an algebraic stack of the form
$$\text{derived affine scheme}/\BA^n$$
for some $n$. Hence, the direct image functor
$$\QCoh(\LocSys_{\sP,\sigma}(X))\to \QCoh(\on{pt}/\on{Aut}(\sigma))$$
is conservative. Hence, it is enough to show that the action of $\Rep(\sM)^{\otimes X}$ on $\QCoh(\on{pt}/\on{Aut}(\sigma))$ is lisse.

\medskip

Finally, applying the (conservative) pullback functor along $\on{pt}\to \on{pt}/\on{Aut}(\sigma)$, we obtain that it is enough
to show that the action of $\Rep(\sM)^{\otimes X}$ on $\Vect$, corresponding to the $\sM$-local system $\sigma$, is
lisse.

\medskip

However, the latter is evident: for $V\in \Rep(\sM)^c$ and $k\in \Vect$, the corresponding object
$$\CS_{\{*\}}(V)(k)\in \Vect\otimes \Dmod(X)\simeq \Dmod(X)$$
is just $V_\sigma$, the twist of $V$ by $\sigma$.

\end{proof}

\sssec{}

The next conjecture proposes to describe the subcategory $\CM^{\on{lisse}}\subset \CM$ in terms of set-theoretic support:

\begin{conj}  \label{c:lisse as supp}
A compact object $m\in \CM$ belongs to $\CM^{\on{lisse}}$ if and only if it can be written as a colimit of objects $m_\alpha$, for each of
which $\on{supp}(m_\alpha)$ is contained in a finite union of closed subsets of
the form $Y_{\sP,\sigma}$.
\end{conj}

\section{The notion of universal local acyclicity (ULA)} \label{s:ULA}

In this section we review the notion of universal local acyclicity, for a general sheaf-theoretic context
with the exclusion of (a). In the particular case of D-modules we will express it via the forgetful
functor to quasi-coherent sheaves.

\ssec{The abstract ULA property}   \label{ss:ULA abstract}

\sssec{}

In this subsection $\CC$ will be an abstract symmetric monoidal DG category. The example to keep in mind is
$\CC=\Dmod(Y)$ for a scheme $Y$ with respect to $\sotimes$, so $\CC$ is very far from being rigid.

\sssec{}

Let $\CM$ be a $\CC$-module category, which is \emph{dualizable as a $\CC$-module}. This means that there exists
another $\CC$-module category $\CM^{\vee,\CC}$ and functors
$$\CC \overset{\on{unit}_\CM^\CC}\longrightarrow \CM^{\vee,\CC} \underset{\CC}\otimes \CM
\text{ and } \CM \underset{\CC}\otimes \CM^{\vee,\CC} \overset{\on{counit}_\CM^\CC}\longrightarrow \CC$$
that satisfy the usual duality axioms.

\medskip

Note in this case we have a natural identification
\begin{equation} \label{e:M dual}
\CM^{\vee,\CC}\underset{\CC}\otimes \CN \simeq \on{Funct}_{\CC\mmod}(\CM,\CN), \quad \CN\in \CC\mmod.
\end{equation}

\sssec{Example}  \label{sss:induced module}

Take $\CM=\CC\otimes \CM_0$ where $\CM_0$ is a plain DG category. If $\CM_0$ is dualizable as a DG category,
then $\CM$ is dualizable as a $\CC$-module and
$$\CM^{\vee,\CC}\simeq \CC\otimes \CM_0^\vee.$$

\sssec{} \label{sss:abs ULA}

Let $m\in \CM$ be an object. We shall say that $m$ is ULA over $\CC$ if there exists an object $m^{\vee,\CC}\in \CM^{\vee,\CC}$
equipped with a functorial identification
$$\Maps_{\CM}(c\otimes m,m')\simeq \Maps_\CC(c,\on{counit}_\CM^\CC(m'\underset{\CC}\otimes m^{\vee,\CC})), \quad c\in \CC,\,\, m'\in \CM.$$

\medskip

The datum of such an identification is equivalent to that of a pair of maps
$$\one_\CC\overset{\mu}\to \on{counit}_\CM^\CC(m \underset{\CC}\otimes m^{\vee,\CC}) \text{ and }
m^{\vee,\CA}\underset{\CC}\otimes m \overset{\epsilon}\to \on{unit}_\CM^\CC(\one_\CC)$$
such that the composite
\begin{multline}  \label{e:axiom ULA}
m \overset{\mu}\to \on{counit}_\CM^\CC(m \underset{\CC}\otimes m^{\vee,\CC})  \underset{\CC}\otimes m \simeq
(\on{counit}_\CM^\CC \underset{\CC}\otimes \on{Id})(m\underset{\CC}\otimes m^{\vee,\CC} \underset{\CC}\otimes m) \overset{\epsilon}\to \\
\to (\on{counit}_\CM^\CC \underset{\CC}\otimes \on{Id})(m  \underset{\CC}\otimes \on{unit}_\CM^\CC(\one_\CC)) \simeq m
\end{multline}
is the identity map, and similarly for $m^{\vee,\CC}$.

\medskip

Note that if an object $m\in \CM$ is ULA, then so is the corresponding $m^{\vee,\CC}\in \CM^{\vee,\CC}$ with
$$(m^{\vee,\CC})^{\vee,\CC}\simeq m.$$

\sssec{}  \label{sss:inner Hom comp}

In what follows we will assume that $\CC$ and $\CM$ are compactly generated (as DG categories). We will also assume
that the unit object $\one_\CC$ is compact.
(But we do \emph{not} assume that the tensor product operation on $\CC$ or the action of $\CC$ on $\CM$ preserve compactness.)

\medskip

Note that for a pair of objects $m,m'\in M$, we have a well-defined $\ul\Hom(m,m')\in \CC$, satisfying
$$\Maps_\CC(c,\ul\Hom(m,m')):=\Maps_{\CM}(c\otimes m,m').$$

For a fixed $m$, the functor $m'\mapsto \ul\Hom(m,m')$ preserves colimits if and only if
$c\otimes m$ are compact for all compact $c\in \CC$; in particular in this case $m$ itself is compact.

\medskip

For $c\in \CC$, we have a tautological map
\begin{equation} \label{e:inner Hom}
c\otimes \ul\Hom(m,m')\to \ul\Hom(m,c\otimes m').
\end{equation}

We have:

\begin{lem}  \label{l:ULA inner Hom}
An object $m\in \CM$ is ULA if and only if the functor
$$m'\mapsto \ul\Hom(m,m'), \quad \CM\to \CC$$
preserves colimits and \eqref{e:inner Hom} is an isomorphism for all $c\in \CC$.
\end{lem}

\begin{proof}

If $m$ is ULA we have
$$\ul\Hom(m,m')\simeq \on{counit}_\CM^\CC(m' \underset{\CC}\otimes m^{\vee,\CC}).$$

The other direction follows from \eqref{e:M dual} for $\CN=\CC$.

\end{proof}

\begin{cor} \label{c:rigid ULA}
If $\CC$ is rigid, any compact $m\in \CM$ is ULA.
\end{cor}

\begin{proof}

It is enough to check that the objects $c\otimes m$ are compact and the maps \eqref{e:inner Hom}
are isomorphisms for $c$ that are dualizable in $\CC$. In the rigid case this is automatic.

\end{proof}

\ssec{Adding self-duality}

Let $\CC$ be as above (a compactly generated symmetric monoidal DG category with a compact unit).

\sssec{}  \label{sss:abs Verdier}

Assume that $\CC$ contains a compact object, to be denoted
$\wt\one_\CC$, so that the pairing
$$\CC\otimes \CC\mapsto \Vect, \quad c_1,c_2\mapsto \CHom_\CC(\wt\one_\CC,c_1\otimes c_2)$$
defines a self-duality
\begin{equation}  \label{e:Verd abs}
\CC^\vee \simeq \CC.
\end{equation}

Let $\BD_\CC$ denote the corresponding contravariant self-equivalence on $\CC^c$, i.e.,
$$\CHom_\CC(\BD_\CC(c_1),c_2))\simeq \CHom_\CC(\wt\one_\CC,c_1\otimes c_2), \quad c_1\in \CC^c.$$

\medskip

By construction
$$\wt\one_\CC\simeq \BD_\CC(\one_\CC).$$

\sssec{Example}

Let $\CC=\Dmod(Y)$, where $Y$ is a scheme of finite type over a ground field $k$ of characteristic $0$.
We regard $\CC$ as a symmetric monoidal category via the $\sotimes$ tensor product.

\medskip

Then we can take $\wt\one_\CC$ to be the ``constant sheaf" $k_Y\in \Dmod(Y)$, i.e.,
the Verdier dual of the dualizing sheaf $\omega_\CY$. The functor $\BD_\CC$ is
the Verdier duality self-equivalence of $\Dmod(Y)^c$, see \secref{sss:module geometry}.

\sssec{}

Let $\CM$ be a $\CC$-module category.  Note that for a pair of objects $m,m'\in \CM$ and $c\in \CC^c$
we have a canonically defined map
\begin{equation} \label{e:ULA map}
\CHom(c\otimes m,m')\to \CHom_\CM(\BD_\CC(c)\otimes c\otimes m,\BD_\CC(c)\otimes m')\to
\CHom_\CM(\wt\one_\CC\otimes m,\BD_\CC(c)\otimes m'),
\end{equation}
where the last arrow comes from the canonical map
$$\wt\one_\CC\to \BD_\CC(c)\otimes c.$$

\sssec{}

Assume that $\CM$ is compactly generated as a DG category.  Let $\BD_\CM$ denote the canonical contravariant equivalence
$$\CM^c \to (\CM^\vee)^c.$$
Let $\langle-,-\rangle$ denote the tautological pairing
$$\CM\times \CM^\vee\to \Vect.$$

\medskip

Consider
$\CM^\vee$ as a $\CC$-module category, where the action of $c\in \CC$ on $\CM^\vee$ is the functor dual to that
of the action of $c$ on $\CM$.

\medskip

Define a functor
\begin{multline} \label{e:internal pairing}
\langle-,-\rangle_\CC:\CM\times \CM^\vee\to \CC, \\
\CHom_\CC(c,\langle m,m'\rangle_\CC):=\langle \BD_\CC(c)\otimes m,m' \rangle \simeq \langle m,\BD_\CC(c)\otimes m' \rangle, \quad
m\in \CM, m'\in \CM^\vee.
\end{multline}

\medskip

By construction, \eqref{e:internal pairing} factors via a functor
\begin{equation} \label{e:internal pairing bis}
\CM \underset{\CC}\otimes \CM^\vee  \to \CC.
\end{equation}

\sssec{} \label{sss:internal pairing bis}

Assume that \eqref{e:internal pairing bis} is the counit of a duality, so we find ourselves in the context of \secref{ss:ULA abstract}.
In particular, we can identify
$$\CM^{\vee,\CC}\simeq \CM^\vee.$$

\sssec{}

We claim:

\begin{prop} \label{p:ULA via morphism}
For $m\in \CM^c$ the following conditions are equivalent:

\medskip

\noindent{\em(i)} The object $m$ is ULA over $\CC$;

\medskip

\noindent{\em(ii)} The object $m$ is ULA, with
$$m^{\vee,\CC}=\BD_\CM(\wt\one_\CC\otimes m)$$
and $\mu$ being the canonical map
$$\one_\CC\to \langle m,\BD_\CM(\wt\one_\CC\otimes m)\rangle_\CC.$$

\medskip

\noindent{\em(iii)} The object $\wt\one_\CC\otimes m$ is compact, and
the map \eqref{e:ULA map} is an isomorphism for every $c\in \CC^c$ and $m'\in \CM$.

\end{prop}

\begin{proof}

The implication (ii) $\Rightarrow$ (i) is tautological. To prove (iii) $\Rightarrow$ (ii) we note that if \eqref{e:ULA map} is an isomorphism,
then the object $\BD_\CM(\wt\one_\CC\otimes m)$ satisfies the requirements of $m^{\vee,\CC}$. It remains prove (i) $\Rightarrow$ (iii).

\medskip

Let $m$ be ULA over $\CC$, and let $m^{\vee,\CC}$ be the corresponding object of $\CM^{\vee,\CC}$,
which we now identify with $\CM^\vee$. We have:
$$\CHom_\CM(c\otimes m,m')\simeq \CMaps_\CC(c, \langle m',m^{\vee,\CC}\rangle_\CC)\simeq
\langle \BD_\CC(c)\otimes m', m^{\vee,\CC} \rangle, \quad c\in \CC^c,m'\in \CM.$$

In particular, taking $c=\wt\one_\CC$, we obtain
$$\CHom_\CM(\wt\one_\CC\otimes m,m')\simeq \langle m', m^{\vee,\CC} \rangle, \quad m'\in \CM.$$
Hence, we obtain that $m^{\vee,\CC}$ is compact and
$$\BD_\CM(m^{\vee,\CC})\simeq \wt\one_\CC\otimes m;$$
in particular $\wt\one_\CC\otimes m$ is compact.

\medskip

With respect to the last identification, the map \eqref{e:ULA map} becomes the map
\begin{multline*}
\CHom_\CM(c\otimes m,m')\simeq \CMaps_\CC(c, \langle m',m^{\vee,\CC}\rangle_\CC)\simeq
\langle \BD_\CC(c)\otimes m', m^{\vee,\CC} \rangle \simeq \\
\simeq \CHom_{\CM}(\BD_\CM(m^{\vee,\CC}),\BD_\CC(c)\otimes m')\simeq
\CHom_{\CM}(\wt\one_\CC\otimes m,\BD_\CC(c)\otimes m'),
\end{multline*}
and hence is an isomorphism.

\end{proof}

\sssec{Example}

Let $\CM$ be of the form $\CC\otimes \CM_0$ for a compactly generated DG category $\CM_0$. Then the pairing
\eqref{e:internal pairing bis} does indeed define the counit of a duality. The composite identification
$$\CC^\vee \otimes \CM_0^\vee \simeq (\CC\otimes \CM_0)^\vee \simeq (\CC\otimes \CM_0)^{\vee,\CC}
\simeq \CC\otimes \CM_0^\vee$$
is given by the self-duality \eqref{e:Verd abs}.

\ssec{The ULA condition in the geometric situation}

In this subsection we will take $\CC$ to be $\Dmod(Y)$, where $Y$ is a scheme of finite type.
Let $f:Z\to Y$ be a scheme over $Y$; take $\CM:=\Dmod(Z)$, which is acted on by $\Dmod(Y)$ via
$$\CF_Y,\CF_Z\mapsto f^!(\CF_Y)\sotimes \CF_Z.$$

We will bring the concept of ULA developed above in contact with the more familiar geometric
geometric notion.

\sssec{} \label{sss:module geometry}

First, we note that $\Dmod(Y)$ fits into the paradigm of \secref{sss:abs Verdier} with
$$\BD_{\Dmod(Y)}:=\BD^{\on{Verdier}}_Y.$$

The corresponding object $\wt\one_{\Dmod(Y)}$ is $\sfe_Y$, the ``constant sheaf", i.e.,
$$\CHom(\sfe_Y,\CF)=\Gamma_{\on{dR}}(Y,\CF).$$

\medskip

Similarly, we use Verdier duality on $Z$ to identify $\Dmod(Z)^\vee$ with $\Dmod(Z)$.

\medskip

In this case, the pairing \eqref{e:internal pairing} identifies with
$$\langle \CF_1,\CF_2\rangle_{\Dmod(Y)}=f_*(\CF_1\sotimes \CF_2).$$

\sssec{}

We claim that the resulting pairing \eqref{e:internal pairing bis} provides the counit
of the adjunction.

\medskip

Indeed, the unit map
$$\Dmod(Y)\to \Dmod(Z)\underset{\Dmod(Y)}\otimes \Dmod(Z)$$
is determined by the condition that the corresponding object
$$\on{unit}_{\Dmod(Z)}^{\Dmod(Y)}(\one_{\Dmod(Y)})\in \Dmod(Z)\underset{\Dmod(Y)}\otimes \Dmod(Z)\simeq
\Dmod(Z\underset{Y}\times Z)$$
equals
$$(\Delta_{Z/Y})_*(\omega_Z),$$
where
$$\Delta_{Z/Y}:Z\to Z\underset{Y}\times Z$$
is the relative diagonal map.

\sssec{}

Thus, from \propref{p:ULA via morphism} we obtain:

\begin{cor} \label{c:ULA geom}
For a compact object $\CF\in \Dmod(Z)$ the
following conditions are equivalent:

\medskip

\noindent{\em(i)} There exists an object $\CF^\vee\in \Dmod(Z)^c$ equipped with maps
$$\mu:\omega_Y\to f_*(\CF\sotimes \CF^\vee) \text{ and }
\epsilon: \CF^\vee\underset{Y}\boxtimes \CF \to (\Delta_{Z/Y})_*(\omega_Z)$$
(here $-\underset{Y}\boxtimes -$ stands for the !-pullback of $-\boxtimes -$ along $Z\underset{Y}\times Z\to Z\times Z$)
such that the composite
\begin{multline} \label{e:ULA axiom geom}
\CF \simeq f^!(\omega_Y)\sotimes \CF \overset{\mu}\to f^!(f_*(\CF\sotimes \CF^\vee))\sotimes \CF
\simeq (p_2)_* \circ (\Delta_{Z/Y}\underset{Y}\times \on{id}_Z)^!
(\CF\underset{Y}\boxtimes \CF^\vee \underset{Y}\boxtimes \CF) \overset{\epsilon}\to \\
\to (p_2)_*\circ (\Delta_{Z/Y}\underset{Y}\times \on{id}_Z)^!(\CF\underset{Y}\boxtimes (\Delta_{Z/Y})_*(\omega_Z)) \simeq \CF
\end{multline}
(here $p_2$ denotes the second projection $Z\underset{Y}\times Z\to Z$)
is the identity map, and similarly for $\CF^\vee$.

\medskip

\noindent{\em(ii)} Same as (i), but for $\CF^\vee=\BD^{\on{Verdier}}_Z(f^!(\sfe_Y)\sotimes \CF)$ (in particular, $f^!(\sfe_Y)\sotimes \CF$
is compact), with $\mu$ being the map
$$\omega_Y\to f_*(\CF\sotimes \BD^{\on{Verdier}}_Z(f^!(\sfe_Y)\sotimes \CF)),$$
corresponding to the identity element under
\begin{multline*}
\CHom_{\Dmod(Y)}\left(\omega_Y, f_*(\CF\sotimes \BD^{\on{Verdier}}_Z(f^!(\sfe_Y)\sotimes \CF))\right)\simeq
\Gamma_{\on{dR}}\left(Y,\sfe_Y\sotimes f_*(\CF\sotimes \BD^{\on{Verdier}}_Z(f^!(\sfe_Y)\sotimes \CF))\right)\simeq \\
\overset{\text{projection formula}}\simeq \Gamma_{\on{dR}}\left(Z,f^!(\sfe_Y)\sotimes \CF\sotimes \BD^{\on{Verdier}}_Z(f^!(\sfe_Y)\sotimes \CF)\right)\simeq
\CHom_{\Dmod(Z)}(f^!(\sfe_Y)\sotimes \CF,f^!(\sfe_Y)\sotimes \CF)).
\end{multline*}

\medskip

\noindent{\em(iii)} For any $\CF_Y\in \Dmod(Y)^c$ and $\CF'\in \Dmod(Z)$, the map
\begin{multline} \label{e:ULA map geom}
\CHom_{\Dmod(Z)}(f^!(\CF_Y)\sotimes \CF,\CF')\to \\
\to \CHom_{\Dmod(Z)}(f^!(\BD^{\on{Verdier}}_Y(\CF_Y))\sotimes f^!(\CF_Y)\sotimes \CF,f^!(\BD^{\on{Verdier}}_Y(\CF_Y))\sotimes\CF')\to \\
\to \CHom_{\Dmod(Z)}(f^!(\sfe_Y)\sotimes \CF,f^!(\BD^{\on{Verdier}}_Y(\CF_Y))\sotimes\CF')
\end{multline}
is an isomorphism.

\end{cor}

\sssec{}

We will call objects of $\Dmod(Z)$ satisfying the equivalent conditions of \corref{c:ULA geom} ULA over $Y$ (or ULA with respect to $f$).

\begin{rem}
Note that condition (iii) in \corref{c:ULA geom} can be rephrased as follows:

\medskip

For $\CF_Y\in \Dmod(Y)$, the object
$$f^*(\CF_Y)\overset{*}\otimes (f^!(\sfe_Y)\sotimes \CF)\in \Dmod(Z)$$
is defined and the canonical map
$$f^*(\CF_Y)\overset{*}\otimes (f^!(\sfe_Y)\sotimes \CF) \to f^!(\CF_Y)\sotimes \CF$$
is an isomorphism.
\end{rem}

\sssec{}

When we view this definition from the point of view of condition (i), we obtain the following assertion that reads that
``local acyclicity implies universal local acyclicity":

\begin{cor} \label{c:no need universal}
Let $g:Y'\to Y$ be a morphism of schemes, and let set $Z':=Y'\underset{Y}\times Z$. If $\CF$ is ULA over $Y$,
then its !-pullback to $Z'$ is ULA over $Y'$.
\end{cor}

In the context of $\ell$-adic sheaves, an assertion parallel to \corref{c:no need universal}
had been proved by O.~Gabber, see \cite[Corollary 6.6]{LZ}.

\ssec{An aside: the notion of ULA in other sheaf-theoretic contexts}

We will now place ourselves in a general sheaf-theoretic context of \secref{sss:sheaves}, excluding (a) (the exclusion is because we want our
categories to be compactly generated).

\sssec{}

Let $f:Z\to Y$ be a map of schemes. We consider $\Shv(Z)$ as a module category over $\Shv(Y)$ via
the operation
$$f^!(-)\sotimes -.$$

However, the notion of ULA developed from \secref{ss:ULA abstract} is not directly applicable in this context, since $\Shv(Z)$
is not in general dualizable as a $\Shv(Y)$-module category.  However, we claim that the following
analog of \corref{c:ULA geom} holds:

\medskip

\begin{prop} \label{p:ULA geom}
For a compact object $\CF\in \Shv(Z)$ the
following conditions are equivalent:

\medskip

\noindent{\em(i)} There exists an object $\CF^\vee\in \Shv(Z)^c$ equipped with maps
$$\mu:\omega_Y\to f_*(\CF\sotimes \CF^\vee) \text{ and }
\epsilon: \CF^\vee\underset{Y}\boxtimes \CF \to (\Delta_{Z/Y})_*(\omega_Z)$$
such that the composite
\begin{multline} \label{e:ULA axiom geom arb}
\CF \simeq f^!(\omega_Y)\sotimes \CF \overset{\mu}\to f^!(f_*(\CF\sotimes \CF^\vee))\sotimes \CF
\simeq (p_2)_* \circ (\Delta_{Z/Y}\underset{Y}\times \on{id}_Z)^!
(\CF\underset{Y}\boxtimes \CF^\vee \underset{Y}\boxtimes \CF) \overset{\epsilon}\to \\
\to (p_2)_*\circ (\Delta_{Z/Y}\underset{Y}\times \on{id}_Z)^!(\CF\underset{Y}\boxtimes (\Delta_{Z/Y})_*(\omega_Z)) \simeq \CF
\end{multline}
is the identity map, and similarly for $\CF^\vee$.

\medskip

\noindent{\em(ii)} Same as (i), but for $\CF^\vee=\BD^{\on{Verdier}}_Z(f^!(\sfe_Y)\sotimes \CF)$ with $\mu$ being the canonical map
$$\omega_Y\to f_*(\CF\sotimes \BD^{\on{Verdier}}_Z(f^!(\sfe_Y)\sotimes \CF)).$$

\medskip

\noindent{\em(iii)} For any base change
\begin{equation} \label{e:base change ULA}
\CD
\wt{Z} @>{g_Z}>>  Z \\
@V{\wt{f}}VV  @VV{f}V \\
\wt{Y}  @>{g_Y}>>  Y,
\endCD
\end{equation}
and $\wt\CF:=g_Z^!(\CF)$, for any $\CF_{\wt{Y}}\in \Shv(\wt{Y})^c$ and $\wt\CF'\in \Shv(\wt{Z})$, the map
\begin{equation} \label{e:ULA map geom arb}
\CHom_{\Shv(\wt{Z})}(\wt{f}^!(\CF_{\wt{Y}})\sotimes \wt\CF,\wt\CF')
\to \CHom_{\Shv(\wt{Z})}(\wt{f}^!(\sfe_{\wt{Y}})\sotimes \wt\CF,\wt{f}^!(\BD^{\on{Verdier}}_{\wt{Y}}(\CF_{\wt{Y}}))\sotimes \wt\CF')
\end{equation}
is an isomorphism.

\end{prop}

\begin{proof}

%

Clearly, (ii) implies (i). Let us show that (iii) implies (ii).

\medskip

First, note that the RHS in \eqref{e:ULA map geom} a priori identifies with
\begin{equation} \label{e:ULA map geom arb bis}
\CHom_{\Shv(\wt{Y})}\left(\CF_{\wt{Y}},\wt{f}_*(\wt\CF'\sotimes \BD^{\on{Verdier}}_{\wt{Z}}(\wt{f}^!(\sfe_{\wt{Y}})\sotimes \wt\CF))\right).
\end{equation}

\medskip

Take $\wt{Y}:=Z$; so that the square \eqref{e:base change ULA} becomes
$$
\CD
Z\underset{Y}\times Z @>{p_2}>> Z  \\
@V{p_1}VV  @VV{f}V  \\
Z @>{f}>> Y.
\endCD
$$
Take $\wt\CF':=(\Delta_{Z/Y})_*(\omega_Z)$ and
and $\CF_{\wt{Y}}:=\BD^{\on{Verdier}}_Z(f^!(\sfe_Y)\sotimes \CF)$. The expression in \eqref{e:ULA map geom arb bis} identifies with
\begin{equation} \label{e:ULA map geom arb spec}
\CHom_{\Shv(Z)}\left(\BD^{\on{Verdier}}_Z(f^!(\sfe_Y)\sotimes \CF),\BD^{\on{Verdier}}_Z(f^!(\sfe_Y)\sotimes \CF))\right).
\end{equation}

\medskip

The required map $\epsilon$ is obtained via \eqref{e:ULA map geom} from the identity element in
\eqref{e:ULA map geom arb spec}.

\medskip

Finally, let us assume (i) and deduce (iii). 
If we have a datum as in (i) for $f:Z\to Y$, its !-pullback defines a similar data for
$\wt{f}:\wt{Z}\to \wt{Y}$. So we can assume that $g=\on{id}$.

\medskip

The datum of $\mu$ defines a map
\begin{multline}  \label{e:ULA map geom arb abs}
\CHom_{\Shv(Z)}(f^!(\CF_Y)\sotimes \CF,\CF')\to
\CHom_{\Shv(Z)}(f^!(\CF_Y)\sotimes \CF\sotimes \CF^\vee,\CF'\sotimes \CF^\vee)\to \\
\to \CHom_{\Shv(Y)}\left(f_*(f^!(\CF_Y)\sotimes \CF\sotimes \CF^\vee),f_*(\CF'\sotimes \CF^\vee)\right)\simeq \\
\simeq \CHom_{\Shv(Y)}\left(\CF_Y\sotimes f_*( \CF\sotimes \CF^\vee),f_*(\CF'\sotimes \CF^\vee)\right)
\overset{\mu}\to
\CHom_{\Shv(Y)}(\CF_Y,f_*(\CF'\sotimes \CF^\vee)).
\end{multline}

The datum of $\epsilon$ defines an inverse map
\begin{multline*}
\CHom_{\Shv(Y)}(\CF_Y,f_*(\CF'\sotimes \CF^\vee)) \to
\CHom_{\Shv(Z)}(f^!(\CF_Y),f^!(f_*(\CF'\sotimes \CF^\vee))) \simeq \\
\simeq \CHom_{\Shv(Z)}\left(f^!(\CF_Y), p_2{}_*(p_1^!(\CF'\sotimes \CF^\vee))\right)
\to \CHom_{\Shv(Z)}\left(f^!(\CF_Y)\sotimes \CF, p_2{}_*(p_1^!(\CF'\sotimes \CF^\vee))\sotimes \CF\right) \simeq \\
\simeq \CHom_{\Shv(Z)}\left(f^!(\CF_Y)\sotimes \CF,p_2{}_*(p_1^!(\CF'\sotimes \CF^\vee)\sotimes p_2^!(\CF))\right)\simeq \\
\simeq \CHom_{\Shv(Z)}\left(f^!(\CF_Y)\sotimes \CF,p_2{}_*(p_1^!(\CF^\vee)\sotimes p_2^!(\CF))\sotimes p_1^!(\CF'))\right) \overset{\epsilon}\to \\
\to \CHom_{\Shv(Z)}\left(f^!(\CF_Y)\sotimes \CF,p_2{}_*((\Delta_{Z/Y})_*(\omega_Z)\sotimes p_1^!(\CF'))\right) \simeq
\CHom_{\Shv(Z)}(f^!(\CF_Y)\sotimes \CF,\CF').
\end{multline*}

In particular, taking $\CF_Y=\sfe_Y$, we obtain an identification between
$$\CHom_{\Shv(Z)}(\BD_Z^{\on{Verdier}}(\CF),\CF')\simeq
\Gamma_{\on{dR}}(Z,\CF'\sotimes \CF^\vee)\simeq \Gamma_{\on{dR}}(Y,f_*(\CF'\sotimes \CF^\vee))
\simeq \CHom_{\Shv(Y)}(\sfe_Y,f_*(\CF'\sotimes \CF^\vee))$$
and
$$\CHom_{\Shv(Z)}(f^!(\sfe_Y)\sotimes \CF,\CF'),$$
functorial in $\CF'$. Hence,
$$\CF^\vee\simeq \BD_Z^{\on{Verdier}}(f^!(\sfe_Y)\sotimes \CF).$$

Under this identification, the map \eqref{e:ULA map geom arb} goes over to the composition of \eqref{e:ULA map geom arb abs} and
$$\CHom_{\Shv(Y)}(\CF_Y,f_*(\CF'\sotimes \CF^\vee)) \simeq \CHom_{\Shv(Z)}(\BD^{\on{Verdier}}_Z(\CF^\vee),f^!(\BD^{\on{Verdier}}_Y(\CF_Y))\sotimes \CF'),$$
hence it is an isomorphism.

\end{proof}

\sssec{}

We will call objects of $\Dmod(Z)$ satisfying the equivalent conditions of \propref{p:ULA geom} ULA over $Y$ (or ULA with respect to $f$).

\begin{rem}
Let us compare the notion of ULA in the three contexts (b), (b') and (b''). Note that we have fully faithful embeddings
of the corresponding categories (b'') $\to$ (b') $\to$ (b).

\medskip

We note that if $\CF$ is an object of a smaller category, then, if we view the from the point of view of condition (ii)
of \propref{c:ULA geom}, it
is ULA in the smaller category if and only if it is such from the point of view of a bigger category.

\medskip

Note that this property is \emph{non-obvious} from the point of view of condition (iii) of \propref{c:ULA geom}, since for the larger category
we are testing the isomorphism on a larger collection of objects (one that are denoted $\wt\CF'$).

\end{rem}

\begin{rem}

Note that the property of being ULA is \emph{by design} stable under base change: this is obvious from the point of view of each
of the conditions (i), (ii), (iii) in \propref{c:ULA geom}.

\medskip

That said, the proof of \propref{c:ULA geom} shows that it is enough to require that the map \eqref{e:ULA map geom arb} be an
isomorphism for $(\wt{Y},g_Y)$ being the pair $(Z,f)$.

\end{rem}

\ssec{The ULA property for D-modules}

In this subsection we will come back to the ULA property in the context of D-modules.

\medskip

We will be working over a ground field $k$ of characteristic $0$. Let $Y$ be a \emph{smooth} scheme of finite type over
$k$. Consider the symmetric monoidal category $\CC:=\Dmod(Y)$. In this subsection we will give a hands-on criterion for objects
(in some class of $\Dmod(Y)$-module categories) to be ULA.

\sssec{}

We will regard the category $\Dmod(Y)$ as $\QCoh(Y_{\on{dR}})$, where
$Y_{\on{dR}}$ is the de Rham prestack of $Y$, i.e.,
$$\Maps(S,Y_{\on{dR}})=\Maps(S_{\on{red}},Y).$$

A key observation observation is that $Y_{\on{dR}}$ is \emph{1-affine} (see \cite[Definition 1.3.7]{Ga5} for
what this means, and \cite[Theorem 2.6.3]{Ga5} for the statement of the result).

\medskip

In other words, a datum of a $\QCoh(Y_{\on{dR}})$-module category $\CM$ is equivalent to a datum of
a \emph{sheaf of categories over $Y_{\on{dR}}$}, i.e., an assignment
$$(S\overset{y}\to Y_{\on{dR}}) \rightsquigarrow \CM_{S,y}\in \QCoh(S)\mmod,\quad S\in \on{Sch}^{\on{aff}},$$
$$(S'\overset{f}\to S) \rightsquigarrow \CM_{S',y\circ f}\simeq \QCoh(S')\underset{\QCoh(S)}\otimes \CM_{S,y},$$
equipped with a data of homotopy-compatibility for compositions.

\medskip

Explicitly, each $\CM_{S,y}$ is recovered as
$$\CM_{S,y}\simeq \QCoh(S)\underset{\QCoh(Y_{\on{dR}})}\otimes \CM.$$

Vice versa, $\CM$ can be identified with
$$\underset{S\overset{y}\to Y_{\on{dR}}}{\on{lim}}\, \CM_{S,y},$$
equipped with an action of
$$\underset{S\overset{y}\to Y_{\on{dR}}}{\on{lim}}\, \QCoh(S)=\QCoh(Y_{\on{dR}}).$$

\medskip

We will denote by $\CM_Y$ the value of this sheaf categories on $Y$, where we regard $Y$ as equipped with the canonical projection
to $Y_{\on{dR}}$. Explicitly,
$$\CM_Y\simeq \QCoh(Y)\underset{\QCoh(Y_{\on{dR}})}\otimes \CM.$$

\sssec{Examples}

If $\CM=\CM_0\otimes \Dmod(Y)$, then $\CM_Y=\CM_0\otimes \QCoh(Y)$.

\medskip

Let now $\CM=\QCoh(Z_{\on{dR}})$. Then
$$\CM_Y=\QCoh(Z_{\on{dR}}\underset{Y_{\on{dR}}}\times Y).$$

If $Z\to Y$ is a smooth map, one can identify $\QCoh(Z_{\on{dR}}\underset{Y_{\on{dR}}}\times Y)$ with
the derived category of (quasi-coherent sheaves of) modules over the ring $\on{D}_{Z/Y}$ of
\emph{vertical differential operators}; this is a subring of $\on{D}_Z$, generated by functions and vertical vector fields,
i.e., vector fields along the fibers of the map $Z\to Y$.

\sssec{}

The adjoint pair
$$\ind:\QCoh(Y)\rightleftarrows \QCoh(Y_{\on{dR}}):\oblv$$
induces an adjoint pair
\begin{equation} \label{e:ind module}
\ind:\CM_Y \rightleftarrows \CM:\oblv.
\end{equation}

The functor $\oblv$ is conservative (indeed, since $Y$ is smooth, any map $S\to Y_{\on{dR}}$ can be lifted to $Y$).

\medskip

In what follows we will assume that $\CM_Y$ is compactly generated as a DG category. The conservativity of
$\oblv$ implies that in this case $\CM$ is also compactly generated.

\sssec{}

Since $\CM_Y$ is compactly generated, and hence dualizable, all the categories $\CM_{S,y}$ are dualizable
(as DG categories, or equivalently, as $\QCoh(S)$-module categories). The assignment
$$(S,y)\mapsto \CM_{S,y}^\vee$$
is also a sheaf of categories over $Y_{\on{dR}}$.

\medskip

Set
$$\CM^{\vee,Y_{\on{dR}}}:=\underset{S\overset{y}\to Y_{\on{dR}}}{\on{lim}}\, \CM_{S,y}^\vee.$$

The 1-affineness of $Y_{\on{dR}}$ implies that $\CM^{\vee,Y_{\on{dR}}}$ is the dual of $\CM$ as a
$\QCoh(Y_{\on{dR}})$-module category. Denote by $\langle-,-\rangle_{Y_{\on{dR}}}$ the resulting pairing
$$\CM\times \CM^{\vee,Y_{\on{dR}}}\to \QCoh(Y_{\on{dR}}).$$

\medskip

Define the pairing
$$\langle-,-\rangle: \CM\times \CM^{\vee,Y_{\on{dR}}}\to \Vect$$
by
$$\langle-,-\rangle:=\Gamma_{\on{dr}}(Y,-)\circ \langle-,-\rangle_{Y_{\on{dR}}},$$
where $\Gamma_{\on{dr}}(Y,-)$ is the functor of de Rham cohomology
$$\QCoh(Y_{\on{dR}})\simeq \Dmod(Y)\to \Vect.$$

\begin{lem}  \label{l:duality Dmod mod}
The functor $\langle-,-\rangle$ defines an equivalence $\CM^{\vee,Y_{\on{dR}}}\simeq \CM^\vee$.
\end{lem}

\begin{proof}

The adjunction
$$(\Delta_Y)_{\on{dR},*}:\Dmod(Y)\rightleftarrows \Dmod(Y)\otimes \Dmod(Y):\Delta^!_Y$$
(here we identify $\Dmod(Y)\otimes \Dmod(Y)\simeq \Dmod(Y\times Y)$)
defines an adjunction
\begin{multline*}
\CM^{\vee,Y_{\on{dR}}}\underset{\Dmod(Y)}\otimes \CM
\simeq (\CM^{\vee,Y_{\on{dR}}}\otimes \CM)\underset{\Dmod(Y)\otimes \Dmod(Y)}\otimes \Dmod(Y)
\rightleftarrows  \\
\rightleftarrows  (\CM^{\vee,Y_{\on{dR}}}\otimes \CM)\underset{\Dmod(Y)\otimes \Dmod(Y)}\otimes (\Dmod(Y)\otimes \Dmod(Y))
\simeq  \CM^{\vee,Y_{\on{dR}}}\otimes \CM
\end{multline*}

We define the unit object
$$\on{unit}_\CM\in \CM^{\vee,Y_{\on{dR}}}\otimes \CM$$
to be the image of the unit
$$\on{unit}^{Y_{\on{dR}}}_\CM\in \CM^{\vee,Y_{\on{dR}}}\underset{\Dmod(Y)}\otimes \CM$$
under the above left adjoint
$$\CM^{\vee,Y_{\on{dR}}}\underset{\Dmod(Y)}\otimes \CM \to  \CM^{\vee,Y_{\on{dR}}}\otimes \CM.$$

The fact that $\on{unit}_\CM$ and $\langle-,-\rangle$ satisfy the duality axioms follows by diagram chase.

\end{proof}

\sssec{}

Let
$$\wt\one_{\QCoh(Y_{\on{dR}})}\in \QCoh(Y_{\on{dR}})$$
be the object corresponding to $k_Y\in \Dmod(Y)$. Note that it satisfies the assumption of \secref{sss:abs Verdier}, and the
corresponding self-duality of $\Dmod(Y)$ is $\BD_Y^{\on{Verdier}}$.

\medskip

It follows from \lemref{l:duality Dmod mod} that any $\CM$ as above satisfies the assumption of
\secref{sss:internal pairing bis}.

\sssec{}

We now claim:

\begin{thm}  \label{t:ULA Dmod}
For $m\in \CM^c$ the following conditions are equivalent:

\medskip

\noindent{\em(i)} The object $m$ is ULA over $\Dmod(Y)$;

\medskip

\noindent{\em(ii)} For every $\CF\in \Dmod(Y)^c$, the object $\CF\otimes m\in \CM$ is compact;

\medskip

\noindent{\em(iii)} The object $\oblv(m)\in \CM_Y$ is compact.

\end{thm}

The rest of this subsection is devoted to the proof of this theorem.

\sssec{}

The implication (i) $\Rightarrow$ (ii)
follows from \secref{sss:inner Hom comp} and \lemref{l:ULA inner Hom}. To prove (iii) $\Rightarrow$ (ii),
it suffices to show that for $\CF_0\in \QCoh(Y)^c$, we have
$$\ind(\CF_0)\otimes m\in \CM^c.$$
However,
$$\ind(\CF_0)\otimes m\simeq \ind(\CF_0\otimes \oblv(m)),$$
and the assertion follows from the fact that the functor $\ind$ preserves compactness.

\begin{proof}[Proof of {\em(iii)} $\Rightarrow$ {\em(i)}]

This is in fact a general assertion. Let $\phi:\CC\to \CC_0$ be a symmetric monoidal functor
that admits a $\CC$-linear left adjoint. Let $\CM$ be a dualizable $\CC$-module category and set
$$\CM_0:=\CC_0\underset{\CC}\otimes \CM.$$

Consider the corresponding functor
$$\phi_\CM:=(\phi\otimes \on{Id}_\CM):\CM\to \CM_0;$$
it admits a left adjoint, given by $(\phi^L\otimes \on{Id}_\CM)$.

\medskip

Consider the inner Hom object $\ul\Hom(m,m')\in\CC$. Note that we have
\begin{equation} \label{e:inner Hom compat}
\phi(\ul\Hom(m,m')) \simeq \ul\Hom_0(\phi_\CM(m),\phi_\CM(m')),
\end{equation}
where $\ul\Hom_0(-,-)$ denotes inner Hom in $\CC_0$.

\medskip

Assume now that $\phi$ is conservative. It follows from \eqref{e:inner Hom compat} and \lemref{l:ULA inner Hom} that
if $\phi_\CM(m)\in \CM_0$ is ULA over $\CC_0$, then $m$ is ULA over over $\CC$.

\medskip

We apply this to $\phi$ being
$$\oblv:\Dmod(Y)\to \QCoh(Y).$$

Assumption (iii) says that $\oblv_Y(m)$ is compact. It is then automatically ULA over $\QCoh(Y)$
by \corref{c:rigid ULA}.

\end{proof}

The implication (ii) $\Rightarrow$ (iii) follows by
$$\ind \circ \oblv(m)\simeq \on{D}_Y\otimes m$$
from the following more general assertion:

\begin{prop} \label{p:ind crit comp}
Let $m\in \CM_Y$ be such that $\ind(m)\in \CM$ is compact. Then $m$ is compact.
\end{prop}

\begin{proof}

We need to show that the functor
$$m'\mapsto \CHom_{\CM_Y}(m,m')$$
preserves colimits. We will do so by expressing it via the functor
$$m'\mapsto \CHom_{\CM}(\ind(m),\ind(m'))\simeq \CHom_{\CM_Y}(m,\oblv\circ \ind(m')),$$
while the latter preserves colimits by assumption.

\medskip

With no restriction of generality, we can assume that $Y$ is affine.

\medskip

Consider the formal completion of $Y$ in $Y\times Y$
$$Y \overset{p_1}\leftarrow Y^\wedge \overset{p_2}\to Y.$$

Let $\CM_{Y^\wedge}$ be the value of $\CM$ over $Y^\wedge$. Set
$$\CM'_{Y^\wedge}:=\IndCoh(Y^\wedge)\underset{\QCoh(Y^\wedge)}\otimes \CM_{Y^\wedge}$$

The maps $p_1,p_2$ define the functors
$$p_i^*:\CM_Y \to \CM_{Y^\wedge}$$
and
$$(p_i^{\IndCoh})_*:\CM'_{Y^\wedge}\to \CM_Y.$$

\medskip

Note that the functor
$$m'\mapsto \oblv\circ \ind(m')$$
identifies with
$$m'\mapsto (p_1^{\IndCoh})_*(\on{D}_Y\otimes p_2^*(m')),$$
where we regard $\on{D}_Y$ as an object of $\IndCoh(Y^\wedge)$.

\medskip

From here it follows $\oblv\circ \ind(m')$ carries an action of the algebra $\CO_{Y\times Y}$ of functions on $Y\times Y$
(via its action on $\on{D}_Y$), and we have a functorial identification
$$m'\simeq \CHom_{\CO_{Y\times Y}}(\CO_Y,\oblv\circ \ind(m')).$$

From here, we obtain that $\CHom_{\CM_Y}(m,\oblv\circ \ind(m'))$ carries an action of $\CO_{Y\times Y}$ and
$$\CHom_{\CM_Y}(m,m')\simeq \CHom_{\CO_{Y\times Y}}\left(\CO_Y,\CHom_{\CM_Y}(m,\oblv\circ \ind(m'))\right).$$

\medskip

Now, the required assertion follows from the fact that the functor
$$\CHom_{\CO_{Y\times Y}}(\CO_Y,-):\CO_{Y\times Y}\mod\to \Vect$$
commutes with colimits, since $\CO_Y$ is a compact object of $\CO_{Y\times Y}$, the latter because $Y$ is smooth.

\end{proof}

This completes the proof of \thmref{t:ULA Dmod}.

\end{document}